\documentclass[a4paper, 
headsepline=off, DIV=12, titlepage=false, 11pt]{scrartcl}
\setkomafont{author}{\sffamily\scshape}

\title{
\semiLARGE
Dirichlet $\bm{L}$-series at $\bm{s = 0}$
and the scarcity 
of Euler systems}
\date{}
\author{Dominik Bullach \and David Burns \and Alexandre Daoud \and Soogil Seo}
 
 \usepackage[backend=bibtex, style=alphabetic, sorting=nyt, maxbibnames=9, maxalphanames=5, minalphanames=1
 ]{biblatex}
\renewbibmacro{in:}{%
  \ifentrytype{article}{}{\printtext{\bibstring{in}\intitlepunct}}}
  \DeclareFieldFormat{journaltitle}{#1\isdot}
  \DeclareFieldFormat[article, inproceedings, unpublished, phdthesis]{title}{\textit{#1}}

\input{preamble}

\begin{document}

\maketitle

\let\thefootnote\relax

\setcounter{tocdepth}{2}

\begin{abstract}
    We study Euler systems for $\mathbb{G}_m$ over a number field $k$. Motivated by a distribution-theoretic idea of Coleman, we formulate a conjecture regarding the existence of such systems that is elementary to state and yet strictly finer than Kato's equivariant Tamagawa number conjecture for Dirichlet $L$-series at $s=0$. To investigate the conjecture, we develop an abstract theory of `Euler limits' and, in particular, prove the existence of canonical `restriction' and `localisation' sequences in this theory. By using this approach we obtain a variety of new results, ranging from a proof, modulo standard $\mu$-vanishing  hypotheses, of our central conjecture in the case $k$ is $\QQ$ or imaginary quadratic to a proof of the `minus part' of Kato's conjecture in the case $k$ is totally real. 
\end{abstract}

\vspace{-0.4cm}
\tableofcontents

\section{Introduction}

The mysterious link between $L$-series and arithmetic, manifestations of which include the analytic class number formula and the Conjecture of Birch and Swinnerton-Dyer, is a key theme in modern arithmetic geometry. 
The theory of Euler systems has, since the 1980s, been crucial to investigations of this link, though finding concrete examples of such systems has proved to be a difficult problem. In the present article, we  consider in detail the apparent scarcity of Euler systems in the setting of the multiplicative group $\mathbb{G}_m$ over number fields, 
and find that this scarcity itself, when made precise, has important consequences regarding the formulation and study of special value conjectures.
\\ 
To help motivate our approach we recall that, in 1989, Coleman conjectured an intrinsically global, distribution-theoretic analogue of the fact that norm-compatible families of units in towers of local cyclotomic fields arise by evaluating a power series at roots of unity, as had been proved in \cite{coleman}.  Hitherto, however, a resolution of this conjecture has seemed out of reach, with comparatively little supporting evidence and no proof strategy apparent (see \cite{Seo1} or  \cite{BurnsSeo} for a discussion of the history). 

To generalise Coleman's idea, we note that, after suitable reinterpretation, his conjecture implies that every Euler system (of rank one) for $\mathbb{G}_m$ over $\Q$ should, modulo certain torsion considerations, arise as a product of Galois-conjugates of the system of cyclotomic units. We then further note that, for any number field $k$, the collection of Euler systems of any given rank for $\mathbb{G}_m$ over $k$ is a module over the algebra $\mathcal{R}_k \coloneqq \varprojlim_{E}\Z[\Gal(E/k)]$, where $E$ runs over all finite abelian extensions of $k$ (and the transition maps in the inverse limit are the  natural projection maps). Then, roughly speaking, our central conjecture will assert that every Euler system over $k$ that satisfies certain natural, and explicit, auxiliary conditions should be an $\mathcal{R}_k$-multiple of the `Rubin--Stark system' $c_k^\mathrm{RS}$ that is defined in \cite{Rub96}. 
This straightforward prediction (which, in the sequel, we refer to as the `Scarcity Conjecture') is stated precisely as Conjecture \ref{scarcity-conjecture} and lies at the centre of our approach.
\\
At this point, it is important to note $\mathcal{R}_k$ is neither Noetherian nor compact and hence that various standard algebraic techniques cannot be applied to the study of these questions. To overcome such difficulties, we develop an abstract theory of `Euler limits'
that simultaneously incorporates, amongst other things, the theories of inverse limits, Euler systems and Perrin-Riou functionals. This general theory is then  the main theoretical advance that we make in this article and can be expected to have applications beyond those that we discuss here. In particular, the theory concerns systems that are defined integrally (that is, over $\Z$) rather than either $p$-adically (for a fixed prime $p$)  or adelically, and hence crucially incorporates techniques from both classical, and (what one might call) `horizontal', Iwasawa theory. For instance, these differing techniques can be combined with a detailed analysis of the classical embedding problem for number fields to prove the existence of canonical $p$-adic `restriction' and global `localisation' exact sequences for modules arising from Euler limits, and also to carefully analyse completion functors in this setting. 

These general results can then be used to reduce the classification of Euler systems to a family of $p$-primary problems (for all $p$) and thereby to prove, under natural hypotheses, that the Scarcity Conjecture is implied by the sort of divisibilities in `Iwasawa Main Conjecture'-type statements that Euler systems are already expected to satisfy (and can often be verified using existing techniques). This approach therefore provides both a conceptual underpinning of the Scarcity Conjecture and also an effective means of obtaining strong supporting evidence, such as the following result concerning Euler systems of rank one. 
\begin{thmintro}\label{rank one intro} After inverting $2$ and all primes that ramify in $k$, and assuming standard $\mu$-vanishing hypotheses, the Scarcity Conjecture is valid if $k$ is either $\Q$ or imaginary quadratic.
\end{thmintro}

This result is later stated precisely as  Theorem~\ref{thm-scarcity-rank-one}, whilst concrete evidence in support of the Scarcity Conjecture for Euler systems of rank greater than one can be found, for example, in Theorem~\ref{bdss-result}. 

With some additional effort, the methods developed here will allow one to prove the result of Theorem~\ref{rank one intro} without inverting any rational primes. We also stress that, whilst its proof relies on key properties of cyclotomic and elliptic units, for the reasons given above the content of Theorem~\ref{rank one intro} goes far beyond what has previously been shown in these classical settings. For example, the argument used to prove it also leads to a proof of Coleman's original distribution-theoretic conjecture and, beyond that, to an explicit characterisation of all of the distributions that are considered by Coleman (many of which are not cyclotomic in nature). For brevity, however, the treatment of these important issues concerning Theorem~\ref{rank one intro} is deferred to the supplementary article \cite{scarcity2}.

In the present article, we instead choose to focus  on the link between our study of Euler systems and the `equivariant' strengthening of the Tamagawa Number Conjecture of Bloch and Kato \cite{bloch-kato} that is formulated by Kato in \cite{kato93a, kato93b} (where the conjecture is referred to as the `generalised Iwasawa Main Conjecture’). More precisely, our approach is related to Kato's conjecture in the case of leading terms at zero of  Dirichlet $L$-series and, for ease, we refer to this as  `eTNC($\mathbb{G}_m$)'. We recall, in particular, that eTNC($\mathbb{G}_m$) is known to imply a wide range of previously formulated refinements of Stark's Conjecture (cf.\@ Remark~\ref{etnc cons}). \\
To explain this link, we recall that, aside from the analytically-defined system $c_k^\mathrm{RS}$, there also exists (unconditionally) a family of Galois-cohomological Euler systems $c_k^\mathrm{coh}$ for $\mathbb{G}_m$ over $k$ (as constructed by Sano and the second author in  \cite{sbA}). Then, roughly speaking, a special case of the  Scarcity Conjecture implies that  $c_k^\mathrm{coh} = \lambda\cdot c_k^\mathrm{RS}$ for an element $\lambda$ of $\mathcal{R}_k$, the analytic class number formula implies that any such element $\lambda$ belongs to $\mathcal{R}_k^\times$, and the resulting relation between $c_k^\mathrm{RS}$ and $c_k^\mathrm{coh}$ implies eTNC($\mathbb{G}_m$) over all abelian extensions of $k$. In this way, then, eTNC($\mathbb{G}_m$) is seen to be a direct consequence of the analytic class number formula and the scarcity of Euler systems, thereby providing a straightforward philosophy to underpin eTNC($\mathbb{G}_m$).
To the best of the authors' knowledge, excluding the basic analogy to Deligne's proof of the Weil Conjectures for varieties over function fields, no heuristic of any sort has previously been available for Kato's conjecture. \\ 
In fact, it turns out that the Scarcity Conjecture is strictly finer than eTNC($\mathbb{G}_m$) and also encodes precise information about the structure of the Selmer module of $\mathbb{G}_m$ over the abelian closure of $k$ (for details see \cite{scarcity2}). Fortunately, however, the theory of Euler limits is flexible enough to provide an effective strategy for proving eTNC($\mathbb{G}_m$) without requiring one to first prove the Scarcity Conjecture. The resulting approach then has a significant advantage over previous strategies in this context since it avoids delicate issues relating to Iwasawa-theoretic descent that have been key obstacles to progress on eTNC($\mathbb{G}_m$). 

The point here is that, for any given prime $p$, the methods of $p$-adic Iwasawa theory involve passing to the limit over $\Z_p$-power extensions of $k$ and therefore primarily concern extensions that ramify at $p$-adic places. Hence, when considering extensions of $k$ that are unramified at any such places, this approach can introduce undesired Euler factors that are not invertible (in the presence of `trivial zeroes') and so cannot easily be removed. Indeed, previous strategies to deal with this issue have assumed (at least one of) the conjectures of Gross--Kuz'min, of Leopoldt and of Mazur--Rubin and Sano, the known validity  of each of which remains restricted to a small number of special cases. However, since Euler systems are essentially characterised by their values on extensions that ramify at all $p$-adic places, our theory is able to completely avoid any reliance on the first two of these conjectures and to restrict dependence on the third to the (much easier) setting of tamely ramified cyclic extensions. 

In this way we can therefore directly leverage existing techniques to obtain concrete new evidence for Kato's conjecture such as in the following result. Before stating this result, we recall that eTNC($\mathbb{G}_m$) predicts, for each finite abelian extension of number fields $K/k$, an equality of graded invertible $\ZZ[\Gal(K/k)]$-modules (cf.\@ Remark~\ref{etnc cons}\,(a)). 
\newpage

\begin{thmintro} \label{intro etnc} Fix a finite abelian extension of number fields $K/k$ and set $G \coloneqq \Gal(K/k)$. Then 
$\mathrm{eTNC}(\mathbb{G}_m)$ for $K/k$ is valid in both of the following cases: 
\begin{liste}
\item $k$ is totally real, $K$ is CM and one extends scalars from $\Z[G]$ to $\Z[1/2][G](1-\tau)$, where $\tau$ denotes the element of $G$ corresponding with complex conjugation;
\item $k$ is imaginary quadratic and such that for every finite abelian extension (and every odd prime $p$) a certain $p$-adic Iwasawa $\mu$-invariant vanishes, and one extends scalars from $\Z[G]$ to $\Z[1/2][G]$.  
\end{liste}
\end{thmintro}

Claim (a) of this result verifies what is often referred to as the `minus part' eTNC$(\mathbb{G}_m)^-$ of eTNC$(\mathbb{G}_m)$, and our approach  reduces its proof to a statement in Iwasawa theory that is easily seen to follow from the `Strong Brumer--Stark Conjecture' recently proved (outside $2$) by Dasgupta and Kakde in \cite{DasguptaKakde}. In particular, in this setting the horizontal Iwasawa theory of Euler limits allows us, via Theorem~\ref{p adic restriction sequence}, to avoid technical hypotheses used in other attempts to derive 
$\mathrm{eTNC}(\mathbb{G}_m)^-$ from the seminal results of \cite{DasguptaKakde}, either directly as in Nickel \cite{nickel2021strong} (though the deduction  in loc.\@ cit.\@ of the $p$-part of eTNC$(\mathbb{G}_m)^-$ for extensions that are tamely ramified at $p$ does play a role in our argument) or by strengthening relevant aspects of the arguments in \cite{DasguptaKakde} as in Atsuta and Kataoka \cite{atsuta-kataoka}, whilst at the same time avoiding difficult hypotheses related to the Gross--Kuz'min Conjecture that arose in earlier attempts to use Iwasawa theory in this context. The result of claim (a) also itself has interesting consequences: for example, it implies the validity, after inverting $2$, of the `integral Gross--Stark Conjecture' from \cite{Gross88} (cf.\@ Remark~\ref{etnc cons}\,(b)) and this fact has recently been used by Honnor to remove, up to $2$-power torsion, the `root of unity ambiguity' in the $p$-adic analytic formula for Brumer--Stark units that is proved by Dasgupta and Kakde in \cite[Th.\@ 1.6]{DasguptaKakde2} (for details see  \cite{Honnor23}). 
In addition, claim (a) also combines with work of Atsuta and Kataoka \cite{atsuta-kataoka-Fitting} to imply an explicit description for the Fitting ideals of the minus parts of class groups in the relevant cases.\\
A precise version of claim (b) of Theorem~\ref{intro etnc} is stated as Theorem~\ref{precise intro B(ii)} and strongly improves upon previous results towards eTNC($\mathbb{G}_m$) over imaginary quadratic fields. In addition, in recent work of Hofer and the first author \cite{BullachHofer}, it is shown that the hypothesis on $\mu$-invariants that occurs in this result, and which Iwasawa has conjectured to always hold, can sometimes be avoided. However, this observation relies on techniques that seem to be restricted to the setting of imaginary quadratic fields, whilst the approach used here appears, in principle, to be completely general and thereby applicable to many different contexts. 

The arguments proving Theorem~\ref{intro etnc} also lead us, at the same time, to a new proof of eTNC$(\mathbb{G}_m)$ for $K/k$ in the case that $k = \Q$ and one extends scalars from $\Z[G]$ to $\mathbb{Z}[1/2][G]$. This result was first proved as the main result of Greither and the second author in \cite{BurnsGreither}, but the proof obtained here is much simpler since it avoids the extensive, and delicate, descent calculations in Iwasawa theory that are key to the argument of loc.\@ cit. \\
Going beyond the special cases considered in Theorems \ref{rank one intro} and \ref{intro etnc}, our approach provides a concrete strategy for deriving the validity of the Scarcity Conjecture over an arbitrary number field $k$, and hence also of eTNC($\mathbb{G}_m$) and the numerous conjectures it implies, from a single, explicit, `integrality' prediction of Rubin \cite{Rub96} concerning the system $c_k^\mathrm{RS}$ (that is usually referred to as the `Rubin--Stark Conjecture'). In addition, the theory of higher-rank Euler, Kolyvagin and Stark systems developed by Sakamoto, Sano and the second author in \cite{bss} is an effective method for realising this strategy and we will discuss this aspect of the theory  elsewhere. It should, however, be noted that the Rubin--Stark Conjecture itself is of a different nature and that its verification may well require both an explicit construction of $c_k^\mathrm{RS}$ by geometric means and the proof of an appropriate 
explicit reciprocity law, as has already been achieved in the case $k=\Q$  by Urban in \cite{Urban}.\\
In brief, the main contents of this article are as follows. In \S\,\ref{esGm} we define, for each non-negative rank, several special families of Euler systems over $k$, prove some important preliminary results and then formulate the Scarcity Conjecture. In \S\,\ref{consSC} we establish the precise link between the Scarcity Conjecture and eTNC($\mathbb{G}_m$).
In \S\,\ref{elimits} we introduce an abstract theory of `Euler limits' and prove key results in this theory. In \S\,\ref{euler systems and limits} we prove a `Uniformisation Theorem' for Euler systems over $k$ that are valued in a field, combine this with general results concerning Euler limits to derive explicit criteria for the validity of the Scarcity Conjecture and of eTNC($\mathbb{G}_m$) and finally use higher-rank Kolyvagin systems to provide evidence for these criteria. In \S\,\ref{iwasawa theory section} we combine criteria established in \S\,\ref{euler systems and limits} with several existing results to obtain further   evidence in support of the Scarcity Conjecture and also prove precise versions of Theorems \ref{rank one intro} and \ref{intro etnc}. \\
For the reader's convenience, we end this section by specifying some general notation that will be used throughout the article. \\
We write $\N_0$ for $\N \cup \{0\}$ and $\mathbb{P}$ for the set of prime numbers. We also write $\overline{\Q}$ for the algebraic closure of $\Q$ in $\C$ and refer to finite extensions of $\Q$ in $\overline{\Q}$ as `number fields'. We write $\cl_k$ for the ideal class group of a number field $k$. \\
For a commutative ring $R$, an $R$-module $M$ and $r \in \N_0$, we write $\exprod_R^r M$ for the $r^{th}$-exterior power over $R$ of $M$.
Given a homomorphism $f\: M \to N$ of $R$-modules we write $f^{(r)}$ for the induced map $\exprod_R^r M \to \exprod_R^r N$ and, except in cases of ambiguity, we  abuse notation by writing $f$ in place of $f^{(r)}$. We write $M^\ast \coloneqq \Hom_R (M, R)$ for the $R$-linear dual of $M$ and, if $\p$ denotes a prime ideal of $R$, then we let $M_\p$ be the localisation of $M$ at $\p$. \\ 
For an abelian group $A$, we set $A [n] \coloneqq \{ a \in A \mid n \cdot a = 0 \}$ for $n\in \mathbb{N}$ and $A_\tor \coloneqq \bigcup_{n \in \N} A [n]$. We write $A_\tf$ for the torsion-free quotient $A/A_\tor$ of $A$, $A^\vee$ for its Pontryagin dual  
$\Hom_\Z (A, \Q/\Z)$,
$A_\mathrm{div}$ for ${\bigcap}_{n \in \N}(nA)$ and $A_{p\cdot\mathrm{div}}$ for each $p\in \mathbb{P}$ for ${\bigcap}_{n \in \N}(p^nA)$. 
If there is no danger of confusion, we 
abbreviate the functor $(-) \otimes_\Z A$ to $(-)\!\cdot\! A$ (or, occasionally, just $(-) A$) and, if $A$ is  a $\Z_p$-module for some $p\in \mathbb{P}$, we similarly abbreviate $(-) \otimes_{\Z_p} A$. The duals $M^\ast$ and $M^\vee$ of a $\Z [A]$-module $M$ are in general endowed with contragredient $A$-action. If $A$ is finite, we write $A^\circ$ for its character group 
$\Hom_\Z (A, \C^\times) = \Hom_\Z(A,\overline{\Q}^\times)$. For a subset $\Upsilon$ of $A^\circ$ that is stable under the natural action of $\Gal(\overline{\Q}/\Q)$, we write $[\Upsilon]_\Q$ for the set of orbits of this action and also, interchangeably,  for a set of representatives of these orbits. We  write $\mathbf{1}_A$ for the trivial character of $A$ and, for $\chi \in A^\circ$, use the primitive idempotent $e_\chi \coloneqq |A|^{-1}\sum_{a \in A} \chi (a) \sigma^{-1}$ of $\overline{\Q}[A]$. We note, in particular, that $|A|\cdot e_{\mathbf{1}_A}$ coincides with the norm element $\NN_A \coloneqq  \sum_{a \in A} a \in \Z [A]$ of $A$.  \\
For any finite subset $\mathcal{S}$ of $\mathbb{P}$, we often use the subring of $\Q$ defined by 
\[ \Z_\mathcal{S} \coloneqq \Z \bigl[1/p \mid p \in \mathcal{S}\bigr] \quad (\text{so that } \ZZ_\emptyset = \ZZ). \]
Finally, we shall write $\mathcal{P} (X)$ for the set of subsets of a set $X$.

\textbf{\sffamily Acknowledgements}\, We are very grateful to Takenori Kataoka for his careful reading of an earlier version of this article and, in particular, for identifying an error in one of our arguments. We are also most grateful to an anonymous referee for their very careful reading of the article and many helpful comments and suggestions. The second author is grateful to  Mahesh Kakde, Masato Kurihara and Takamichi Sano for inspiring discussions on these, and related, subjects. He and the first author are also grateful to Matthew Honnor, S\"oren Kleine, Andreas Nickel, and Daniel Mac{\'i}as Castillo for helpful discussions.

The first and third authors wish to acknowledge the financial support of the Engineering and Physical Sciences Research Council [EP/L015234/1, EP/N509498/1, EP/W522429/1], the EPSRC Centre for Doctoral Training
in Geometry and Number Theory (The London School of Geometry and Number Theory),
University College London and King’s College London. In addition, the fourth author was supported by
 a National Research Foundation of Korea (NRF) grant (RS-2025-16064775) funded by the government of Korea (MSIT).

Finally we would like to point out that this article 
incorporates various ideas from an earlier unpublished version of an article of Hofer and the first author \cite{BullachHofer} and we are very grateful to Martin Hofer for his permission to include them. 
\smallskip\\
\textbf{\sffamily Afternote}\, After this article was submitted for publication, the result of Theorem~\ref{intro etnc}\,(a) was extended by Dasgupta, Kakde, and Silliman \cite{DasguptaKakdeSilliman} to also include the `2-part'. As a consequence, the validity of $\mathrm{eTNC} (\mathbb{G}_m)^-$ is now known to be valid for finite abelian CM extensions of totally real fields without the need to invert $2$.

\section{Euler systems for $\mathbb{G}_m$}\label{esGm}

For any number field $E$ we write $S_\infty(E)$ and $S_\fin(E)$ for the sets of archimedean and finite (non-archimedean) places of $E$ respectively, and $S_p(E)$ for each $p \in \mathbb{P}$ for the subset of $S_\fin(E)$ comprising all $p$-adic places. We also write $d_E$ for the absolute discriminant of $E$ and, for an  extension $F$ of $E$ we write $S_\ram(F /E)$ for set of places of $E$ that ramify in $F$.

If $S$ is a set of places of $E$, we write $S_F$ for the set of places of $F$ that lie above those in $S$. We will however omit the explicit reference to $F$ if it is clear from  context. For example, $\bigO_{F, S}$ will  denote the ring of $(S_F\cap S_\fin(F))$-integers of $F$.
 We also write  $Y_{F, S}$ for the free abelian group on $S_F$ and  $X_{F,S}$ for the subgroup of $Y_{F,S}$ comprising elements whose coefficients sum to zero. 
 
For each $v \in S_\fin(E)$ we write $\ord_v \: E^\times \to \Z$ for the normalised valuation at $v$ and $\NN v$ for its absolute norm $ | \bigO_E/\p_v |$, where $\p_v$ is the ideal of $\bigO_E$ corresponding to $v$. 

If $F / E$ is a Galois extension and $v$ is a finite place of $E$ that is unramified in $F$, then we write $\Frob_v$ for the arithmetic Frobenius of $v$ in $\gal{F}{E}$. We recall that an archimedean place of $E$ is said to `ramify' in $F$ if its decomposition group is non-trivial (and hence has order two); for each archimedean place $v$ of $E$ that is unramified (that is, does not ramify) in $F$, we will therefore write $\mathrm{Frob}_v$ for the trivial element of $\gal{F}{E}$. 

\subsection{Multi-rank Euler systems}\label{ES-Gm-section}

In the sequel we  fix a number field $k$ and for any Galois extension $K$ of $k$ in $\overline{\Q}$ set
\[ \cG_K \coloneqq \gal{K}{k} \,\,\,\text{  and  } \,\,\, S (K) \coloneqq S_\ram (K / k). 
\]
If $K/k$ is abelian, then for each place $v$ of $k$ we write $\cG_{K, v}$ for the decomposition subgroup in $\cG_K$ of (any place of $K$ above) $v$. We also set $G_k \coloneqq \cG_{\overline{\Q}}$ and write $k^{\mathrm{ab}}$ for the maximal abelian extension of $k$ in $\overline{\Q}$ so that $G_k^{\mathrm{ab}}$ identifies with the abelianisation $G_k^\mathrm{ab}$ of $G_k$. \\

\subsubsection{Abelian extensions} Our theory concerns the collection of fields
\[ \Omega_k\coloneqq \{ K:\, k\subsetneq K\subset k^{\mathrm{ab}}, \, [K:k] \,\,\text{is finite}, \, S(K)\cap S_\fin(k) \not= \emptyset\}. \]
Ray classfields provide a cofinal (with respect to inclusion) subset of $\Omega_k$. We recall, in particular, that a `modulus' is a finite formal product $\m$ of places of $k$ and the associated ray classfield $k(\m)$ is an extension of $k$ with the following properties: $[k(\m):k]$ is finite, $k(\m)\subset k^{\mathrm{ab}}$,   $S(k(\m))$ is contained in the set of places that divide $\m$ and, for each $K\in \Omega_k$, there exists a modulus $\m$ divisible  only by places in $S(K)$ such that $K\subseteq k(\m)$. The ray classfield  $k(1)$ with modulus the empty product of places is therefore equal to the Hilbert classfield of $k$ (and so $k(1)\notin \Omega_k$).

For $K \in \Omega_k$, we say that the set $S(K)$ is `minimal' if it is a minimal (with respect to inclusion) element of $\{ S (E) \mid E \in \Omega_k \}$, and we then define subsets of $\Omega_k$ by setting 
\begin{align*}
 \Omega_k^\mathrm{min} &\coloneqq \{ K \in \Omega_k : S(K) \text{ is minimal} \},\\
 \Omega_k^\mathrm{exc} &\coloneqq \{ K \in \Omega_k^\mathrm{min} : | S (K) | > 1 \}.
\end{align*}
Fields in $\Omega_k^\mathrm{exc}$ constitute `exceptional cases' since for many $k$ they cannot exist and, in all cases, they are very rare -- for example, for any $K \in \Omega_k^{\mathrm{exc}}$ it can be shown that $S(K)$ contains a place $\p$ for which  $k(\p^n) = k(1)$ for every $n\in \N$. 

For a subfield $E$ of $\overline{\Q}$ and $M\in \N$ we define groups of roots of unity by setting 
\[ \mu_E \coloneqq E^\times_\tor \,\,\text{ and }\,\, \mu_M \coloneqq \mu_{\overline{\Q}}[M] = \{ x\in \overline{\Q}^\times: x^M = 1\}.\]

The following result records a useful observation of Rubin concerning ray classfields.

\begin{lem}\label{rubin lemma} 
    Fix $M\in \N$ and $\q\in S_\fin(k)$ with $\q\nmid M$. Then $k (\q)$ has the following properties.
\begin{itemize}
    \item[(i)] $S(k(\q)) \subseteq \{\q\}$ and so $S(k(\q)) \cap S_\infty(k) = \emptyset$.
    \item[(ii)] The inertia degree of $\q$ in $k(\q)$ is equal to $[k(\q):k(1)]$ and is divisible by $M$ if and only if $\q$ splits completely in $k (\mu_M, (\cO_k^\times)^{1 / M})$. 
\end{itemize}
\end{lem}

\begin{proof}
Claim (i) and the first assertion in claim (ii) are clear. The second assertion in claim (ii) is proved by Rubin in \cite[Lem.\@ 4.1.2]{Rubin-euler} but, for completeness, we give the argument.

Class field theory identifies the inertia subgroup $\gal{k (\q)}{k (1)}$ of $\q$ in $\cG_{k(\q)}$ with the cokernel of the natural map $\theta_\q \: \cO_k^\times \to (\cO_k / \q)^\times$. Hence, if  
 $M$ divides $[k (\q) : k (1)]$, then  $M$ divides $|(\cO_k / \q)^\times| = \NN \q - 1$ and so $\q$ splits completely in $k (\mu_M)$. \\
In addition, any such $\q$ splits completely in $k (\mu_M, (\cO_k^\times)^{1 / M})$ if and only if every element of $\cO_k^\times$ is an $M$-th power in the completion $k_\q$ of $k$ at $\q$. Then, since $\q$ does not divide $M$, Hensel's Lemma implies that this last condition is satisfied if and only if every element of $\im(\theta_\q)$ is an $M$-th power in the cyclic group $(\cO_k / \q)^\times$, or equivalently the order of $\cok(\theta_\q)$ is divisible by $M$. This proves claim (ii). 
\end{proof}

\subsubsection{Euler systems} For each extension $K$ of $k$ we set 
\[ V_K \coloneqq S_\infty (k)\setminus S (K) = \{v \in S_\infty(k): v \,\,\text{ splits completely in }\, K\}\]
and also  
 \begin{equation}\label{r_K def} r_K \coloneqq \begin{cases} | V_K|, &\text{ if $K \neq k$},\\
 |S_\infty ( k)| - 1, &\text{ if $K = k$.}\end{cases}\end{equation}

 \begin{definition}\label{rank def} A \emph{rank function} (or simply a \emph{rank}) for $k$ is a function 
 \[ \bm{r}\:\Omega_k \to \N_0\] 
 that factors through the function
  $\Omega_k  \to \mathcal{P} (S_\infty (k))$ sending each $K$ to $V_K$ (so that $\bm{r}(K) = \bm{r}(K')$ if $V_{K} = V_{K'}$). If $\bm{s}$ is another rank function for $k$, we  write $\bm{r} > \bm{s}$ to denote that $\bm{r} (E) > \bm{s} (E)$ for every $E \in \Omega_k$ (with $\bm{r} < \bm{s}$, $\bm{r} \geq \bm{s}$ and $\bm{r} \leq \bm{s}$ being defined similarly). We identify each $r \in \N_0$ with the (constant) rank function $K \mapsto r$. The function $K \mapsto r_K$ is referred to as the \emph{maximal rank function} for $k$ and denoted  $\bm{r}_\mathrm{max}$ (the justification for this terminology becomes clear in Theorem~\ref{coleman-over-the-reals} below). 
 \end{definition}

Finally, for fields 
$E, F \in \Omega_k$ with $E \subseteq F$, any $s \in \N_0$ and any field $\cQ$, we write  %
\[ \NN^{s}_{F / E}: \cQ\! \cdot\! \exprod^{s}_{\Z [\cG_F]} \bigO_{F, S (F)}^\times \to \cQ\!\cdot\!  \exprod^{s}_{\Z [\cG_E]} \bigO_{E, S(F)}^\times\]
for the homomorphism of $\cQ[\cG_F]$-modules induced by the field-theoretic norm $\NN_{F / E} \: F^\times \to E^\times$. 

\begin{definition} \label{euler-systems-definition-1} Let $\bm{r}$ be a rank function for $k$ and $\cQ$ a field. A $\cQ$-rational \emph{Euler system} for $k$ of rank $\bm{r}$ is a collection of elements
\[
(c_E)_E \in {\prod}_{E \in \Omega_k} \bigl(\cQ \cdot \exprod^{\bm{r}(E)}_{\Z [\cG_E]} \cO_{E, S(E)}^\times\bigr)
\]
that satisfies both of the following conditions: 
\begin{itemize}
\item[(i)] for all  
$E$ and $F$ in $\Omega_k$ with $E \subseteq F$, one has an equality 
\begin{equation}\label{es dist rel}
\NN^{\bm{r} (F)}_{F / E} (c_F) = \Big ( {\prod}_{v \in S (F) \setminus S (E)} (1 - \Frob_v^{-1} ) \Big) c_E
\end{equation}
in the graded module $\bigoplus_{i = 0}^\infty \big ( \cQ\cdot\exprod^i_{\Z [\cG_E]} \bigO_{E, S(F)}^\times \big)$;  
\item[(ii)] for every $E$ in $\Omega_k^\mathrm{exc}$, one has $\NN^{\bm{r} (E)}_{E / k} (c_E) = 0$.
\end{itemize}
We write $\ES^{\bm{r}}_k (\cQ)$ for the $\cQ \llbracket G_k^{\mathrm{ab}}\rrbracket$-module of $\cQ$-rational Euler systems of rank $\bm{r}$ for $k$ and then set $\ES_k (\cQ) \coloneqq \ES^{\bm{r}_\mathrm{max}}_k (\cQ)$. 
\end{definition}

\begin{rk}\label{arch euler factors} We clarify the conditions that occur in the above definition.
\begin{liste}  
\item Condition (i) is the natural `Euler system distribution relation' in a higher rank setting. In addition, if any place $v$ in $S(F) \setminus S(E)$ is archimedean, then $v$ splits completely in $E$ so that the corresponding Euler factor $1 - \Frob_v^{-1}$  in the formula (\ref{es dist rel}) is $0$ and the condition simplifies to give  $\NN^{\bm{r} (F)}_{F / E} (c_F) = 0$.
\item Condition (ii) is a minor technical restriction as it only concerns exceptional fields (and so is satisfied vacuously if $\Omega_k^\mathrm{exc} = \emptyset$). In addition, it is automatically satisfied by the Euler systems that will be the main object of our study (see Remark~\ref{symmetric euler systems rk}\,(b) and Conjecture~\ref{scarcity-conjecture}). 
\end{liste}
\end{rk}

It will often be useful for us to consider the projection of an Euler system to subsets of $\Omega_k$.
More precisely, for any $\cX \subseteq \Omega_k $, we write $\ES_k^{\bm{r}, \cX} (\cQ)$ for the image of $\ES_k^{\bm{r}} (\cQ)$ under the natural projection map
\[
\varrho^\cX \: {\prod}_{E \in \Omega_k} \bigl(\cQ \cdot \exprod^{\bm{r} (E)}_{\Z [\cG_E]} \cO_{E, S(E)}^\times\bigr)
\to {\prod}_{E \in \cX}\bigl(\cQ \cdot \exprod^{\bm{r} (E)}_{\Z [\cG_E]} \cO_{E, S(E)}^\times\bigr), \quad (m_E)_{E\in \Omega_k} \mapsto (m_E)_{E \in \cX}.
\]

\begin{rk}\label{omega def} For any $\mathcal{V}\subseteq \mathcal{P}(S_\infty (k))$, we set 
\[\Omega_k^\mathcal{V}\coloneqq \{ K \in \Omega_k: V_K \in \mathcal{V}\}.\]
If $\mathcal{V} = \{ V \}$ for some $V \subseteq S_\infty (k)$, then we abbreviate $\Omega_k^\mathcal{V}$ to $\Omega_k^V$ and, for the set  
$\cX = \Omega_k^{V}$, we shorten all adornments $\cX$ (as, for example, for the map $\varrho^\cX$ defined above) to $V$. \end{rk}

Finally, to record an important consequence of the Euler system relations we require further notation. For  $K \in \Omega_k$, we abbreviate the trivial character $\mathbf{1}_{\cG_K}$ of $\cG_K$ to $\mathbf{1}_K$ and, for any finite set of places $\Pi$ of $k$ and each $\chi \in \cG_K^\circ$, we set 
\begin{equation} \label{order-of-vanishing-formula}
r_\Pi (\chi) \coloneqq \dim_\C (e_\chi (\C \cdot\cO_{K, \Pi}^\times))
= \begin{cases}
| \{ v \in (\Pi \cup S_\infty (k)) \mid \chi (\cG_{K, v}) = 1 \} | & \text{ if } \chi \neq \mathbf{1}_K, \\
|\Pi \cup S_\infty (k)| - 1 & \text{ if } \chi = \mathbf{1}_K.
\end{cases}
\end{equation}
We then define a subset  of $\cG_K^\circ$ by setting %
\[ \Upsilon_K \coloneqq \{ \chi \in \cG_K^\circ: \, r_{S(K)} (\chi) = r_K\}.\] 
We note  $\Upsilon_K$ is a union of orbits of the action of $G_{\Q}$ on $\cG^\circ_K$, and hence that the idempotent 
\begin{equation}\label{eK def} e_K \coloneqq {\sum}_{\chi \in \Upsilon_K}e_\chi\in \overline{\Q}[\cG_K]\end{equation}
belongs to $\Q [\cG_K]$.  

\begin{lem} \label{faithful-lem}
For $\cQ\subseteq \C$, $c = (c_E)_E\in \ES^{\bm{r}}_k (\cQ)$ and $E \in \Omega_k$ the following claims are valid.
\begin{liste}
    \item $(1 - e_E) \cdot c_E =0$.
    \item If $\bm{r} (E) > r_E$, then $c_E = 0$. 
\end{liste}
\end{lem}

\begin{proof}
Let $\chi \in \cG_E^\circ$ and set 
$E_\chi\coloneqq  E^{\ker(\chi)}$. We use the injective map
\[
\nu_{E / E_\chi} \: \cQ \exprod^{\bm{r} (E)}_{\Z [\cG_{E_\chi}]} \bigO_{E_\chi, S(E)}^\times 
\to \cQ \exprod^{\bm{r} (E)}_{\Z [\cG_{E}]} \bigO_{E, S(E)}^\times,
\quad a \mapsto 
\begin{cases}
[E : E_\chi]^{1 - \bm{r} (E)} a & \text{ if } \bm{r} (E) \geq 1, \\
\NN_{\gal{E}{E_\chi}} a & \text{ if } \bm{r}(E) = 0
\end{cases}
\]
and also recall  $\NN^0_{E / E_\chi}$ is the natural restriction map $\cQ [\cG_E] \to \cQ [\cG_{E_\chi}]$. Then one has 
\[
[E : E_\chi] e_\chi  \cdot c_E  = e_\chi \cdot \NN_{\gal{E}{E_\chi}} \cdot c_E = e_\chi (\nu_{E / E_\chi} \circ \NN_{E / E_\chi}^{\bm{r} (E)})( c_E )
\]
and so, to prove claim (a), it suffices to show that $e_\chi \NN_{E / E_\chi}^{\bm{r} (E)}( c_E ) = 0$ for all $\chi \in \cG_E^\circ\setminus \Upsilon_E$. \\
To do this, we first consider the case that $\chi \neq \mathbf{1}_E$. 
 We now fix $l\in \mathbb{P}\setminus \{2\}$ that is large enough to ensure $E \cap k (\mu_l, (\cO_k^\times)^{1 / l}) = k$. By Cebotarev's Density Theorem we may then choose a $\p \in S_\fin(k)\setminus  S(E)$ that has full decomposition group in $\cG_{E_\chi}$ and also splits completely in $k(\mu_l, (\cO_k^\times)^{1 / l})$. By Lemma~\ref{rubin lemma} the ray class field $F \coloneqq k(\p)$ is then a non-trivial extension of $k(1)$. It follows that $F$, and hence also $E_\chi F$, is a ramified extension of $k$, which implies that the value $c_{E_\chi F}$ is well-defined. We also note that $\bm{r} (E) = \bm{r} (EF)$. The Euler systems distribution relations now imply that
\begin{align*}
(1 - \Frob_\p^{-1}) \cdot \NN^{\bm{r}(E)}_{E / E_\chi} (c_E) & = \NN^{\bm{r}(E)}_{EF / E_\chi} (c_{EF})
= \Big( \prod_{v \in S(E) \setminus S(E_\chi)} (1 - \Frob_v^{-1}) \Big) \cdot \NN^{\bm{r} (E)}_{E_\chi F / E_\chi} (c_{E_\chi F}).
\end{align*}
If $\chi \not \in \Upsilon_{E}$, then there exists a place $v \in S(E) \setminus S (E_\chi)$ such that $\chi (\Frob^{-1}_v) = 1$, hence $e_\chi$ annihilates the above Euler product. On the other hand, $1 - \chi (\Frob_\p^{-1}) \neq 0$ because $\chi \neq \bm{1}_E$ and $\Frob_\p$ generates $\cG_{E_\chi}$, so we conclude that $e_\chi \NN^{\bm{r}(E)}_{E / E_\chi} (c_E) = 0$, as required to prove claim (a) for non-trivial characters. \\
Turning now to the case $\chi =\mathbf{1}_E$,
we note that if $E \in \Omega_k^\mathrm{exc}$, then $\NN^{\bm{r} (E)}_{E / k} (c_E) = 0$ by Definition~\ref{euler-systems-definition-1}(ii), and so   $e_{\bm{1}_E} c_E = 0$. In the rest of this argument, we therefore assume  $E \not \in \Omega_k^\mathrm{exc}$. Observe that $\mathbf{1}_E\in \Upsilon_E$ if and only if $| S(E) \cup S_\infty (k) |= r_E + 1$ or, equivalently (as $r_E \leq | S_\infty (k)|$), $|S(E) | = 1$.
If this condition is not satisfied, then $E \not \in \Omega_k^\mathrm{min}$ since $E \not \in \Omega_k^\mathrm{exc}$, and so there exists $E' \in \Omega_k$ with $S(E') \subsetneq S ( E)$. Then, if $\m$ is a large-enough power of the conductor of $E$ to ensure $E'\subseteq k (\m)$, the distribution relation (\ref{es dist rel})  implies 
\[
\NN^{\bm{r} ( E)}_{E / k} (c_E) = 
\NN^{\bm{r}(E)}_{k(\m) / k} (c_{k (\m)}) = \NN^{\bm{r}(E)}_{E' / k} \big ( \bigl({\prod}_{\p \in S(E) \setminus S(E')} (1 - \Frob_\p^{-1})\bigr)( c_{E'}) \big) = 0,
\]
and hence that $e_{\bm{1}_E} c_E = 0$, as required to complete the proof of claim (a).\\
To prove claim (b), we assume  $\bm{r} (E) > r_E$. Then (\ref{order-of-vanishing-formula}) implies that the space $\exprod^{\bm{r} (E)}_{e_\chi \C [\cG_E]} e_\chi \C \cO^\times_{E, S(E)}$ vanishes, and so $e_\chi c_E = 0$, for all $\chi\in \Upsilon_E$. It follows that $e_Ec_E = 0$ and hence that  claim (a) implies $c_E = e_Ec_E  + (1-e_E)c_E = 0$, as required.
\end{proof}

\subsection{Special families}

For the purpose of arithmetic applications an Euler system  must be `integral' in a suitable sense. However, defining a precise notion of integrality for  higher-rank systems is a delicate task. In this subsection we use exterior biduals and $T$-modification, pioneered by Rubin \cite{Rub96} and Gross \cite{Gross88}, respectively, to address this question. We then introduce certain special families of such systems that will play an important role in our theory. 

If $R$ is a commutative Noetherian ring, then for any $R$-module $M$ we write $M^\ast = \Hom_R (M, R)$ for its dual. We recall that, for each $s \in \N_0$, the $s$-th \textit{exterior power bidual} of $M$ is defined to be the module 
\[\bidual^s_R M \coloneqq \left ( \exprod^s_R M^\ast \right )^\ast .\]
We note that if $R = \Z [A]$ for a finite abelian group $A$, then $\bidual^s_R M$ coincides with the lattice first introduced by Rubin in \cite[\S\,2]{Rub96} (cf.\@ \cite[Rem.~A.9]{sbA}). 

\subsubsection{Rubin lattices}\label{rubin lattice section}

For a finite abelian extension $K$ of $k$ we write $\mathscr{P}_K^\mathrm{ad}$ for the collection of non-empty finite sets of places of $k$ that are 
 disjoint from $S(K)\cup S_\infty(k)$ and contain no place that divides $|\mu_K|$. We refer to elements of $\mathscr{P}_K^\mathrm{ad}$ as `admissible sets for $K$'. 
For any $T\in \mathscr{P}_K^\mathrm{ad}$ the group 
\[
\bigO_{K, S(K), T}^\times = \{ u \in \bigO_{K, S(K)}^\times \mid u \equiv 1 \text{  modulo all places in }T_K\}\]
of `$T$-modified $S(K)$-units' of $K$ is a finite index subgroup of $\bigO_{K, S(K)}^\times$ that is $\Z$-torsion free.  
For $T\in \mathscr{P}_K^\mathrm{ad}$ we set 
\[ \delta_T = \delta_{T,K} \coloneqq {\prod}_{v \in T} (1 - \NN v\cdot \Frob_v^{-1}) \in \ZZ[\G_K],\]
and then, for any $s\in \N_0$, define a $\G_K$-submodule of $\Q \exprod^s_{\Z [\cG_K]} \bigO^\times_{K, S(K)}$ by setting 
\[
\mathfrak{L}^s_K \coloneqq \Big \{
a \in \Q \exprod^s_{\Z [\cG_K]} \bigO_{K, S(K)}^\times 
\; \big | \;
\delta_{T,K} (a) \in \bidual^s_{\Z [\cG_K]} \bigO_{K, S(K), T}^\times \text{ for all } T \in \mathscr{P}_K^\mathrm{ad}\Big\}.
\]
If $s = r_K$, then we will suppress the superscript $r_K$ in the notation.\\  
The argument of \cite[Ch.\@ IV, Lem.\@ 1.1]{Tate} implies $\Ann_{\Z [\cG_K]}(\mu_K)$ is generated over $\ZZ$ by the set $\{\delta_{T,K} \mid T \in \mathscr{P}_K^\mathrm{ad}\}$. By combining this fact together with the defining conditions of $\fL^s_K$, one deduces that  
\begin{equation}\label{tate pop} \Ann_{\Z [\cG_K]}(\mu_K) \cdot \fL^s_K \; \subseteq \; \bidual^s_{\Z [\cG_K]}\bigO_{K, S(K)}^\times.\end{equation}

In particular, since $\bidual^s_{\Z [\cG_K]} \cO_{K, S(K)}$ is finitely-generated and  $\Ann_{\Z [\cG_K]} (\mu_K)$ contains $|\mu_K|$, it follows that $\fL^s_K$ is finitely-generated as a $\Z [\cG_K]$-module.

\begin{rk} \label{lattice-remark} 
If $s \in \{0,1\}$, then the result of \cite[Ch.\@ IV, Prop.\@ 1.2]{Tate}
 implies that the inclusion (\ref{tate pop}) uniquely characterises $\fL^s_K$. In other words, in this case, an element $a$ of $\Q \exprod^s_{\Z [\cG_K]} \bigO_{K, S(K)}^\times$ belongs to $\fL_K^s$ if and only if $\Ann_{\Z [\cG_K]} (\mu_K)\cdot a \subseteq \bidual^s_{\Z [\cG_K]} \bigO^\times_{K, S(K)}$.
 Similarly, if $p\in \mathbb{P}$ is such that $\mu_K \otimes_\Z \Z_p$ is $\cG_K$-cohomologically trivial, then the arguments of Popescu in \cite[Th.\@ 5.5.1\,(3)]{Popescu02} show that, for any $s\in \N_0$, an element $a$ of $\Q_p \exprod^s_{\Z [\cG_K]} \bigO_{K, S(K)}^\times$ belongs to $\Z_p \fL_K^s$ if and only if $\Ann_{\Z_p [\cG_K]} (\mu_K \otimes_\Z \Z_p)\cdot a \subseteq  \bidual^s_{\Z_p [\cG_K]} (\Z_p \bigO^\times_{K, S(K)})$. In general, the inclusion (\ref{tate pop}) allows one to explicitly relate $\fL^s_K$ to the lattices defined by Popescu in \cite[Def.\@ 2.1.1]{Popescu02}. To be more precise, setting \[
 U^\mathrm{ab}_{K, S(K)} \coloneqq \big \{ u \in \bigO_{K, S(K)}^\times 
\; \big | \; K \big (  u^{1 / | \mu_K|} \big) / k \text{ is abelian} \big \},
\]
one finds that $a\in \Q \exprod^s_{\Z [\cG_K]} \bigO_{K, S(K)}^\times$ satisfies  $\Ann_{\Z [\cG_K]} (\mu_K) \cdot a \subseteq \bidual^s_{\Z [\cG_K]} \bigO^\times_{K, S(K)}$ if and only if $|\mu_K|\cdot f(a) \in (U^\mathrm{ab}_{K, S(K)})^{\ast \ast} $ for all $f \in \exprod^{s - 1}_{\Z [\cG_K]} (\bigO_{K, S(K)}^\times)^\ast$.
\end{rk}

\begin{rk}\label{surj of annihilator} The explicit description of the modules $\mathrm{Ann}_{\Z[\cG_{K}]}(\mu_{K})$  given above has the following useful consequence: for any finite abelian extensions  $K$ and $K'$ of $k$ with $K \subseteq K'$, the homomorphism $\mathrm{Ann}_{\Z[\cG_{K'}]}(\mu_{K'}) \to \mathrm{Ann}_{\Z[\cG_K]}(\mu_K)$ that is induced by the restriction map $\Z[\cG_{K'}] \to \Z[\cG_K]$ is surjective (cf.\@ \cite[Lem.\@ 3.9]{bdss}). 
\end{rk}

\vskip 0.1truein

In the next result, we establish several useful properties of the lattices $\fL^s_K$ with respect to change of $K$. In its statement we use, for each finite abelian extension $E$ of $k$ with $K \subseteq E$, the isomorphism of $\Q[\cG_K]$-modules that is defined as follows   
\[
\nu_{E / K} \: \Q \exprod^s_{\Z [\cG_K]} \bigO_{K, S(E)}^\times \xrightarrow{\sim} \bigl(\Q \exprod^s_{\Z [\G_E]} \bigO_{E, S(E)}^\times\bigr)^{\gal{E}{K}}, 
\quad a \mapsto 
\begin{cases}
[E : K]^{1 - s} a \quad & \text{ if } s > 0, \\
\NN_{\gal{E}{K}} a & \text { if } s = 0.
\end{cases}
\]

\begin{lem}\label{lattice lemma2} 
For a finite abelian extension $E\!$ of $k$ with $\!K\!\subseteq\! E$, the following claims are valid. 
\begin{liste}
\item $\nu_{E / K}(\fL^s_K)\subseteq ( \fL^s_E)^{\gal{E}{K}}$,
 \item the composite $\nu_{E / K} \circ \NN^s_{E / K}$ coincides with multiplication by  $\NN_{\gal{E}{K}}$,
\item there exists an integer $N$, depending on $K$ but independent of $E$, such that, for every $a \in \fL^s_K$ and $t \in \N$, one has $a \in t \cdot \fL_K^s$ if $\nu_{E / K} (a) \in Nt \cdot \fL_E^s$,
\item $\nu_{E / k}$ induces an isomorphism $ |\mu_k|^{-1} \exprod^s_\Z (\cO_{k, S(E)}^\times)^{\ast \ast} \cong (\fL_E^s)^{\cG_E}$. 
\end{liste}
\end{lem}

\begin{proof} 
To prove claim (a) we must show that $\delta_T \nu_{E / K} (a) \in \bidual^s_{\Z [\cG_E]} \cO_{E, S(E), T}^\times$ for all $a\in \fL^s_K$ and $T\in \mathscr{P}_E^\mathrm{ad}$. For this, we note that, since $S(K)\subseteq S(E)$ and $\mu_K \subseteq \mu_E$, one has $\mathscr{P}_E^\mathrm{ad} \subseteq \mathscr{P}_K^\mathrm{ad}$. In addition, for each $T\in \mathscr{P}_E^\mathrm{ad}$ the module $\cO^\times_{E, S(K), T}$ is $\Z$-torsion free and hence reflexive. By Lemma~\ref{algebraic-lemma-injection} (and Remark~\ref{compatibility nu maps}) below we therefore have that
\[
\nu_{E / K} (\delta_{T, K} a) = \delta_{T, E} \nu_{E / K} (a) \in \bidual^s_{\Z [\cG_E]} \cO^\times_{E, S(E), T},
\]
as required to prove claim (a).  \\
Claim (b) is straightforward to check from the definitions (see also \cite[Rem.\@ 4.14]{BKS}). \\
To prove claim (c), let $a \in \fL_K^s$ be an element with the property that $\nu_{E / K} (a) = M \cdot x$ for some $x \in (\fL_E^s)^{\gal{E}{K}}$ and $M \in \N$. We set $b \coloneqq \nu_{E / K}^{-1} (x)\in \Q \exprod^s_{\Z [\cG_K]} \cO^\times_{K, S(K)}$ and claim that, for every $T\in \mathscr{P}_E^\mathrm{ad}$, one has $\delta_{T, K} b\in \bidual^s_{\Z [\cG_K]} \cO^\times_{K, S(K), T}$ . To show this we note,  $\delta_{T, E} x \in \big (\bidual^s_{\Z [\cG_E]} \cO^\times_{E, S(E), T} \big)^{\gal{E}{K}}$ for each such $T$, and so $\delta_{T, K} b \in \bidual^s_{\Z [\cG_K]} \cO^\times_{K, S(E), T}$ by Lemma~\ref{algebraic-lemma-injection}. 
Moreover,  the cokernel of the inclusion $\cO^\times_{K,S(K), T} \to \bigO_{K, S(E), T}^\times$ is torsion-free so that the corresponding restriction map $(\bigO_{K, S(E), T}^\times)^\ast \to (\cO^\times_{K,S(K), T})^\ast$ is surjective and hence 
\[ \bidual^s_{\Z [\cG_K]} \cO^\times_{K, S(K), T} = \Big (\Q \exprod^s_{\Z [\cG_K]} \cO^\times_{K, S(K)} \Big) \cap \Big ( \bidual^s_{\Z [\cG_K]} \cO^\times_{K, S(E), T} \Big)\]
(cf.\@ the argument of \cite[Lem.\@ 4.7\,(ii)]{BKS}). It follows that $\delta_{T, K} b\in \bidual^s_{\Z [\cG_K]} \cO^\times_{K, S(K), T}$ and hence that $b$ satisfies the defining condition of $\fL^s_K$ at every $T\in \mathscr{P}_E^\mathrm{ad}$. Since $\Ann_{\Z [\cG_K]} (\mu_K)$ is generated by the set $\{\delta_{T, K}: T \in \mathscr{P}_E^\mathrm{ad}\}$, this implies $b$ belongs to the lattice
\[
\widetilde{\fL}^s_K \coloneqq \big \{ c \in \Q \exprod^s_{\Z [\cG_K]} \cO^\times_{K, S(K)} \mid \Ann_{\Z [\cG_K]} (\mu_K) \cdot c \in \bidual^s_{\Z [\cG_K]} \cO^\times_{K, S(K)} \big \}.
\]
The equation $a = M \cdot b$ therefore implies $a\in M\cdot \widetilde{\fL}^s_K$. 
Since $\widetilde{\fL}^s_K$ is a finitely-generated $\Z$-module that contains $\fL_K^s$ (by \ref{tate pop}) and generates the same $\Q$-vector spaces as $\fL_K^s$, the quotient $\widetilde{\fL}^s_K / \fL_K^s$ is finite.
Hence, setting $N \coloneqq | \widetilde{\fL}^s_K / \fL_K^s|$ and taking $M = Nt$ shows that 
$
a \in Nt \widetilde{\fL}^s_K \subseteq t \fL_K^s$,
as required to prove claim (c).
\\
To prove claim (d) we fix $a\in (\fL^s_E)^{\cG_E}$ and set $b \coloneqq \nu_{E / k}^{-1} (a)\in \Q \exprod^s_\Z \cO^\times_{k, S(E)}$. Then, for $T \in \mathscr{P}^\mathrm{ad}_E$, one has $\delta_{T, E} a\in \big ( \bidual^s_{\Z [\cG_E]} \cO^\times_{E, S(E), T} \big)^{\cG_E}$ and hence $\delta_{T, k} b\in \exprod^s_\Z \cO^\times_{k, S(E), T}$ by Lemma~\ref{algebraic-lemma-injection}. 
In particular, one has $| \mu_k | a \in \exprod^s_\Z (\cO_{k, S(E)}^\times)^{\ast \ast}$. Conversely, if $|\mu_k| a\in \exprod^s_\Z (\cO_{k, S(E)}^\times)^{\ast \ast}$, then the argument of \cite[Th.\@ 5.5.1\,(3)]{Popescu02} implies $\delta_{T, k} a \in \exprod^s_\Z \cO_{k, S(E), T}^\times$ for all $T \in \mathscr{P}_k^\mathrm{ad}$. Just as in the proof of claim (a), we can conclude from this that $\nu_{E / k} (a)\in (\fL^s_E)^{\cG_E}$, as required. 
\end{proof}

\begin{rk} \label{lattice lemma2 S-version rk}
Let $\mathcal{S}$ be any finite subset of $\mathbb{P}$.  
Then the proof of Lemma~\ref{lattice lemma2} shows that all of the stated claims remain valid for the $\Z_\mathcal{S} [\cG_K]$-lattices $\Z_\mathcal{S} \fL_K^s$ if in claim (c) one restricts to integers outside $\mathcal{N}$ that are coprime to $\mathcal{S}$. 
\end{rk}

The following algebraic result is an analogue of \cite[Prop.\@ A.4]{sbA}.

\begin{lem} \label{algebraic-lemma-injection}
Let $G$ be a finite abelian group and $M$ a finitely-generated $\Z[G]$-module. Then, for each $s \in \N_0$ and each subgroup $H \subseteq G$, there exists an  isomorphism of $\Z[G/H]$-modules
\[
\bidual^s_{\Z [G / H]} (M^{\ast \ast})^H \stackrel{\simeq}{\longrightarrow} \Big ( \bidual^s_{\Z [G]} M \Big)^H
\]
that is induced by the assignment $\NN^s_{H} a \mapsto \NN_H a$ on $\Q \exprod^s_{\Z [G / H]} M^H \to \Q \exprod^s_{\Z [G]} M$.
\end{lem}

\begin{proof}
We may assume $s \in \N$ since the claim is easily verified directly if $s = 0$. Choose a free presentation $F_1 \to F_0 \to M^\ast \to 0$ of the linear dual $M^\ast$ of $M$, where $F_0$ and $F_1$ are finitely-generated free $\Z [G]$-modules. Dualising this presentation then gives an exact sequence
\begin{equation*} \label{free-presentation-M}\begin{tikzcd}
0 \arrow{r} & M^{\ast \ast} \arrow{r} & F_0^\ast \arrow{r} & F_1^\ast.
\end{tikzcd} \end{equation*}
Upon applying the general result of Sakamoto \cite[Lem.\@ B.12]{Sakamoto20} to both this sequence and its $H$-invariants, we therefore  obtain a composite isomorphism of the required form 
\begin{align} \label{line-1}
\bidual^s_{\Z [G / H]} (M^{\ast \ast})^H & \cong 
\ker \Big \{ \exprod^s_{\Z [G / H]} (F_0^\ast)^H \to (F_0^\ast)^H \otimes_{\Z [G / H]} \exprod^{s - 1}_{\Z [G / H]} (F_1^\ast)^H \Big \} \notag\\
& \cong \ker \Big \{ \Big ( \exprod^s_{\Z [G]} F_0^\ast \Big)^H \to \Big ( F_0^\ast \otimes_{\Z [G]} \exprod^{s - 1}_{\Z [G]} F_1^\ast \Big)^H \Big \} \notag\\
    & \cong  \Big ( \bidual^s_{\Z [G]} M \Big)^H ,\notag
\end{align}
where the second isomorphism is induced by the isomorphism $\exprod^s_{\Z [G / H]} (F_0^\ast)^H \cong \big ( \exprod^s_{\Z [G]} F_0^\ast \big)^H $ that sends each $\NN_H^s a$ to $\NN_H a$. 
\end{proof}

\begin{rk} \label{compatibility nu maps}
    Assume $G = \gal{E}{k}$  and $M = \cO_{E, S(E), T}^\times$ for an abelian extension $E/k$ and a set $T \in \mathscr{P}_E^{\mathrm{ad}}$. Then, for any subgroup $H$ of $G$, one has $(M^{\ast\ast})^H = \cO^\times_{K, S(E), T}$ with $K \coloneqq E^H$ and the isomorphism $\bidual^s_{\Z [\cG_K]} \cO_{K, S(E), T}^\times \xrightarrow{\simeq} ( \bidual^s_{\Z [\cG_E]} \cO_{E, S(E), T}^\times)^H$ in Lemma~\ref{algebraic-lemma-injection} coincides with the restriction of the map $\nu_{E / K}$ defined just before Lemma~\ref{lattice lemma2}. To verify this, one uses the fact that, for every $b \in \bidual^s_{\Z [\cG_K]} \cO_{K, S(E), T}^\times$, one has $b = \NN_H^s a$ for some $a \in \Q \exprod^s_{\Z [\cG_E]} \cO_{E, S(E), T}^\times$ and hence that the explicit definition of $\nu_{E / K}$ implies $\nu_{E / K} (b) = \NN_H a$.  
\end{rk}

\subsubsection{Integrality restrictions}\label{integrality restrictions section}

In this section, we fix a Dedekind domain $\cR$ of characteristic zero and write $\cQ$ for its field of fractions. We also fix $\cX \subseteq \Omega_k$ and  a rank function $\bm{r}$ for $k$. We shall then define several natural $\cR\llbracket G_k^{\mathrm{ab}}\rrbracket$-submodules of $\ES^{\bm{r}, \cX}_k (\cQ)$ that play an important role in our theory (and are in part motivated by the properties of the `Rubin--Stark Euler system' discussed in the next section).   

We start by specifying the notion of `integrality' that is  central to our approach.

\begin{definition} \label{definition-integral-Euler-systems}
A system $c$ in $\ES^{\bm{r}, \cX}_k (\cQ)$
is said to be \emph{$\cR$-integral} if one has 
\begin{equation}\label{int res cond}
c_E \in \cR \!\cdot\! \fL_E^{\bm{r}(E)} \quad \text{ for all } E \in \cX.
\end{equation}
We write
$\ES^{\bm{r}, \cX}_k (\cR)$ for the $\cR\llbracket G_k^{\mathrm{ab}}\rrbracket$-submodule of $\ES^{\bm{r}, \cX}_k (\cQ)$ comprising $\cR$-integral systems. \\
(If $\cX = \Omega_k$, then we suppress explicit reference to $\cX$ in the notation.) 
\end{definition}

In most cases, the module $\ES_k^{\bm{r}} (\cR)$ contains a non-zero submodule that is elementary in nature. To construct this submodule, we fix $\p\in S_\fin(k)$ and set 
\[ k(\p^\infty) \coloneqq {\bigcup}_{n \in \N} k (\p^n) \,\,\text{ and }\,\, t_\p\coloneqq |(\cG_{k (\p^\infty)})_{\mathrm{tor}}|\]
(noting class field theory implies $(\cG_{k (\p^\infty)})_{\mathrm{tor}}$ is finite). If $E\in \Omega_k$ is such that $S(E) = \{\p\}$, then  Lemma~\ref{lattice lemma2}\,(d) implies $\nu_{E / k}$ induces an isomorphism
\[|\mu_k|^{-1}\!\cdot\! \exprod^{\bm{r}(E)}_{\cR} (\cR\cO^\times_{k, \{ \p \}})^{\ast \ast} \stackrel{\simeq}{\longrightarrow} 
(\cR\fL_E^{\bm{r}(E)})^{\cG_E}.\]
In particular, if $[E : k]^{-1} t_\p \in \cR$ (which occurs if either $\NN\p\in \cR^\times$ or $[k(\p^\infty):k]$ is finite) and for   $E\in \Omega_k$ and  $a\in |\mu_k|^{-1}\!\cdot\!\exprod^{\bm{r}(E)}_{\cR} (\cR \cO^\times_{k, \{ \p \}})^{\ast \ast}$, we set 
\[
c_\p (a)_E \coloneqq \begin{cases}
([E : k]^{-1} t_\p)\cdot \nu_{E / k} (a) & \text{ if } E \subseteq k (\p^\infty), \\
0 & \text{ otherwise,}
\end{cases}\]
then it is easily seen that $c_\p (a) \coloneqq (c_\p (a)_E)_{E\in \Omega_k}$ defines a system in $\ES_k^{\bm{r}}(\cR)^{G_k^\mathrm{ab}}.$ In this way one therefore obtains a canonical homomorphism of $\cR \llbracket G_k^{\mathrm{ab}}\rrbracket$-modules 
\begin{equation*}
\varrho^\mathrm{iso} \: {\prod}_{\p} \bigl(|\mu_k|^{-1} \exprod^{\bm{r} (k(\p))}_{\cR} (\cR \cO^\times_{k, \{ \p \}})^{\ast \ast}\bigr) \to 
\ES_k^{\bm{r}} (\cR)^{G_k^{\mathrm{ab}}}, \quad
(a_\p)_\p \mapsto \Big ( {\prod}_\p c_\p (a_\p) \Big), 
\end{equation*}
in which the direct product on the left is over $\{ \p\in S_\fin(k): \NN\p \in \cR^\times \,\text{or}\, k(\p^\infty)\in \Omega_k\}$.\\
By construction, the components of a system in the image of this map are trivial except possibly on fields of prime-power conductor and there are also no non-trivial distribution relations between its components at fields of coprime conductors. This motivates the following definition.  

\begin{definition}\label{local construction} The image $\ES_k^{{\bm{r}}, \cX}(\cR)^{\mathrm{iso}}$ of the composite $\varrho^\cX \circ \varrho^\mathrm{iso}$
 is the $\cR \llbracket G_k^{\mathrm{ab}}\rrbracket$-module of ($\cX$-restricted) \emph{isolated ($\cR$-integral) Euler systems}.  \end{definition}

For the Dedekind domains $\cR$ that are of most interest to us, the map $\varrho^\mathrm{iso}$ is surjective and so $\ES_k^{{\bm{r}}}(\cR)^{\mathrm{iso}} = \ES_k^{\bm{r}} (\cR)^{G_k^{\mathrm{ab}}}$ (see Lemma~\ref{invariant systems description} below). For the moment, however, we note only that, for any $\cR$, the module $\ES_k^{{\bm{r}}}(\cR)^{\mathrm{iso}}$  is non-zero provided $\bm{r} (E) \leq |S_\infty(k) |$ for all $E \in \Omega_k^{S_\infty (k)}$ and, in addition, 
there exists either $p \in \mathbb{P}$ with $p \in \cR^\times$ or $\p \in S_\fin(k)$ with $k(\p^\infty)\in \Omega_k$. 

\begin{rk}\label{local vanishing}
The set $\{ \p \in S_\fin(k): k(\p^\infty) \in \Omega_k\}$ is empty if and only if $k$ is either $\Q$ or imaginary quadratic. To explain this, 
we fix a prime $p$ that splits completely in $k$ and also $\p\in S_p(k)$. 
Then global class field theory identifies $\gal{k (\p^\infty)}{k (1)}$ with the quotient of $\Z_p^\times$ by the closure of the image of the global units of $k^\times$. If $k$ is not $\Q$ or an imaginary quadratic field, then the global units are infinite and so said quotient must be finite. This combines with Lemma~\ref{rubin lemma}\,(ii) to imply that $k(\p^\infty) / k (1)$ is finite and non-trivial (and so $k(\p^\infty)\in \Omega_k$) if $\p\in S_p(k)$ and $p$ splits completely in $k (\mu_\ell, (\cO_k^\times)^{1 / \ell})$ for some $\ell \in \mathbb{P}\setminus \{p\}$. 
On the other hand, if $k$ is either $\Q$ or imaginary quadratic, then every extension $k (\p^\infty)/k$ is infinite.
\end{rk} 

The submodule of $\ES_k (\cR) = \ES_k ^{\bm{r}_\mathrm{max}}(\cR)$ specified by the following definition will also play an important role in our theory. 

\begin{definition}\label{sym def} For a field $\cF$ of characteristic zero, an \emph{$\cF$-rational symmetric Euler system} is a family $c$ in $\ES_k(\cF)$ for which there exists an element $c_k$ of $\cF\exprod^{r_k}_\Z \bigO_{k}^\times$ (the `initial value' of $c$) with the following property: for all $E \in \Omega_k$ and  $\p \in S_\fin(k) \cap S(E)$, one has 
\begin{equation}\label{symm def eq}
 \NN^{r_E}_{E / k} (c_E) = \Big ({\prod}_{v \in S(E) \setminus \{ \p \}} (1 - \Frob_v^{-1})\Big) \cdot \Ord_\p^{-1} (c_k)
\end{equation}
in $\bigoplus_{i = 0}^\infty \big ( \cF \exprod^i_{\Z} \bigO_{k, S(E)}^\times \big)$, where $\Ord_\p = \Ord_{\p,\cF}$ is the isomorphism $\cF \exprod^{r_k + 1}_\Z \bigO_{k, \{ \p \}}^\times \to  \cF \exprod^{r_k}_\Z \bigO_{k}^\times$ induced by the normalised valuation map $\ord_\p \: \bigO_{k, \{ \p \}}^\times \to \Z$ at $\p$. The collection of such systems is an $\cF \llbracket G_k^{\mathrm{ab}}\rrbracket$-submodule $\ES_k (\cF)^{\mathrm{sym}}$ of $\ES_k(\cF)$. An \emph{$\cR$-integral symmetric Euler system} is then an element of the $\cR \llbracket G_k^{\mathrm{ab}}\rrbracket$-submodule $\ES_k (\cR)^{\mathrm{sym}} \coloneqq \ES_k(\cR)\cap \ES_k (\cQ)^{\mathrm{sym}}$ of $ \ES_k (\cR)$.
\end{definition}

\begin{rk}  \label{symmetric euler systems rk} To clarify the above definition, we fix $c$ in $\ES_k(\cR)$ and $E \in \Omega_k$. 
\begin{liste}
    \item If $|S(E)| = 1$, then $E \subseteq k (\p^\infty)$ for some finite place $\p$ of $k$. In this case the element $\NN^{r_E}_{E / k} (c_E)$ of $  \cQ \exprod^{r_k + 1}_\Z \bigO_{k, \{ \p \}}^\times$ is independent of the choice of $E$, and $c$ is symmetric (in the above sense) if and only if $\Ord_\p(\NN^{r_E}_{E / k} (c_E))$ is independent of $\p$.
    \item If $|S(E)| >1$, then the equality (\ref{symm def eq}) directly implies that $\NN^{\bm{r} (E)}_{E / k} (c_E)=0$. In particular, the condition (ii) can be omitted from Definition \ref{euler-systems-definition-1} without changing the set $\ES_k (\cR)^{\mathrm{sym}}$.
\end{liste}
\end{rk}

The following consequence of the Cebotarev Density Theorem for symmetric systems is important in the sequel.  

\begin{lem}\label{ivc versus local}\label{no isolated system satisfies initial value} 
If $\cR$ denotes either $\Z_\cS$ for a finite subset $\cS$ of $\mathbb{P}$ or $\Z_p$ for some $p \in \mathbb{P}$, then one has $\ES_k(\cR)^{\mathrm{iso}} \cap \ES_k(\cR)^{\mathrm{sym}} = (0)$. \end{lem}

\begin{proof} It is enough to show that for every $c\in \ES_k(\cR)^{\mathrm{iso}} \cap \ES_k(\cR)^{\mathrm{sym}}$ one has $c_k = 0$. Indeed, if $\p\in S_\fin(k)$ and $E$ is any finite ramified extension of $k$ in $k (\p^\infty)$, then $c_k = 0$ implies $\NN^{r_E}_{E / k} (c_E) = 0$ and hence that $c_E = 0$ since $c_E$ is fixed by $\cG_E$. \\
Now, if we assume there exists a place $\p\in S_\fin(k)$ of $k$ such that $\NN \p \not \in \cR^\times$
and $k(\p^\infty) / k$ is infinite, then by definition $c_{k (\p^n)} = 0$ for any $n$ large enough such that $k(\p^n) \in \Omega_k$, and so the initial value condition implies $c_k = 0$. We may therefore assume that all but finitely many places of $k$ are such that $k(\p^\infty) / k$ is a finite extension. \\
Then, since $c$ is isolated, the element $c_{k (\p)}$ is fixed by $\cG_{k (\p)}$ for every such $\p$ with $k(\p) \in \Omega_k$, and so the element 
\[
(\nu_{k(\p) / k} \circ \Ord_\p^{-1} )(c_k) = ( \nu_{k (\p) / k}\circ \NN^{r_k + 1}_{k(\p) / k}) (c_{k (\p)})
= [k(\p) : k] c_{k (\p)}
\]
is divisible by $[k (\p) : k]$ in $\fL_{k (\p)}$. It now follows from Lemma~\ref{lattice lemma2}\,(d) that $\Ord_\p^{-1}(c_k)$ is divisible by $[k(\p) : k]$ in 
$|\mu_k|^{-1} \exprod^{r_k + 1}_{\cR} \cR \cO_{k, \{ \p\}}^\times$.
\textcolor{black}{Let $\langle [\p] \rangle$ denote the subgroup of the ideal class group $\cl_k$ of $k$ generated by the class of $\p$. Then there exists a natural exact sequence 
\[
0 \to \cO_k^\times \to \cO_{k, \{ \p \}}^\times \xrightarrow{\ord_\p} \Z \to \langle [\p] \rangle \to 0
\]
which implies $\Ord_\p$ induces an isomorphism
\[
\exprod^{r_k + 1}_\cR \cR \cO_{k, \{ \p \}}^\times \xrightarrow{\simeq} | \langle [\p] \rangle | \exprod^{r_k}_\cR \cR \cO_k^\times. 
\]
Hence, since $| \langle [\p] \rangle |$ divides $[k(1):k]$, the element $c_k$ is divisible by $[k (\p) : k (1)]$ in $|\mu_k|^{-1} \exprod^{r_k}_{\cR} \cR \cO_k^\times$.} We now fix $p \in \mathbb{P}$ with $p \notin \cR^\times$. 
It then suffices to show the degrees $[k(\p) : k (1)]$ are divisible by an unbounded power of $p$ (as $\p$ varies). To do this, we fix $n \in \N$ and consider the finite Galois extension $L$ of $k$ obtained by adjoining $p^n$-th roots of all elements in $\cO_k^\times$. Then it is enough to note that, by Cebotarev's Density Theorem, there exist infinitely many places $\p$ of $k$, not lying above $p$, that are completely split in $L$ and, for any such place, the degree $[k (\p):k(1)]$ is divisible by $p^n$ (by Lemma~\ref{rubin lemma}). This concludes the proof of the claimed result.
\end{proof}

To describe another distinguished family of Euler systems, we fix $K, L \in \Omega_k$ with $K \subseteq L$ and $V_L = V_K$, and set $H \coloneqq \gal{L}{K} \subseteq \cG_L$. We also fix a place $v$ in $S(L)\setminus S_\infty(k)$ that splits completely in $K$ and is {\em tamely ramified} in $L$, a place $w = w_K$ of $K$ above $v$, and $T \in \mathscr{P}^\mathrm{ad}_L$. We write $\rec_v \: K^\times \to H$ for the composite of the canonical embedding of $K$ into its completion $K_w$ at $w$ and the local reciprocity map $K^\times_w \to H_w \subseteq H$, where $H_w$ denotes the decomposition group of $w$ inside $H$. Writing $I_H$ for the kernel of the projection map $\Z [H] \to \Z$
and $\rho_H \: \Z [H] \to \Z [H] / I_H^2$ for the associated projection map, we thereby  obtain a canonical homomorphism  of $\ZZ[\cG_K]$-modules 
\begin{equation}\label{rec-prime def}
\Rec'_v \: \cO_{K, S(K)}^\times \to (I_H / I_H^2) \otimes_\Z \Z [\cG_K], 
\quad a \mapsto {\sum}_{\sigma \in \cG_K} \rho_H (\rec_v (\sigma a) - 1) \otimes \sigma^{-1} 
\end{equation}
and hence also an induced composite homomorphism (cf.\@ \cite[Prop.~2.7]{Sano})

\begin{align*} \Rec_v \: \bidual^{r_K + 1}_{\Z [\cG_K]} \cO_{K, S(L), T}^\times \xrightarrow{\Rec'_v}&\, \big (\bidual^{r_K}_{\Z [\cG_K]} \cO_{K, S(L), T}^\times\big )\otimes_{\Z}  (I_H / I_H^2)\\
\xrightarrow{\nu_{L/K}}&\, \big ( \bidual^{r_K}_{\Z [\cG_L]} \cO_{L, S(L),T}^\times ) \otimes_{\Z}  ( \Z [H] / I_H^2) 
\end{align*}
 \\
in which the second arrow is induced by the isomorphism $\bidual^{r_K}_{\Z [\cG_K]} \cO_{K, S(L), T}^\times 
\stackrel{\simeq}{\to} \big( \bidual^{r_K}_{\Z [\cG_L]} \cO_{L, S(L), T}^\times \big)^H$ from Lemma~\ref{algebraic-lemma-injection} (and hence by $\nu_{L/K}$, cf.\@ Remark~\ref{compatibility nu maps}) and the inclusion $I_H \subset \Z[H]$.\\
We further recall from \cite[Lem.\@ 5.1~(iii)]{Rub96} that the map 
 \[
\Ord_v \coloneqq \Ord_{v, K / k} \: K^\times \to \Z [\cG_K], \quad a \mapsto {\sum}_{\sigma \in \cG_K} \ord_{w} (\sigma a) \sigma^{-1}
\]
induces an isomorphism (which we denote by the same symbol) 
\[
\Ord_v \: e_K \cQ \exprod_{\Z [\cG_K]}^{r_K + 1} \cO_{K, S(K) \cup \{ v\}}^\times 
\stackrel{\simeq}{\to} 
e_K \cQ \exprod_{\Z [\cG_K]}^{r_K} \cO_{K, S(K)}^\times. 
\]
(Note that the map $\Ord_\p$ occurring in Definition \ref{sym def} is a special case of this construction with $K = k$.) 
We finally define a normalising factor at $v$ by setting 
\[ P_{L / K,\{v\}}\coloneqq {\prod}_{v' \in S(L) \setminus ( S(K) \cup \{v\})} (1 - \Frob_{v'}^{-1}) \in \Z[\cG_K].\] 

\begin{definition}\label{con def} An element $c$ of $\ES_k(\cR)$ is an \emph{($\cR$-integral) congruence system} if, for every  $p\in \mathbb{P}$ with both $p \not \in \cR^\times$ and $p \mid 2 d_k$, it satisfies the following condition. Writing $\widehat{\cR}^p$ for the $p$-completion of $\cR$, for all data $L/K, T$ and $v$ as specified above, one has 
\[ \Ord_v^{-1} (\textcolor{black}{P_{L / K,\{v\}} \cdot }\delta_{K,T} c_{K})\in \bidual^{r_K + 1}_{\widehat{\cR}^p [\cG_K]} \widehat{\cR}^p  \cO_{K, S(L), T}^\times,\]
and also 
\begin{multline*}
\delta_{L, T}{\sum}_{\sigma \in H} (\sigma c_{L}) \otimes \rho_H (\sigma^{-1}) = (-1)^{r_K} \cdot  (\Rec_v\circ \Ord_v^{-1}) \big ( 
P_{L / K, \{v\}}\cdot \delta_{K, T}c_{K} \big ) \\
\text{in}\quad \big( \bidual^{r_K}_{\widehat{\cR}^p [\cG_L]}\widehat{\cR}^p \cO_{L, S(L),T}^\times  \big) \otimes_{\Z} ( \Z[H]/ I_H^2) .\end{multline*}
The collection $\ES_k(\cR)^{\mathrm{con}}$ of all such systems is an $\cR \llbracket G_k^{\mathrm{ab}} \rrbracket$-submodule of $\ES_k(\cR)$. 
\end{definition}

\begin{rk}
    Whilst congruence relations play an important role in early articles concerning (rank-one) Euler systems, such as Thaine \cite{thaine}, it was subsequently shown  by Rubin \cite[Ch.~IV, \S~8]{Rubin-euler} that a weaker Iwasawa-theoretic form of these congruences can be directly deduced from Euler distribution relations. However, distribution relations on their own are not sufficient for our theory (see, for example, Remark~\ref{scarcity remarks}\,(d)) and the congruences described above provide an appropriate replacement in higher rank for the congruences used by Thaine. The precise form of these congruences is motivated by conjectures of Mazur--Rubin and Sano (see Lemma~\ref{RS-properties-lemma}\,(iii) and Remark~\ref{MRS remark}) and the role that they play in our approach is described in Proposition~\ref{mrs and soogils argument}.
\end{rk}

\subsection{The Scarcity Conjecture}

\subsubsection{The Rubin--Stark Euler system} \label{RS section}

Let $\mathscr{S}$ denote the set $S_\infty(k) \cup S_\fin(k)$ of all places of $k$. Then, by fixing a bijection $\mathscr{S} \cong \N$ we may regard $\mathscr{S}$ as a totally ordered set $(\mathscr{S},\preceq) = \set{v_i}_{i \in \N}$ so that $v_i \preceq v_j$ if and only if $i \leq j$. We choose this bijection in such a way that the first $|S_\infty(k)|$ places in $\mathscr{S}$ are the places in $S_\infty(k)$. Throughout this article all exterior powers $\exprod_{v \in \Sigma}$ indexed over a given finite set of places $\Sigma$ of $k$ will be arranged with respect to the ordering $\preceq$. We also fix, for each place $v \in \mathscr{S}$, an extension $\overline{v}$ of $v$ to our fixed choice of algebraic closure $\overline{\Q}$ of $\Q$. For any subfield $K \subseteq \overline{k}$ we write $v_K$ for the restriction of $\overline{v}$ to $K$, and also set
\[
S^\ast (K) \coloneqq S (K) \cup S_\infty (k).
\]
Let $K$ be a finite abelian extension of $k$ and assume to be given disjoint finite sets of places $\Sigma$ and $T$ of $k$ such that $\Sigma$ contains $S^\ast (K)$. Then, for any character $\chi \in \cG_K^\circ$, the \emph{$T$-modified $\Sigma$-imprimitive Dirichlet $L$-series} for $K / k$ and $\chi$ is defined by setting 
\[
L_{k, \Sigma, T} (\chi, s) = {\prod}_{v \in T} (1 - \chi (\Frob_v) \NN v^{1 - s}) \cdot {\prod}_{v \not \in \Sigma} (1 - \chi (\Frob_v) \NN v^{- s})^{-1},
\]
where $s$ is a complex variable of real part $\text{Re} (s) > 1$. Further, it is well-known that $L_{k, \Sigma, T} (\chi,s)$ admits a meromorphic continuation to $\C$ and that \cite[Ch.\@ I, Prop.\@ 3.4]{Tate} implies 
\begin{equation}\label{tate order formula} \mathrm{ord}_{s=0}(L_{k, \Sigma, T} (\chi, s)) = r_\Sigma (\chi),\end{equation}
with $r_\Sigma (\chi)$ as in (\ref{order-of-vanishing-formula}). We then define the leading term of the $T$-modified $\Sigma$-imprimitive equivariant $L$-series of $K / k$ to be
\[
\theta^\ast_{K / k, \Sigma, T} (0) = {\sum}_{\chi \in \cG_K^\circ} \big ( \lim_{s \to 0} s^{- r_\Sigma (\chi)} L_{k, \Sigma, T} (\chi, s) \big) e_{\chi^{-1}} 
\]
and note that $\theta^\ast_{K / k, \Sigma, T} (0)\in \R [\cG_K]^\times$. 
Finally, we recall that the Dirichlet regulator map 
\[
\lambda_{K, \Sigma} \: \bigO_{K, \Sigma}^\times \to \R \otimes_\Z X_{K, \Sigma}, \quad
a \mapsto - {\sum}_{w \in \Sigma_K} \log |a|_w \cdot w
\]
induces an isomorphism of $\R[\cG_K]$-modules 
\begin{equation} 
\lambda^{(r)}_{K, \Sigma}: \R \otimes_\Z \exprod^r_{\Z [\cG_K]} \bigO_{K, \Sigma}^\times \stackrel{\simeq}{\longrightarrow}
\R \otimes_\Z \exprod^r_{\Z [\cG_K]} X_{K, \Sigma},
\end{equation}
that will often also be abbreviated to $\lambda_{K, \Sigma}$ (if no confusion is possible).   

\begin{definition}\label{Rubin--Stark-definition}\
 
\begin{liste}
    \item Let $K$ be a finite abelian extension of $k$, let $V$ be a proper subset of $\Sigma$ comprising places that split completely in $K$, set  $r \coloneqq | V|$, and fix a place $\mathfrak{p} \in \Sigma \setminus V$. Then the $r$-th order \emph{Rubin--Stark element} for the data $K / k, \Sigma, T$ and $V$ is the unique element $\varepsilon^{V}_{K / k, \Sigma, T}$ of $\R \otimes_\Z \exprod^r_{\Z [\cG_K]} \bigO_{K, \Sigma}^\times$ with 
\[
\lambda_{K, \Sigma} ( \varepsilon^{V}_{K / k, \Sigma, T}) = 
e_{K, \Sigma, r} \theta^\ast_{K / k, \Sigma, T} (0) \cdot \exprod_{v \in V} (v_K - \mathfrak{p}_{K}),
\]
where $e_{K, \Sigma, r}$ denotes the sum of $e_\chi$ over all characters $\chi \in \cG^\circ_K$ with $r_\Sigma (\chi) = r$. (This element $\varepsilon^V_{K / k, \Sigma, T}$ does not depend on the choice of place $\mathfrak{p} \in \Sigma \setminus V$ -- see \cite[Prop.~3.3]{Sano2015}.)
\item The \emph{Rubin--Stark system} for $k$ is  the family
\[
\varepsilon_k = ( \varepsilon_{K/k})_K \in {\prod}_{K \in \Omega_k} \bigl(\R \otimes_\Z\exprod^{r_K}_{\Z [\cG_K]} \bigO_{K, S(K)}^\times\bigr),
\]
where we set $\varepsilon_{K / k} \coloneqq \varepsilon^{V_K}_{K / k, S^\ast (K), \emptyset}$.
\end{liste}
\end{definition}

\begin{bspe1} \label{Rubin--Stark-examples}
In several cases, the above definition can be made more explicit for $r \coloneqq r_K$.
\begin{liste}
\item (\textit{Cyclotomic units}) Take $k$ to be $\Q$, $K$ to be a finite real abelian extension of $\Q$, and $V$ to be $S_\infty (\Q) = \{ v_1 \}$. Then one has  
\[
\varepsilon_{K / \Q} = (1/2)\otimes N_{\Q (\xi_m) / K} (1 - \xi_m) 
\quad \in \Q \otimes_\Z \bigO^\times_{K, S(K)}.
\]
Here $m = m_K$ is the conductor of $K$ and $\xi_m = \iota^{-1} ( e^{2 \pi i / m})$ where $\iota \: \overline{\Q} \hookrightarrow \C$ is the embedding corresponding to the choice of place $\overline{v_1}$ fixed at the beginning of the section. (See \cite[Ch.~III, \S\,5]{Tate} for a proof of this fact.)

\item (\textit{Stickelberger elements}) Let $k$ be a totally real field, $K$ a finite abelian CM extension of $k$, and $V = \emptyset$. In this setting the Rubin--Stark element is given by
\[
\varepsilon_{K / k} = \theta_{K / k, S^\ast (K), \emptyset} (0)
\coloneqq {\sum}_{\chi \in \cG_K^\circ} L_{k, S^\ast (K), \emptyset} (\chi, 0) e_{\chi^{-1}} 
. 
\]
\item  (\textit{Elliptic units}) Let $k$ be an imaginary quadratic field and $K \in \Omega_k$. Fix a place $\p \in S (K)$ and write $\mathfrak{f} = \p^n$ for a power of $\p$ large enough so that the natural map $\bigO_k^\times \to ( \bigO_k/\mathfrak{f})^\times$ is injective. 
Write $\mathfrak{m}$ for the conductor of $K$ and let $\a \subsetneq \bigO_k$ be an auxiliary prime ideal coprime to $6 \mathfrak{f} \m$.
Then the Rubin--Stark Conjecture holds for 
$K$ (see, for example, \cite[Ch.~IV, Prop.~3.9]{Tate}) with the elliptic unit
\[
\varepsilon_{K / k} = (\Frob_\a - \NN \a )^{-1} \cdot \NN_{k (\ff \m) / K} ( \psi (1; \ff \m, \a))
\in \Q \otimes_\Z \bigO_{K,S(K)}^\times.
\]
(This follows from Kronecker's second limit formula; see, for instance,\@  \cite[Lem.~2.2\,e)]{Fla09}).
\end{liste}
\end{bspe1}

In the next result we record several key properties of the Rubin--Stark system. In claim (ii) of this result we refer to the  conjecture formulated by Rubin in \cite[Conj.\@ B$'$]{Rub96} and in claim (iii) to the `refined class number formula' that was conjectured independently by Mazur and Rubin in \cite{MazurRubin} and by Sano in \cite{Sano}. 

\begin{lem} \label{RS-properties-lemma} 
\begin{itemize}
\item[(i)] Write $\zeta^\ast_k (0) = \theta_{k / k, S_\infty (k)}^\ast (0)$ for the leading term at zero of the Dedekind $\zeta$-function of $k$ and $\varepsilon_{k , k}$ for the unique element of $\R \exprod^{r_k}_{\Z} \bigO_{k}^\times$ that satisfies  
\[
\lambda_{k, S_\infty (k)} ( \varepsilon_{k , k}) = 
 -\zeta^\ast_{k} (0) \cdot  {\bigwedge}_{v \in S_\infty (k) \setminus \{ v_1 \}} (v - v_1) 
 \quad \in \R \exprod^{r_k}_{\Z} X_{k, S_\infty (k)}.
\]
Then $\varepsilon_k$ belongs to $ \ES_k(\R)^{\mathrm{sym}}$ and $\varepsilon_{k,k}$ is its initial value. 
\item[(ii)] $\varepsilon_k\in \ES_k(\Z)^{\mathrm{sym}}$ if the Rubin--Stark Conjecture  is valid for all $L/k$ with $L$ in $\Omega_k$.
\item[(iii)]  $\varepsilon_k \in \ES_k(\Z)^{\mathrm{con}}$ if the Mazur--Rubin--Sano Conjecture is valid for all $L/k$ with $L$ in $\Omega_k$.
\end{itemize}
\end{lem}

\begin{proof} The fact that $\varepsilon_k$ validates the distribution relations in Definition \ref{euler-systems-definition-1}\,(i) is proved by Rubin in \cite[Prop.\@ 6.1]{Rub96} (see also \cite[Prop.\@ 3.5]{Sano}). Hence, in view of Remark~\ref{symmetric euler systems rk}\,(b), to prove (i) it is enough  to show, for every field $E \in \Omega_k$ and place $\p \in S_\fin(k) \cap S(E)$, that 
\[
 \NN^{r_E}_{E / k} (\varepsilon_{E / k}) = \Big ({\prod}_{v \in S(E) \setminus \{ \p \}} (1 - \Frob_v^{-1})\Big) \cdot \Ord_\p^{-1} (\varepsilon_{k, k})
\]
in $\bigoplus_{i = 0}^\infty \big ( \R \exprod^i_{\Z} \bigO_{k, S(E)}^\times \big)$. We also note that this condition is satisfied trivially except in the case   $S(E) \setminus \{ \p \} = \emptyset$, or equivalently $S^\ast (E) = \{\p\}\cup S_\infty(k)$ and $V_E = S_\infty (k)$.\\
It is thus enough to verify the above property in the latter case and, to do this, we write $v^\ast \: \R Y_{k, S^\ast (E)} \to \R$ for the dual map of a place $v \in S^\ast (E)$ (considered as an element of $Y_{k, S^\ast (E)}$). Then $\exprod_{v \in S_\infty (k) \setminus \{ v_1 \}} v^\ast$ defines an isomorphism $\R \exprod^{r_k}_{\Z} X_{k, S_\infty (k)} \cong \R$ with the property that
\[
\big ( \exprod_{v \in S_\infty (k) \setminus \{ v_1 \}} v^\ast \big) \big( 
\zeta^\ast_{k} (0) \cdot  {\bigwedge}_{v \in S_\infty (k) \setminus \{ v_1 \}} (v - v_1)
\big) = \zeta^\ast_k (0),
\]
and so it suffices to show that
\begin{align*}
    \Big(\big ( \exprod_{v \in S_\infty (k) \setminus \{ v_1 \}} v^\ast \big) \circ \lambda_{k,S_\infty(k)}^{(r_k)} \circ \Ord_\p \circ \NN^{r_E}_{E / k} \Big)( \varepsilon_{E / k}) = -\zeta^\ast_k (0)
\end{align*}
For this, we first note that, by \cite[Prop.\@ 6.1]{Rub96}, one has $\NN^{r_E}_{E / k} ( \varepsilon^{V_E}_{E / k, S^\ast(E), \emptyset} )= \varepsilon^{V_E}_{k / k, S^\ast(E), \emptyset}$, which is the unique element of $\R \exprod^{r_E}_{\Z} \bigO_{k, S^\ast(E)}^\times$ such that
\begin{align*}
\lambda_{k, S^\ast(E)}^{(r_E)} ( \varepsilon^{V_E}_{k / k, S^\ast(E), \emptyset}) & = \zeta^\ast_{k, S^\ast(E)} (0) \cdot {\bigwedge}_{v \in S_\infty(k)} (v - \p) \\
& = - \zeta^\ast_{k, S^\ast(E)} (0) \cdot ( \p - v_1) \wedge {\bigwedge}_{j=2}^{j=r_k+1}(v_j - v_1).
\end{align*}
Moreover, it is easy to see that 
 $(\log \NN \fp) \cdot \ord_\p = \p^\ast \circ \lambda_{k, S^\ast(E)}$.
 In addition, for any subset $M \subseteq S^\ast (E) \setminus \{ v_1\}$ one has $(\exprod_{v \in M} v^\ast) \circ \lambda^{(|M|)}_{k, M \cup \{v_1\}} = \exprod_{v \in M} (v^\ast \circ \lambda_{k, M \cup \{ v_1\}})$
 and, for any $M'$ containing $M$, \textcolor{black}{the restriction of $\lambda_{k, M'}$ to $\R \cO^\times_{k, M}$ is equal to $\lambda_{k, M' \cap (M \cup S_\infty (k))}$}. It follows that 
\begin{align*}
&(\exprod_{v \in S_\infty (k) \setminus \{ v_1 \}} v^\ast ) \circ \lambda^{(r_k)}_{k, S_\infty (k)} \circ \Ord_\p \\
= &\ord_\p \wedge \big( {\bigwedge}_{v \in S_\infty (k) \setminus \{ v_1 \}} (v^\ast \circ \lambda_{k, S^\ast (E)}) \big)  \\
 = &(\log \NN \fp)^{-1} \cdot  (\p^\ast \circ \lambda_{k, S^\ast(E)}) \wedge \big( {\bigwedge}_{v \in S_\infty (k) \setminus \{ v_1\}} (v^\ast \circ \lambda_{k, S^\ast (E)}) \big)  \\
 = &(-1)^{r_k} \cdot (\log \NN \fp)^{-1} \cdot {\bigwedge}_{v \in S^\ast (E) \setminus \{v_1\}} (v^\ast \circ \lambda_{k, S^\ast (E)}) \\
 = &(-1)^{r_k} \cdot (\log \NN \fp)^{-1} \cdot \big( {\bigwedge}_{v \in S^\ast (E) \setminus \{ v_1\}} v^\ast \big) \circ \lambda_{k, S^\ast (E)}^{(r_E)}.
\end{align*}
We may thus calculate that
\begin{align*}
  & \phantom{\,=\,}  \big(\big ( {\exprod}_{v \in S_\infty (k) \setminus \{ v_1 \}} v^\ast \big) \circ \lambda_{k,S_\infty(k)}^{(r_k)} \circ \Ord_\p \circ \NN^{r_E}_{E / k} \big)( \varepsilon_{E / k})\\
& =  (-1)^{r_k} \cdot (\log \NN \fp)^{-1}
 \cdot \big( {\bigwedge}_{v \in S^\ast (E) \setminus \{ v_1\}} v^\ast \big)\big(  - \zeta^\ast_{k, S^\ast (E)} (0) \cdot ( \p - v_1) \wedge {\bigwedge}_{j=2}^{j=r_k+1}(v_j - v_1) \big) \\
& = (-1)^{r_k} \cdot (-1)^{r_k} \cdot (-1) \cdot (\log \NN \p)^{-1} \cdot \zeta^\ast_{k, S^\ast (E)} (0)  \\
& = - \zeta^\ast_k (0),
\end{align*}
as required to prove (i).  

Claim (ii) follows from (i) and the fact that the validity of the Rubin--Stark Conjecture for the stated fields directly implies $\varepsilon_k$ belongs to the submodule $\ES_k(\Z)$ of $\ES_k(\R)$. 

To verify claim (iii), we fix data $L / K$, $T$ and $v$ as in Definition \ref{con def}. We then recall that the 
Mazur--Rubin--Sano Conjecture for $(L / K / k, S (L), T, V_K, V_K \cup \{ v \})$  incorporates by its very formulation (see \cite[Conj.\@ 5.4]{BKS} and the sentence that immediately precedes it) the validity of the Rubin--Stark conjecture for $(K / k, S^\ast (L),T, V_K \cup \{ v \})$. In particular, if the 
 Mazur--Rubin--Sano Conjecture is valid in the above case, then one has 
\begin{align*}
\Ord_v^{-1} ( P_{L / K,\{v\}} \delta_{K, T} \varepsilon_{K / k}) & = \Ord_v^{-1} ( \varepsilon^{V_K}_{K / k, S^\ast (L) \setminus \{ v \}, T}) \\
& = (-1)^{r_K} \cdot  \varepsilon^{V_K \cup \{ v \}}_{K / k, S^\ast (L), T} \in \bidual^{r_K + 1}_{\Z_p [\cG_K]} \Z_p \cO_{K, S(L), T}^\times,  
\end{align*}
where the second equality follows from \cite[Prop.~3.6]{Sano} and the containment from the (implied) validity of the relevant case of the Rubin--Stark Conjecture. Given this observation, it is then clear that the validity of the 
Mazur--Rubin--Sano Conjecture for all $L/k$ with $L$ in $\Omega_k$ implies that $\varepsilon_k$ verifies all conditions in 
Definition \ref{con def} and so belongs to $\ES_k(\Z)^{\mathrm{con}}$.
\end{proof}

\begin{rk}\label{MRS remark} The strong restrictions imposed on $v$ and $p$ in Definition \ref{con def}, and the fact that only a single place in $S(L)\setminus S_\infty(k)$ is considered, means  the conditions that are required to ensure $\varepsilon_k\in \ES_k(\ZZ)^{\mathrm{con}}$ are much weaker than are the general properties predicted by the conjectures of Mazur--Rubin and Sano. This fact will play a key role in later sections. \end{rk}

\subsubsection{Statement of the Scarcity Conjecture} 

For any 
subset $\cX$ of $\Omega_k$ and any finite subset $\mathcal{S}$ of $\mathbb{P}$, we consider the $\Z_\mathcal{S} \llbracket G_k^{\mathrm{ab}} \rrbracket$-submodules of $\ES^{\cX}_k(\Z_\mathcal{S})$ that are obtained by setting 
\[ \ES^{\cX}_k(\Z_\mathcal{S})^\mathrm{sym} \coloneqq \varrho^\cX(\ES_k(\Z_\mathcal{S})^\mathrm{sym})
\quad \text{ and } \quad
\ES^{\cX}_k(\Z_\mathcal{S})^\mathrm{con} \coloneqq \varrho^\cX(\ES_k(\Z_\mathcal{S})^\mathrm{con})
.\] 

In particular, we note that Lemma~\ref{RS-properties-lemma} implies, modulo conjectures of Rubin and of Mazur--Rubin--Sano, that the `$\cX$-restricted Rubin--Stark system' 
\[ \varepsilon_k^\cX \coloneqq \varrho^\cX (\varepsilon_k) = ( \varepsilon_{E / k})_{E \in \cX}\]
belongs to both $\ES^{\cX}_k(\Z_\mathcal{S})^\mathrm{sym}$ and $\ES^{\cX}_k(\Z_\mathcal{S})^\mathrm{con}$. 

We now state the central conjecture of this article. This conjecture simultaneously continues ideas of Coleman, of Rubin and of Sano and three of the current authors (for more details see Remarks \ref{etnc cons}\,(b) and \ref{comp bdss} below).

\begin{conj}[The `Scarcity Conjecture'] \label{scarcity-conjecture} 
For each finite subset $\cS$ of $\mathbb{P}$, and each subset $\cX$ of $\Omega_k\setminus \Omega_k^\emptyset$,  one has 
\[
\ES_k^\cX (\Z_\mathcal{S})^\mathrm{sym} \cap 
\ES_k^\cX (\Z_\mathcal{S})^\mathrm{con}
= \Z_\mathcal{S} \llbracket G_k^{\mathrm{ab}} \rrbracket\cdot \varepsilon_k^\cX.
\] 
\end{conj}

\begin{rk}\label{scarcity remarks}\
\begin{liste}  
\item The requirement that $\cX\cap \Omega_k^\emptyset = \emptyset$ (and hence that $\cX$ contains no field $K$ with $r_K = 0$) is forced by the fact Euler systems of rank zero and of positive rank are seemingly of an essentially different nature. In particular, the recent work of Sakamoto in \cite{Sakamoto22} suggests many Euler systems of rank zero do not belong to the submodule generated by the system of Stickelberger elements discussed in Example~\ref{Rubin--Stark-examples}(b). 
\item If $\cX\cap \Omega_k^{S_\infty (k)} = \emptyset$, then (as a direct consequence of the definition of symmetric systems) one has 
$\ES^{\cX}_k(\Z_\mathcal{S})^\mathrm{sym} = \ES_k^\cX (\Z_\cS)$ and so Conjecture \ref{scarcity-conjecture} predicts that the $\Z_\mathcal{S} \llbracket G_k^{\mathrm{ab}} \rrbracket$-module $\ES_k^\cX (\Z_\mathcal{S})^\mathrm{con}$ is generated by $\varepsilon_k^\cX$.
\item If $\{\ell \in \mathbb{P}: \ell \mid 2 d_k\}\subseteq \cS$, then $\ES_k (\Z_\cS)^\mathrm{con} = \ES_k (\Z_\cS)$ and so, in this case, Conjecture~\ref{scarcity-conjecture} predicts that the $\Z_\mathcal{S} \llbracket G_k^{\mathrm{ab}} \rrbracket$-module $\ES_k^\cX (\Z_\mathcal{S})^\mathrm{sym}$ is generated by $\varepsilon_k^\cX$.
\item If $k = \Q$, $\cS = \{ 2 \}$ and $\cX = \Omega_\Q^{\{\infty\}}$ with $\infty$ the archimedean place of $\Q$, then (c) predicts the $\Z[1/2] \llbracket G_k^{\mathrm{ab}} \rrbracket$-module $\ES_\Q^{\{\infty\}}(\Z[1/2])^\mathrm{sym}$ is generated by the cyclotomic system $\varepsilon_\Q^{\{\infty\}}$ from Example~\ref{Rubin--Stark-examples}\,(a). This prediction will be proved in Theorem~\ref{thm-scarcity-rank-one}. In addition, in the supplementary article \cite{scarcity2} we use further analysis to show, more strongly, that $\ES_\Q^{\{\infty\}}(\Z)^\mathrm{sym}= \Z\llbracket G_k^{\mathrm{ab}} \rrbracket\cdot \varepsilon_\Q^{\{\infty\}}$ and that a system in $\ES_\Q^{\{\infty\}}(\Z)$ is symmetric if and only if it verifies certain classical congruences  considered by Thaine in \cite{thaine} (and hence that many systems in $\ES^{\{\infty\}}_\Q(\Z)$ are not symmetric). 
\item The theory developed below will in fact suggest the possibility of formulating a stronger version of Conjecture \ref{scarcity-conjecture} (see  Remark~\ref{tempting conjecture}).
\end{liste}
\end{rk}

The interest of Conjecture \ref{scarcity-conjecture} will be explained, at least partly, by the main result of the next section. Then, in the remainder of the article, we shall develop, and apply, a general strategy for the proof of the conjecture (and hence of the results discussed in the Introduction).

\section{Scarcity and Tamagawa numbers}\label{consSC}

\subsection{Statement of the main result} \label{conj-implies-etnc-section}

For a finite group $\Delta$ we write $e_\Delta$ for the idempotent $|\Delta|^{-1} \sum_{\delta \in \Delta} \delta$ of the group ring $\Q[\Delta]$. For each $V\in \mathcal{P}(S_\infty ( k))$ and finite abelian extension $K$ of $k$, we then define an idempotent of $\Z[1/2] [\cG_K]$ by setting 
\begin{equation}\label{epsKV def} \epsilon_{K, V} \coloneqq 
{\prod}_{v \in (S_\infty(k) \setminus V)} (1 - e_{\cG_{K, v}})
\cdot {\prod}_{v \in V} e_{\cG_{K, v}}.\end{equation}
For any subset $\mathcal{V}$ of $\mathcal{P} (S_\infty (k))$ we then define 
\[ \epsilon_{K, \mathcal{V}} \coloneqq {\sum}_{V \in \mathcal{V}} \epsilon_{K, V}.\]
For convenience, in the case $\mathcal{V} = \{ V_K \}$ we also use the abbreviation  
\begin{equation}\label{epsK def} \epsilon_K \coloneqq \epsilon_{K,\{ V_K \}} = {\prod}_{v \in S_\infty(k)\setminus V_K} (1 - e_{\cG_{K, v}}).\end{equation} 
We note  $\epsilon_K$ is non-zero and, in general, differs from the idempotent $e_K$ defined in (\ref{eK def}). 

\begin{remark}\label{explicit epsilon description} We record several useful properties of the above elements. 

\begin{liste}
\item For $V\in \mathcal{P}(S_\infty ( k))$ one has $\epsilon_{K,V} = \sum_{\chi}e_\chi$, where $\chi$ runs over all characters in $\cG_K^\circ$ with the property that, for each $v \in S_\infty(k)$, one has $\chi(\cG_{K,v}) = 1\Longleftrightarrow v \in V$. In particular, one has $\epsilon_{K,V}\cdot\epsilon_{K,V'} = 0$ for all $V' \in \mathcal{P}(S_\infty(k))\setminus \{V\}$ and so $\epsilon_{K,\cV}$ is an idempotent for every $\cV\subseteq \mathcal{P}(S_\infty(k))$. 

\item If $V_K = S_\infty$ (as is the case, for example, if $k$ is totally imaginary or $K$ is totally real), then $\epsilon_K = 1$.  If $k$ is totally real and $K$ is CM, then $V_K = \emptyset$ and $\epsilon_K = (1 - \tau)/2$, where $\tau$ denotes the element of $\cG_K$ induced by complex conjugation. In general, it is clear that $\epsilon_K$ belongs to $2^{- |S_\infty|+r_K}\! \cdot\!\Z [\cG_K]$.

\item If $K'$ is any finite abelian extension of $k$ that contains $K$, then, for each set $\cV$, the natural projection $\QQ[\cG_{K'}]\to \QQ[\cG_K]$ sends $\epsilon_{K',\cV}$ to $\epsilon_{K,\cV}$.
\end{liste}
\end{remark}

Let now $\cR$ denote either $\Z_\cS$ for a finite subset $\cS$ of $\mathbb{P}$ or $\Z_p$ for some $p \in \mathbb{P}$, and write $\cQ$ for its field of fractions. Then, for any finite abelian extension $K$ of $k$, any subgroup $H$ of $\cG_K$ and any idempotent $e$ of $\cQ[H]$, we shall use `$\mathrm{TNC}(K,e\cR[H])$' to denote the equivariant Tamagawa Number Conjecture (or `eTNC' for short in the sequel) for the pair $(h^0 (\Spec K), e\cR[H])$. Whilst this conjecture is a special case of the general conjectures formulated by Kato  in \cite{kato93a, kato93b}, for our purposes it will be convenient to use the precise formulation that is given in \cite[Conj.\@ 3.1]{BKS} for $\cR = \Z_\cS$ and in \cite[Conj.~2.3]{BKS2} for $\cR = \Z_p$. 
\smallskip

In this section we shall prove the following result.

\begin{thm} \label{scarcity-implies-etnc}
Let $K$ be a finite abelian extension of $k$ and $\cS$ a finite subset of $\mathbb{P}$. Then, for any set $\cV\subseteq \mathcal{P} (S_\infty (k)) \setminus \{\emptyset\}$, the validity of Conjecture \ref{scarcity-conjecture} for the pair $(\Omega_k^\cV, \mathcal{S})$ implies that of $\mathrm{TNC}(K,\epsilon_{K, \mathcal{V}}\Z_\mathcal{S} [\cG_K])$. 
\end{thm}

\begin{rk}\label{etnc cons} The phrase `eTNC($\mathbb{G}_m$) for $K/k$' that is used in the Introduction refers to the conjecture $\mathrm{TNC}(K,\Z[\cG_K])$. We recall, in particular, that this conjecture has been shown to imply a wide range of more explicit conjectures including the Rubin--Stark Conjecture, the `refined class number formula' conjectured by Mazur--Rubin \cite{MazurRubin} and Sano \cite{Sano}, the `integral Gross--Stark conjecture' (from \cite{Gross88}), the `Lifted Root Number Conjecture' of Gruenberg, Ritter and Weiss \cite{GRW}, and the central conjecture of Chinburg in \cite{Chinburg}. Details of these, and other similar, deductions can be found, for example, in \cite{BurnsFlach98} and \cite{BKS}.
\end{rk}

\subsection{The analytic class number formula}\label{acnf section} 

As a first step in the proof of Theorem~\ref{scarcity-implies-etnc}, we will establish, in Proposition~\ref{analztic-class-number-formula-argument} below, a useful reinterpretation of the analytic class number formula in our setting. Before doing so, however, it is convenient to review the properties of an important family of complexes. 

For any finite abelian extension $E / k$, finite set of places $\Sigma \supseteq S^\ast (E)$, and finite set $T\subset S_\fin(k)$ with $T \cap \Sigma= \emptyset$, we  use the $T$-modified `Weil-\'etale cohomology' complex $\text{R} \Gamma_{c, T} ( ( \bigO_{E, \Sigma})_\mathcal{W}, \Z)$ of the constant sheaf $\Z$ that is constructed in \cite[Prop.~2.4]{BKS}, and consider its linear dual  
\[
C^\bullet_{E, \Sigma, T} = \mathrm{RHom}_\Z ( \text{R} \Gamma_{c, T} ( ( \bigO_{E, \Sigma})_\mathcal{W}, \Z), \Z) [-2].
\] 
Whenever $T = \emptyset$ we will suppress the respective subscript in the notation. 
Lemma~\ref{properties-weil-etale} below is taken from \cite[Prop.\@ 3.1]{bdss} and sets out the essential properties of the complex $C^\bullet_{E, \Sigma, T}$. 

For a commutative Noetherian ring $R$ we write $D(R)$ for the derived category of $R$-modules and $D^{\mathrm{perf}}(R)$ for the full triangulated subcategory of $D(R)$ comprising complexes that are `perfect' (that is, isomorphic in $D(R)$ to a bounded complex of finitely-generated projective $R$-modules). 

\begin{lem} \label{properties-weil-etale}
    For any data $E / k$, $\Sigma$ and $T$ as above, the complex $C^\bullet_{E, \Sigma, T}$ belongs to $D^{\mathrm{perf}}(\Z [\cG_E])$ and has all of the following properties.
    \begin{liste}
    \item The complex $C^\bullet_{E, \Sigma, T}$ is acyclic outside degrees zero and one, and there are canonical identifications of $\Z [\cG_E]$-modules
    \[
    H^0 ( C^\bullet_{E, \Sigma, T}) = \cO^\times_{E, \Sigma, T}, 
    \qquad 
    H^1 (C^\bullet_{E, \Sigma, T})_\tor = \Cl_{E, \Sigma, T} 
    \quad \text{ and } \quad
    H^1 (C^\bullet_{E, \Sigma, T})_\tf = X_{E, \Sigma}. 
    \]
    Here we write $\Cl_{E, \Sigma, T}$ for  
    the  quotient of the group of fractional ideals of $\mathcal{O}_{E,\Sigma}$ that are prime to $T_E$ by the subgroup of principal ideals with a generator congruent to $1$ modulo all places in $T_E$ (so $\Cl_{E, \Sigma, T}$ is the `$\Sigma_E$-ray class group mod $T_E$' of $E$). 
    \item Let $T'$ be a finite set of places of $k$ that contains $T$ and is disjoint from $\Sigma$, and write $\mathbb{F}_{E, T' \setminus T}^\times$ for the $\cG_E$-module $\bigoplus_{w \in (T' \setminus T)_E} ( \cO_E/\p_w)^\times$, where $\p_w$ is the prime ideal of $\cO_E$ corresponding with the place $w$. Then 
    there exists a canonical exact triangle in $D^\mathrm{perf} (\Z [\cG_E])$
    \begin{cdiagram}
    C^\bullet_{E, \Sigma, T'} \arrow{r} & C^\bullet_{E, \Sigma, T} \arrow{r} & \mathbb{F}_{E, T' \setminus T}^\times [0] \arrow{r} & 
    \phantom{X}.
    \end{cdiagram}%
    \item Let $\Sigma'$ be a finite set of places of $k$ that contains $\Sigma$ and is disjoint from $T$. Then there exists a canonical exact triangle in $D^\mathrm{perf} (\Z [\cG_E])$ 
    \begin{equation} \label{triangle-changing-S}
    \begin{tikzcd}
    C^\bullet_{E, \Sigma, T} \arrow{r} & C^\bullet_{E, \Sigma', T} \arrow{r} & 
    {\bigoplus}_{v \in \Sigma' \setminus \Sigma} \Big [
     \Z [\cG_E] \xrightarrow{1 - \Frob_v^{-1}} \Z [\cG_E]
    \Big ] \arrow{r} & \phantom{X},
    \end{tikzcd} 
    \end{equation}
    where each complex that occurs in the direct sum is concentrated in degrees zero and one.
    \item If $F$ is a finite abelian extension of $k$ with $E \subseteq F$ and $S(F) \subseteq \Sigma$, then there exists a natural isomorphism $\,C^\bullet_{F, \Sigma, T} \otimes^\mathbb{L}_{\Z [\cG_F]} \Z [\cG_E] \cong C^\bullet_{E, \Sigma, T}\,$ in $D^\mathrm{perf} ( \Z [\cG_E])$.
    \end{liste}
\end{lem}

\begin{rk} \label{T-modification Fitt remark}
For any set $T \in \mathscr{P}^\mathrm{ad}_E$, the $\Z [\cG_E]$-module $\mathbb{F}_{E, T}^\times$ that occurs in Lemma~\ref{properties-weil-etale}\,(b) lies in an exact sequence of $\Z [\cG_E]$-modules of the form 
\begin{cdiagram}
0 \arrow{r} & {\bigoplus}_{v \in T} \Z [\cG_E] \arrow{rrr}{( x\, \mapsto\,  x\cdot\delta_{\{ v\}})_{v}} & & & {\bigoplus}_{v \in T} \Z [\cG_E] \arrow{r} &
\mathbb{F}_{E, T}^\times \arrow{r} & 0,
\end{cdiagram}%
(cf.\@ \cite[(4.16)]{Chinburg1985}). This sequence implies that $\Fitt^0_{\Z [\cG_E]} (\mathbb{F}_{E, T}^\times) = \Z [\cG_E] \cdot \delta_{T, E}$.
\end{rk}

Fix $r \in \N_0$. As before, we write $e_{E, \Sigma, r}$ for the idempotent of $\Q [\cG_E]$ which is defined as the sum
of all primitive orthogonal idempotents $e_\chi$ associated to those $\chi \in \cG^\circ_E$ with $r_{\Sigma} (\chi ) = r$. Then, for any $e_\chi$ that occurs in the sum $e_{E, \Sigma, r}$, one has
\[
e_\chi \C X_{E, \Sigma} \cong e_\chi \C Y_{E, V_\chi} \cong e_\chi \C [\cG_E]^{r},
\]
where, \textcolor{black}{with $\p$ a fixed finite place in $\Sigma$ as in Definition \ref{Rubin--Stark-definition}}, 
\[
V_\chi = \begin{cases}
\{ v \in \Sigma \mid \chi (v) = 1 \} & \text{ if } \chi \neq \bm{1}_E, \\
\Sigma \setminus \{ \mathfrak{p} \} & \text{ if } \chi = \bm{1}_E \text{ and } S_\infty (k) \cup \set{\mathfrak{p}} \subseteq \Sigma, \\
\Sigma \setminus \{ v_1 \} & \text{ if } \chi = \bm{1}_E \text{ and } S_\infty (k) = \Sigma.
\end{cases}
\] 
In particular, $e_{E, \Sigma, r} \Q X_{E, \Sigma}$ is a free $e_{E, \Sigma, r}\Q [\cG_E]$-module of rank $r$ (for which we have fixed a basis by virtue of our fixed choice of extension to $E$ for every place of $k$ at the beginning of \S\,\ref{RS section} and, in the case of the trivial character, also of the places $\p$ and $v_1$). 
Given this, we consider the following composite `projection' map 
\begin{align} 
\nonumber 
\Theta^r_{E / k, \Sigma, T} \: \Det_{\Z [\cG_E]} ( C^\bullet_{E, \Sigma, T}) & \hookrightarrow 
\Det_{\Q [\cG_E]} (\Q C^\bullet_{E, \Sigma, T}) \\ \nonumber 
& \stackrel{\simeq}{\longrightarrow}
\Det_{\Q [\cG_E]} ( \Q H^0 (C^\bullet_{E, \Sigma, T} ) ) \otimes_{\Q [\cG_E]} \Det^{-1}_{\Q [\cG_E ]}( \Q H^1 (C^\bullet_{E, \Sigma, T} ) ) \\ \nonumber 
& \xrightarrow{\cdot e_{E, \Sigma, r}} 
\Big ( e_{E, \Sigma, r} \Q \exprod^r_{\Z[\cG_E]} \bigO^\times_{E, \Sigma} \Big)
\otimes_{\Q [\cG_E]}  
\Big( e_{E, \Sigma, r} \Q \exprod^r_{\Z[\cG_E]}  X_{E, \Sigma}^\ast \Big) \\ \label{projection-map}
& \stackrel{\simeq}{\longrightarrow}  e_{E, \Sigma, r} \Q \exprod^r_{\Z [\cG_E]} \bigO_{E, \Sigma}^\times, 
\end{align}
where the second map is the passage-to-cohomology map, the third arrow is multiplication by the idempotent $e_{E, \Sigma, r}$, and the last map is induced by our fixed choice of basis for the $e_{E, \Sigma, r}\Q [\cG_E]$-free module
$e_{E, \Sigma, r} \Q X_{E, \Sigma}$. 

In the sequel we fix, for $p\in \mathbb{P}$, an algebraic closure $\overline{\Q_p}$ of $\Q_p$ and write $\C_p$ for its completion. 

\begin{prop} \label{analztic-class-number-formula-argument}
Fix a finite abelian extension $K$ of $k$ and a subset $V \subsetneq \Sigma$ of cardinality $r$ that comprises places splitting completely in $K$. Then, for each prime $p$, the following assertions are equivalent:
\begin{liste}
    \item $\Z_p \cdot\im (\Theta^r_{K / k, \Sigma, T}) = \Z_p [\cG_K]\cdot \varepsilon_{K / k, \Sigma, T}^V$,
    \item $\Z_p\cdot\im (\Theta^r_{K / k, \Sigma, T}) \subseteq \Z_p [\cG_K]\cdot \varepsilon^V_{K / k, \Sigma, T}$,
   \item $\Z_p \cdot \im (\Theta^r_{K / k, \Sigma, T}) \supseteq \Z_p [\cG_K]\cdot \varepsilon_{K / k, \Sigma, T}^V$.
\end{liste}
\end{prop}

\begin{proof}
We shall show that condition (b) implies (a). The converse is clear, and the proof that (c) implies (a) is analogous. \\
We first note that if $T'$ is a finite set of places which is both disjoint from $\Sigma$ and such that $T'' \coloneqq T \cup T'$ is admissible, then $\delta_{T'} = \delta_{T', K}$ is a non-zero divisor. Since $\varepsilon_{K / k, \Sigma, T''}^V = \delta_{T'} \varepsilon_{K / k, \Sigma, T}^V$ and $\im (\Theta^r_{K / k, \Sigma, T''}) = \delta_{T'}\cdot \im (\Theta^r_{K / k, \Sigma, T})$ (the latter as a consequence of the exact triangle in Lemma~\ref{properties-weil-etale}\,(b) and Remark~\ref{T-modification Fitt remark})
we may therefore assume that $T$ is admissible.
In this case, it is known that the complex $\Z_p \otimes^\mathbb{L}_\Z C^\bullet_{K, \Sigma, T}$ admits a representative of the form $[P \to P]$ with $P$ a free $\Z_p [\cG_K]$-module of finite rank (cf.\@ \cite[\S\,5.4]{BKS}). Moreover, $\Z_p \cdot \Det_{\Z [\cG_K]} ( C^\bullet_{K, \Sigma, T})$ is a free $\Z_p [\cG_K]$-module of rank one. We fix a basis $\fz_K$ of the latter and set $z^\mathrm{b}_K \coloneqq \Theta^r_{K / k, \Sigma, T} (\fz_K)$.  \\
By assumption, we may write $z^\mathrm{b}_K = q_K \cdot \varepsilon^V_{K / k, \Sigma, T}$ for some $q_K \in \Z_p [\cG_K]$.
By construction of  $\Theta^r_{K / k, \Sigma, T}$, the element $z^\mathrm{b}_K$ is annihilated by the idempotent $1 - e_{K, \Sigma, r}$. Since the same is true for the Rubin--Stark element $\varepsilon_{K / k, \Sigma, T}^V$, it suffices to prove $q_K e_{K, \Sigma, r} $ is a unit in $\Z_p [\cG_K] e_{K, \Sigma, r}$. \\
As before we write $\lambda_{K, \Sigma}$ for the $\Sigma$-Dirichlet regulator map of $K$. Let $\varpi = \varpi_{K / k, \Sigma, T}$ be the unique element of $\C_p\Det_{\Z [\cG_K]} (C^\bullet_{K, \Sigma, T})$ that corresponds with the leading term $\theta^\ast_{K / k, \Sigma, T} (0)$ under the isomorphism $\vartheta_{C^\bullet_{K, \Sigma, T}, \lambda_{K, \Sigma}}$ defined in Lemma~\ref{norm-map-Lemma} below. In particular, one has an equality $(\C_p \otimes_\Z \Theta^r_{K / k,\Sigma, T}) (\varpi) = \varepsilon_{K / k, \Sigma, T}^V$.\\
Since $\theta^\ast_{K / k, \Sigma, T} (0)\in \C_p [\cG_K]^\times$, the element $\varpi$ is a $\C_p [\cG_K]$-basis of $\C_p\Det_{\Z [\cG_K]} (C^\bullet_{K, \Sigma, T})$
and so there is a \textcolor{black}{unique element $x \in \C_p [\cG_K]^\times$ such that $ \varpi = x \fz_K$. By construction, $(q_K - x^{-1}) \varpi$ is contained in $\ker(\C_p \otimes_\Z \Theta^r_{K / k,\Sigma, T})$ and so the injectivity of $\C_p \otimes_\Z \Theta^r_{K / k,\Sigma, T}$ on the $e_{K, \Sigma, r}$-part implies that $q_K e_{K, \Sigma, r}  =  x^{-1} e_{K, \Sigma, r}$. It therefore suffices to show that $x^{-1} e_{K, \Sigma, r}$ is a unit in $\Z_p [\cG_K] e_{K, \Sigma, r}$.} \\
To do this, we shall show that the norm $\NN_{\Z_p [\cG_K] e_{K, \Sigma, r} / \Z_p} (x^{-1}e_{K, \Sigma, r})$ is a unit in $\Z_p$. Given this, 
Lemma~\ref{norm-map-Lemma}\,(a) then implies the claim. \\
Note that the set $\Upsilon_{K, \Sigma, r}$ of all characters $\chi \in \cG_K^\circ$ with the property that $r_\Sigma (\chi) = r$ is stable under the action of the absolute Galois group $G_\Q$, whence $\sum_{\chi \in \Upsilon_{K, \Sigma, r}} \chi$ is a rational valued character. 
By the Artin induction theorem (see \cite[Ch.\@ II, Th.\@ 1.2]{Tate}) there thus exists a natural number $m$ and for each subgroup $H$ of $\cG_K$ an integer $m_H$ such that
\[
m \cdot {\sum}_{\chi \in \Upsilon_{K, \Sigma, r}} \chi = {\sum}_{H} m_H \cdot \mathrm{Ind}_H^{\cG_K} (\bm{1}_H)
= {\sum}_{H} m_H \cdot \big ( {\sum}_{\chi (H) = 1} \chi \big).
\]
Writing 
$\pi_{K / K^H} \: \C_p [\cG_K] \to \C_p [\cG_{K^H}]$ for the natural restriction map, we then deduce that
\[
\NN_{\Z_p [\cG_K] e_{K, \Sigma, r} / \Z_p} (x^{-1} e_{K, \Sigma, r})^m = \big( {\prod}_{\chi \in \Upsilon_{K, \Sigma, r}} \chi (x) \big)^{ - m} = {\prod}_H \NN_{\C_p [\cG_{K^H}] / \C_p} ( \pi_{K / K^H} (x) )^{- m_H}
\]
and it therefore suffices to show that each factor $\NN_{\C_p [\cG_{K^H}] / \C_p} ( \pi_{K / K^H} (x) )$ is in $\Z_p^\times$.  \\
By construction, $\pi_{K^H / K}(x) \fz_{K^H}$ is the unique element of $\C_p \Det_{\Z [\cG_{K^H}]} (C^\bullet_{K^H, \Sigma, T})$ that corresponds with the leading term $\theta^\ast_{K^H / k, \Sigma, T}$ via the map $\vartheta_{C^\bullet_{K^H, \Sigma, T}, \lambda_{K^H, \Sigma}}$ defined in Lemma~\ref{norm-map-Lemma} below. It then follows from Lemma~\ref{norm-map-Lemma} that 
\begin{equation} \label{element-determinant}
\mathcal{F}_{\C_p [\cG_{K^H}] / \C_p} ( \pi_{K^H / K}(x) \fz_{K^H, T}) = \NN_{\C_p [\cG_{K^H}] / \C_p} (\pi_{K^H / K}(x)) \cdot \mathcal{F}_{\Z_p [\cG_{K^H}] / \Z_p} ( \fz_{K^H, T})
\end{equation}
corresponds, by \cite[Ch.\@ IV, \S\,1, Prop.\@ 1.8]{Tate}, under $\vartheta_{\tilde{C}^\bullet_{K^H, \Sigma, T}, \lambda_{K^H, \Sigma}}$ with
\[
\NN_{\C_p [\cG_{K^H}] / \C_p} ( \theta^\ast_{K^H / k, \Sigma, T} (0)) = \theta^\ast_{K^H / K^H, \Sigma, T} (0) = \zeta_{K^H, \Sigma, T}^\ast (0),
\]
where $\zeta_{K^H, \Sigma, T}^\ast (0)$ denotes the leading term of the $\Sigma$-truncated and $T$-modified Dedekind $\zeta$-function of $K^H$ at $s = 0$.
The analytic class number formula therefore implies that the element in (\ref{element-determinant}) is a basis of $\Z_p \Det_{\Z} (C^\bullet_{K^H, \Sigma, T})$. On the other hand, since $\fz_{K^H, T}$ is a basis of $\Z_p \Det_{\Z [\cG_{K^H}]} (C^\bullet_{K^H, \Sigma, T})$, the element $\mathcal{F}_{\Z_p [\cG_{K^H}] / \Z_p} ( \fz_{K^H, T})$ is a $\Z_p$-basis as well (by Lemma~\ref{norm-map-Lemma}\,(b)). We therefore deduce that $\NN_{\C_p [\cG_{K^H}] / \C_p} (\pi_{K^H / K}(x))$ is a unit in $\Z_p$, as claimed. 
\end{proof}

In the following general algebraic result we fix a commutative Noetherian ring $R$ and a finitely-generated free $R$-algebra $A$. We write $\NN_{A / R} \: A \to R$ for the `norm' map that sends each $a$ in $A$ to the determinant of the $R$-linear endomorphism $x\mapsto ax$ of $A$. 

\begin{lem} \label{norm-map-Lemma}
Suppose to be given a perfect complex $C^\bullet$ of $A$-modules of the form
$P_0 \longrightarrow P_1$, where $P_0$ and $P_1$ are finitely-generated free $A$-modules, and $P_0$ is placed in degree $0$. Then the following claims are valid. 
\begin{liste}
\item An element $a \in A$ is a unit if and only if $\NN_{A / R} (a)$ is a unit in $R$. 
\item There is an $\NN_{A / R}$-semilinear map $\mathcal{F}_{A / R} \: \Det_A (C^\bullet) \to \Det_R (\widetilde{C}^\bullet)$, where in the last term $\widetilde{C}^\bullet$ indicates $C^\bullet$, regarded as a complex of $R$-modules. The map $\mathcal{F}_{A / R}$ sends each $A$-basis of $\Det_A (C^\bullet)$ to an $R$-basis of $\Det_R (\widetilde{C}^\bullet)$.  
\item Assume that $R$ and $A$ are semi-simple algebras and that $\lambda \: H^0 (C^\bullet) \cong H^1 (C^\bullet)$ is an isomorphism of $A$-modules.  Define the isomorphism
\[  \vartheta_{C^\bullet, \lambda} \:  \Det_A (C^\bullet) \cong \Det_A (H^0 (C^\bullet)) \otimes_A \Det_A (H^1 (C^\bullet))^{-1} \cong A,\]
where the first map is the canonical `passage-to-cohomology' isomorphism and the second the composite of $\Det_A(\lambda)\otimes_A1$ and the evaluation map on $\Det_A (H^1 (C^\bullet))$. Then there exists a commutative diagram
\begin{cdiagram}
\Det_A (C^\bullet) \arrow{r}{\vartheta_{C^\bullet, \lambda}} \arrow{d}[left]{\mathcal{F}_{A / R}} & A \arrow{d}{\NN_{A / R}} \\
\Det_R (\widetilde{C}^\bullet) \arrow{r}{\vartheta_{\widetilde{C}^\bullet, \lambda}} & R.
\end{cdiagram}%
\end{liste}
\end{lem}

\begin{proof} 
The first claim is a restatement of \cite[Ch.\@ III, \S\,9.6, Prop.\@ 3]{Bourbaki}.
The second and third claims are both derived via an explicit computation from the general result of  \cite[Ch.\@ III, \S\,9.6, Prop.\@ 6]{Bourbaki}. \end{proof} 

\subsection{The proof of Theorem~\ref{scarcity-implies-etnc}}

\subsubsection{Commutative diagrams in Galois cohomology}\label{basic-eulers-section}

We now suppose to be given finite sets of places $U \subseteq U'$ of $k$ that each contain the set $S^\ast(E) \coloneqq S_\infty(k) \cup S(E)$. We then  define an isomorphism of $\ZZ[\cG_E]$-modules
\begin{align} \nonumber 
    \Delta_{E,U',U, T} \: \Det_{\ZZ[\cG_E]}(C^\bullet_{E,U', T}) &\xrightarrow{\simeq} \Det_{\ZZ[\cG_E]}(C^\bullet_{E, U, T}) \otimes_{\ZZ[\cG_E]} \bigotimes_{v \in U' \setminus U}  \Det_{\ZZ[\cG_E]}\Big [
     \Z [\cG_E] \xrightarrow{1 - \Frob_v^{-1}} \Z [\cG_E]
    \Big ]\\
    &\xrightarrow{\simeq} \Det_{\ZZ[\cG_E]}(C^\bullet_{E, U, T}),
    \label{def delta map}
\end{align}
where the first map is induced by the exact triangle in Lemma~\ref{properties-weil-etale}\,(c) and the second by trivialising the latter determinant with respect to the identity.

If $U'$ contains $S^\ast (F)$ for some a further finite abelian extension $F$ of $k$ with $E \subseteq F$, we then also define a map
\begin{align} \nonumber
    \nabla_{F / E, U', T} \: \Det_{\ZZ[\cG_F]}(C^\bullet_{F,U', T}) & \to \Det_{\ZZ[\cG_F]}(C^\bullet_{F,U', T}) \otimes_{\Z [\cG_F]} \Z [\cG_E] \\  \nonumber
    & \xrightarrow{\simeq} \Det_{\ZZ[\cG_E]}(C^\bullet_{F,U', T} \otimes^\mathbb{L}_{\Z [\cG_F]} \Z [\cG_E]) \\ \label{def nabla map}
    & \xrightarrow{\simeq} \Det_{\ZZ[\cG_E]}(C^\bullet_{E,U', T}).
\end{align}
Here the first map is the canonical projection, the second is induced by the base change property of the determinant functor, and the third is the isomorphism in Lemma~\ref{properties-weil-etale}\,(d). \\
Particular instances of these constructions then give rise to the composite map
\begin{equation} \label{definition-transition-map}
i_{F / E, T} \coloneqq \Delta_{E,S^\ast(F), S^\ast(E), T} \circ \nabla_{F / E, S^\ast (F), T}
\: \Det_{\Z [\cG_F]} ( C^\bullet_{F, S^\ast (F), T} ) \to \Det_{\Z [\cG_E]} ( C^\bullet_{E, S^\ast (E), T} ).
\end{equation}
If $T$ is empty, then we will usually suppress reference to it in the notation and simply write $\Delta_{E, U', U}$, $\nabla_{F / E, U'}$, and $i_{F / E}$.\\
In the sequel a key role will be played by the functoriality of the projections $\Theta^{r_E}_{E / k, S^\ast (E), T}$ with respect to the various maps that occur in the descent homomorphisms $i_{F / E, T}$. This behaviour is explicitly described in the following result. 

\begin{lem} \label{commutative-diagrams-lemma}
Let $E$ and $F$ be finite abelian extensions of $k$ such that $E \subseteq F$. Let $V' \subseteq S^\ast (F)$ be the set of places which split completely in $E / k$ and write $r' = | V'|$ for its cardinality. For simplicity, we also put $V = V_E$, $r = r_E$, and $W = V' \setminus V$. 
\begin{liste}
\item For any finite set of places $U$ of $k$ containing $S^\ast (F)$, the following diagram commutes.
\begin{cdiagram}
\Det_{\Z [\cG_F]} (C^\bullet_{F, U, T} ) \arrow{rr}{\Theta^{r'}_{F, U, T}} \arrow{d}{\nabla_{F / E, U, T}} & & \Q \exprod^{r'}_{\Z [\cG_F]} \bigO_{F, U, T}^\times
 \arrow{d}{\NN^{r'}_{F / E}} \\
\Det_{\Z [\cG_{E}]} (C^\bullet_{E, U, T}) \arrow{rr}{\Theta^{r'}_{E, U, T}} & & \Q \exprod^{r'}_{\Z [\cG_{E}]} \bigO_{E,U, T}^\times \;
\end{cdiagram}%
\item Assume $F \in \Omega_k$ and $E \in \Omega_k \cup \{ k \}$, and set $\omega^\ast_{ F / E} \coloneqq (-1)^{r |W|}$ for $E \in \Omega_k$ and $\omega_{ F / k}^\ast \coloneqq -1$. Then, with $W' \coloneqq S^\ast(F)\setminus W$, the diagram 
\begin{cdiagram}
\Det_{\Z [\cG_{E}]} (C^\bullet_{E, S^\ast (F), T}) \arrow{rr}{\Theta^{r'}_{E, S^\ast (F), T}} \arrow[d, "\Delta_{E,S^\ast(F), W', T}", "\simeq"' slantedswap] & & \Q \exprod^{r'}_{\Z [\cG_{E}]} \bigO_{E, S^\ast (F), T}^\times \arrow{d}{\omega_{F / E}^\ast \cdot\Ord_{E,W}} \\
\Det_{\Z [\cG_{E}]} (C^\bullet_{E, W', T}) \arrow{rr}{\Theta^r_{E, W', T}} & & \Q \exprod^{r}_{\Z [\cG_{E}]} \bigO_{E, S^\ast (F) \setminus W, T}^\times \; 
\end{cdiagram}%
commutes, where  
$\mathrm{Ord}_{E,W}$ denotes the exterior product over $v$ in $W$ of the maps 
\[
\Ord_{E,v} \: \bigO_{E, S^\ast (F), T}^\times \to \Z[\cG_{E}], \quad a \mapsto {\sum}_{\sigma \in \cG_{E}} \ord_{v_{E}} (\sigma a) \sigma^{-1}.
\]
\item Assume $E \neq k$. Then, for any finite sets $U$ and $U'$  of places of $k$ with $S^\ast(E)\subseteq U \subseteq U'$, the following diagram commutes   
\begin{cdiagram}
       \Det_{\Z [\cG_{E}]} ( C^\bullet_{E, U', T} ) \arrow{rr}{\Theta^{r}_{E /  k, U', T}} \arrow[d,"\Delta_{E, U',U, T}", "\simeq" slantedswap] & & \Q \exprod^r_{\Z [\cG_{E}]} \bigO^\times_{E, U', T}  \\
       \Det_{\Z [\cG_E]} ( C^\bullet_{E, U, T} ) \arrow{rr}{\Theta^r_{E / k, U, T}}  & &  \Q \exprod^r_{\Z [\cG_{E}]} \bigO^\times_{E, U, T} \arrow{u}[right]{\prod_{v \in U'\setminus U} (1 - \Frob_v^{-1})}. \; 
\end{cdiagram}%
\item For any finite sets $U$ and $U'$  of places of $k$ with $S^\ast(F)\subseteq U \subseteq U'$, the following diagram commutes.
\begin{cdiagram}
\Det_{\Z [\cG_{F}]} ( C^\bullet_{F, U', T} ) \arrow{rr}{\Delta_{F, U',U, T}} \arrow{d}{\nabla_{F / E, U'}} &  &
\Det_{\Z [\cG_{F}]} ( C^\bullet_{F, U, T} ) \arrow{d}{\nabla_{F / E, U, T}} \\
    \Det_{\Z [\cG_{E}]} ( C^\bullet_{E, U', T} ) \arrow{rr}{\Delta_{E, U',U}} &  &
\Det_{\Z [\cG_{E}]} ( C^\bullet_{E, U, T} )
\end{cdiagram}
\end{liste}
\end{lem}

\begin{proof}
For parts (a) and (c) see the proof of \cite[Th.\@ 3.8]{bdss}.  To prove part (b), we recall that the left hand vertical map in the given square is induced by the exact triangle (\ref{triangle-changing-S}) in Lemma~\ref{properties-weil-etale}\,(c) applied with the sets $\Sigma' = S^\ast (F)$ and $\Sigma = S^\ast (F) \setminus W$.
To prove commutativity, we may first base change to $\Q [\cG_E]$. Then, by the definition of the projection maps $\Theta^{r'}_{E / k, S^\ast (F), T}$ and $\Theta^{r}_{E / k, S^\ast (F) \setminus W, T}$ (which involves passing to cohomology and trivialising the top degree cohomology), it suffices to show that the composite homomorphism 
\begin{align*}
e_{E} \Q \exprod^{r'}_{\Z [\cG_{E}]} \cO^\times_{E, S^\ast (F), T} & = 
e_{E} \Det_{\Q [\cG_{E}]} ( \Q H^0 ( C^\bullet_{E, S^\ast(F), T})) \\
 & \longrightarrow e_{E} \Det_{\Q [\cG_{E}]} ( \Q H^0 ( C^\bullet_{E, S^\ast (F) \setminus W, T})) = e_{E} \Q \exprod^{r}_{\Z [\cG_{E}]} \cO^\times_{E, S^\ast (F) \setminus W, T}, 
\end{align*}
that is induced by the exact triangle (\ref{triangle-changing-S})
along with the trivialisation
\[
\Det_{\Z [\cG_{E}]} \big ( \Big [
\Z [\cG_{E}] \stackrel{0}{\longrightarrow} \Z [\cG_{E}]
\Big ]  \big )
\cong \Z [\cG_{E}]
\]
for all $v \in W$, coincides with the map $\omega^\ast_{F / E}\cdot\Ord_W$ that is defined in the statement of part (b).
To do this, we note that the Artin--Verdier duality theorem  identifies $C^\bullet_{E,S}$ with the complex that is  denoted by  $\Psi_S$ in \cite[\S\,3.2]{BurnsFlach98}. In particular, the argument of \cite[Prop.~3.2]{BurnsFlach98} implies  the long exact sequence of cohomology of (\ref{triangle-changing-S}) splits into the two short exact sequences 
\begin{cdiagram}[row sep=tiny]
0 \arrow{r} & \Q \cO^\times_{E, S^\ast (F) \setminus W, T} \arrow{r} & \Q \cO^\times_{E, S^\ast (F), T} \arrow{r}{f} & \Q Y_{E, W} \arrow{r} & 0 \\
0 \arrow{r} & \Q X_{E, S^\ast (F) \setminus W} \arrow{r} & \Q X_{E, S^\ast (F)} \arrow{r} & 
\Q Y_{E, W} \arrow{r} & 0,
\end{cdiagram}%
in which the map $f$ is given by
\[
f \: \cO^\times_{E, S^\ast (F), T} \to Y_{E, W}
\cong {\bigoplus}_{v \in W} \Z [\cG_{E}], \quad a \mapsto {\sum}_{v \in W} \ord_w (\sigma a) \sigma^{-1}.
\]
Given this description, the required claim follows by a straightforward, and explicit, calculation (just as in the argument of \cite[Lem.\@ 4.2, 4.3]{BKS}). \\
The commutativity of the diagram in (d) follows from a standard functoriality property of the determinant functor and the fact that the triangles (\ref{triangle-changing-S}) and isomorphisms in Lemma~\ref{properties-weil-etale}\,(d) combine into a canonical commutative diagram
\begin{cdiagram}[row sep=small, column sep=small]
C^\bullet_{F, U, T} \arrow{r} \arrow{d} & C^\bullet_{F, U', T} \arrow{r} \arrow{d} & {\bigoplus}_{v \in U' \setminus U} 
 \Big [
     \Z [\cG_F] \xrightarrow{1 - \Frob_v^{-1}} \Z [\cG_F]
    \Big ] \arrow{r} \arrow{d} & \phantom{X} \\
C^\bullet_{E, U, T} \arrow{r} & C^\bullet_{E, U', T} \arrow{r} & {\bigoplus}_{v \in U' \setminus U} 
\Big [
    \Z [\cG_E] \xrightarrow{1 - \Frob_v^{-1}} \Z [\cG_E]
    \Big ] \arrow{r} & \phantom{X}. 
\end{cdiagram}
This proves the claimed result. 
\end{proof}

\subsubsection{Completion of the proof}

The link between the constructions in \S\,\ref{basic-eulers-section} and the theory of integral Euler systems is provided by the following result.  

\begin{prop} \label{basic-Euler-systems}
The collection of morphisms $(\Theta^{r_E}_{E / k, S^\ast (E)})_{E\in \Omega_k}$ induces a homomorphism of $\Z\llbracket G_k^{\mathrm{ab}}\rrbracket$-modules
\[ \Theta_{k} \: {\varprojlim}_{E \in \Omega_k} \Det_{\Z [\cG_E]} (C^\bullet_{E, S^\ast (E)} ) \to 
\ES_k (\Z)^{\mathrm{sym}} \cap \ES_k (\Z)^{\mathrm{con}},\]
where the transitions morphisms in the limit are the maps $i_{F/E}$  defined in (\ref{definition-transition-map}) for $E\subseteq F$. 
\end{prop}

\begin{proof} Existing results in the literature imply $\im(\Theta_k)$ is contained in $\ES_k (\Z)^{\mathrm{con}}$. More precisely, the argument of \cite[Th.\@ 3.8\,(i)]{bdss} shows that every $c$ in $\im(\Theta_k)$ satisfies the Euler system distribution relations; any such $c$ also satisfies the integrality condition (\ref{int res cond}) with $\cR = \Z$  because the proof of Theorem~3.8\,(ii) of \textit{loc.\@ cit.\@} implies $\im(\Theta^{r_E}_{E / k, S^\ast (E)})\subseteq \mathfrak{L}_E$ for every $E$ in $\Omega_k$; finally, the argument of \cite[Th.\@ 5.16]{BKS} shows that any such $c$ validates all of the congruences that occur in the Mazur--Rubin--Sano Conjecture. 

It therefore suffices for us to fix a system $c$ in $\im(\Theta_k)$ and show that it is symmetric. To do this, we fix a field $E \in \Omega_k$. We note first that if either $V_E \neq S_\infty (k)$ or $|S(E)| > 1$, then 
$\bm{1} (e_{E, S^\ast (E), r_E}) = 0$ and so, since the definition of $\Theta^{r_E}_{E / k, S^\ast (E)}$ involves multiplication by $e_{E, S^\ast (E), r_E}$, one has
$\NN_{E / k}^{r_E} (c_E) = 0$. Thus, the explicit condition in Definition \ref{sym def} is satisfied trivially in any such case. We may therefore assume in the sequel that both $V_E = S_\infty (k)$ and $S(E) = \{ \p \}$ for some finite place $\p$. \\
By definition, there exists a system $\fz = (\fz_E)_E \in \varprojlim_{E \in \Omega_k} \Det_{\Z[\cG_E]} (C^\bullet_{E, S^\ast (E)})$ such that one has $\Theta^{r_E}_{E / k, S^\ast (E)} (\fz_E) = c_E$. Observe that $\fz$ also uniquely specifies an element $\fz_k$ of $\Det_\Z (C^\bullet_{k, S_\infty (k)})$. 
Now, we have a commutative diagram
\begin{cdiagram}[column sep = 4.5em]
\Det_{\Z[\cG_E]} (C^\bullet_{E, S^\ast (E)}) \arrow{r}{\Theta^{r_E}_{E / k, S^\ast (E)}} 
\arrow{d}[left]{\nabla_{ E/ k, S^\ast (E)}} & \bidual^{r_E}_{\Z [\cG_E]} \cO^\times_{E, S (E)} \arrow{d}{\NN^{r_E}_{E / k}} \\
\Det_\Z (C^\bullet_{k, S^\ast (E)}) \arrow{r}{\Theta^{r_E}_{k / k, S^\ast (E)}} \arrow{d}[left]{\Delta_{k, S^\ast (E), S_\infty (k)}} & 
\exprod^{r_E}_\Z \cO^\times_{k, \{ \p \}} \arrow{d}{- \Ord_\p} \\
\Det_\Z (C^\bullet_{k, S_\infty (k)}) \arrow{r}{\Theta^{r_E}_{k / k, S_\infty (k)}} & 
\exprod^{r_E - 1}_\Z \cO^\times_k,
\end{cdiagram}%
in which the upper and lower squares are the relevant instances of the diagrams in Lemma~\ref{commutative-diagrams-lemma}\,(a) and Lemma~\ref{commutative-diagrams-lemma}\,(b). From the commutativity of this diagram, it then follows that $c$ satisfies the condition in Definition \ref{sym def} with its initial value $c_k$ equal to $\Theta_{k/k,S_\infty(k)}^{r_k}(\fz_k)$.
\end{proof}

If we now assume Conjecture \ref{scarcity-conjecture} to be valid for the pair $(\cX, \cS)$, then every $c$ in $\varrho^\cX (\im (\Theta_k))$ can be written as $c = q \cdot \varepsilon^\cX_k$ for some $q \in \Z_\mathcal{S} \llbracket G_k^{\mathrm{ab}} \rrbracket$. 
This shows that Conjecture \ref{scarcity-conjecture} gives rise to an  inclusion $\varrho^\cX (\im (\Theta_k)) \subseteq \Z_\mathcal{S} \llbracket G_k^{\mathrm{ab}} \rrbracket \varepsilon_k^\cX$.
Theorem~\ref{scarcity-implies-etnc} therefore now follows immediately upon applying the following result with $\cX$ taken to be the set $\Omega_k^\cV$ (which is easily seen to satisfy the hypotheses required by the result) for every $p\in \mathbb{P}\setminus \cS$.

\begin{prop} \label{every-basic-RS-implies-etnc}
Fix a finite abelian extension $K/k$, $p\in \mathbb{P}$ and $\cV\subseteq \mathcal{P} (S_\infty (k))$. Let $\cX$ be a subset of $\Omega_k$ that  satisfies the following condition: for each intermediate field $L$ of $K / k$ such that  $V_L \in \cV$, one has
\begin{enumerate}[label=$\bullet$]
    \item $L \in \cX$ if $S(L) \not \subseteq S_\infty (k)$,
    \item $k(\p)L \in \cX$ for almost all $\p \in S_\fin(k)$ if $S(L) \subseteq S_\infty (k)$ and $k( \p) \in \Omega_k$.
\end{enumerate}
Then, if $\varrho^\cX (\im( \Theta_k)) \subseteq \Z_p \llbracket G_k^{\mathrm{ab}} \rrbracket \varepsilon^\cX_k$, the conjecture  $\mathrm{TNC} ( K, \epsilon_{K, \mathcal{V}} \Z_p [\cG_K])$ is valid.
\end{prop}

\begin{proof}
For each $\chi$ in $\cG_K^\circ$, we set $K(\chi) \coloneqq K^{\ker(\chi)}$. We also define a subset  of $\cG_K^\circ$ by setting 
\[ \Upsilon_{K, \mathcal{V}}^\diamond \coloneqq \{ \chi \in \cG_K^\circ: \, \chi (\epsilon_{K, \mathcal{V}}) \neq 0\}.\]
Then for $\chi \in \Upsilon^\diamond_{K, \mathcal{V}}$, the properties in Remark~\ref{explicit epsilon description}\,(a) and (c) combine to imply $V_{K(\chi)}\in \cV$. In particular, if $K(\chi)/k$ is unramified at all finite places, then we can fix 
 a (finite) prime $\p_\chi$ of $k$ that is inert in the (cyclic) extension $K(\chi)$ and is such that the compositum of $K(\chi)$ and $k(\p_\chi)$ belongs to $\cX$. Then, for each  $\chi$ in $\Upsilon_{K, \mathcal{V}}^\diamond$, we obtain a field in $\cX$ by setting  
\[
    K_\chi \coloneqq \begin{cases} K(\chi) &\text{if $\chi$ is ramified at a finite place},\\
        K(\chi)k(\p_\chi) &\text{otherwise.}\end{cases}\]

We now write $N$ for the compositum of $K$ and the fields $K_\chi$ for each $\chi$ in $\Upsilon_{K, \mathcal{V}}^\diamond$, set $\Gamma \coloneqq \cG_N$ and regard $\epsilon_{K,\mathcal{V}}$ as an idempotent of $\QQ[\Gamma]$ in the obvious way (this idempotent does not usually identify with $\epsilon_{N,\mathcal{V}}$ but this fact is not important in the sequel). 

Then, by the known functoriality properties of the eTNC (cf.\@ \cite[Prop. 4.1]{BurnsFlach01}), it is sufficient to show that $\varrho^\cX (\im( \Theta_k)) \subseteq \Z_p \llbracket G_k^{\mathrm{ab}} \rrbracket \varepsilon^\cX_k$ implies the validity of $\mathrm{TNC} (N, \epsilon_{K, \mathcal{V}} \Z_p [\Gamma])$.  

Further, since the $\Z[\Gamma]$-module $\Det_{\Z [\Gamma]} (C^\bullet_{N, S^\ast (N)})$ is locally free, Roiter's Lemma~\cite[(31.6)]{CurtisReiner} allows us to choose an element $\fz_N$ of $ \Det_{\Z[\Gamma]} (C^\bullet_{N, S^\ast (N)})$ that generates a (free, rank-one) $\Z[\Gamma]$-submodule of finite prime-to-$p$ index. In particular, if we set $D^\bullet_{N,S^\ast(N)} \coloneqq \Z_p \otimes_\ZZ C^\bullet_{N,S^\ast(N)}$, then the image $1\otimes \fz_N$ of $\fz_N$ in the free $\ZZ_p[\Gamma]$-module $\Z_p \otimes_\ZZ \Det_{\Z [\Gamma]} (C^\bullet_{N, S^\ast (N)}) \cong
\Det_{\Z_p [\Gamma]} (D^\bullet_{N,S^\ast(N)})$ is a basis.

Now, since the transition maps $i_{F / E} = i_{F/E,\emptyset}$ defined in (\ref{definition-transition-map}) are surjective, we can lift $\fz_N$ to an element $a$ of the corresponding inverse limit $\varprojlim_{E \in \Omega_k} \Det_{\Z [\cG_E]} (C^\bullet_{E, S^\ast (E)} )$. This element $a$ then gives rise, by Proposition~\ref{basic-Euler-systems}, to an Euler system $z^\text{b} \coloneqq \Theta_{k} (a)$. In particular, if we assume  $\varrho^\cX (\im( \Theta_k))\subseteq \Z_p \llbracket G_k^{\mathrm{ab}} \rrbracket \cdot\varepsilon^\cX_k$, then there exists an element $q = (q_E)_E$ of $\Z_p \llbracket G_k^{\mathrm{ab}}\rrbracket$ with $\varrho^\cX (z^\text{b}) = q \cdot \varepsilon^\cX_k$. More concretely, therefore, this element $q$ would satisfy %
\begin{equation} \label{defining-property}
z^\text{b}_E = q_E \cdot \varepsilon_{E / k}
\quad \text{ for every } E \in \cX. 
\end{equation}

We claim that these equalities imply $\epsilon_{K, \mathcal{V}}\cdot q_N$ is a unit of $\epsilon_{K, \mathcal{V}}\Z_p[\Gamma]$. To show this, we recall that for each $\chi$ in $\Upsilon_{K, \mathcal{V}}^\diamond$ the field $K_\chi$ is a subfield of $N$. The element $1 \otimes z^\text{b}_{K_\chi}$ is thus equal to the image of the $\Z_p [\Gamma]$-basis $1\otimes \fz_N$ of $\Det_{\Z_p [\Gamma]} (D^\bullet_{N, S^\ast (N)})$
under the composite homomorphism
\begin{cdiagram}
\Det_{\Z_p [\Gamma]} (D^\bullet_{N, S^\ast (N)}) 
\arrow[rr, "i_{N / K_\chi}",twoheadrightarrow] & &
\Det_{\Z_p [\cG_{K_\chi}]} (D^\bullet_{K_\chi, S^\ast (K_\chi)}) 
\arrow{rrr}{\Theta^{r_{K_\chi}}_{K_\chi / k, S^\ast (K_\chi)}}
& & &
\Z_p \fL_{K_\chi}
\end{cdiagram}%
and so the equality (\ref{defining-property}) (with $E = K_\chi$) implies 
\begin{equation} \label{for-later-reference}
    ( \Theta^{r_{K_\chi}}_{K_\chi / k, S^\ast (K_\chi)} \circ i_{N / K_\chi})( \fz_N) = z^\mathrm{b}_{K_\chi} = q_{K_\chi} \cdot \varepsilon_{K_\chi / k}.
\end{equation}
This combines with the implication (b) $\Rightarrow$ (a) of Proposition~\ref{analztic-class-number-formula-argument} (with $K$ replaced by $K_\chi$ and $V$ by $V_{K_\chi}$) to imply that  $z^\mathrm{b}_{K_\chi}$ and $\varepsilon_{K_\chi / k}$ generate the same $\Z_p [\cG_{K_\chi}]$-submodule of $\Z_p \fL_{K_\chi}$. In particular, since this submodule is free as an $e_{K_\chi}\Z_p[\cG_{K_\chi}]$-module, we deduce that the element $e_{K_\chi}\!\cdot q_{K_\chi}$ is a unit in $e_{K_\chi}\Z_p[\cG_{K_\chi}]$. In addition, since $r_{S^\ast (K_\chi)} (\chi) = r_{K_\chi}$, the orthogonality relations of characters imply that $\chi (e_{K_\chi}) = 1$ and so  $\chi (q_N) = \chi (q_{K_\chi})\in \Z_p [\im \chi]^\times$. Since this is true for every $\chi\in \Upsilon_{K, \mathcal{V}}^\diamond$, and  $\Upsilon^\diamond_{K, \mathcal{V}}$ is stable under the action of $G_\Q$ on $\Gamma^\circ$, one therefore has
 \[
 \NN_{(\epsilon_{K, \mathcal{V}}\Z_p [\Gamma]) / \Z_p} ( \epsilon_{K, \mathcal{V}}\cdot q_N) = 
 {\prod}_{\chi \in \Upsilon^\diamond_{K, \mathcal{V}}} \chi ( q_N) 
 = {\prod}_{\chi \in [\Upsilon^\diamond_{K, \mathcal{V}}]_\QQ} \NN_{\Z_p [\im (\chi)] / \Z_p} ( \chi (q_N)) \in \Z_p^\times,
 \]
and hence, by Lemma~\ref{norm-map-Lemma}\,(a), that $\epsilon_{K, \mathcal{V}}\cdot q_N\in \epsilon_{K, \mathcal{V}}\Z_p[\Gamma]^\times$, as claimed.

This fact implies that $\epsilon_{K, \mathcal{V}} q_N^{-1} \fz_N$ is an $\epsilon_{K, \mathcal{V}}\Z_p[\Gamma]$-basis of $\epsilon_{K, \mathcal{V}}\cdot \Det_{\Z_p [\Gamma]} (D^\bullet_{N, S^\ast (N)})$. To prove the Proposition it therefore suffices, by \cite[Prop.\@ 2.5]{BKS2}, to verify for every  $\chi \in \Upsilon_{K, \mathcal{V}}^\diamond$ that the composite homomorphism 
\begin{align*}
e_\chi \C_p \Det_{\Z_p [\Gamma]} (D^\bullet_{N, S^\ast (N)}) 
& \xrightarrow{f_{N / K(\chi)}}
e_\chi \C_p \Det_{\Z_p [\cG_{K(\chi)}]} (D^\bullet_{K(\chi), S^\ast (N)}) \\
& \xrightarrow{\Theta^{r_\chi}_{K(\chi)/ k, S^\ast (N)}}
e_\chi \C_p \exprod^{r_\chi}_{\Z_p [\cG_{K(\chi)}]} U_{K(\chi), S^\ast (N)}
\end{align*}
sends $e_\chi q_N^{-1} \fz_N$ to $e_\chi \varepsilon^{V_\chi}_{K(\chi)/ k, S^\ast (N)}$. Here $V_\chi$ is any choice of subset of $S^\ast (N)$ of cardinality  $r_\chi \coloneqq r_{S^\ast (N)} (\chi)$ that only contains places that split completely in $K(\chi)$, and the map $f_{N / K(\chi)}$ appearing above is induced by the map $\nabla_{N / K (\chi), S^\ast (N)}$ from (\ref{def nabla map}) via scalar extension. \\
To verify this we set $W \coloneqq V_\chi\setminus V_{K_\chi}$ and $m \coloneqq |W|$, and claim that it suffices to prove
\begin{align}\label{etnc-needed-equality}
    (\Ord_{K(\chi), W} \circ \Theta^{r_\chi}_{K(\chi)/k,S^\ast(N)} \circ f_{N/K(\chi)})(e_\chi q_N^{-1}\fz_N) = (-1)^{mr_{K_\chi}}\cdot e_\chi \varepsilon_{K(\chi)/k, S^\ast(N)\setminus W, \varnothing}^{V_{K_\chi}}.
\end{align}
Indeed, since $\Ord_{K(\chi), W}$ is injective on the $e_{K(\chi)/k, S^\ast(N), \varnothing}$-component of $\bidual_{\ZZ_p[\cG_{K(\chi)}]}^{r_{\chi}} U_{K(\chi), S^\ast(N)}$ (as proved in \cite[Lem.\@ 5.1\,(i)]{Rub96}), this equality combines with the fact that 
\begin{equation*} \label{ord-relation}
\Ord_{K(\chi), W} (\varepsilon^{V_\chi}_{K(\chi)/ k, S^\ast (N), \emptyset}) = (-1)^{mr_{K_\chi}} \cdot \varepsilon^{V_{K_\chi}}_{K(\chi)/ k, S^\ast (N) \setminus W, \emptyset}
\end{equation*}
(as proved by Sano \cite[Prop.~3.6]{Sano}) to imply the claimed result. 

Turning now to the verification of (\ref{etnc-needed-equality}), we first note (\ref{for-later-reference}) implies that one has 
\begin{equation}\label{first last one}
 ( \Theta^{r_{K_\chi}}_{K_\chi / k, S^\ast (K_\chi)} \circ i_{N / K_\chi})( q_N^{-1} \fz_N) = \varepsilon_{K_\chi / k}.
\end{equation}
Using Lemma~\ref{commutative-diagrams-lemma}, we may therefore calculate that
\begin{align} \nonumber
    & (\Theta^{r_{K_\chi}}_{K(\chi) 
    ,S^\ast (K_\chi)} \circ \Delta_{K (\chi), S^\ast (N), S^\ast (K_\chi)} \circ f_{N / K (\chi)})(e_\chi q_N^{-1}\fz_N)\\ \nonumber
     = \, & ( \Theta^{r_{K_\chi}}_{K (\chi), S^\ast (K_\chi)} \circ \Delta_{K (\chi), S^\ast (N), S^\ast (K_\chi)}   \circ  f_{K_\chi / K (\chi)} \circ f_{N / K_\chi})(e_\chi q_N^{-1}\fz_N)\\  \nonumber
     = \, & ( \Theta^{r_{K_\chi}}_{K (\chi), S^\ast (K_\chi)} \circ  f_{K_\chi / K (\chi)} \circ \Delta_{K_\chi, S^\ast (N), S^\ast (K_\chi)}  \circ f_{N / K_\chi})(e_\chi q_N^{-1}\fz_N)\\  \nonumber
    =\, &(\NN_{K_\chi/K(\chi)}^{r_{K_\chi}}\circ\Theta^{r_\chi}_{K_\chi, S^\ast (K_\chi)} \circ i_{N/K_\chi})(e_\chi q_N^{-1}\fz_N)\\  \nonumber
    =\, &\NN_{K_\chi/K(\chi)}^{r_{K_\chi}}(e_\chi\varepsilon_{K_\chi / k})\\
    =\, &e_\chi\varepsilon_{K(\chi)/k, S^\ast (K_\chi),\emptyset}^{V_{K_\chi}}.  \label{last brick equality}
\end{align}
Here the first equality follows from the fact that $f_{N / K (\chi)} = f_{K_\chi / K (\chi)} \circ f_{N / K_\chi}$, the second from Lemma~\ref{commutative-diagrams-lemma}\,(d), the third from Lemma~\ref{commutative-diagrams-lemma}\,(a) and the definition of the map $i_{N / K_\chi}$, the fourth from (\ref{first last one}), and the last from the properties of Rubin--Stark elements. \\
To proceed, we set $W' = S^\ast (N) \setminus W$ and calculate
\begin{align*}
   & \phantom{=}\; (-1)^{m r_{K_\chi}}  \cdot (\Ord_{K (\chi),W} \circ \Theta^{r_\chi}_{K(\chi)/ k, S^\ast (N)} \circ f_{N / K(\chi)}) ( e_\chi q_N^{-1} \fz_N)\\
   &= (\Theta_{K(\chi), W'}^{r_{K_\chi}} \circ \Delta_{K(\chi), S^\ast (N), W'} \circ f_{N/K(\chi)})(e_\chi q_N^{-1}\fz_N)\\
   &= (\Theta_{K(\chi), W'}^{r_{K_\chi}} \circ \Delta_{K(\chi), S^\ast (K_\chi),W'} \circ \Delta_{K(\chi), S^\ast (N),S^\ast (K_\chi)} \circ f_{N/K (\chi)})(e_\chi q_N^{-1}\fz_N)\\
   &= \big ( \prod_{v \in W' \setminus S^\ast (K_\chi)} ( 1 - \chi ( \Frob_v)^{-1} ) \big) \cdot
   (\Theta_{K (\chi), S^\ast (K_\chi)}^{r_{K_\chi}} \circ \Delta_{K(\chi), S^\ast (N),S^\ast (K_\chi)} \circ f_{N/K (\chi)})(e_\chi q_N^{-1}\fz_N)\\
   &=  \big ( {\prod}_{v \in W' \setminus S^\ast (K_\chi)} ( 1 - \chi ( \Frob_v)^{-1} ) \big) \cdot e_\chi\varepsilon_{
    K(\chi)/k,S^\ast (K_\chi),\emptyset}^{V_{K_\chi}} \\
    & = e_\chi\varepsilon_{
    K(\chi)/k,W',\emptyset}^{V_{K_\chi}}.
\end{align*}
Here the first equality follows via an application of Lemma~\ref{commutative-diagrams-lemma}\,(b), the second by the definition of the respective -$\Delta$ homomorphisms, the third by appealing to Lemma~\ref{commutative-diagrams-lemma}\,(c), the fourth by (\ref{last brick equality}), and the last by properties of Rubin--Stark elements. 

This concludes the proof of (\ref{etnc-needed-equality}), and hence of Proposition~\ref{every-basic-RS-implies-etnc}. 
\end{proof}

\begin{rk}\label{comp bdss} The results of this section show that the conjecture formulated by Sano and three of the present authors in \cite{bdss} cannot be valid. To explain this we write $k^{\rm s}$ for the maximal abelian extension of $k$ in which all places in $S_\infty = S_\infty(k)$ split completely, $\Omega$ for  $\Omega^{S_\infty}_k$, $\mathcal{R}^{\rm s}$ for $\Z\llbracket\cG_{k^{\rm s}}\rrbracket$ and $\mathcal{A}^{\rm s}$ for the ideal $\varprojlim_{E \in \Omega}\mathrm{Ann}_{\Z[\cG_E]}(\mu_E)$ of $\mathcal{R}^{\rm s}$, where the limit is taken with respect to the natural projection maps $\Z[\cG_{E'}] \to \Z[\cG_E]$ for fields $E \subseteq E'$ in $\Omega$. Then the central conjecture (Conjecture 2.5) of loc.\@ cit.\@ predicts
 the $\mathcal{R}^{\rm s}$-module $\ES_k$ of Euler systems defined in \cite[Def.~2.3]{bdss} is such that $\mathcal{A}^{\rm s}\cdot \ES_k \subseteq \mathcal{A}^{\rm s}\cdot \varrho^{S_\infty}(\im (\Theta_k))$. However, the validity of this prediction is not consistent with the result of Proposition~\ref{basic-Euler-systems} for two reasons: firstly, the intersection $\mathcal{A}^{\rm s}\cdot (\ES_k\cap \ES_k^{S_\infty}(\ZZ)^{\mathrm{iso}})$ is in general non-zero (and, in this regard, recall Lemma~\ref{ivc versus local}) and, secondly, the final observation in Remark~\ref{scarcity remarks}(d) implies that, in general, there are systems in $\mathcal{A}^{\rm s}\cdot \ES_k$ that are not  symmetric. 
\end{rk}

\section{The theory of Euler limits}\label{elimits}

In this section we develop a general theory of `Euler limits' that is central to our analysis of Euler systems.  Throughout, we fix a number field $k$ and, for convenience, use the abbreviations 
\[ \Omega = \Omega_k,\quad  S_\infty = S_\infty (k)\quad\text{and}\quad S_p =  S_p (k).\]
We also fix a subring $\cR$ of $\C$ with field of fractions $\cQ$.

\subsection{Definitions and examples}

\begin{definition}
Given a subset $\cX$ of $\Omega$, we define a \emph{triangular system} on $\cX$ to consist of the following data for every pair of fields $E$ and $F$ in $\cX$ with $E \subseteq F$:
\begin{itemize}
\item[$\bullet$] an $\cR [\cG_F]$-module $M_{F,E}$; 
\item[$\bullet$] maps of $\cR[\cG_F]$-modules  $\rho_{F / E} \: M_{F,F} \to M_{F, E}$ and $j_{F / E} \: M_{E,E} \to M_{F, E}$.
\end{itemize}
If $E = F$, we abbreviate $M_{F,E}$ to $M_E$ and refer to it as a \emph{diagonal term} of the system. 
\end{definition}

\begin{bspe1} \label{triangular-system-examples} Triangular systems arise naturally in several ways. \

\begin{liste}
\item \textit{(Projective systems)} Each projective system $(\widetilde{M}_F, \widetilde{\rho}_{F / E})_{F \in \cX}$ of $\cR [\cG_F]$-modules gives rise to a triangular system by setting  $M_{F, E} \coloneqq \widetilde{M}_E$ for each $E \subseteq F$ and taking $\rho_{ F / E}$ to be $\widetilde{\rho}_{F / E}$ and $j_{F / E}$ to be the identity map on $\widetilde{M}_E$.

\item \textit{(Biduals of units)} For any rank function $\bm{r}$ one obtains a triangular system on $\cX$ by setting, for each $E$ and $F$ in $\cX$ with $E\subseteq F$,
\[
M_{F, E} \coloneqq 
\begin{cases}
\cR \fL_E^{\bm{r} (E)} \quad & \text{ if } E = F, \\
\bigoplus_{i = 0}^\infty \cQ \exprod^i_{\Z [\cG_E]} \cO_{E, S(F)}^\times \quad & \text{ if } E \neq F, 
\end{cases}
\]
and by taking $\rho_{F / E}$ to be the map induced by the norm $ \cO_{F, S(F)}^\times \to  \cO_{E, S(F)}^\times$ and $j_{F/E}$ to be the (injective) map induced by the inclusion $\cO_{E, S(E)}^\times \subseteq \cO_{E, S(F)}^\times$. 
\item \textit{(Exterior powers of duals)} 
For any $t \in \N_0$, and any pair of fields $E, F \in \cX$ with $E \subseteq F$, the composite homomorphism  
\[ (\bigO_{F, S(F)}^\times)^\ast \to \Hom_{\ZZ[\cG_F]}(\bigO_{E, S(E)}^\times,\ZZ[\cG_F]^{\gal{F}{E}}) \to (\bigO_{E, S(E)}^\times)^\ast,\]
where the first map is restriction and the second is induced by sending  $\NN_{\gal{F}{E}}$ to $1$, induces a homomorphism of $\ZZ[\cG_F]$-modules 
$\Phi^t_{F / E} \: \exprod^t_{\Z [\cG_F]} (\bigO_{F, S(F)}^\times)^\ast \to \exprod^t_{\Z [\cG_E]} (\bigO_{E, S(E)}^\times)^\ast$.
Given a rank function $\bm{t}$, we therefore obtain a triangular system by setting
\[
M_{F, E} \coloneqq \begin{cases}
\exprod^{\bm{t}(E)}_{\cR [\cG_E]} (\cR \cO^\times_{E, S(E)})^\ast \quad &\text{ if } E = F, \\
\bigoplus_{i = 0}^\infty \exprod^{i}_{\cR [\cG_E]} (\cR \cO^\times_{E, S(E)})^\ast \quad &\text{ if } E \neq F,
\end{cases}
\]
and by taking $\rho_{F / E}$ to be $\Phi^{\bm{t}(E)}_{F / E}$ and $j_{F / E}$ to be the natural inclusion map. 
\end{liste}
\end{bspe1}

We next associate to each triangular system a natural notion of limit.  

\begin{definition}\label{S euler def} Fix a finite set $\Pi$ of places of $k$ and, for each field $E$ in $\Omega$, set 
\[ \Pi(E) \coloneqq \Pi\cup S(E) = \Pi \cup S_\ram( E / k).\]
(Note, in particular, that $\Pi(E) = S(E)$ if $\Pi = \emptyset$ and  $\Pi(E) = S^\ast(E)$ if $\Pi = S_\infty$.) 
\begin{liste}
\item The $\Pi$-relative \emph{Euler factor} for an extension $F / E$ of fields in $\Omega$ is the element 
\[
P_{F / E, \Pi} \coloneqq 
{\prod}_{v \in S (F) \setminus \Pi(E)} (1 - \Frob_v^{-1}) \in \Z [\cG_E].\]
(In particular, if $\Pi = \emptyset$, then $P_{F / E, \Pi}$ is the standard factor $\prod_{v \in S (F) \setminus S(E)} (1 - \Frob_v^{-1})$ that also occurs in the definition of an Euler system.)  
\item Let $\calT \coloneqq \{(M_{F, E}, \rho_{F / E}, j_{F / E})\}_{E, F \in \cX}$ be a triangular system on a subset $\cX$ of $\Omega$. Then the  $\Pi$-relative \emph{Euler limit}\index{Euler limit} of $\calT$ is the $\cR \llbracket G_k^{\mathrm{ab}} \rrbracket$-submodule  
\[
\Big \{ (m_F)_F \in {\prod}_{F \in \cX} M_{F}
\; \big | \;  \rho_{F / E} (m_F) = P_{F / E, \Pi} \cdot j_{F / E}(m_E)
\text{ for all } E, F \in \cX \text{ with } E \subseteq F
\Big \}\]
of the direct product $\prod_{F \in \cX} M_{F}$. 
We denote this submodule by $\eullim{}{\Pi}\calT$ or, more simply (when data defining the system $\calT$ is clear from context), by $
\eullim{F \in \cX}{\Pi}
M_{F}$.
\end{liste}
\end{definition}

In the rest of this section we fix a set of places $\Pi$ as above. In addition, for a (not necessarily finite) extension of fields $E \subseteq F$ that are abelian over $k$, we write %
\[ \pi_{F / E} \: \C \llbracket \cG_F \rrbracket \to \C \llbracket \cG_E \rrbracket\]
for the  natural restriction map (and we use the same notation to denote the restriction of $\pi_{F/E}$ to $R\llbracket \cG_F\rrbracket$ for any subring $R$ of $\C$). 

\begin{bspe1}\label{es remark} The concept of Euler limit incorporates several natural constructions.\ 
\begin{liste}
\item \textit{(Euler systems)} 
Take $\cX = \Omega$ and $\Pi = \emptyset$ and fix a rank function $\bm{r}$ for $k$. Then the module of Euler systems $\ES^{\bm{r},\cX}_k (\cR)$ (specified in Definition \ref{definition-integral-Euler-systems}) coincides with the \textcolor{black}{submodule of the} Euler limit of the triangular system in  Example~\ref{triangular-system-examples}\,(b) \textcolor{black}{that comprises all elements $(m_F)_{F\in \cX}$ with the property that 
$e_{\bm{1}_F}m_F= 0$ for every $F \in \cX\cap \Omega^{\mathrm{exc}}$. }%
\item \label{rk reduction rem} \textit{(Perrin-Riou functionals)} 
Take $\cX  = \Omega$ and $\Pi  = \emptyset$ and fix a rank function $\bm{t}$ for $k$. We define the module of \textit{Perrin-Riou functionals} $\mathrm{PR}_k^{\bm{t}} (\cR)$ to be  the Euler limit of the triangular system defined in Example~\ref{triangular-system-examples}\,(c) with transition maps $\rho_{F / E} \coloneqq P_{F / E, \emptyset} \cdot \Phi_{F / E}^{\bm{t} (E)}$. Then, for any other rank function $\bm{s}$ for $k$ with $\bm{s}\geq \bm{t}$, the argument of Perrin-Riou in \cite[Lem.\@ 1.2.3]{PerrinRiou98} (see also \cite[\S\,6]{Rub96}) shows that the assignment 
$((\theta_E)_E,(c_E)_E) \mapsto (\theta_E(c_E))_E$
 induces a homomorphism of $\cR\llbracket G_k^{\mathrm{ab}}\rrbracket$-modules 
\[
\mathrm{PR}_k^{\bm{t}} ( \cR) \otimes_{\cR\llbracket G_k^{\mathrm{ab}}\rrbracket} \ES^{\bm{s}}_k ( \cR) \to \ES^{\bm{s} - \bm{t}}_k (\cR).\]
This method of `rank-reduction' has been used in the literature to obtain useful classical (rank-one) Euler systems from higher-rank Euler systems.  
\item \textit{(Inverse limits)} 
Each projective system $(M_F, \rho_{F / E})_{F \in \cX}$ of $\cR[\cG_F]$-modules gives rise to an associated projective system $(M'_F, \rho'_{F / E})_{F \in \cX}$ in which one has $M'_F \coloneqq M_F$ and  $\rho'_{F/E} \coloneqq P_{F / E, \Pi} \cdot\rho_{F / E}$ for all $E \subseteq F$. The identity map on $\prod_{F \in \cX} M_F$ restricts to induce an embedding of the inverse limit $\varprojlim_{F \in \cX} M_F$ into the Euler limit $\eullim{F \in \cX}{\Pi} M'_F$ of the triangular system associated to $(M'_F, \rho'_{F / E})_{F \in \cX}$ by  Example~\ref{triangular-system-examples}\,(a). 
\item \textit{(Idempotent projections)} 
Let $(M_F, \rho_{F / E})_{F \in \cX}$ be a projective system of $\Z$-torsion-free $\cR[\cG_F]$-modules and
recall the idempotents 
\[ e_{F, \Pi} \coloneqq e_{F, S_\infty\cup \Pi(F), r_F}\]
from Definition \ref{Rubin--Stark-definition}. (We note, in particular, that if $\Pi=\emptyset$, then $e_{F,\Pi}$ is the idempotent $e_{F}$ defined in (\ref{eK def}).) We define a triangular system $(M'_{F, E}, \rho'_{F / E}, j_{F / E})_{F \in \cX}$ by setting
\[
M'_{F, E} \coloneqq 
\begin{cases}
e_{F, \Pi} M_F \quad & \text{ if } F = E, \\
\cQ \otimes_\cR M_{E} & \text{ if } F \neq E,
\end{cases}
\]
and $\rho'_{F / E}\coloneqq P_{F / E, \Pi} \cdot \rho_{F / E}$, and taking $j_{F / E}$ to be the natural inclusion of $e_{E, \Pi} M_E$ into $\cQ \otimes_{\cR} M_E$. We then write $\eullim{F \in \cX}{\Pi} e_{F, S} M_F$ for the Euler limit of the triangular system $(M'_{F, E}, \rho'_{F / E}, j_{F / E})_{F \in \cX}$. For $E$ and $F$ in $\Omega$ with $E \subseteq F$, a straightforward calculation shows   
$P_{F / E,\Pi} \cdot ( \pi_{F / E} (e_{F, \Pi}) - e_{E, \Pi}) = 0$. This fact implies that the assignment 
\[ 
\varprojlim_{F \in \cX} M_F \to\eullim{F \in \cX}{\Pi} e_{F, \Pi} M_F, \quad (m_F)_F \mapsto (e_{F,\Pi} m_F)_F
\]
is a well-defined map of $\cR \llbracket G_k^{\mathrm{ab}} \rrbracket$-modules.  
\end{liste}
\end{bspe1}

In the next two sections we shall prove the existence of natural exact sequences relating respectively to `restriction' and `localisation' functors that are associated to the Euler limits of triangular systems. 

\subsection{The $p$-adic restriction exact sequence}

\subsubsection{Statement of the $p$-adic restriction sequence}
In the sequel we will consider subsets $\cX$ of $\Omega$ that are `large' in the following sense.

\begin{hypothesis}[Closure Hypothesis] \label{hyp X} 
    Fix a subset $\mathcal{V}$ of $\mathcal{P}(S_\infty)$ and let $\cX$ be a subset of $\Omega^\cV$ that has  the following two closure properties:
    \begin{romanliste}
        \item \textit{(Composita)} 
        If $K \in \cX$ and $E \in \Omega^{S_\infty}$, then   $KE \in \cX$.
        \item \textit{(Subfields)} 
        If $K \in \cX$, then for every  intermediate field $L$ of $K / k$ with $V_L \in \cV$, one has 
       \begin{itemize}
           \item $L \in \cX$ if $L / k$ is ramified at a finite place,
           \item $k(\p)L \in \cX$ for almost all $\p \in S_\fin(k)$ if $L / k$ is unramified at all finite places and $k( \p) \in \Omega_k$. 
       \end{itemize}
    \end{romanliste}
\end{hypothesis}

\begin{bspe1}\label{cX examples} 
The following two families of extensions satisfy Hypothesis \ref{hyp X} (for suitable sets $\cV$) and will play an important role in the sequel. 
\begin{liste}
\item For any $\cV\subseteq \mathcal{P} (S_\infty)$, the set $\cX = \Omega^\cV$ satisfies Hypothesis \ref{hyp X} with respect to $\cV$.
    \item For any totally real field $k$, the set $\cX$ of all finite abelian CM extensions of $k$ that are ramified at some finite prime satisfies Hypothesis \ref{hyp X} with respect to  $\cV = \{ \emptyset \}$. 
\end{liste}
\end{bspe1}

For each finite set of places $\Sigma$ of $k$, we define a subset of $\cX$ by setting  
\begin{equation}\label{cX_Sigma def} \cX_\Sigma \coloneqq \{E \in \cX: \Sigma\subseteq S(E)\} \end{equation}
and then consider the natural `restriction' map on Euler limits
\begin{equation}
\res^\Pi_\Sigma \: \eullim{F \in \cX}{\Pi} M_F \to \eullim{F \in \cX_\Sigma}{\Pi} M_F, \quad (m_F)_{F\in \cX} \mapsto (m_F)_{F\in \cX_\Sigma}.\end{equation}

\begin{rk} For $E \in \cX$ and $F \in \cX_\Sigma$ with $E \subseteq F$ (if such $F$ exists), the defining relation of the Euler limit relative to $\Pi \cup \Sigma$ has no Euler factor for the extension $E / F$. In many cases, therefore, the limit $\eullim{F \in \cX_\Sigma}{\Pi}M_{F}$ can be regarded as a submodule of $ \eullim{F \in \cX}{\Pi \cup \Sigma} M_{F}$.
\end{rk}

In the next result we shall establish an important property of the maps $\res^\Pi_\Sigma$ for triangular systems that satisfy the following natural hypothesis.

\begin{hypothesis} \label{inj-hypothesis}
Fix $p\in \mathbb{P}$. Then the data $\{ (M_{F, E}, \rho_{F / E}, j_{F / E}) \}_{E, F \in \cX}$ constitutes a triangular system in which the transition maps $j_{F / E}$ are injective (and will be suppressed in the notation) and, for all pairs  of fields $E \subseteq F$ contained in $\cX$, there exists an injective homomorphism of $\cG_E$-modules 
    $\iota_{F / E} \: M_E \to  M_{F}^{\Gal(F/E)}$,
    that has both of the  following properties: 
\begin{romanliste}
    \item the restriction of the composite $\iota_{F / E} \circ \rho_{F / E}$ to the full pre-image $M'_F$ in $M_{F}$ of $M_{E}$ under $\rho_{F/E}$
    factors through multiplication by $\NN_{\gal{F}{E}}$, 
    \item for every $E \in \cX$ there exists an integer $N$, depending only on $E$, such that, for every integer $t$ and for every $m$ in $M_E$, one has $m\in p^{t}\cdot M_E$ if $\iota_{F/E}(m) \in p^{t+N}\cdot M_F$.  
\end{romanliste}
\end{hypothesis}

\begin{bspe1}\label{inj examples} Hypothesis \ref{inj-hypothesis} is satisfied in several natural cases with $\mathcal{R}$
 a Dedekind domain and $p$ a prime with $p \notin \cR^\times$. For example, one may take $\cR$ to be $\Z_p$ or $\Z_\mathcal{S}$ for a finite subset $\cS$ of $\mathbb{P}\setminus \{p\}$. 
\begin{liste}
\item Hypothesis \ref{inj-hypothesis}\,(ii) is satisfied in this case if, for each $E$, the module $M_E$ is $\Z$-torsion-free and there exists a natural number $N$, depending only on $E$, such that, for every $F$, the $p$-part of the order of  $\coker(\iota_{F/E})_{\mathrm{tor}}$ divides $p^N$. To justify this, set $\iota \coloneqq \iota_{F/E}$ and note that if $\iota(m) = p^t \cdot y$ for some integer $t \geq 0$ and $y\in M_F$, then the image of $y$ in $\coker(\iota)$ is annihilated by $p^t$, and hence also, by assumption, by $c p^N$ for some natural number $c$ that is coprime with $p$. If $z$ is the (unique) element of $M_E$ with $\iota(z) = c p^N \cdot y$, then $c p^N\cdot m = p^t \cdot z \in p^t M_E$. Thus, because $M_E$ is $\Z$-torsion-free and $c$ is a unit modulo $p^{t - N}$, we conclude that $m$ belongs to $p^{t - N} M_E$, as required. 
\item The triangular system that is associated (via Example~\ref{triangular-system-examples}\,(a)) to the canonical projective system $(\cR [\cG_F], \pi_{F / E})_{F \in \cX}$ satisfies Hypothesis \ref{inj-hypothesis} with each $\iota_{F / E}$ taken to be induced by the map $\cR [\cG_E] \to \cR [\cG_F]$ sending each $a$ to $\NN_{\gal{F}{E}}\cdot\tilde a$ for any (and therefore every) $\tilde a\in \cR [\cG_F]$ with $\pi_{F/E}(\tilde a) = a$. 
For this choice, the validity of (i) is clear and the validity of (ii) follows from Example~\ref{inj examples}\,(a) and the fact that $\coker (\iota_{F / E})_\tor$ vanishes.
\item Lemma~\ref{lattice lemma2}\,(c) (see also Remark~\ref{lattice lemma2 S-version rk}) shows that the triangular system of biduals in Example~\ref{triangular-system-examples}\,(b) satisfies Hypothesis \ref{inj-hypothesis} with $\iota_{F/ E}$ taken to be the map $\nu_{F / E}$ from Lemma~\ref{lattice lemma2}. 
\end{liste}
\end{bspe1}

We write $s$ for the greatest (non-negative) integer such that $k$ contains a primitive $p^s$-th root of unity. We then define a field  
\begin{equation}\label{angle field} k \langle p \rangle \coloneqq k (\mu_{p^{s + 1}}, (\cO_k^\times)^{1 / p^{\min\{ s, 1\}}})\end{equation}
and, for each set of places $\Sigma$ of $k$, a subfield of $k \langle p \rangle$ by setting 
\begin{equation}\label{k_{p} def} k_{\Sigma} \langle p \rangle  \coloneqq \text{composite of} \,\,  \big \{E : k\subseteq E\subseteq k \langle p \rangle, E/k\,\,\text{is abelian}, 
\, \{\p \in S_p\setminus \Sigma: \cG_{E,\p} = 1\}\not=\emptyset \big \}.
\end{equation}
We can now state the main result of this subsection. 

\begin{thm}[The $p$-adic restriction sequence] \label{p adic restriction sequence}
    Fix the following data: 
\begin{itemize}
\item $p\in \mathbb{P}\setminus \{2\}$,
    \item a subset $\cX$ of $\Omega$ that satisfies Hypothesis \ref{hyp X}, 
    \item a triangular system $\{ (M_{F, E}, \rho_{F / E}, j_{F / E})\}_{E, F \in \cX}$ that satisfies Hypothesis \ref{inj-hypothesis} with respect to $p$ and is such that each diagonal term $M_E$ is a $\Z_p [\cG_E]$-lattice,
    \item finite subsets $\Pi$ and $\Sigma \neq \emptyset$ of $S_\fin (k)$.
\end{itemize}
    Then there exists an exact sequence of $\Z_p \llbracket G_k^{\mathrm{ab}} \rrbracket$-modules 
    \[ 
    0 \to \eullim{F \in \cX}{\Pi}( M_F^{\gal{k^{\mathrm{ab}}}{k_{\Pi (F)}\langle p \rangle}}) \xrightarrow{\,\,\subseteq\,\,} 
     \eullim{F \in \cX}{\Pi} M_F \xrightarrow{\overline{\res}^\Pi_{\Sigma}}  \eullim{F \in \cX_{\Sigma}}{\Pi} \big 
     ( M_F / M_F^{\gal{k^\mathrm{ab}}{ k_{\Pi (F)} \langle p \rangle}}\big ),\]
in which $\overline{\res}^\Pi_{\Sigma}$ is induced by the restriction map $\res^\Pi_{\Sigma}$.
\end{thm}

\begin{rk} The result of Theorem~\ref{p adic restriction sequence} is equivalent to asserting that $\ker(\overline{\res}^\Pi_{\Sigma})$  is independent of the choice of $\Sigma$. In the stated generality, however, it seems more difficult to obtain concrete information about the cokernel of $\overline{\res}^\Pi_{\Sigma}$. 
\end{rk}

\begin{rk} Our methods also prove a variant of Theorem~\ref{p adic restriction sequence} for $p = 2$, see Proposition~\ref{p adic restriction intermediate step} and Remark~\ref{p = 2 remark restriction sequence} for more details.
\end{rk}

We next record an important consequence of Theorem~\ref{p adic restriction sequence} that will be derived in \S\,\ref{grs proof section}.  

\begin{cor}[The global restriction sequence] \label{global restriction sequence}
    Fix the following data:  
\begin{itemize}
\item a finite subset $\cS$ of $\mathbb{P}$,
    \item a subset $\cX$ of $\Omega$ that satisfies Hypothesis \ref{hyp X}, 
    \item a triangular system $\{ (M_{F, E}, \rho_{F / E}, j_{F / E})\}_{E, F \in \cX}$ for which  each diagonal term $M_E$ is a $\Z_\cS [\cG_E]$-lattice and Hypothesis \ref{inj-hypothesis} is satisfied for all but finitely many $p$,
    \item finite subsets $\Pi$ and $\Sigma \neq \emptyset$ of $S_\fin (k)$.
\end{itemize}
Then there exists an exact sequence of $\Z_\cS\llbracket G_k^{\mathrm{ab}} \rrbracket$-modules 
    \[ 
    0 \to \bigl(\eullim{F \in \cX}{\Pi} M_F\bigr)^{G_k^\mathrm{ab}} \xrightarrow{\,\,\subseteq\,\,}   
     \eullim{F \in \cX}{\Pi} M_F \xrightarrow{\overline{\res}^\Pi_{\Sigma}}   \eullim{F \in \cX_{\Sigma}}{\Pi} \big ( M_F / M_F^{\cG_F}\big ),
    \]
    in which $\overline{\res}^\Pi_{\Sigma}$ is induced by the restriction map $\res^\Pi_{\Sigma}$.
\end{cor}

The proof of Theorem~\ref{p adic restriction sequence} itself will be given in \S\,\ref{prs proof section} after we establish several necessary preliminary results. In fact, since this argument is rather involved, for the convenience of the reader we now end this subsection by outlining the general strategy to be used.

\begin{rk}\label{strategy remark} To summarise the key ideas of the proof of Theorem~\ref{p adic restriction sequence}, we discuss the triangular system $(\ZZ_p [\cG_E], \pi_{F / E}, \nu_{F / E})$ described in Example~\ref{inj examples}\,(b) and also assume that $\Pi = \emptyset$ and $\Sigma = \{ \p \}$ for some fixed prime $\p$ of $k$. We stress, however, that, modulo some relatively minor technical difficulties, the general case uses exactly the same approach.\\
We now fix an element 
$(m_F)_F$ of $\eullim{F \in \cX}{\emptyset} \ZZ_p[\cG_F]$ for which one has $m_F = 0$ for all $F \in \cX_\Sigma$ (so that $(m_F)_F \in \ker(\overline{\res}_\Sigma^\emptyset)$) and explain how we show this restricts the values of $m_F$ for all $F \in \cX$. To do this, we fix $F \in \cX$ and note that for every extension $E$ of $k$ for which $EF \in \cX$, one has $m_{FEk(\p)} = 0$ and hence also, by the Euler system relations, that  
\[
0 = \pi_{FE k(\p) / FE} (0) = \pi_{FE k(\p) / FE} (  m_{FEk(\p)}) = (1 - \Frob_\p^{-1}) \cdot  m_{FE}.
\]
Writing $H_{E,\p}$ for the subgroup of $\cG_{FE}$ generated by $\Frob_\p$, one thus has $m_{FE}\in \ZZ_p[\cG_{FE}]\cdot \NN_{H_{E,\p}}$ since $\ZZ_p[\cG_{FE}]$ is a cohomologically-trivial $H_{E,\p}$-module, and hence also (by the Euler system relations) that  
\[
P_{FE / F, \emptyset} \cdot m_F = \pi_{FE / F} ( m_{F E}) \in n_{E,\p}\cdot \ZZ_p[\cG_F]\]
with $n_{E,\p} \coloneqq | H_{E,\p}\cap \gal{FE}{F} |$. This containment implies that the value $m_F$ would be strongly restricted by the existence of a family of extensions $E$ of $k$ that are each ramified at very few primes (so that the terms $P_{FE / F, \emptyset}$ can be controlled) and for which $n_{E,\p}$ can nevertheless be divisible by an arbitrarily large power of $p$. In particular,  as an important part of the proof of Theorem~\ref{p adic restriction sequence}, we shall apply results of Neukirch regarding classical embedding problems to construct a family of extensions $E$ of $k$ that have carefully prescribed local behaviour (for more details see Remark~\ref{neukirch difficulty rem} and Lemma~\ref{induction-step-revisited} below).\end{rk}

\subsubsection{Constructing extensions with prescribed local behaviour}
\label{Neukirch section}

A key step in the proof of Theorem~\ref{p adic restriction sequence} will be provided by the following result concerning the existence of cyclic extensions of number fields with prescribed local behaviour. 

\begin{prop} \label{ultimate neukirch result} Fix the following data:
\begin{itemize}
    \item $p \in \mathbb{P}$, 
    \item a finite abelian extension $K$ of $k$,
    \item a finite set $\cT\subset S_\fin(k)$ with  $S_p \subseteq \cT$ and such that the $\cT$-class group of $k$ is trivial. 
\end{itemize}
Set $h \coloneqq \ord_p (|\cl_k|)$ and write $s$ for the greatest integer such that $k$ contains a primitive $p^s$-th root of unity. Fix a natural number $n$ with $n > s$, define a finite Galois extension of $k$ by
\[
\textcolor{black}{k (p, \cT, n) \coloneqq 
\begin{cases}
 k (\mu_{p^{2n + h + 2s}}, ( \cO_{k, \cT}^\times)^{1 / p^s})
\quad & \text{ if } p \text{ is odd},\\
 k (\mu_{p^{2n + h + 2s + 1}}, ( \cO_{k, \cT}^\times)^{1 / p^{s + 1}})
& \text{ if } p = 2,
\end{cases}}
\]
and set 
\[ L = L (p, \cT, n) \coloneqq K \cap k (p, \cT, n).\]
Then, for any element $\sigma$ of $\cG_K$ whose restriction to 
$L$
is trivial, there exists a finite Galois extension $E = E(p,\cT,n,\sigma)$ of $k$ that has all of the following properties:
\begin{enumerate}[label=(\roman*)]
    \item $\cG_E$ is a cyclic $p$-group,
    \item every $v \in \cT$ is unramified in $E$ and the order of  $\Frob_v$ on $E$ is at least $p^{n}$,
    \item $S(E)$ contains no infinite place and at most two finite places, 
    \item for each $v \in S(E)$ the following claims are valid:
    \begin{itemize}
        \item[(a)] the inertia subgroup of $v$ in $\cG_E$ has order at least $p^{n}$,
        \item[(b)] $v$ is unramified in $K$ and the restriction of $\Frob_v$ to $K$ is equal to $\sigma$.
    \end{itemize}
\end{enumerate}
\end{prop}

\begin{rk}\label{neukirch difficulty rem}
    To give an idea of the essential difficulty involved in proving this result, we assume (for simplicity) that $p > 2$, $k = \Q$, $K= \Q (\mu_f)$ for $f \in \N$ and $\cT = \{ \ell \}$ for a prime $\ell$. Then, for any element $\sigma$ of $\cG_{\Q (\mu_f)}$ that acts trivially on  $\Q (\mu_{p^n})$, one can fix an integer $a$ that is prime to $f$, congruent to $1$ modulo $p^n$ and such that $\sigma (\zeta) = \zeta^a$ for all $\zeta \in \mu_f$. Next, by applying Dirichlet's theorem on primes in arithmetic progressions, one derives the existence of infinitely many primes $q$ that are congruent to $a$ modulo $fp^n$ (so that $p^n$ divides $(q - 1)$) and are such that $\Q (\mu_q)$ contains a cyclic extension $E$ of degree $p^n$ that satisfies condition (i), (iii) and (iv) of Proposition~\ref{ultimate neukirch result}. However, in this setting, condition (ii) of the latter result requires that the order of $\ell$ modulo $q$ is at least $p^n$, and this is  difficult to ensure whilst also simultaneously controlling the ramification in $E$. 
    Indeed, while the first question can be addressed by simply using the Grunwald--Wang theorem, the second question requires a finer use of the theory of embedding problems, as discussed by Neukirch in \cite[\S\,8]{Neukirch73}.
    In fact, the verification of the hypotheses needed for the application of Neukirch's general theory is called a `difficult arithmetic problem' in \cite[bottom of p.~107]{Neukirch73}, and  we are only able to satisfactorily overcome this problem
    by using the Cebotarev density theorem in a way that develops the idea described by Neukirch in \cite[bottom of p.~108]{Neukirch73}.  
\end{rk}

We now fix $N\in \N$ and set 
\[ q \coloneqq p^N.\]
Then, to help the reader, before proving Proposition~\ref{ultimate neukirch result} we review several results obtained by Neukirch in \cite[\S\,7 and \S\,8]{Neukirch73} concerning the possible existence of cyclic extensions of $k$ of degree $q$ that realise given local extensions. To do this, we set $G_k \coloneqq \gal{\overline{\Q}}{k}$ and, for each place $v$ of $k$, write $G_{k_v}$ for the decomposition subgroup in $G_k$ of some fixed place of $\overline{\Q}$ above $v$. Then, for a given collection of places $v$ together with a morphism $\varphi_v \: G_{k_v} \to \nZ{q}$ for each such $v$, Neukirch considers the problem of determining whether there exists a surjective morphism $G_k \to \nZ{q}$ whose restriction to $G_{k_v}$ agrees with $\varphi_v$ for all $v$ in the given set.  \\
To be more precise, we fix a set of places $\cM$ of $k$ that contains both $S_\infty$ and $ S_p $ and a finite subset $\cT$ of $\cM$. We then define 
\[
\Lambda (\cM, \cT) \coloneqq {\prod}_{v \in \cT} \{ 0 \} \times {\prod}_{v \in \cM \setminus \cT}^{'} \Hom ( G_{k_v}, \nZ{q} ) \times {\prod}_{v \not \in \cM} \Hom_{\mathrm{nr}} (G_{k_v}, \nZ{q})
\]
where $\Hom_{\mathrm{nr}} (G_{k_v}, \nZ{q})$ is the subset of $\Hom (G_{k_v}, \nZ{q})$
comprising all morphisms that are `non-ramified' at $v$ (that is, vanish on the inertia subgroup of $v$) and
$\prod'$ is the restricted product (that is, the subset of the cartesian product comprising all elements $(\varphi_v)_v$ with the property that $\varphi_v$ is non-ramified for all but finitely many $v$). We then consider the group 
\[
\Delta (k, \nZ{q})^\cM_\cT \coloneqq
\coker \big ( \Hom (G_k, \nZ{q} ) \to \big ( {\prod}_v' \Hom (G_{k_v}, \nZ{q} ) \big) / \Lambda (\cM, \cT) \big ) 
\]
in which the restricted product runs over all places $v$ of $k$ and the unlabelled arrow is induced by the canonical diagonal restriction maps. 

We now assume to be given a morphism $\phi_0 \: G_k \to \nZ{q}$ that is unramified outside $\cM$ and, for each $v \in \cT$, a morphism $\varphi_v \: G_{k_v} \to \nZ{q}$. We  write $\phi_{0, v}$ for the restriction of $\phi_0$ to $\G_{k_v}$ and define $\eta = \eta(\phi_0, (\varphi_v)_{v \in \cT})$ to be the class in $\Delta (k, \nZ{q})^\cM_\cT$ of the element $(\eta_v)_{v}$ of ${\prod}_v' \Hom (G_{k_v}, \nZ{q} )$ that is defined by setting
\begin{equation}\label{obstruction class def}
\eta_v \coloneqq \begin{cases}
\phi_{0, v} - \varphi_v \quad & \text{ if } v \in \cT,\\
\phi_{0, v} & \text{ if } v \not \in \cT.
\end{cases}
\end{equation}

\begin{rk}\label{interpret obstruction}
The above class $\eta$ is a natural `obstruction' to solving the given existence problem since, as is proved in \cite[Satz~(2.4)]{Neukirch73} (cf.\@ also \cite[Bsp.~3 in \S\,3]{Neukirch73}), it vanishes if and only if there exists a morphism $\phi \: G_k \to \nZ{q}$ that is unramified outside $\cM$ and, for all $v \in \cT$, coincides with $\varphi_v$ when restricted to $G_{k_v}$. 
\end{rk}

To investigate the vanishing of $\eta$, it is convenient to reinterpret $\Delta (k, \nZ{q})^\cM_\cT$ in the manner described by the following result. To state  the result, we use the subgroup of $k^\times$ specified by 
\[
U^\cM_\cT (q) \coloneqq \{ a \in k^\times \mid \forall v \not \in \cM : \ord_v (a) \equiv 0\!\! \mod q, \, \forall v \in \cM \setminus \cT : a \in (k_v^\times)^q \}. 
\]

\begin{lem}[{\cite[Satz (7.1)]{Neukirch73}}] \label{neukirch obstruction lemma 1}
There is a canonical isomorphism
\[
\Delta (k, \nZ{q})^\cM_\cT \xrightarrow{\simeq} \Hom_\Z ( U^\cM_\cT (q) / (k^\times)^q, \nZ{q} )
\]
that can be explicitly described as follows. Let $f$ be an element of $\Delta (k, \nZ{q})^\cM_\cT$ represented by a family $(f_v)_v \in {\prod}_v' \Hom (G_{k_v}, \nZ{q} )$. Then the above isomorphism sends $f$ to the map 
\[
U^\cM_\cT (q) / (k^\times)^q \to \nZ{q}, \quad 
a \mapsto {\sum}_v f_v ( \rec_v (a)),
\]
where $\rec_v \: k_v^\times \to G_{k_v}^\mathrm{ab}$ is the local reciprocity map and $G_{k_v}^\mathrm{ab}$  the abelianisation of $G_{k_v}$. 
\end{lem}

By combining Kummer theory with the above result, we next give a convenient Galois-theoretic description of $\Delta (k, \nZ{q})^\cM_\cT$. To do this, we set
\[
k_1\coloneqq k (\mu_q) 
\quad \text{ and } \quad 
k_2 = k^\cM_{2,\cT} \coloneqq k_1 ( \sqrt[q]{a} \mid a \in U^\cM_\cT (q)). 
\]

\begin{rk} \label{functoriality obstruction extension rk}
Neukirch refers to the Galois extension $k_2/k_1$ as the `obstruction extension' associated to the problem at hand and shows it has the following functorial properties.  
Let $\cM'$ be a set of places of $k$ disjoint from $\cM \cup \cT$, and write $H_{\cM'} $ for the subgroup of $\gal{k_2}{k_1}$ generated by $\{ \Frob_v \mid v \in \cM'_{k_1} \}$. Then, since every place in $\cM'_{k_1}$ splits completely in $k^{\cM \cup \cM'}_{2,\cT}$, one has $k^{\cM \cup \cM'}_{2,\cT} \subseteq k_2^{H_{\cM'}}$, with equality if $p$ is odd, and so $\gal{k^{\cM \cup \cM'}_{2,\cT}}{k_1}$ is a quotient of $\gal{k_2}{k_1} / H_{\M'}$. (For more details see \cite[Satz (7.3)]{Neukirch73}.) 
\end{rk}

The following result is then a straightforward extension of \cite[Satz (7.4)]{Neukirch73} that allows us to incorporate the case $p = 2$.

\begin{lem} \label{neukirch obstruction lemma 2}
The canonical map
\[
\gal{k_2}{k_1} \to \Hom_\Z ( U^\cM_\cT (q) / (k^\times)^q, \mu_q), 
\quad 
\sigma \mapsto \{ a \mapsto \sigma ( \sqrt[q]{a}) / \sqrt[q]{a} \}
\]
is injective. In addition, its cokernel is of order dividing two and vanishes if $p$ is odd.
\end{lem}

\begin{proof} We write $Q$ for the quotient module $\faktor{( U^\cM_\cT( q) \cap (k_1^\times)^q)}{(k^\times)^q} $ , and consider the tautological exact sequence
\begin{cdiagram}
0 \arrow{r} & Q 
\arrow{r} & \faktor{U^\cM_\cT (q)}{(k^\times)^q} 
\arrow{r} & \faktor{U^\cM_\cT (q)}{( U^\cM_\cT(q) \cap (k_1^\times)^q)}
\arrow{r} & 0.
\end{cdiagram}%
Then Kummer theory identifies $\gal{k_2}{k_1}$ with $\Hom_\Z ( \faktor{U^\cM_\cT (q)}{( U^\cM_\cT( q) \cap (k_1^\times)^q)}, \mu_q)$ and so, upon taking Kummer duals of the above sequence, we obtain an exact sequence 
\begin{cdiagram}
0 \arrow{r} & \gal{k_2}{k_1} \arrow{r} & 
\Hom_\Z \big( \faktor{U^\cM_\cT (q)}{(k^\times)^q}, \mu_q \big) 
\arrow{r} & \Hom_\Z  ( Q, \mu_q ) 
\arrow{r} & 0.
\end{cdiagram}%
This sequence directly proves the first claim and also shows that the order of the cokernel of the given map is equal to the order of $Q$. 

To bound the latter  order, we observe that $Q$ identifies with a subgroup of   
\begin{equation}\label{theta def}
\ker \big ( k^\times\!/(k^\times)^q \xrightarrow{\theta} k_1^\times\!/( k_1^\times)^q \big ) 
\cong H^1 (\gal{k_1}{k}, \mu_q),
\end{equation}
where $\theta$ is the natural map and the isomorphism results from Kummer theory and the relevant inflation-restriction sequence. In addition, by \cite[Satz (4.8)]{Neukirch73}, the group $H^1 (\gal{k_1}{k}, \mu_q)$ vanishes if either $q$ is odd or $k \cap \Q (\mu_{2^n})$ is complex. It therefore only remains to consider the case that $q  = 2^N$ and $\sqrt{-1}\notin k$. In this case, the inflation-restriction sequence combines with the latter vanishing result to give an isomorphism
\[
H^1 (\gal{k (\sqrt{-1})}{k}, \mu_{k ( \sqrt{-1})}) \cong H^1 (\gal{k_1}{k}, \mu_q). 
\]
In addition, since $\gal{k (\sqrt{-1})}{k}$ is cyclic,  a Herbrand quotient argument implies that  
\[
| H^1 (\gal{k (\sqrt{-1})}{k}, \mu_{k ( \sqrt{-1})}) | = | \widehat{H}^0 (\gal{k (\sqrt{-1})}{k},\mu_{k ( \sqrt{-1})}) | = 
| \{ \pm 1 \} | = 2, 
\]  
as required to conclude the proof of the second claim.
\end{proof}

Before proving Proposition~\ref{ultimate neukirch result}, we make one final observation about the field $k_2$.  

\begin{lem} \label{max ab subext lemma}
Fix $q \coloneqq p^N$ as above, and set
\[
k_3 = k_{3, \cT} \coloneqq k ( \mu_{q}, (\cO_{k, \cT}^\times)^{1 / q}).
\]
Then, under the hypotheses of Proposition~\ref{ultimate neukirch result}, the following claims are valid. 
\begin{liste}
\item One has $k_2\subseteq k_3$, with equality if every place in $\cM \setminus \cT$ splits completely in $k_3$.
\item Write $s$ for the greatest integer such that $k$ contains a primitive $p^s$-th root of unity and assume $N > s$. Then the maximal abelian extension of $k$ in $k_3$ is equal to $k (\mu_{qp^{s}}, (\cO_{k, \cT}^\times)^{1 / p^s})$ if either $p$ is odd or $\sqrt{-1}\in k$, and is otherwise contained in
\textcolor{black}{$k (\mu_{q 2^{s}}, (\cO_{k, \cT}^\times)^{1 / 2^{s + 1}})$}. 
\end{liste}
\end{lem}

\begin{proof} To prove the inclusion $k_2\subseteq k_3$ it is sufficient, by Kummer theory, to show that the image  in $k^\times / (k^\times)^q$ of each element of $U^\cM_\cT (q)$  is represented by an element of $\cO_{k, \cT}^\times$. To do this, we fix  $a$ in $U^\cM_\cT (q)$ and write $Z (a)$ for the (finite) set of places $v$ of $k$ that are outside $\cT$ and satisfy $\ord_v (a) \neq 0$. Then, since the $\cT$-class group of $k$ is assumed to be trivial, the ideal associated to each $v \in Z( a)$ is the product of an ideal supported only at places in $\cT$ with an element $\pi_v$ of $k^\times$ which must therefore belong to  $\cO_{k, \cT \cup \{ v \}}^\times$ and satisfy $\ord_v (\pi_v) = 1$.
By the explicit definition of $U^\cM_\cT (q)$, one also knows $q$ divides $\ord_v (a)$ for each $v \in Z (a)$. It follows that $a\cdot\prod_{v \in Z(a)} \pi_v^{- \ord_v (a)}$ belongs to  $\cO_{k, \cT}^\times$ and has the same image as $a$  in $k^\times / (k^\times)^q$, as required to show $k_2\subseteq k_3$.\\
To complete the proof of claim (a) we note that if every place in $\cM \setminus \cT$ splits completely in $k_3$, then the explicit definition of $k_3$ implies that $\cO_{k, \cT}^\times \subset (k_v^\times)^q$ for every place $v \in \cM \setminus \cT$. In this case therefore one has $\cO_{k, \cT}^\times \subseteq U^\cM_\cT (q)$ and so  $k_3\subseteq k_2$, as required. \\
To prove claim (b), we write $L'$ for the maximal abelian extension of $k$ in $k_3$. 
 Then $L'$ contains $k_1 = k (\mu_q)$ and so Kummer theory implies the existence of a subgroup $U$ of $(\cO_{k,\cT}^\times \cdot (k_1^\times)^{q})/ (k_1^\times)^{q}$ such that $\gal{L'}{k_1} \cong \Hom ( U, \mu_{q})$. In addition, since $L'$ is abelian over $k$, the conjugation action of $\cG_{k_1}$ on $\gal{L'}{k_1}$ is trivial. To make the latter condition explicit, we fix elements  $\tau \in \gal{L'}{k_1}$ and $\rho \in \cG_{k_1}$, a lift $\widetilde{\rho}$ of $\rho$ to $\cG_{L'}$ and an element $a$ of $\cO_{k, \cT}^\times$ that represents a class in $U$. Then, upon fixing a choice of primitive $q$-th root of unity $\zeta_q \in \mu_q$ and for some integer $n (a,\widetilde{\rho})$, one has 
\[ \widetilde{\rho} (\sqrt[q]{a}) = \zeta_{q}^{n (a,\widetilde{\rho})} \sqrt[q]{ \rho (a)} =\zeta_{q}^{n (a,\widetilde{\rho})} \sqrt[q]{a}.\]
In particular, since $\tilde\rho \circ\tau = \tau\circ\tilde\rho$ in $\gal{L'}{k_1}$, one therefore has
\[ \tilde\rho(\tau(\sqrt[q]{a})) = \tau (\widetilde{\rho} ( \sqrt[q]{a})) = \tau ( \zeta_{q}^{n (a,\widetilde{\rho})} \sqrt[q]{a}) 
= \zeta_{q}^{n (a,\widetilde{\rho})}(\tau ( \sqrt[q]{a})/\sqrt[q]{a}) \sqrt[q]{a} 
 =  (\tau ( \sqrt[q]{a})/\sqrt[q]{a})\tilde\rho(\sqrt[q]{a}),
\]
and so the element $\tau ( \sqrt[q]{a}) /\sqrt[q]{a}$ of $\mu_q \subset k_1$ is fixed by $\tilde\rho$. Since this is true for every $\rho$ in $\cG_{k_1}$, it follows that $\tau (\sqrt[q]{a}) /\sqrt[q]{a}$ belongs to $k^\times$ and so must be a $p^s$-th root of unity (by definition of $s$). 

We now write $\langle a \rangle$ for the subgroup of $k_1^\times / (k_1^\times)^q$ generated by $a$, so that Kummer theory gives a natural isomorphism 
\[
\gal{k_1 (\sqrt[q]{a})}{k_1} \stackrel{\simeq}{\to} \Hom ( \langle a \rangle, \mu_{q}), \quad \tau \mapsto \tau (\sqrt[q]{a}) /\sqrt[q]{a}.
\]
Then, since the above argument shows the order of every element in the image of this  isomorphism to be at most $p^s$ (which, by assumption, is a proper divisor of $q$), the order of $\langle a\rangle$ is at most $p^s$, and so  $a^{p^s}$ belongs to the kernel of the map $\theta$ in (\ref{theta def}). Now, as observed earlier,  $\theta$ is injective if either $p$ is odd or $\sqrt{-1}\in k$, and in all other cases $|\ker(\theta)|$ divides $2$, and so the class of $a^{p^s}$ in $k^{\times}/(k^{\times})^{q}$ is trivial, respectively of order dividing $2$.

We now complete the proof of (b) in the case that the class of $a^{p^s}$ in $k^{\times}/(k^{\times})^{q}$ is non-trivial, so that $p = 2$, $\sqrt{-1}\notin k$ and  $|\ker(\theta)|=2$, and merely note that the case in which the class of $a^{p^s}$ is trivial can be dealt with in an entirely similar (and, in fact, slightly easier) fashion. 

\textcolor{black}{We first claim that the unique non-trivial element $x$ of $\ker(\theta)$ is represented by an element $b_0$ of $\cO_{k, \cT}^\times$.
 To verify this, we choose any $b'_0$ in $k^\times$ that represents $x$ and write $Z (b'_0)$ for the set of places $v$ of $k$ outside $\cT$ for which $\ord_v (b'_0) \neq 0$. Then, since $b'_0 \in (k_1^\times)^q$, one has $k (\sqrt[q]{b'_0}) \subseteq k_1$ so that $k (\sqrt[q]{b_0}) / k$ is unramified outside $S_2(k)$ and hence $q$ divides $\ord_v (b'_0)$ for each $v \in Z (b'_0)$. In particular, if (just as in the proof of claim (a))  for each $v \in Z (b'_0)$ we fix $\pi_v\in \cO_{k, \cT}^\times$ with $\ord_v (\pi_v) = 1$, then $b_0 \coloneqq b'_0\cdot \prod_{v \in Z (b'_0)} \pi_v^{- \ord_v (b'_0)}\in \cO_{k, \cT}^\times$ represents $x$, as required. \\
 Then, as $x$ has order $2$, one has $b_0^2 = b_1^q$  and thus, noting the current assumptions ensure $q$ is even, also $b_0 = \pm b_1^{q / 2}$ for some $b_1 \in \cO_{k, \cT}^\times$. Hence, since  $a^{2^s}$ and $b_0$ both represent $x$, one has $a^{2^s} = \pm b_1^{q / 2} b^{q}$ for some  $b\in \cO_{k, \cT}^\times$. Upon taking $(2^sq)$-th roots of this last equality, one deduces the existence of  $\zeta\in \mu_{2^{s}q}$ such that $\zeta\sqrt[q]{a} = \sqrt[2^{s + 1}]{\pm b_1 b^2}$, and thus $k_1 (\sqrt[q]{a}) \subset k (\mu_{q 2^{s}},  ( \cO_{k, \cT}^\times)^{1 / 2^{s + 1}})$.} Since $L'$ is defined to be the compositum of all such fields $k_1 (\sqrt[q]{a})$, it is therefore also contained in $k(\mu_{2^{s}q}, ( \cO_{k, \cT}^\times)^{1 / 2^{s + 1}})$, as required to complete the proof of claim (b) in this case. 
\end{proof}
\textit{Proof of Proposition~\ref{ultimate neukirch result}:} We are now ready to prove Proposition~\ref{ultimate neukirch result}, the notation and hypotheses of which we henceforth assume.
In addition, in this argument we take
\[
m \coloneqq 2n + h \quad \text{ and } \quad 
q \coloneqq \begin{cases}
p^{m + s} \quad & \text{ if } p \text{ is odd},\\
p^{m + s + 1} & \text{ if } p = 2,
\end{cases}
\]
and use the auxiliary fields 
\[ k_1 \coloneqq k (\mu_q), \quad k_3 \coloneqq k_1( (\cO^\times_{k, \cT})^{1 / q})\quad \text{and}\quad F \coloneqq K \cdot k_3.\]
Then it is clear from the definition of the fields $k(p,\cT,n)$ that $L \subseteq k_3$ and so Lemma~\ref{max ab subext lemma}\,(b) implies $K \cap k_3= L$. 
Thus, since the restriction to $L$ of the element $\sigma$ fixed in Proposition~\ref{ultimate neukirch result} is trivial, we can use the canonical isomorphism $\gal{F}{L} \cong \gal{K}{L} \times \gal{k_3}{L}$ to choose an element $\overline{\sigma}$ of $\gal{F}{L}$ whose restrictions to $K$ and to $k_3$ are respectively equal to $\sigma$ and to the identity automorphism. By Cebotarev's Density Theorem, applied to the Galois extension $F / k$, we can then choose a place $\q$ of $k$ that does not belong to $\cT$, is unramified in $F$ and such that $\overline{\sigma}$ is equal to $\Frob_\mathfrak{Q}$ for some place $\mathfrak{Q}$ of $F$ lying above $\q$.  This condition ensures $\q$ splits completely in $k_3$  and also that the restriction of $\Frob_\q$ to $K$ is $\sigma$. In particular, 
by Lemma~\ref{rubin lemma} (and the definition of $k_3$), the degree $[k ( \q): k(1)]$ is divisible by $q$.
Now, since $\gal{k ( \q)}{k(1)}$ is isomorphic to a quotient of the unit group $( \cO_k / \q)^\times$, it is cyclic and so this argument implies $\cG_{k (\q)}$ contains an element of order $q$, and hence also has a cyclic quotient of order $q$. This implies $k (\q)$ contains a subfield $E_0$ that is a cyclic extension of $k$ of degree $q$ and, consequently, that there 
exists a surjective homomorphism of groups 
\[
\phi_0 \: G_k \to \cG_{E_0} \cong \nZ{q},
\]
in which the isomorphism is induced by a choice of generator of $\cG_{E_0}$. For each place $v$ of $k$ we write $\phi_{0, v}$ for the restriction of $\phi_0$ to the decomposition group $G_{k_v}$. 

For  $v$ in $\cT$ we also write $k^{\mathrm{nr},m}_v$ for the unique non-ramified extension of $k_v$ of degree $p^{m}$, and define $\varphi_v$ to be the composite homomorphism 
\[ G_{k_v} \to \gal{k^{\mathrm{nr}, m}_v}{k_v} \cong \faktor{( q p^{-m}\Z)}{q \Z} \subseteq \nZ{q},\]
in which the first arrow is the canonical projection and the isomorphism sends the image of $\Frob_v$  in $\gal{k^{\mathrm{nr}, m}_v}{k_v}$ to the class of $q p^{-m}$. 
We then write $\eta = \eta(\phi_0,(\varphi_v)_{v \in \cT})$ for the obstruction class (in the sense of 
(\ref{obstruction class def}) and Remark~\ref{interpret obstruction}) that is represented by the family of homomorphisms $\eta_v \in \Hom ( G_{k_v}, \nZ{q} )$ specified by    
\begin{equation}\label{obs spec} \eta_v \coloneqq \begin{cases} \phi_{0, v} - \varphi_v, &\text{ if $v \in \cT$}\\
 \phi_{0,v}, &\text{ if $v \notin\cT$.}
                          \end{cases}\end{equation}
To analyse  $\eta$ we set $\cM \coloneqq \cT \cup \{ \q \} \cup S_\infty$ with $\q$ the prime fixed above
and claim that 
\begin{equation}\label{divisibility property} \eta \in p^{t}\cdot \Delta(k, \Z / q \Z)^\cM_{\cT \cup S_\infty}
\quad \text{ with }\quad 
t = t(p) \coloneqq \begin{cases} s, &\text{if $p$ is odd},\\
s+1, &\text{if $p =2$.}\end{cases}\end{equation}
To verify this, we note that the family $(\phi_{0,v})_v$ of restrictions of $\phi_0$ represents the trivial element of $\Delta(k, \Z / q \Z)^\cM_{\cT \cup S_\infty}$ (by the very definition of the latter group) and hence that $\eta$ is also represented by the family $(\eta'_v)_v$ obtained by setting
\begin{equation}\label{obs spec2}
\eta_v' \coloneqq \begin{cases}
    -\varphi_v \quad & \text{ if } v \in \cT, \\
    0 & \text{ if } v \not \in \cT.
\end{cases}
\end{equation}
To prove (\ref{divisibility property}) we are thus reduced to showing $\varphi_v \in p^{t}\cdot \Hom (G_{k_v}, \Z / q \Z)$ for each $v \in \cT$. In addition, the explicit definition of each $\varphi_v$ ensures $\im(\varphi_v) \subseteq ( q p^{-m})\cdot (\nZ{q} )\subseteq p^{t}\cdot (\nZ{q})$. Hence, if we fix an element $y\in \nZ{q}$ with 
$p^{t}\cdot y = \varphi_v(f_v)$, then $\varphi_v = p^{t}\cdot \varphi_v'$, where $\varphi'_v$
is defined as the composite morphism $G_{k_v} \to \gal{k_v^{\mathrm{nr}, m + t}}{k_v} \to \Z / q \Z$ with the first arrow being the natural restriction map and the second arrow the unique homomorphism that sends the image of $\Frob_v$ in $\gal{k_v^{\mathrm{nr}, m + t}}{k_v}$ to the class of $y$.

Having verified (\ref{divisibility property}), we next note that, since $\q$ splits completely in $k_3$, Lemma~\ref{max ab subext lemma}\,(a) implies $k_3$ coincides with the field $k_2 = k^\cM_{2, \cT \cup S_\infty} \coloneqq k_1 ( U^\cM_{\cT \cup S_\infty} (q)^{1 / q})$ defined earlier. In addition, 
 if we fix a choice of primitive $q$-th root of unity $\zeta \in \mu_q$, then the isomorphism 
 \[ 
 \Hom (U^\cM_{\cT \cup S_\infty} (q) / (k^\times)^q, \Z / q \Z) \cong \Hom (U^\cM_{\cT\cup S_\infty} (q) / (k^\times)^q, \mu_q), \quad \theta \mapsto [x \mapsto  \zeta^{\theta(x)}]
 \]
 combines with  the maps 
in Lemmas \ref{neukirch obstruction lemma 1} and \ref{neukirch obstruction lemma 2} to give a homomorphism 
\[ 
\Psi \: 
\gal{k_3}{k_1} = \gal{k_2}{k_1} \to \Delta (k, \faktor{\Z}{q \Z})^\cM_{\cT {\cup S_\infty }}\]
that is bijective if $p$ is odd, respectively injective and with cokernel of order at most $2$ if $p=2$. Furthermore, if $p=2$, then 
(\ref{divisibility property}) implies  that $\eta \in 2\cdot \Delta(k, \Z / q \Z)^\cM_{\cT {\cup S_\infty }}$. In all cases, therefore, there exists a (unique) element $\tau$ of  $\gal{k_3}{k_1}$ with $\Psi (\tau) = \eta$ and we claim that this element acts trivially on $L$. \\
 To show this, we recall the integer $t$ defined in (\ref{divisibility property})
 and observe $L$ is contained in the field $k(p,\cT,n)$ that is generated over $k_1$ by $\mu_{qp^{s}}$ and all elements of the form $\sqrt[p^{t}]{a}$ with $a \in \cO_{k, \cT}^\times$. It is thus enough for us to show that 
\begin{equation}\label{required eq} \tau (\sqrt[p^{t}]{a}) =\sqrt[p^{t}]{a} \,\text{ for all } a \in \cO_{k,\cT}^\times\,\,\text{ and }\,\, \tau (\xi) = \xi\,\text{ for all } 
\, \xi \in \mu_{ qp^{s}}.\end{equation}
In order to verify the first of these equalities, we note that the order of $\tau$ divides $qp^{-t}$. Indeed, since $\Psi$ is injective, this is true provided that the order of $\eta$ divides $qp^{-t}$ and so follows from the fact that (\ref{divisibility property}) implies directly that $\eta$ is divisible by $p^{t}$.
To verify the first equality in (\ref{required eq}) it is thus enough for us to show that, if $a \in k^\times$, then every $\kappa\in \gal{k_3}{k_1}$ of order dividing $qp^{-t}$ acts as the identity on $\sqrt[p^{t}]{a}$. To show this, let $\xi = \xi(a,\kappa)$ be the unique $q$-th root of unity with the property that $\kappa ( \sqrt[q]{a}) = \xi  \sqrt[q]{a}$. Then, since $\xi\in k_1$, one has $\kappa(\xi) = \xi$ and so $\kappa^u ( \sqrt[q]{a}) = \xi^u \sqrt[q]{a}$ for all $u\in \ZZ$. In particular, since $\kappa^{qp^{-t}}$ is trivial, the order of $\xi$ divides $qp^{-t}$ and so we obtain the required equality via the computation
\[
\kappa ( \sqrt[p^{t}]{a} ) = ( \kappa ( \sqrt[q]{a}) )^{q p^{-t}} = ( \xi  \sqrt[q]{a})^{q p^{-t}} 
= \sqrt[p^{t}]{a}.
\]
\textcolor{black}{To prove the second equality in (\ref{required eq}) we fix $\xi \in \mu_{qp^{s}}$
and note that $\xi^q\in \mu_{p^s} \subseteq k^\times$. Then, from the explicit descriptions of the maps in Lemmas \ref{neukirch obstruction lemma 1} and \ref{neukirch obstruction lemma 2} used in the above definition of $\Psi$, we know that $\tau (\xi)/\xi = \zeta^{{\sum}_v \eta'_v ( \rec_v (\xi^q))},$ where in the sum $v$ runs over all places of $k$ and $(\eta'_v)_v$ is the representative of $\eta$ specified in (\ref{obs spec2}). Hence, since $\eta'_v = 0$ for all $v \notin \cT$, it is enough for us to show $\eta'_v(\rec_v (\xi^q)) = 0$ for all $v\in \cT$. To do this, we note each such $v$ is finite and so, since $\xi^q$ is a unit of the valuation ring of $k_v$, its image under $\rec_v$ belongs to the inertia subgroup of $G_{k_v}^\mathrm{ab}$. In particular, since our choice of $\varphi_v$ implies that $\eta_v' = -\varphi_v$ defines an unramified extension of $k_v$, one must therefore have $\rec_v (\xi^q)\in \ker(\eta'_v)$, as required.}\\
Having proved $\tau$ acts trivially on $L$, we now note that, since (as observed earlier) $K \cap k_3 = L$, the natural map $\gal{Kk_3}{L} \to \gal{k_3}{L} \times \gal{K}{L}$ is bijective, and so there exists a unique element $g$ of $\gal{K k_3}{L}$ that restricts on $k_3$ to give $\tau$ and on $K$ to give the element $\sigma$ fixed in Proposition~\ref{ultimate neukirch result}. Then, since (\ref{required eq}) implies $g$ belongs to the abelian subgroup $\gal{K k_3}{Lk_1}$ of $\gal{K k_3}{L}$, Cebotarev's Density Theorem allows us to fix a place $\mathfrak{A}$ of $Lk_1$ that has all of the following properties: 
\begin{itemize}
\item $\mathfrak{A}$ has absolute degree one;
\item the restriction $\a$ of $\mathfrak{A}$ to $k$ does not belong to $\cM = \cT\cup \{\q\} {\cup S_\infty}$; 
\item the rational prime lying beneath $\mathfrak{A}$ is unramified in $K k_3$;
\item the restriction of $\Frob_\mathfrak{A}$ to $Kk_3$ is equal to $g$.
\end{itemize}
We note, in particular, that $\a$ is unramified in $K k_3$ and such that the respective restrictions of $\Frob_\a$ to $K$ and to $k_3$  are equal to $\sigma$ and to $\tau$. 

\textcolor{black}{
At this point, we aim to use the functorial properties of obstruction classes described in Remark~\ref{functoriality obstruction extension rk} (the notation of which we assume) to specify 
a quotient of $\Delta (k, \Z / q \Z)^\cM_{\cT {\cup S_\infty }}$ in which the image of $\eta$ vanishes. To do this, we set $\cM' = \{\a \}$. Then it is clear the projection of $\tau$ is trivial in the quotient of $\gal{k_2}{k_1}$ by the subgroup $H_{ \cM'}$ generated by $\Frob_\a = \tau$ and hence, since Remark 
\ref{functoriality obstruction extension rk} identifies $\gal{k_{2, \cT {\cup S_\infty}}^{\cM\cup \cM'}}{k_1}$ with a quotient of $\gal{k_2}{k_1} / H_{\cM'}$, the projection of $\tau$ to 
$\gal{k_{2, \cT {\cup S_\infty }}^{\cM\cup\cM'}}{k_1}$ is also trivial. In particular, since the composites of the respective maps from Lemma~\ref{neukirch obstruction lemma 2} and Lemma~\ref{neukirch obstruction lemma 1} form the horizontal maps in a commutative diagram}
\begin{cdiagram}
    {\gal{k_2}{k_1}} \arrow{r} \arrow{d} & \Delta (k, \nZ{q})^{\cM}_{\cT {\cup S_\infty }} \arrow{d} \\
  \gal{k_{2, \cT {\cup S_\infty }}^{\cM\cup\cM'}}{k_1} \arrow{r}  & \Delta (k, \nZ{q})^{\cM\cup\cM'}_{\cT {\cup S_\infty }}
\end{cdiagram}%
\textcolor{black}{
(in which the vertical maps are the natural projections) and the upper map sends $\tau$ to $\eta$, we can deduce that the projection of $\eta$ to $\Delta (k, \nZ{q})^{\cM\cup \cM'}_{\cT {\cup S_\infty }}$ vanishes.
}

As a consequence of Remark~\ref{interpret obstruction} and the explicit specification (\ref{obs spec}) of $\eta$, we can therefore now deduce the existence of {a cyclic extension $E_1$ of $k$ of order dividing $q$ that is unramified outside $\cM\cup\cM' = \cT \cup \{\q,\a\} {\cup S_\infty}$ and such that, for each $v \in \cT {\cup S_\infty }$, the localisation $E_{1, w}$ of $E_1$ at a place $w$ lying above $v$ {is the field cut out by $\eta_v$. In particular, this extension is trivial if $v$ is an infinite place and for any place $v \in \cT$} our explicit choice of $\varphi_v$ ensures $E_{1, w}$ is the unique non-ramified extension of $k_v$ of degree $p^m$, and hence that $v$ is unramified in $E_1$ and the image of $\Frob_v$ in $\cG_{E_1}$ has order $p^m$. 
This shows, in particular, that $E_1 / k$ is unramified at all places in $\cT {\cup S_\infty}$ and hence $S(E_1) \subseteq \{ \q, \a \}$. We also recall $\q$ and $\a$ are both chosen to be unramified in $K$ and such that their respective Frobenius automorphisms restrict to give $\sigma$ on $K$.
At this stage, we have therefore verified that $E_1$}
has all of the properties (i), (ii), (iii) and (iv)\,(b) that are stated in Proposition~\ref{ultimate neukirch result} with respect to $m$ (rather than $n$).\\
Next we note that, since the maximal unramified extension in $k$ in $E_1$ has degree at most $p^h$, and $\cG_{E_1}$ is cyclic of $p$-power order, there must exist a place $v\in \{ \q, \a \}$ whose inertia subgroup in $\cG_{E_1}$ has order at least $p^{m  - h}  = p^{(2n + h) - h} = p^{2n}$. If $S(E_1) = \{v\}$, then this already shows $E_1$ has all of the properties required by Proposition~\ref{ultimate neukirch result}. We therefore assume $S(E_1) = \{\q,\a\}$ and write $v'$ for the unique place in 
$\{ \q, \a \}\setminus \{v\}$ and $I'$ for its inertia subgroup in $\cG_{E_1}$. If $|I'| \ge p^n$, then the above argument again shows $E_1$ has all required properties. On the other hand, if $|I'| < p^n$, then the subfield $E_1^{I'}$ of $E_1$ is a cyclic extension of $k$ that is unramified outside $v$ and in which $v$ has inertial degree at least $p^{2n} / |I'| > p^n$ and every place in $\cT$ has decomposition subgroup of order at least $p^m / |I'| > p^n$. This shows that $E_1^{I'}$ has all properties required by (i), (ii), (iii) and (iv) of Proposition~\ref{ultimate neukirch result}, and hence concludes the proof of the latter result. 
\qed 

\subsubsection{Consequences of the Cebotarev Density Theorem} 

{We now begin to realise the strategy of our proof of Theorem~\ref{p adic restriction sequence}, as outlined in Remark~\ref{strategy remark}. 
In particular, the following technical consequence of Cebotarev's Density Theorem is closely related to the case $|\Sigma | = 1$ of Theorem~\ref{p adic restriction sequence} and, together with an induction on $|\Sigma|$, will play an important role in the proof of the general case that we present in \S\,\ref{prs proof section}.}\\
In this result we use the fields $k(p,\cT,n)$ that are defined in the statement of Proposition~\ref{ultimate neukirch result}.

\begin{lem} \label{induction-step-revisited}
Fix $p\in \mathbb{P}$ and $\p\in S_\fin(k)$. Let $\cT$ be a finite subset of $S_\fin(k)$ with $S_p \cup \{ \p \}\subseteq \cT$  and such that the $\cT$-class group of $k$ is trivial. For any extension $L$ of $k$ we define the field 
\[ L (p) \coloneqq L \cap k (p, \cT)\quad \text{with}\quad k (p, \cT) \coloneqq {\bigcup}_{n \in \N} k(p, \cT, n).\]
Let $\cX$ be a subset of $\Omega$ that satisfies Hypothesis \ref{hyp X} and $\{ (M_{F, E}, \rho_{F / E}, j_{F / E})\}_{E, F \in \cX}$  
a triangular system that satisfies Hypothesis \ref{inj-hypothesis} and is such that each diagonal term $M_E$ is both $\Z$-torsion-free and 
such that $(M_E)_{p\cdot\mathrm{div}} = (0)$. Fix $(a_L)_L\in  \eullim{L \in \cX}{\Pi} M_{L}$ and $F\in \cX$ such that  
\begin{itemize}
    \item for any field $L \in \cX$ with $F \subseteq L$ the element $a_L$ is fixed by $\gal{L}{L(p)}$.
\end{itemize}
Then, for any  $K \in \cX$, the element $a_K$ is fixed by $\gal{K}{K(p)}$ provided that all of the following conditions are satisfied: 
\begin{itemize}[label=$\bullet$]
\item[(i)]$k \subseteq K \subseteq F$.
    \item[(ii)] $S(F) \setminus \Pi(K) = \{\p\}$.
    \item[(iii)] $V_K = V_F$.
\end{itemize}
\end{lem}

\begin{proof} Fix a field $K$ as above and an element $\sigma$ of $\gal{K}{K(p)}$ and write $P(\sigma)$ for the subset of $S_\fin(k)\setminus  S_p$ comprising places $v$ for which the restriction of $\Frob_v^{-1}$ to $K$ agrees with $\sigma$. Then, for any natural number $t$, Proposition~\ref{ultimate neukirch result} {(applied with $n  \coloneqq t  + s$ and $\sigma \coloneqq \sigma^{-1}$)} implies the existence of a finite cyclic $p$-extension $E$ of $k$ that belongs to $\Omega^{S_\infty }_k$, 
is unramified outside $P(\sigma)$ and is such that the order of the decomposition subgroup $\cG_{E.\p}$ is at least $p^t$. 

We note that, since all archimedean  places split in $E / k$, Hypothesis \ref{hyp X}\,(i) guarantees that both of the composite fields
\[ F' \coloneqq EF\quad \text{and} \quad K' \coloneqq EK\]
belong to $\cX$. Now, by assumption, one has $a_{F'} \in M_{F'}^{\gal{F'}{F'(p)}}$ and so the Euler limit relations combine with the stated conditions (ii) and (iii) to imply that the element 
\[
  (1 - \Frob_\p^{-1}) \cdot a_{K'}  
= P_{F' / K', \Pi} \cdot a_{K'} = \rho_{F' / K'} (a_{F'})
\]
is fixed by the image $\gal{K'}{K'(p)}$ of $\gal{F'}{F'(p)}$ in $\cG_{K'}$.\\
Let us now write $H$ for the intersection of $\gal{K'}{K'(p) K}$ and the subgroup of $\cG_{K'}$ generated by $\Frob_\p$. Then any element $\tau$ of $H$ is a power of $\Frob_\p$ and so $(1 - \tau) \cdot a_{K'}$ is fixed by $H$. However, the element $(1 - \tau)\cdot a_{K'}$ is also annihilated by the norm $\NN_{H}$ and hence, 
since $M_{K'}$ is $\Z$-torsion free, must vanish. It follows that $a_{K'}$ is fixed by $H$.\\
Write $M_{K'}'$ for the full pre-image in $M_{K'}$ of $M_{K}$ under $\rho_{K'/K}$. Then, since $(a_L)_L$ belongs to $ \eullim{L\in \cX}{\Pi} M_{L}$, one has $a_{K'}\in M_{K'}'$ and so Hypothesis \ref{inj-hypothesis}\,(i) implies that 
\[ (\iota_{K'/ K} \circ  \rho_{K'/ K})(a_{K'}) = \NN_{\gal{K'}{K}}\cdot f(a_{K'})\]
for a suitable endomorphism $f$ of $M'_{K'}$. Setting 
\[ n \coloneqq |S(K') \setminus \Pi(K)|,\]
we can therefore deduce that the element 
\begin{align*}
    \iota_{K' / K} \big( (1 - \sigma)^n \cdot a_K\big) 
    & = \iota_{K' / K} \big ( P_{K' / K, \Pi} \cdot a_K \big )\\
   & =  (\iota_{K' / K} \circ  \rho_{K' / K}) (a_{K'}) \\
    & = \NN_{\gal{K'}{K}} \cdot f(a_{K'}) \\
    & = \NN_{\gal{K'}{K}/H} \cdot \NN_H \cdot f(a_{K'}) \\
    & = \NN_{\gal{K'}{K}/H} \cdot |H| \cdot f (a_{K'})
\end{align*}
is divisible by $|H|$ in $M_{K'}$. Note that the restriction-to-$E$ isomorphism $\gal{K'}{K} \cong \gal{E}{E \cap K}$ restricts to an isomorphism
\begin{align*}
 \gal{K'}{K' (p)K} & = 
\gal{K'}{K} \cap \gal{K'}{K'(p)} \cong \gal{E}{E \cap K} \cap \gal{E}{E (p)} \\
& = \gal{E}{E(p)(E \cap K)}.
\end{align*}
Now, as $E$ is unramified at $p$, the field $E(p)$ is contained in the maximal unramified extension of $k$ in $k (p, \cT)$. That is, we can find an integer $\alpha$ that only depends on $K$, $k$, $p$ and $\p$ and is such that, for any value of $t\ge \alpha$ the group $\gal{E}{E(p)(E \cap K)}$ must contain the unique subgroup of the (cyclic) group $\cG_E$ that has index $p^\alpha$.
 However, by construction, the decomposition group $\cG_{E,\p}$ has order at least $p^t$ and so the order of its intersection with $\gal{E}{E(p)(E \cap K)}$ must be at least  $p^{t - \alpha}$. 
 Write $\beta$ for the order of the restriction of $\Frob_\p$ to $K$. Then the element $\Frob_\p^\beta$ belongs to $\gal{K'}{K} \cong \gal{E}{E \cap K}$
and so we deduce that $|H| \geq p^{t - \alpha - \ord_p (\beta)}$. \\
We now write $e_{\bm{1}}$ for the idempotent in $\Q [\langle \sigma \rangle]$ that is associated to the trivial character of $\langle \sigma \rangle$. Then $1 - e_{\bm{1}}$ is contained in the augmentation ideal of $\Q [\langle \sigma \rangle]$. Since the latter ideal is generated by $1 - \sigma$, it follows that $1 - e_{\bm{1}} = x(1 - \sigma)$ for some $x \in \Q [\langle \sigma \rangle]$, and hence also 
\begin{align*}
    x^{n - 1} (1 - \sigma)^{n} & = (1 - \sigma) \big ( x (1 - \sigma) \big)^{n - 1} = (1-\sigma) (1 - e_{\bm{1}})^{n - 1} \\
    & = (1 - \sigma) (1 - e_{\bm{1}}) = (1 - \sigma).
\end{align*}
Thus, if we now fix a natural number $z$ such that $z\cdot x\in \Z [\langle \sigma \rangle]$, then this computation combines with the previous discussion to imply that the element 
\[ \iota_{K' / K} \big( z^{n - 1} \cdot (1 - \sigma) \cdot a_K\big) = z^{n - 1} \cdot x^{n - 1}\cdot\iota_{K'/K}((1 - \sigma)^{n} \cdot a_K)\]
is divisible by $p^{t - \alpha - \ord_p (\beta)}$ in $M_{K'}$. 

Now, since $n\le |S(K')\setminus S(K)| \le |S(E)|$,  property (iii) in Proposition~\ref{ultimate neukirch result} implies $n \le 2$. By taking $t$ large, Hypothesis \ref{inj-hypothesis}\,(ii) therefore 
combines with the assumption  $(M_K)_{p\cdot\mathrm{div}}$ vanishes to
imply  $z \cdot (1 - \sigma) \cdot a_K = 0$. Then, since $M_K$ is torsion free, it follows that $(1 - \sigma) \cdot a_K = 0$ and hence, as $\sigma$ is an arbitrary element of $\gal{K}{K(p)}$, that $a_K$ is fixed by $\gal{K}{K(p)}$, as required.
\end{proof}

\subsubsection{The proof of Theorem~\ref{p adic restriction sequence}}\label{prs proof section}

{
For expositional clarity, we find it convenient to separate this argument into three steps. 
\begin{itemize}
\item[$\bullet$] In Proposition~\ref{p adic restriction intermediate step} we first prove a weaker version of Theorem~\ref{p adic restriction sequence} in which the fields $k_{\Pi (F)} \langle p \rangle$ (as defined in (\ref{k_{p} def})) are replaced by the larger fields $k (p, \cT)$ from   Lemma~\ref{induction-step-revisited}. 
\item[$\bullet$] In Lemma~\ref{p adic restriction sequence another variant} we then strengthen the result of Proposition~\ref{p adic restriction intermediate step} by replacing each $k(p, \cT)$ by its subfield $k\langle p\rangle$ defined in (\ref{angle field}).  
\item[$\bullet$] Finally, we deduce Theorem~\ref{p adic restriction sequence} by showing that Lemma~\ref{p adic restriction sequence another variant} remains valid if each field    $k\langle p\rangle$ is replaced by its subfield $k_{\Pi (F)} \langle p \rangle$.  
\end{itemize} 
We shall therefore start this section with the following result. }

\begin{prop} \label{p adic restriction intermediate step}
Fix  $p\in \mathbb{P}$, $\cX\subseteq \Omega$ that satisfies Hypothesis \ref{hyp X} and a triangular system $\{ (M_{F, E}, \rho_{F / E}, j_{F / E})\}_{E, F \in \cX}$ that satisfies Hypothesis \ref{inj-hypothesis} and is also such that every $M_E$ is both $\Z$-torsion-free and satisfies $(M_E)_{p\cdot\mathrm{div}} = (0)$. Let $\Pi$ and $\Sigma$ be finite subsets of $S_\fin (k)$. 
 Then there exists an exact sequence of $\Z\llbracket G_k^{\mathrm{ab}} \rrbracket$-modules 
\[ 0 \longrightarrow \big ( \eullim{F \in \cX}{\Pi} M_F \big)^{\gal{k^{\mathrm{ab}}}{k  (p, \cT)}} \xrightarrow{\,\,\subseteq\,\,} \eullim{F \in \cX}{\Pi} M_F \longrightarrow
\eullim{F \in \cX_{\Sigma}}{\Pi} \big ( M_F / M_F^{\gal{F}{F(p)}}\big ),\]
 in which $\cT$ is a finite subset of $S_\fin (k)$ that contains $\Sigma \cup S_p (k)$ (and that is constructed in the proof), we have set
 \[ L(p) \coloneqq L \cap k (p, \cT)\]
using the field $k (p, \cT)$ defined in Lemma~\ref{induction-step-revisited}, and 
 the third arrow is induced by the restriction map $\res^\Pi_{\Sigma}$.
\end{prop}

\begin{proof}
Let us first specify the set $\cT$. For every $\p \in \Sigma$ we may apply Lemma \ref{quadratic extension lemma} below (with $\Sigma' = \{ \p \}$) to obtain a chain of fields in $\Omega$
\[ k = L_0^{( \p)}  \subset L_1^{(\p)}\subset \cdots \subset L_t^{(\p)} \subset  L^{( \p)}\]
with $|S(L^{( \p)}_i)\setminus S(L^{( \p)}_{i-1})| = 1$ for $1\le i\le t$, $S(L^{(\p)})\setminus S(L^{( \p)}_t) =\{\p\}$ and $S(L^{(\p)})\cap S_\infty = \emptyset$.
We then take $\cT$ to be any finite subset of $S_\fin (k)$ that contains $\Sigma \cup S_p (k) \cup \bigcup_{\p \in \Sigma} S(L^{(\p)})$ and such that $\cT$-class group of $k$ vanishes.\\
Turning now to the proof that the sequence in the statement of the proposition is exact, we need to show that if $a = (a_K)_K\in \eullim{K \in \cX}{\Pi} M_K$ is such that $a_{K'}\in M_{K'}^{\gal{K'}{K' (p)}}$ for every $K'\in \cX_{\Sigma}$, then  $a_K \in M_K^{\gal{K}{K (p)}}$ for every $K\in \cX$. To do this, we fix such an element $a$ and argue by induction on the quantity 
\[ n (K) \coloneqq |\Sigma \setminus S (K)|.\]

If, firstly, $n(K) = 0$, then $\Sigma\subseteq S(K)$ and so $K \in \cX_{\Sigma}$. In this case, therefore, the given hypothesis on $a$ directly implies $a_K\in M_K^{\gal{K}{K (p)}}$, as required.\\ 
We therefore now assume $n (K) > 0$ and that the assertion is valid for all $F \in \cX$ such that $n (F) < n (K)$. Then, since $n (K) > 0$, we may fix a (finite) place $\p$ in $\Sigma \setminus S(K)$. We  set $F \coloneqq L^{(\p)}K$ and $F_i \coloneqq L^{(\p)}_i K$  for each $i$ with $1\le i\le t$ and thereby  obtain a chain of fields in $\cX$ 
\[ K = F_0 \subseteq F_1\subseteq \cdots \subseteq F_t \subseteq F\]
with $|S(F_i)\setminus S(F_{i-1})| \leq  1$ for $1\le i\le t$, $S(F)\setminus S(F_t) = \{\p\}$ and $V_K = V_F$.

We shall now argue by a downwards induction on $i$ (for $0\le i\le t$) that for any field $E$ in $\cX$ containing $F_i$ the element $a_{E}$ belongs to $M_{E}^{\gal{E}{E (p)}}$. We note in particular that, since $F_0 = K$, this result for $i=0$ implies the claimed result. 

For the inductive argument we first take a field $E_t$ in $\cX$ that contains $F_t$ and write $E$ for the field $ E_tF= E_tL^{(\p)}$. Then $E$ belongs to $\cX$ and $n(E) < n(K)$ since $\p \in S(F) \subseteq S(E)$, and so the inductive hypothesis implies $a_{E}$ belongs to $M_{E}^{\gal{E}{E (p)}}$. 

In addition, if either $S(E_{t}) = S(E)$ or $\p \in \Pi$, then the Euler factor $P_{\Pi, E/ E_t}$ is trivial and so the Euler limit relations directly imply that $a_{E_{t}} = \rho_{E/ E_{t}} ( a_{E})$ is contained in $M_{E_{t}}^{\gal{E_t}{E_t (p)}}$. On the other hand, if $S(E) \setminus \Pi (E_t) = \{ \p \}$ then, as $V_{E_t} = V_E$, the stated hypotheses allow us to apply Lemma~\ref{induction-step-revisited} (with $F$ and $K$ taken to be $E$ and $E_t$) in order to deduce that $a_{E_t}$ is again fixed by $\gal{E_t}{E_t (p)}$, as required. 

This verifies the inductive base and then the inductive step is  established by the same argument with the roles of $F$, $F_t$ and $\p$ played by $F_i$, $F_{i-1}$ and the unique place in $S(F_i)\setminus S(F_{i-1})$ (if such a place exists). This proves the claimed result. \end{proof}

{The following lemma was used in the inductive step of the above proof.}

\begin{lem} \label{quadratic extension lemma}
Let $\Sigma'$ be a finite subset of $S_\fin(k)$. Then there exists $t\in \N_0$ and a chain of fields $k = L_0  \subset L_1\subset \cdots \subset L_t \subset  L$ in $\Omega$ such that $|S(L_i)\setminus S(L_{i-1})| = 1$ for $1\le i\le t$, $S(L)\setminus S(L_t) =  \Sigma'$ and $S(L)\cap S_\infty = \emptyset$.
\end{lem}

\begin{proof} For  $v \in S_\fin(k)$ and $n \in \N$, we write $U_v^{(n)}$ for the group of units in the ring of integers $\cO_v$ of the completion $k_v$ that are congruent to $1$ modulo $(v\cO_v)^n$. %
 Then Hensel's Lemma implies that for any integer $\ell$ there exists a natural number $n_v$ with the property that $( U_v^{(n_v)})^\ell = U_v^{(n_v + \ord_v (\ell))}$ (for details see \cite[Ch.\@ II, Lem.\@ 3.5]{Neukirch-cft}). 

 We define a conductor $\mathfrak{m} \coloneqq (\prod_{v'}v')(\prod_{v\in S_2(k)\setminus \Sigma'} v^{n_v + \ord_v (2)})$ of $k$, where $v'$ runs over all real archimedean places of $k$, and write $k(\mathfrak{m})$ for the ray class field of $k$ of conductor $\mathfrak{m}$. We fix an odd prime $\ell$ that does not ramify in $k$, set $k[\ell] \coloneqq k(\mu_\ell,(\mathcal{O}_{k}^\times)^{1/\ell})$ and consider the diagram
 \begin{cdiagram}[column sep=tiny, row sep=small]
        & k(\mathfrak{m}) \cdot k[\ell] \ar[dash]{dl} \ar[dash]{dr}\\
        k(\mathfrak{m}) & & k[\ell]\\
        & k. \ar[dash]{ul} \ar[dash]{ur}
    \end{cdiagram}%
Then $\ell$ is prime to $|\mu_k|$ so that the maximal abelian extension of $k$ in $k[\ell]$ is $k(\mu_\ell)$ (see the argument of Lemma~\ref{max ab subext lemma}(b)) and so is totally ramified at all places in $S_\ell(k)$. One therefore has $k(\mathfrak{m})\cap k[\ell] = k$. 
In particular, for each element $\sigma$ of $\cG_{k(\mathfrak{m})}$ we can use Cebotarev's Density Theorem to fix a place $\p_\sigma$ of $k$ that is outside $\Sigma' \cup S_2 (k)$, splits completely in $k[\ell]$ and has Frobenius on $k(\mathfrak{m})$ equal to $\sigma$. We note that the degree of $k(\p_\sigma)/k(1)$ is divisible by $\ell$ (by Lemma~\ref{rubin lemma}) and hence that $k(\p_\sigma)/k$ is ramified at $\p_\sigma$ (but not at any archimedean place). \\
Write $k_{\mathfrak{m},1}$ for the subgroup of $k^\times$ comprising elements $x$ with $|x|_v > 0$ for all real places of $k$ and $x\equiv \, 1\, (\mathrm{mod}\,\,
 v^{n_v + \ord_v (2)}\mathcal{O}_v)$ for all $v$ in $S_2(k) \setminus \Sigma'$, and recall that the $\m$-ray class group $\cl_{k, \emptyset, \m}$ of $k$ is defined as the group of fractional ideals that are coprime to $\m$ modulo its subgroup of principal ideals with a generator in $k_{\mathfrak{m}, 1}$. 
Since the set $\Upsilon \coloneqq  \{\p_\sigma \mid \sigma \in \cG_{k(\mathfrak{m})}\}$ generates $\cl_{k, \emptyset, \m}$, we find that the ideal $\prod_{\p \in \Sigma'} \p$ is equal to the product of an ideal only supported on places in $\Upsilon$ with an element $x$ of $k_{\m, 1}$. 
This element $x$ therefore is integral outside $\Sigma' \cup \Upsilon$ and, moreover, has all of the following properties:
\begin{enumerate}[label=$\bullet$]
    \item $\ord_\p (x) = 1$ if $\p \in \Sigma'$,
    \item $x \equiv 1 \, (\!\!\!\!\mod v^{n_v + \ord_v (2)})\,$ if $v$ is $2$-adic and $v \not \in \Sigma'$, 
    \item $|x|_v > 0$ if $v$ is a real archimedean place of $k$.
\end{enumerate}
For any such $x$ the quadratic extension $L' \coloneqq k(\sqrt{x})/k$ is unramified outside $S_2(k) \cup \Sigma' \cup \Upsilon$. Further, since (by construction and the earlier observation concerning principal units) $x$ is a square in $k_v^\times$ for all $2$-adic places $v$ outside  $\Sigma'$, any such $v$ splits completely in $L'$ and so $L' / k$ is unramified outside $ \Sigma' \cup \Upsilon$.
 Finally we note that $L' / k$ is ramified at each place $\p$ in $\Sigma'$ because $x$ is a uniformiser in $k_\p$ for all such $\p$. The field $L'$ is therefore such that $\Sigma' \subseteq S(L') \subseteq \Sigma' \cup \Upsilon$. 
 
We now set $t \coloneqq |\cG_{k(\mathfrak{m})}|$ and fix an ordering $\{\p_j: 1\le j\le t\}$ of the places $\{\p_\sigma: \sigma \in \cG_{k(\mathfrak{m})}\}$. We define $L$ to be the compositum of $L'$ and the fields $\{k(\p_j): 1\le j\le t\}$ and for each integer $i$ with $1\le i\le t$ we write $L_i$ for the maximal extension of $k$ in $L$ that is unramified outside $\{\p_j: 1\le j\le i\}$. This gives a chain of fields of the required sort since $S(L_i) = \{\p_j: 1\le j\le i\}$ for $1\le i\le t$ (since each extension $k(\p_j)/k$ is ramified at precisely $\p_j$) and hence both 
$S(L) = \Upsilon \cup \Sigma'$ and $S(L)\setminus S(L_t) = \Sigma'$. 
\end{proof}

{In the following result, we verify the second step in the strategy described above.}

\begin{lem} \label{p adic restriction sequence another variant}
Assume to be given data of the following sort: 
\begin{itemize}
\item $p\in \mathbb{P}\setminus \{2\}$,
    \item $\cX\subseteq \Omega$ satisfying Hypothesis \ref{hyp X}, 
    \item a triangular system $\{ (M_{F, E}, \rho_{F / E}, j_{F / E})\}_{E, F \in \cX}$ that satisfies Hypothesis \ref{inj-hypothesis} with respect to $p$ and is such that each $M_E$ is a $\Z_p [\cG_E]$-lattice,
    \item finite subsets $\Pi$ and $\Sigma$ of $S_\fin (k)$.
\end{itemize}
Then there exists an exact sequence of $\Z_p \llbracket G_k^{\mathrm{ab}} \rrbracket$-modules  
    \[  0 \to \eullim{F \in \cX}{\Pi} \bigl(M_F^{\gal{F}{F \cap k\langle p\rangle}}\bigr) \xrightarrow{\,\,\subseteq\,\,}   
     \eullim{F \in \cX}{\Pi} M_F \xrightarrow{\res^\Pi_\Sigma}   \eullim{F \in \cX_{\Sigma}}{\Pi} \big ( M_F / M_F^{\gal{F}{F(p)}}\big ),
    \]%
     where the fields $F(p)$ are as defined in Proposition~\ref{p adic restriction intermediate step}. 
\end{lem}

\begin{proof} 
Fix $(a_{F'})_{F'}\in \ker(\res^\Pi_\Sigma)$ and $F$ in $\cX$. Then, by Proposition~\ref{p adic restriction intermediate step},  $a_F$ is fixed by the group $\gal{F}{F (p)}$ where, we recall, $F(p)$ denotes  $F \cap k (\mu_{p^\infty}, (\cO_{k, \cT}^\times)^{1 / p^s})$ for a suitable finite set of places $\cT$ of $k$ that contains $\Sigma \cup S_p  (k)$, and so it suffices to prove $a_F$ is  fixed by the natural action of $\gal{F (p)}{F \cap k\langle p\rangle}$. 
To do this, we note that  Lemma~\ref{max ab subext lemma}\,(b) implies that 
\[ F(p)\cap k (\mu_p, (\cO_k^\times)^{1 / p}) = F \cap k (\mu_{p^\infty}, (\cO_{k, \cT}^\times)^{1 / p^s})\cap k (\mu_p, (\cO_k^\times)^{1 / p}) = F \cap k\langle p\rangle\]
and hence that there is an isomorphism of Galois groups
\[
\gal{F (p) ( \mu_p, (\cO_k^\times)^{1 / p})}{(F \cap k\langle p\rangle)} \cong
\gal{F (p)}{F \cap k\langle p\rangle} \times \gal{k (\mu_p, (\cO_k^\times)^{1 / p})}{(F \cap k\langle p\rangle)}. 
\]
For any $\sigma$ in $\gal{F(p)}{F\cap k\langle p\rangle}$ and natural number $d$, we can therefore use Cebotarev's Density Theorem to fix a set $\Sigma' = \Sigma' (d,\sigma) = \{\mathfrak{q}_i\}_{1\le i\le d}$ of $d$ places of $k$ that do not belong to $\Pi \cup \cT$ and are such that every $\q_i$ has the following properties: $\mathfrak{q}_i$ is unramified in $F$ and such that the restriction of $\rm{Frob}_{\mathfrak{q}_i}$ to $F(p)$ is equal to $\sigma^{-1}$, $\mathfrak{q}_i$ is totally split in $k(\mu_p,(\mathcal{O}_k^\times)^{1/p})$ and hence, by Lemma~\ref{rubin lemma}, such that the ramification degree $e_i$ of $\mathfrak{q}_i$ in $k(\mathfrak{q}_i)$ is divisible by $p$. \\
Write $F'_d$ for the compositum of $F$ and the fields $k(\mathfrak{q}_i)$ for $1\le i\le d$, and $I$ for the subgroup of $\gal{F'_d}{F}$ generated by the inertia subgroups in $\G_{F'_d}$ of each of the places $\mathfrak{q}_i$.
 Then, since $F'_d (p) / k$ is unramified at every place in $\Sigma'$, $I$ is contained in $\gal{F'_d}{F'_d (p)}$ and so $a_{F'_d}$ is fixed by $I$. In particular, since $S(F'_d)\setminus \Pi (F) = \Sigma'$, the Euler distribution relations combine with Hypothesis \ref{inj-hypothesis}\,(i) to imply that the element 
\begin{align*}
 \iota_{F'_d / F} \big ((1 - \sigma)^{d} \cdot a_F \big ) & = \iota_{F'_d / F} \big( \bigl({\prod}_{v \in \Sigma'}(1- \rm {Frob}_{v}^{-1})\bigr)\cdot a_F \big)\\
 & =   \iota_{F'_d / F} ( P_{F_d'/F, \Pi}\cdot a_F ) \\
& = (\iota_{F'_d / F} \circ \rho_{F_d'/F, \Pi}) (a_{F'_d}) 
\end{align*}
\textcolor{black}{is divisible by $p^d$ in $M_{F'_d}$}. It then follows from Hypothesis \ref{inj-hypothesis}\,(ii) that $(1 - \sigma)^{d} \cdot a_F$ belongs to $p^{d - N}\cdot M_F$ for a natural number $N$ that is independent of $d$. 

Now, since $M_F$ is $\ZZ$-torsion-free, the morphism 
\[
M_F \cong \Z [\langle \sigma \rangle] \otimes_{\Z [\langle \sigma \rangle]} M_F 
\longrightarrow \bigoplus_{\chi} (\Z_p [ \chi] \otimes_{\Z [\langle \sigma \rangle]} M_F)_{\mathrm{tf}},
\quad 
\sigma \otimes m \mapsto ( \chi (\sigma) \otimes m)_\chi
\]
is injective, where the sum ranges over all characters of $\langle \sigma \rangle$, we set $\Z_p [\chi] \coloneqq \Z_p [\im (\chi)]$ and write $N_{\mathrm{ tf}}$ for the quotient of a finitely-generated $\Z_p[\chi]$-module by its torsion subgroup. It is therefore enough for us to prove that 
$(1 - \chi (\sigma)) \otimes a_F$ vanishes in $(\Z_p [\chi] \otimes_{\Z [\langle \sigma \rangle]} M_F)_{\mathrm{tf}}$ for all such non-trivial characters $\chi$. To do this, we note that, because $\chi$ is a $p$-power order character, the field  $\Q_p(\chi) \coloneqq \Q_p (\im (\chi))$ is a totally ramified extension of $\Q_p$ with valuation ring $\Z_p[\chi]$ and uniformiser $\pi_\chi \coloneqq 1 - \chi (\sigma)$. In particular, one has $p = u\cdot \pi_\chi^{[\Q_p (\chi) : \Q_p]}$ for some unit $u$ of $\Z_p [\chi]$. 
It then follows from $(1 - \sigma)^{d} a_F \in p^{d - N}\cdot M_F$ and $[\Q_p (\chi) : \Q_p] \geq p - 1$ that 
\begin{align*}
(\chi (\sigma) - 1) \otimes a_F & \in \pi_\chi^{(d - N) [\Q_p (\chi) : \Q_p] - (d - 1)}( \Z_p [\chi] \otimes_{\Z [\langle \sigma \rangle]} M_F)_{\mathrm{tf}} \\
& \subseteq \pi_\chi^{d ( p - 2) - N ( p - 1) + 1}( \Z_p [ \chi] \otimes_{\Z [\langle \sigma \rangle]} M_F)_{\mathrm{tf}}.
\end{align*}
Since, if $p$ is odd, the exponent $d ( p - 2) - N ( p - 1) + 1$ is unbounded as $d$ increases, this shows that $(\chi (\sigma) - 1) \otimes a_F $ is divisible by an arbitrarily large power of $\pi_\chi$ in the $\Z_p [\chi]$-lattice $(\Z_p [\chi] \otimes_{\Z [\langle \sigma \rangle]} M_F)_{\rm{tf}}$. 
It follows that the element $(\chi (\sigma) - 1) \otimes a_F $ of $(\Z_p [\chi] \otimes_{\Z [\langle \sigma \rangle]} M_F)_{\mathrm{tf}}$ vanishes, as required.
\end{proof}

\begin{rk}
\label{p = 2 remark restriction sequence}
If $p = 2$, then the above argument proves the following variant of Lemma~\ref{p adic restriction sequence another variant}: for
every element $a = (a_F)_{F \in \cX}$ of $\ker(\res^\Pi_\Sigma)$ and every field $F$ in $\Omega$ one has $a_F = \sigma^2(a_F)$ for all $\sigma$ in   
$\gal{F}{F \cap k (\mu_{2^{s + 1}}, (\cO_{k}^\times)^{1 / 2})}$.  
\end{rk}

{In the rest of this section we prove Theorem~\ref{p adic restriction sequence} by verifying the third step of the strategy described at the start of the section. For this, we will require the following two results.}

\begin{lem} \label{vanishes on p ramified fields}
    Fix a subset $\cX$ of $\Omega$ that satisfies Hypothesis \ref{hyp X} and a triangular system $\{ ( M_{F, E}, \rho_{F / E}, j_{F / E} ) \}_{E, F \in \cX}$   that satisfies Hypothesis \ref{inj-hypothesis} with respect to $p$ and is also such that $(M_E)_{p\cdot\mathrm{div}} = (0)$ for all $E \in \cX$. Fix finite subsets $\Pi$ and $\Sigma$ of $S_\fin (k)$, an open subgroup $H$ of $G_k^{\mathrm{ab}}$ and an element $a = (a_E)_E\in \eullim{E \in \cX}{\Pi} M_E$ with the property that $a_E$ is fixed by $H$ for all $E\in \cX_\Sigma$. Then, for every $F \in \cX$ with $S_p\subseteq \Pi (F)$, one has  $a_F= 0$. 
\end{lem}

\begin{proof}
    Fix  $F$ in $\cX$ such that $S_p\subseteq \Pi (F)$ and, for every $n \in \N_0$, write $F_n$ for the $n$-th layer of the cyclotomic $\Z_p$-extension $F_\infty$ of $F$ (with the convention that $F_0 = F$).\\
    Then, by Lemma~\ref{quadratic extension lemma} (with $\Sigma' = \Sigma$), there  exists an extension $L$ of $k$ such that $\Sigma\subseteq S(L)$ and $S(L)\cap S_\infty = \emptyset$. In particular, for every $n \in \N_0$ the field $L_n \coloneqq L \cdot F_n$ belongs to $\cX_\Sigma$ and so, by assumption, the element $a_{L_n}$ is fixed by $H$. Since $S (L_n) \setminus \Pi (F_n) = S (L) \setminus \Pi (F)$, the Euler limit relations then show that, for every $n$, the element 
\[
\big( {\prod}_{v \in S (L) \setminus \Pi (F)} (1 - \Frob_v^{-1}) \big) \cdot a_{F_n} = \rho_{L_n / F_n} (a_{L_n})
\in M_{F_n}
\]
is fixed by $H$.

Now, since $S_\ram(F_{n+1}/ F_n)\subseteq S_p$, one has $P_{F_{n + 1} / F_n, \Pi} = 1$ and so the Euler limit relations imply  $a_{F_\infty} \coloneqq (a_{F_n})_{n \in \N_0}$ belongs to $\varprojlim_{n \in \N_0} (\Z_p \otimes_\Z M_{F_n})$, where the limit is taken with respect to the maps $\rho_{F_{n + 1}/F_n}$. In addition, if we fix any non-trivial element $\gamma$ of the open subgroup $\gal{F_\infty}{F} \cap H$ of $\gal{F_\infty}{k}$, then the last displayed equation implies that $a_{F_\infty}$ is annihilated by the element $(\gamma - 1) \cdot \prod_{v \in S(L) \setminus \Pi (F)} (1 - \Frob_v^{-1})$ of $\Z_p \llbracket \cG_{F_\infty} \rrbracket$.
In particular, since each of the elements $\gamma$ and $\Frob_v$ for $v \in S (L) \setminus \Pi (F)$ generates an open subgroup of $\cG_{F_\infty}$, we may apply Lemma~\ref{some-strange-lemma} below to deduce that $a_{F_\infty}$, and hence also $a_F$, vanishes, as required.
\end{proof}

\begin{lem} \label{some-strange-lemma}
Fix a $\Z_p$-extension $k_\infty$ of $k$ in which no finite place splits completely and a field $K \in \cX$. Set $K_\infty \coloneqq K k_\infty$ and, for each $n \in \N_0$, write  $K_n$ for the subfield of $K_\infty$ of degree $p^n$ over $K$. Let $\{ (M_{K_m}, \varphi_{K_m / K_n} ) \}_{n \geq 0}$ be a projective system of $\Z_p [\cG_{K_n}]$-modules that, \textcolor{black}{when regarded as a triangular system as in Example~\ref{triangular-system-examples}\,(a)},
satisfies Hypothesis \ref{inj-hypothesis} with respect to $p$ and is also such that $(M_{K_n})_{p\cdot\mathrm{div}} = (0)$ for every $n \in \N_0$. Then, for any element $\gamma \in \cG_{K_\infty}$ \textcolor{black}{of infinite order}, the element  $\gamma - 1$ acts injectively on $\varprojlim_n M_{K_n}$.
\end{lem}

\begin{proof} 
Fix an element $a = (a_n)_n$ of this limit such that $(\gamma - 1) a = 0$. Then each $a_n$ is fixed by the restriction of $\gamma$ to $K_n$. \\
Now fix $N \in \N_0$. Then, by assumption, $L \coloneqq K_\infty^{\langle \gamma \rangle}$ is a finite extension of $K$ and so $ K_N L$ is 
a finite extension of $k$ and hence $K_N L = K_n$ for some $n\in \N_0$. In addition, the discussion above implies that, for any $m \geq n$, the element $a_{m}$ is invariant under the action of $\gamma$, hence contained in $M_{K_m}^{\gal{K_m}{K_n}}$. Hypothesis \ref{inj-hypothesis}\,(i) therefore implies that
\[
\iota_{K_m / K_n} (a_n) = (\iota_{K_m / K_n} \circ \varphi_{K_m / K_n}) (a_m) = \NN_{\gal{K_m}{K_n}} \cdot f (a_m) = [K_m : K_n] \cdot f(a_m)
\]
for some endomorphism $f$ of $M_{K_m}$. This shows  $\iota_{K_m / K_n} (a_n)$ 
is divisible by $[K_m : K_n] = p^{m - n}$ in $M_{K_m}^{\gal{K_m}{K_n}}$. Then, since $p^{m-n}$ is unbounded as $m$ tends to infinity, this fact combines with Hypothesis \ref{inj-hypothesis}\,(ii) to imply \textcolor{black}{$a_n \in (M_{K_n})_{p\cdot\mathrm{div}}$ and hence, by the stated assumption, that $a_n = 0$}. In particular, since $K_N \subseteq K_n$ one has $a_N = \varphi_{K_n / K_N}(a_n) = 0$ and so, since $N$ was chosen arbitrarily, it follows that $a = 0$, as claimed. 
\end{proof}

We are now at last ready to prove Theorem~\ref{p adic restriction sequence}. 
\medskip

\textit{Proof of Theorem~\ref{p adic restriction sequence}:} We assume the hypotheses of the latter result and fix an element 
$a = (a_{F'})_{F'}$ of $\ker(\res^\Pi_\Sigma)$. Then Lemma~\ref{p adic restriction sequence another variant} implies $a$ is fixed by the open subgroup $\gal{k^{\mathrm{ab}}}{k \langle p \rangle}$ of $G_k^{\mathrm{ab}}$ and so Lemma~\ref{vanishes on p ramified fields} implies  $a_F = 0$ for every $F$ in $\Omega$ with $S_p \subseteq \Pi (F)$. To proceed, we now fix a field $F$ in $\Omega$ and consider the fields  
\[ k^{\split} \coloneqq k_{\Pi (F)} \langle p \rangle \,\,\subseteq \,\, k^\unr  \coloneqq  k_{\Pi (F)}^\unr \langle p \rangle \]
that are respectively obtained as the compositum of all {abelian} extensions of $k$ in $k\langle p\rangle$ in which at least one place of $S_p \setminus \Pi (F)$ splits completely, respectively is non-ramified. {Then, in order to prove Theorem~\ref{p adic restriction sequence}, we must show that each element $a_F$ is fixed by $\gal{F}{F \cap k^\mathrm{s}}$.} In particular, since we have already observed that $a_F = 0$ if $(S_p  \setminus \Pi) \subseteq S(F)$, in the rest of this argument we may, and will, assume that at least one place in $S_p  \setminus \Pi$ is unramified in $F$. 
It follows that $F \cap k \langle p \rangle\subseteq k^\mathrm{nr}$ and so $a_F$ is fixed by $\gal{F}{F \cap k^\mathrm{nr}}$. To prove $a_F$ is fixed by $\gal{F}{F \cap k^\split}$ it is therefore enough for us to 
prove that  
  $e_\chi 
  a_F = 0$ for every character $\chi$ of $\cG_F$ 
 that factors through $\cG_{F \cap k^\unr}$ but not $\cG_{F \cap k^\split}$.  \\
To do this, we fix such a $\chi$ and claim that there can exist no place $v \in S_p \setminus \Pi$ of $k$ that is both unramified in $F$ and such that $\chi (\Frob_v) = 1$. Indeed, since any such place $v$ splits completely in the fixed field $F_\chi$ of $\ker(\chi)$ in $F$, one would have $F_\chi \cap k^\unr \subset F \cap k^\split$ which contradicts the fact that $\chi$ does not factor through  $\cG_{F \cap k^\split}$. We may therefore assume that $\chi (\Frob_v) \neq 1$ for every place $v$ in $S_p  \setminus \Pi (F)$. We now fix $n \in N$ large enough to ensure that the $n$-th layer $F_n$ in the cyclotomic $\Z_p$-extension of $F$ is such that \textcolor{black}{$S_p \subseteq S(F_n)$ (or, equivalently, $S(F_n) \setminus S(F) = S_p \setminus \Pi (F)$)} and so $a_{F_n} = 0$. Then the $\Pi$-relative Euler system relations imply that 
\[
\chi ( P_{F_n / F,\Pi}) \cdot e_\chi a_F =
e_\chi \cdot P_{F_n / F,\Pi} \cdot a_F
= e_\chi\cdot \rho_{F_n / F} (a_{F_n}) = 0.\]
In particular, since the element 
\[ \chi ( P_{F_n / F,\Pi}) = {\prod}_{v \in S_p  \setminus \Pi (F)} (1 - \chi (\Frob_v)^{-1})\]
is non-zero, the above equality implies that $e_\chi a_F$ vanishes, as required to complete the proof of Theorem~\ref{p adic restriction sequence}.\qed 

\subsubsection{The proof of Corollary \ref{global restriction sequence}}\label{grs proof section} In the setting of Corollary \ref{global restriction sequence}, we suppose to be given an element $a = (a_F)_F$ of $\eullim{F \in \cX}{\Pi} M_F$ with the property that $a_F$ is fixed by $\cG_F$ whenever $F$ belongs to $\cX_\Sigma$. 

Then, by assumption, we may choose a large enough odd prime number $p$ such that $p\notin \cS\cup S_\ram(k/\Q)$ and Hypothesis \ref{inj-hypothesis} holds for the triangular system $\{ (M_{F, E}, \rho_{F / E}, j_{F / E}) \}_{E, F \in \cX}$. In particular, $\{ (\Z_p \otimes_\Z M_{F, E}, \rho_{F / E}, j_{F / E}) \}_{E, F \in \cX}$ is then a triangular system that satisfies Hypothesis \ref{inj-hypothesis} and is such that each diagonal term $M_E$ is a $\Z_p [\cG_E]$-lattice.\\
\textcolor{black}{
By assumption, the element $a$ then belongs to the kernel of the map $\overline{\res}^\Pi_\Sigma$ that occurs in Theorem~\ref{p adic restriction sequence} for the triangular system $\{ (\Z_p \otimes_\Z M_{F, E}, \rho_{F / E}, j_{F / E}) \}_{E, F \in \cX}$. 
We may therefore apply the latter result to deduce that each element $a_F$ is fixed by the group $\gal{F}{F \cap k_{\Pi (F)} \langle p \rangle}$, and hence that $a$ is fixed by the open subgroup $\gal{k^{\mathrm{ab}}}{k^{\mathrm{ab}} \cap k \langle p \rangle}$ of $G_k^{\mathrm{ab}}$.
An application of Lemma~\ref{vanishes on p ramified fields} then moreover shows that the element $a_F$ vanishes whenever $S_p \subseteq \Pi (F)$. 
}\\
Now, since $p \notin S_\ram(k/\Q)$, the extension $k (\mu_p) / k$ is totally ramified at every $p$-adic place and so $k_{\Pi (F)} \langle p \rangle = k$ for all $F$ in $\cX$ with $(S_p \setminus \Pi (F)) \not= \emptyset$. Hence, for any such $F$, the element $a_F$ is fixed by $\cG_F$. On the other hand, if $S_p \subseteq \Pi (F)$, then we have already seen above that $a_F=0$ as a consequence of Lemma~\ref{vanishes on p ramified fields}. 

We have therefore proved that the element $a_F$ is fixed by $\cG_F$ for every $F$ in $\cX$, as required to prove Corollary \ref{global restriction sequence}.
\qed

\subsection{The localisation exact sequence} \label{localisation sequence section}

\subsubsection{Statement of the main result} 

In this section we assume that the subset $\Pi$ of $S_\fin(k)$ fixed in previous sections is empty. We also assume to be given both a projective system $\{ (M_F, \varphi_{F / E}) \}_{F \in \cX}$ and a triangular system $\{ (N_{F, E}, \rho_{F / E}, j_{F / E})\}_{E, F \in \cX}$ of $\cR[\cG_F]$-modules. Finally, with $e_F$ the idempotent of $\Q[\cG_F]$ defined in (\ref{eK def}), we assume to be given a family of injective $\cR [\cG_F]$-module homomorphisms $\beta_F \: e_{F} M_F \to N_F$ that gives rise to an (injective) map of $\cR \llbracket G_k^{\mathrm{ab}} \rrbracket$-modules 
\begin{equation}\label{linking hom}
\beta = (\beta_F)_F \: \eullim{F \in \cX}{\emptyset} e_{F} M_F \to \eullim{F \in \cX}{\emptyset} N_F,
\end{equation}
in which the left hand Euler limit is as defined in Example~\ref{es remark}\,(d).  

Our main aim is then to investigate the composite homomorphism 
\begin{equation}\label{beta ast def} \beta_\ast \: \varprojlim_{F \in \cX} \epsilon_F M_F \to \eullim{F \in \cX}{\emptyset} e_{F} M_F
\xrightarrow{\beta} \eullim{F \in \cX}{\emptyset} N_F,
\quad
(m_F)_F \mapsto (e_{F} m_F)_F \mapsto \beta ((e_{F} m_F)_F) \end{equation}
under the assumption that $(M_{F}, \varphi_{F / E})_{F\in \cX}$ satisfies the following technical hypothesis. 

\begin{hypothesis}\label{rigidity hyp}
For all fields $E, F$ in $\cX$ with $E \subseteq F$, there exists an injective map $\iota_{F/E} \: M_E \to M_F$ of $\cR[\cG_F]$-modules such that 
$(M_F,\iota_{F/E})_{F \in \cX}$ is an inductive system of $\cR[G_k^{\mathrm{ab}}]$-modules for which the  following condition is satisfied:
    for some $s\in \NN$ (that is independent of $F/E$) one has both  
    $\iota_{F / E} \circ \varphi_{F / E} = \NN_{\gal{F}{E}}^s$ and $\varphi_{F / E} \circ \iota_{F / E} = [F : E]^s.$
\end{hypothesis}

\begin{bsp}\label{rigidity-hyp-example-group-rings}
For any Dedekind domain $\cR$, the discussion of Example~\ref{inj examples}\,(b) shows that the canonical projective system $(\cR [\cG_E], \pi_{F / E})$ satisfies Hypothesis \ref{rigidity hyp} with $s=1$.
\end{bsp}

We recall the idempotents $\epsilon_K$ defined in (\ref{epsK def}) and then define a subset of $\cG^\circ_K$ by setting 
\begin{equation}\label{diamond def}\Upsilon^\diamond_K\coloneqq \{ \chi\in \cG_K^\circ: \, \chi (\epsilon_K) \neq 0\} = \{\chi\in \cG_K^\circ: \, \forall \, v \in S_\infty,\, \chi (\cG_{K, v}) = 1 \Longleftrightarrow v\in V_K\},\end{equation}
where the equality of sets follows directly from the definition of $\epsilon_K$. We note, in particular, that $\Upsilon^\diamond_K$ contains the set $\Upsilon_{K}$ defined in \S\,\ref{ES-Gm-section} (just before (\ref{eK def})), and hence that the idempotent $e_K$ is such that $e_{K} \epsilon_K = e_{K}$. 

We can now state the main result of this section.

\begin{thm}\label{rigidity theorem} 
 We assume to be given data of the following sort: 
 \begin{itemize}
 \item a finite subset $\cS$ of $\mathbb{P}$,
     \item a subset $V$ of $S_\infty$,
     \item a subset $\cX$ of $\Omega$ that  satisfies Hypothesis \ref{hyp X} with $\cV = \{ V\}$, 
    \item a projective system of $\Z_\mathcal{S} [\cG_F]$-lattices $\{ (M_F, \varphi_{F / E}) \}_{F \in \cX}$ that satisfies  Hypothesis \ref{rigidity hyp},
    \item a triangular system $\{ (N_{F, E}, \rho_{F / E}, j_{F / E}) \}_{F, E \in \cX}$
     for which each $N_E$ is a $\Z_\cS [\cG_E]$-lattice,
    \item an injective homomorphism $\beta = (\beta_F)_F$ of $\Z_\mathcal{S}\llbracket G_k^{\mathrm{ab}} \rrbracket$-modules of the form (\ref{linking hom}). 
\end{itemize}
Then, for any such collection of data, the following sequence is exact:
\[ 0 \to \varprojlim_{F \in \cX} \epsilon_F M_F \xlongrightarrow{\beta_\ast }   
    \eullim{F \in \cX}{\emptyset}\bigl(N_F \cap (\Q\otimes_\Z\beta_{\ast,F}(\epsilon_ FM_F))\bigr) 
    \xlongrightarrow{\Delta}  \prod_{p \in \mathbb{P}\setminus \mathcal{S}} \Big ( \frac{ \eullim{F \in \cX}{\emptyset} (\Z_p \otimes_\Z N_F)}{\beta_{\ast}^{(p)} (\ker (\theta_p))} \Big).
\]%
Here, by abuse of notation, we write $\beta_\ast$ for the homomorphism induced by the map in (\ref{beta ast def}) (and the equalities $e_{F} \epsilon_F = e_{F}$) and $\beta_{\ast, F}$ denotes the map $\epsilon_F M_F \to N_F$ that sends each $m$ to $\beta_F(e_Fm)$. In addition, the map $\Delta$ is induced by the diagonal map 
\[ {\prod}_{F \in \cX} N_F \longrightarrow {\prod}_{F \in \cX} \Big({\prod}_{p\in \mathbb{P}\setminus \cS}(\Z_p\otimes_\Z N_F)\Big),\]
$\theta_{p}$ is the canonical connecting homomorphism   
\[ {\varprojlim}
_{F \in \cX}
(\Z_p\otimes_\Z \epsilon_F M_F) \to {\varprojlim}_{n \in \N}
\bigl({\varprojlim}^1_{F} (\epsilon_F M_F)\bigr)[p^n],
\]
and $\beta^{(p)}_{\ast}$ is the map
\[
{\varprojlim}_{F \in \cX} (\Z_p \otimes_\Z \epsilon_F M_F) \xrightarrow{ (\Z_p\otimes_\Z\beta_{\ast, F})_F} \eullim{F \in \cX}{\emptyset} (\Z_p \otimes_\Z N_F). 
\] 
\end{thm}
\smallskip

An application of this result will play a key role in the proof of Theorem~\ref{rank one intro} from the introduction. Its proof will be completed in \S\,\ref{rigidity proof section} after we have first established several necessary  auxiliary results. 

\subsubsection{Generic Euler limits}

{Recall that $\cQ$ denotes the field of fractions of $\cR$.}
In this subsection we investigate Euler limits that arise from systems of $\cQ[\cG_F]$-modules.  

\begin{proposition}\label{comparison-diamond-projective-limit}
Let $\cX \subseteq \Omega$ be a subset that satisfies Hypothesis \ref{hyp X} with $\cV = \{ V \}$ for some $V \subseteq S_\infty$. Suppose to be given a projective system $\{(M_F, \varphi_{F / E})\}_{F \in \cX}$ of $\cQ[\cG_F]$-modules that satisfies Hypothesis \ref{rigidity hyp}, and take $\Pi = \emptyset$. Then the following claims are valid. 
\begin{romanliste}
\item Write $\Delta$ for the map $\varprojlim_{F \in \cX} \epsilon_F  M_F \to \prod_{F \in \cX} e_{\bm{1}_F} e_F M_F$ sending $(m_F)_F$ to $(e_{\bm{1}_F} e_F m_F)_F$. Then the sequence of $\cQ\llbracket G_k^{\mathrm{ab}}\rrbracket$-modules 
\begin{cdiagram}[column sep=small]
0 \arrow{r} &   \varprojlim_{F \in \cX} \epsilon_F  M_F \arrow{rr}{\theta_1}  & &
\eullim{F \in \cX}{\emptyset} e_{F} M_F \arrow{rr}{\theta_2}  & & \faktor{ \big( \prod_{F \in \cX} e_{\bm{1}_F}e_{F}M_F \big) }{\im(\Delta)}
\end{cdiagram}%
is exact. Here $\theta_1$ sends $(a_F)_F$ to $(e_{F} a_F)_F$, $\theta_2$ sends  $(b_F)_F$ to the class of $(e_{\bm{1}_F}b_F)_F$ and the Euler limit is as defined in Example~\ref{es remark}\,(d).
\item \textcolor{black}{Let $\mathcal{C}$ be a ring that contains $\overline{\QQ}$ and $\cQ$ as subrings. Let $\cQ'$ be a subring of $\mathcal{C}$ that contains $\cQ$  and, for  each $F \in \cX$, set $M_F' \coloneqq \cQ'\otimes_\cQ M_F$.} Then the sequence 
\[ 0 \to \big (\prod_{F \in \cX}\epsilon_FM_F\big) \cap \big( \varprojlim_{F \in \cX} \epsilon_F  M'_F \big) \xrightarrow{\subseteq} \varprojlim_{F \in \cX} \epsilon_F  M'_F
\xlongrightarrow{(a_F)_F \mapsto (e_{F} a_F)_F} \eullim{F \in \cX}{\emptyset} \big( \faktor{(e_{F} M'_F)}{(e_{F}M_F)} \big)\]%
is exact, where the Euler limit is as defined in Example~\ref{es remark}\,(d).
\end{romanliste}
\end{proposition}

\begin{rk} 
{Both the injectivity of $\theta_1$ in claim (i) of this result and the exactness of the sequence in claim (ii) will be verified by detailed technical computations that rely crucially on the definition of the idempotents $e_F$. Since these results are important, to help the reader we shall explicate  the key ideas of these arguments. To avoid technical details (that do not affect the general arguments), we assume $\cX \subseteq \Omega_k^{S_\infty}$ so that $\epsilon_F = 1$ for all $F \in \cX$, and that the class group of $k$ is trivial so all non-trivial finite-order characters of $G_k$ are ramified. We note, in particular, that these conditions combine with Hypothesis \ref{hyp X} to imply that, for each $F \in \cX$ and each $\chi \in \cG_F^\circ$, the field $F_\chi \coloneqq F^{\ker(\chi)}$ also belongs to $\cX$. 
\begin{liste}
\item To show that an element $a = (a_F)_F$ of $\ker(\theta_1)$ vanishes, one must prove $e_\chi a_L = 0$ for all $L \in \cX$ and $\chi\in \cG_L^\circ$. Then, under the above hypotheses, the proof of this fact given below reduces to the following observations for all such $L$ and $\chi$. Since $\chi$ is a faithful character of $\cG_{L_\chi}$, it cannot vanish on the decomposition subgroup of any place 
ramifying in $L_\chi$ and so the definition of $e_{L_\chi}$ implies $\chi ( e_{L_\chi}) \neq 0$, and thus $e_\chi\cdot e_{L_\chi} = e_\chi$. In particular, since the containments $a \in \ker(\theta_1)$ and $L_\chi\in \cX$ combine to imply that $e_{L_\chi} a_{L_\chi} = 0$, one also has $e_\chi a_{L_\chi} = 0$. Thus, since $(a_F)_{F\in \cX} \in \varprojlim_F M_F$, one has $e_\chi a_L = e_\chi a_{L_\chi} = 0,$ as required.  
\item To prove claim (ii), it is enough to show that if an element $(a_F)_F$ of ${\varprojlim}_{F \in \cX} M'_F$ is such that $a_F \in   M_F + (1 - e_F)M'_F$ for every $F \in \cX$, then one has $a_L \in  M_L$ for every $L \in \cX$. This fact is proved by induction on the number of distinct primes dividing the conductor of $L$. In particular, if the class group of $k$ is trivial, then the inductive base is valid since $e_L = 1$ if $L$ has prime-power conductor (so that $M_L + (1 - e_L)M'_L = M_L$), and the inductive step is derived from the fact that if $\chi\in \cG_L^\circ$ satisfies $\chi(e_L) \not= 0$, then the conductor of $L_\chi$ has fewer prime factors than does the conductor of $L$.   
\end{liste}}
\end{rk}

The proof of Proposition~\ref{comparison-diamond-projective-limit} will occupy the remainder of this section.

\begin{proof}[Proof of Proposition~\ref{comparison-diamond-projective-limit}]
Before proving claim (i), we introduce some convenient notation.
Let $\infty$ denote the product of the real places of $k$ and write $k(\infty)$ for the narrow Hilbert class field of $k$.
For each \textcolor{black}{orbit $\chi$ in $[\cG_{k(\infty)}^\circ]_\Q$, we fix a place 
\[ \p_\chi\in S_\fin(k)\] 
for which $\chi (\Frob_{\p_\chi})$ generates $\chi ( \cG_{k(\infty)})$ and also $[k(\p_\chi) : k(1)] > 1$.
(To justify the existence of such a place $\p_\chi$ one can, for example, choose $\ell \in \mathbb{P}$ with $\ell > [k ( 1) : k]$ and then use Cebotarev's Density Theorem to find a place $\p_\chi$ in $S_\fin(k)$ that splits completely in $k ( \mu_\ell, (\cO_k^\times)^{1 / \ell})$ and is such that the restriction of $\Frob_{\p_\chi}$ to $k(\infty)^{\ker (\chi)}$ generates $\gal{k(\infty)^{\ker (\chi)}}{k}$. Such a place $\p_\chi$ then satisfies $[k(\p_\chi) : k(1)] > 1$ by Lemma \ref{rubin lemma}\,(ii).)
}\\ 
For each field $F \in \cX$ and character $\chi \in \cG_F^\circ$, we then set 
\begin{equation} \label{def F chi}
    F_\chi \coloneqq \begin{cases} k(\p_\chi) F^{\ker (\chi)} &\text{if $\chi$ is unramified at all finite places},\\
                                 F^{\ker (\chi)} &\text{if $\chi$ is ramified at a finite place.}\end{cases}
\end{equation}                                 
In particular, \textcolor{black}{since $F \in \cX$, Hypothesis \ref{hyp X}\,(ii) implies that $F_\chi\in \cX$ for every $\chi$ in the subset $\Upsilon_F^\diamond$ of $\cG_F^\circ$ defined in (\ref{diamond def}).} 
 We now claim that for each $y_1 = (y_{1,F})_F \in \eullim{F \in \cX}{\emptyset} (1 - e_{\bm{1}_F}) e_{F} M_F$, respectively $y_2 = (y_{2,F})_F\in \varprojlim_{F \in \cX}  (1 - e_{\bm{1}_F})\epsilon_F  M_F$,  each $F \in \cX$ and both $i \in \{1,2\}$ one has  
\begin{equation} \label{character-decomposition}
 y_{i,F} = {\sum}_{\chi \in \Upsilon_{F,i} }  [FF_\chi : F_\chi]^{-s} e_\chi (\varphi_{F F_\chi / F} \circ \iota_{FF_\chi / F_\chi}) (y_{i,F_\chi}) 
\end{equation}
with 
\[
\Upsilon_{F,1} \coloneqq \Upsilon_{F}\setminus \{ \bm{1}_F \}
\quad \text{ and } \quad 
\Upsilon_{F,2} \coloneqq \Upsilon_{F}^\diamond\setminus \{ \bm{1}_F \}.
\]
In fact, since $y_{1,F}\in (1 - e_{\bm{1}_F}) e_F M_F$, respectively $y_{2,F} \in (1 - e_{\bm{1}_F})\epsilon_F  M_F$, each element $y_{i,F}$ is only supported on characters in $\Upsilon_{F,i}$, and so the above equality is valid provided that, for $i\in \{1,2\}$ and $\chi\in \Upsilon_{F,i}$, one has  
\begin{equation} \label{character-decomposition chi}
e_\chi y_{i,F} = [FF_\chi : F_\chi]^{-s} e_\chi (\varphi_{F F_\chi / F} \circ \iota_{FF_\chi / F_\chi}) (y_{i,F_\chi}). 
\end{equation}
To show this for $i=1$, we note that \textcolor{black}{the Euler limit relations for the triangular system associated to $(M_F, \varphi_{F / E})_{F \in \cX}$ as in Example~\ref{es remark}\,(d) imply that
\[
P_{F F_\chi / F, \emptyset}\cdot (y_{1,F}- \varphi_{F F_\chi / F} (y_{1,F F_\chi})) = 0. 
\]
In particular, since the Euler factor $P_{F F_\chi / F,\emptyset}$ is either $1$ or $1 - \Frob_{\p_\chi}^{-1}$ and so $\chi ( P_{F F_\chi / F,\emptyset})\neq 0$,  this last equality implies} that if $i =1$, then 
\[ e_\chi y_{i,F} = e_\chi \varphi_{F F_\chi / F} (y_{i,F F_\chi})\,\,\text{ for every }\,\, \chi \in  \Upsilon_{F,i}.\]
In a similar way, since $\chi \in \Upsilon_{F,1}\subseteq\Upsilon_F$, one has $\chi (P_{FF_\chi / F_\chi,\emptyset}) \neq 0$ and so, if $i =1$, then  
\[ e_\chi y_{i,F_\chi} = e_\chi \varphi_{FF_\chi / F_\chi} (y_{i,FF_\chi})\,\,\text{ for every }\,\, \chi \in  \Upsilon_{F,i}.\]
Note that these two displayed equalities are also clearly true if $i =2$ since  $y_2$ belongs to the inverse limit $\varprojlim_{F \in \cX}  (1 - e_{\bm{1}_F})\epsilon_F  M_F$. By using these respective equalities, we can therefore now obtain the required equality (\ref{character-decomposition chi})  via the calculation
\begin{align} \nonumber
    e_\chi y_{i,F}  =\,& e_\chi \varphi_{F F_\chi / F} (  y_{i,F F_\chi}) \\ \nonumber 
     =\,& [FF_\chi : F_\chi]^{-s} e_\chi \varphi_{F F_\chi / F} ( \NN_{\gal{FF_\chi}{F_\chi}}^s y_{i,FF_\chi}) \\ \nonumber
    =\,&[FF_\chi : F_\chi]^{-s} e_\chi (\varphi_{F F_\chi / F} \circ \iota_{FF_\chi / F_\chi} \circ \varphi_{FF_\chi / F_\chi}) (y_{i,FF_\chi}) \\ \label{character-part}
    =\,&[FF_\chi : F_\chi]^{-s} e_\chi (\varphi_{F F_\chi / F} \circ \iota_{F F_\chi / F_\chi}) (y_{i,F_\chi}),
\end{align}
in which the third equality is a consequence of  Hypothesis \ref{rigidity hyp}. \\
Having verified (\ref{character-decomposition}),  we now prove claim (a). For this, we observe it suffices to show the  map
\begin{equation*} 
 \mathcal{J} \: \varprojlim_{F \in \cX}  (1 - e_{\bm{1}_F})\epsilon_F  M_F
 \to 
\eullim{F \in \cX}{\emptyset} (1 - e_{\bm{1}_F}) e_{F} M_F, 
\quad x = (x_F)_F \mapsto (e_{F} x_F)_F
\end{equation*}
is bijective. Indeed, from Example~\ref{es remark}\,(d) one knows $\theta_1$ is a well-defined homomorphism of $\cR\llbracket G_k^{\mathrm{ab}}\rrbracket$-modules and, in addition, the explicit definition of $\Delta$ directly implies $\im(\theta_1) \subseteq \ker(\theta_2)$. {If, as we shall do below, we can prove the} bijectivity of $\mathcal{J}$, then we {will} directly deduce $\ker(\theta_2) = \im(\theta_1)$ and, in addition, that \textcolor{black}{if $a = (a_F)_F$ belongs to $\ker(\theta_1)$, so that $e_F a_F = 0$ for all $F \in \cX$, then $ (1 - e_{\bm{1}_F}) a_F = 0$ for all $F \in \cX$. To show that any such $a$ vanishes, and hence complete the proof of the exactness of the given sequence, it is then enough to prove $e_{\bm{1}_F} a_F = 0$ for all $F \in \cX$ or, equivalently (since $a\in \varprojlim_{F \in \cX} \epsilon_F  M_F$ \textcolor{black}{and the system $\{ (M_F, \varphi_{F / E}) \}_{F \in \cX}$ is assumed to satisfy Hypothesis \ref{rigidity hyp}}), that $e_{\bm{1}_F} a_F = 0$ for some $F \in \cX$. This is immediate if $e_{\bm{1}_F} \epsilon_F = 0$ for all $F\in \cX$ and so we may assume that $\cX$ satisfies Hypothesis \ref{hyp X} with respect to a set $\cV$ containing $S_\infty$. To deal with this case, we first use the Cebotarev density theorem (and Lemma~\ref{rubin lemma}) to  choose a prime $\p$ with $[k (\p) : k (1)] > 1$. Then, for any $F_0 \in \cX$, one has $F_0 k(\p) \in \cX$ (by Hypothesis \ref{hyp X}\,(i)) and hence also  $k(\p) \in \cX$ (by Hypothesis \ref{hyp X}\,(ii)). In particular, if we now take $F$ to be $k (\p)$, then one has $e_{\bm{1}_F} e_F = e_{\bm{1}_F}$ and so  the assumed vanishing of $e_F a_F$ implies that $e_{\bm{1}_F} a_F = 0$, as required. }\\
At this stage the proof of claim (i) has been reduced to showing $\mathcal{J}$ is bijective and we do this by showing one obtains an inverse function by using the assignment 
 $y_1 \mapsto \widetilde{y_1} \coloneqq (\widetilde{y}_{1,F})_F $, with 
\begin{align*} \widetilde{y}_{1,F} \coloneqq \, & y_{1,F} + {\sum}_{\chi \in 
{\Upsilon_{F, 2} \setminus \Upsilon_{F, 1}}}  [FF_\chi : F_\chi]^{-s} e_\chi (\varphi_{F F_\chi / F} \circ \iota_{FF_\chi / F_\chi}) (y_{1,F_\chi})\\
=&\, {\sum}_{\chi \in {\Upsilon_{F, 2}}} 
[FF_\chi : F_\chi]^{-s} e_\chi (\varphi_{FF_\chi / F} \circ \iota_{FF_\chi / F_\chi}) (y_{1, F_\chi}),
\end{align*}
where the equality follows from  (\ref{character-decomposition}) with $i=1$.\\
As a first step, we show $(\widetilde{y}_{1,F})_F\in \varprojlim_{F \in \cX} (1 - e_{\bm{1}_F}) \epsilon_F M_F$, so that there is a well-defined map 
\[
\mathcal{I} \:  \eullim{F \in \cX}{\emptyset} (1 - e_{\bm{1}_F}) e_{F} M_F \to \varprojlim_{F \in \cX}   (1 - e_{\bm{1}_F}) \epsilon_F M_F, \quad 
y_1 \mapsto \widetilde{y}_1
= (\widetilde{y}_{1, F})_F
.\]
For this, it is enough to fix arbitrary fields $E, F \in \cX$ with $E \subseteq F$ and then to compute  
\begin{align*}
    \varphi_{F / E} (\widetilde{y}_{1,F}) & = \varphi_{F / E} \Big ( 
    {\sum}_{\chi \in {\Upsilon_{F, 2}} } 
[FF_\chi : F_\chi]^{-s} e_\chi (\varphi_{F F_\chi / F} \circ \iota_{FF_\chi / F_\chi}) (y_{1,F_\chi})  \Big) \\
& = {\sum}_{\chi \in {\Upsilon_{F, 2}}} 
[FF_\chi : F_\chi]^{-s} \pi_{F / E}(e_\chi) (\varphi_{FF_\chi / E} \circ \iota_{FF_\chi / F_\chi}) (y_{1,F_\chi}) \\
& = {\sum}_{\chi \in {\Upsilon_{F, 2}}} 
[FF_\chi : F_\chi]^{-s} [FF_\chi : EF_\chi]^s \pi_{F / E}(e_\chi) (\varphi_{EF_\chi / E}\circ \iota_{EF_\chi / F_\chi}) (y_{1,F_\chi}) \\
& = {\sum}_{\chi \in {\Upsilon_{E, 2}}} 
[EE_\chi : E_\chi]^{-s} e_\chi (\varphi_{EE_\chi / E} \circ \iota_{EE_\chi / E_\chi}) (y_{1,E_\chi}).
\end{align*}
Here the third equality is a consequence of Hypothesis \ref{rigidity hyp} and the last uses the fact  that $\pi_{F / E} (\epsilon_F)= \epsilon_E$ {(so $\sum_{\chi \in \Upsilon_{F, 2}} \pi_{F / E} (e_\chi) = \sum_{\chi \in \Upsilon_{E, 2}} e_\chi$)} since $E$ and $F$ are both contained in $\cX$.\\
It is thus enough to now show that the composites $\mathcal{J} \circ \mathcal{I}$ and $\mathcal{I} \circ \mathcal{J}$ are the respective identity maps. In the first case, this follows directly from the fact that, by definition, one has  $\Upsilon_F= \{\chi \in \cG_F^\circ: e_\chi e_F \not= 0\}$, and so $e_F \widetilde{y}_{1,F} = y_{1,F}$. To check the second case, we observe that, for any  
  $y_2 = (y_{2,F})_F$ in $\varprojlim_{F \in \cX}  (1 - e_{\bm{1}_F})\epsilon_F  M_F$ and $F \in \cX$, one has  
\begin{align*} 
    \widetilde{e_F y_{2,F}} & = {\sum}_{\chi \in {\Upsilon_{F, 2}}}  [FF_\chi : F_\chi]^{-s} e_\chi (\varphi_{F F_\chi / F} \circ \iota_{FF_\chi / F_\chi}) (e_{F_\chi} y_{2,F_\chi}) \\
    & = {\sum}_{\chi \in {\Upsilon_{F, 2}}}  [FF_\chi : F_\chi]^{-s} e_\chi (\varphi_{F F_\chi / F} \circ \iota_{FF_\chi / F_\chi}) (y_{2,F_\chi}) \\
    & = y_{2,F}.
\end{align*}
Here the first equality follows from the definition of $\widetilde{e_F y_{2,F}}$, the second from the fact that, for each  $\chi \in \Upsilon^\diamond_F$, the formula  (\ref{order-of-vanishing-formula}) combines with the definition of $F_\chi$ to imply $r_{S(F_\chi)} (\chi) = r_{F_\chi}$, and hence $e_\chi e_{F_\chi} = e_\chi$, and the last equality follows from (\ref{character-decomposition}) with $i=2$. 
This shows $(\mathcal{I} \circ \mathcal{J}) ( y_2) = \mathcal{I} ( (e_F y_{2,F})_F) = (\widetilde{(e_F y_{2,F})_F})_F = y_2$ and thereby completes the proof of claim (i).
\\
Turning now to claim (ii), it is enough to prove that 
\begin{equation} \label{aim in this proof}
\Big({\prod}_{F\in \cX} ( \epsilon_F M_F + (1 - e_F)\epsilon_F M'_F)\Big) \cap{\varprojlim}_{F \in \cX} \epsilon_F M'_F \subseteq {\prod}_{F\in \cX} \epsilon_F M_F. 
\end{equation}
To verify this, we fix  $F \in \cX$ and introduce the following notation. 
For any field $E \in \Omega$ we let $\Xi (E) \subseteq \Upsilon_F^\diamond$ be the (possibly empty) subset comprising all characters $\chi$ such that $F_\chi = E$. Note that by construction each field $F_\chi$ only depends on the orbit in $[\Upsilon^\diamond_F]_\Q$ to which $\chi$ belongs and hence that
 $\Xi (E)$ is stable under the action of $G_\Q$. This implies that the associated idempotent $e_{\Xi (E)} = \sum_{\chi \in \Xi (E)} e_\chi$ belongs to $\Q [\cG_F]$.\\
To investigate the $F$-component $m_F$ of an element $m = (m_E)_E$ of the left hand side of (\ref{aim in this proof}) we use the following decomposition \textcolor{black}{in $\mathcal{C}\otimes_\cQ M_F$} 
\begin{align} \nonumber
    m_F & = 1\cdot m_F = \Big ({\sum}_{\chi \in \cG_F^\circ} e_\chi\Big)\cdot m_F = {\sum}_{\chi \in \cG_F^\circ} e_\chi m_F \\  \nonumber
    & = {\sum}_{\chi \in \Upsilon_F^\diamond} [FF_\chi : F_\chi]^{-s} e_\chi (\varphi_{F F_\chi / F} \circ \iota_{F F_\chi / F_\chi}) (m_{F_\chi}) \\ \nonumber
    & = {\sum}_{E \in \cX}\Big(\sum_{\chi\in \Xi (E)}[FE : E]^{-s} e_\chi (\varphi_{FE / F} \circ \iota_{FE / E}) (m_{E}) \Big) \\ \label{lambda-decomposition}
    & = {\sum}_{E \in \cX} [FE : E]^{-s} e_{ \Xi (E)}  (\varphi_{FE / F} \circ \iota_{FE / E}) (m_{E}).
    \end{align}
Here the fourth equality follows from (\ref{character-part}) and the fifth from the fact each  $F_\chi$ belongs to $\cX$. \\
We shall now use an induction on the number $n(F)$ of (non-archimedean) prime divisors of the conductor of $F$ to show that $m_F\in \epsilon_F M_F$ for all $F$. \\
Let us first assume that $F$ has prime-power conductor. 
From (\ref{lambda-decomposition}) we see that it is enough to show that $e_{\Xi (E)} m_E \in \epsilon_E M_E$ for all
fields $E$ which are of the form $E = F_\chi$ for some character $\chi \in \Upsilon^\diamond_F$. Fix such a field $E$.
Now, $\chi \in \Upsilon_E$ by the construction of $E = F_\chi$ (which in the case $\chi$ is unramified at all finite places involves the choice of a prime ideal $\p_\chi$ with $\chi ( \Frob_{\p_\chi}) \neq 1$, cf.~(\ref{def F chi})) 
 and so $e_\chi (1 - e_E) \epsilon_E = 0$. Hence, since $m_E$ is assumed to belong to $\epsilon_E M_E + (1 - e_E) \epsilon_E M'_E$, we have $ e_{\Xi (E)} m_E \in \epsilon_E M_E$, as claimed. \\
We now assume to be given $n \in \N$ such that $m_E \in \epsilon_E M_E$ for every  $E\in \cX$ with $n (E) \leq n$. We then fix $F\in \cX$ with $n (F) = n + 1$ and let $E$ be a field of the form $F_\chi$ for some $\chi \in \Upsilon_F^\diamond$. If $\chi$ is ramified at a finite place, then $F_\chi = F^{\ker(\chi)}\subseteq F$ and so $n (E) \leq n(F)$. If $\chi$ is unramified at all finite places, on the other hand, then $F_\chi$ is defined to be a prime-power conductor field and so $n(E) = 1 < n(F)$. In all cases therefore one has $n (E) \leq n(F)$. \\
If $n(E) < n(F)$, then, by the induction hypothesis, one has $m_E \in \epsilon_E M_E$. On the other hand, if $n(E) = n(F)$, and $y_F \in \epsilon_F M_F$ and $i_F \in (1 - e_{F}) \epsilon_F M'_F$ are such that $m_F = y_F + i_F$, then, by reversing the calculation in (\ref{character-part}), one checks that 
\begin{align*}
 [FE : E]^{-s} e_{ \Xi (E)}  (\varphi_{FE / F} \circ \iota_{FE / E}) (m_{E}) & 
 ={\sum}_{\chi\in \Xi (E)} e_\chi \cdot m_{F} = {\sum}_{\chi\in \Xi (E)}  e_\chi(y_F + i_F) \\
&  = {\sum}_{\chi\in \Xi (E)} e_\chi y_F = e_{\Xi(E)} y_F.
\end{align*}
Here the third equality is valid since, under the present hypothesis, 
each $\chi$ in $\Xi (E)$ is not trivial on the decomposition group of any prime divisor of the conductor of $F$ so that   
 $e_\chi = e_\chi e_{F}$ and hence also $e_\chi (i_F) = 0$. 
 
These observations imply that  $e_{ \Xi (E)}  (\varphi_{FE / F} \circ \iota_{FE / E}) (m_{E})\in \epsilon_EM_E$ for every subfield $E$ of $F$ of the form $F_\chi$ for some $\chi \in \Upsilon_F^\diamond$ and hence, via the decomposition (\ref{lambda-decomposition}), that $m_F\in \epsilon_FM_F$, as required to complete the proof of Proposition~\ref{comparison-diamond-projective-limit}. 
\end{proof}

\subsubsection{Profinite completions}\index{profinite completion}
\label{profinite completions section}

In the setting of Theorem~\ref{rigidity theorem}, we now assume to be given an element $a$ of the Euler limit $\eullim{F \in \cX}{\emptyset}\bigl(N_F \cap (\Q\otimes_\Z\beta_{F,\ast}(\epsilon_ FM_F))\bigr)$ and consider the modules  
\begin{align*}
M & \coloneqq \varprojlim_{F \in \cX} \epsilon_F M_F,
&  M(a) & \coloneqq \eullim{F \in \cX}{\emptyset} ( \beta_{\ast, F} (\epsilon_F M_F) + \Z_\cS [\cG_F] a_F)\,\,\,\text{and}\,\,\,  
 X(a) \coloneqq M(a)/\beta_\ast (M). %
\end{align*}

Then, to prove Theorem~\ref{rigidity theorem}, we must show that $X(a)=(0)$ if, for all $p \in \mathbb{P}\setminus \mathcal{S}$ and all $F \in \cX$, the element $a_F$ belongs to a certain submodule of $\Z_p \otimes_\Z N_F$. To help determine when the module $X(a)$ vanishes we will use completion functors.
More precisely, for any $\ZZ$-algebra $H$ and $H$-module $A$, we shall use the modules 
\[ \widehat{A} \coloneqq {\varprojlim}_{n \in \N} A/(nA) \quad \text{ and } \quad
\widehat{A}^p \coloneqq {\varprojlim}_{n \in \N} A/(p^nA).\]
Observe that the assignment $A \mapsto \widehat{A}$ (resp.\@ $A \mapsto \widehat{A}^p$) defines a functor from the category of $H$-modules to the category of $\widehat{H}$-modules (resp.\@ $\widehat{H}^p$-modules). For the reader's convenience, we recall some useful properties of these functors in the following result.  

\begin{lemma}\label{completion-lemma-revisited}  For an abelian group $A$ and prime $p$ the following claims are valid. 
    \begin{liste}
\item {The kernels of the canonical maps $A\to \widehat{A}$ and $A\to \widehat{A}^p$ are $A_{\mathrm{div}}$ and $A_{p\cdot\mathrm{div}}$ respectively.}  
    \item If $A$ is finitely-generated, then the canonical maps  $A \otimes_\ZZ \widehat{\ZZ} \to \widehat{A}$ and $\Z_p\otimes_\Z A\to \widehat{A}^p$   are bijective. 
     \item{$\widehat{(-)}^p$ is an idempotent functor.}
    \item{If $A$ is $\Z$-torsion-free, then every short exact sequence $0 \to A_1 \to A_2 \to A \to 0$ of $H$-modules induces an exact sequence $0 \to \widehat{A_1} \to \widehat{A_2} \to \widehat{A} \to 0$ of $\widehat{H}$-modules and similarly for the functor $\widehat{(-)}^p$,}
        \item{If $A$ is $\Z$-torsion-free, then so too are the groups $\widehat{A}$ and $\widehat{A}^p$.}
        \end{liste}
\end{lemma}

\begin{proof} Claim (a) is clear and claim (b) is well-known. In addition, claim (c) is both straightforward to prove directly and also follows immediately from the general result \cite[Th.\@ 15]{matlis} of Matlis (since $\widehat{A}^p$ is equal to the completion of the $\ZZ$-module $A$ at the ideal  generated by $p$).

For both claims (d) and (e), it is enough to consider the functor $A \to \widehat{A}$. To prove claim (d) in this case we note first that, since $A$ is torsion-free, for each  $n\in \mathbb{N}$ the Snake Lemma applies to the following exact commutative diagram  
 \begin{cdiagram} 0 \arrow[r] & A_1\arrow[r, "\psi"]\arrow[d, "n"] & A_2 \arrow[r, "\phi"]\arrow[d,"n"] & A \arrow[r]\arrow[d,"n"] &0\\
0 \arrow{r} & A_1 \arrow[r,"\psi"] & A_2 \arrow[r,"\phi"] & A \arrow[r] &0 \end{cdiagram}%
to give an exact sequence $0 \to A_1/nA_1 \xrightarrow{\psi/n} A_2/nA_2 \xrightarrow{\phi/n} A/nA \to 0$. It is then enough to note that the latter sequences are compatible (with respect to the natural projection maps) as $n$ varies and that, by the Mittag--Leffler criterion, exactness of the sequences is preserved when one passes to the inverse limit over $n$ since, for each multiple $m$ of $n$, the projection map $A_1/mA_1 \to A_1/nA_1$ is surjective. 

Finally, to prove claim (e) we must show that if $x = (x_n)_n$ is an element of $\widehat{A}$ with the property that $px = 0$ for some prime $p$, then $x=0$. But, since $A$ is torsion-free, for each $n$ the element $x_{np}$ is the image in $A/(npA)$ of an element $\hat x_{np}$ of $nA$. Since $x_n$ is equal to the image of $\hat x_{np}$ in $A/(nA)$ one therefore has $x_n = 0$, as required. \end{proof}

We can now state the main result of this subsection.

\begin{proposition}\label{profinite-injective-revisited}
    \textcolor{black}{Assume that, for each  $p\in \mathbb{P}\setminus \mathcal{S}$,  the image of $a$ in $\eullim{E \in \cX}{\emptyset} (\Z_p \otimes_\Z N_E)$ belongs to $\beta_\ast (\ker \theta_p)$}. Then the natural composite homomorphism of \textcolor{black}{$\Z_\cS \llbracket G_k^{\mathrm{ab}} \rrbracket$}-modules 
    \begin{align*}
        X(a) \to \widehat{X(a)} \to {\prod}_{p\in \mathbb{P}\setminus \mathcal{S}} \widehat{X(a)}^p
    \end{align*}
   is injective. 
\end{proposition}

Before proving this result, we establish a useful preliminary result.  

\begin{lem} \label{X-torsion-free}
Assume that, for each $p\in \mathbb{P}\setminus \mathcal{S}$, the image of $a$ in $\eullim{E \in \cX}{\emptyset} (\Z_p \otimes_\Z N_E)$ belongs to $\beta_\ast (\ker \theta_p)$. Then the module $X(a)$ is $\Z$-torsion-free.
\end{lem}

\begin{proof}
We must show that, for each $p\in \mathbb{P}\setminus \mathcal{S}$, the module $X(a)$ is $p$-torsion-free. To do this we note that, by the given assumption,  there exists an element $x^{(p)} = (x^{(p)}_E)_{E \in \cX}$ of $\ker \theta_p$ such that the image of $a$ in $\eullim{E \in \cX}{\emptyset} (\Z_p \otimes_\Z N_E)$ is equal to $\beta_\ast^{(p)} (x^{(p)})$. This implies $\beta_{\ast, E} (\epsilon_E M_E) + \Z_\cS [\cG_E] a_E$ is contained in $\Z_p \otimes_\Z \beta_{\ast, E} (\epsilon_E M_E) = \beta_{\ast, E}^{(p)} ( \Z_p \otimes_\Z \epsilon_E M_E)$ and so there exists an injective map
\[
X(a) = \faktor{M (a)}{\beta_\ast (M)} \hookrightarrow 
\faktor{\big (\varprojlim_{E \in \cX} (\Z_p \otimes_\Z \beta_{\ast, E}(\epsilon_E M_E) \big) \big) }{\beta_\ast (M)}.
\]
It is therefore enough to show that the last module in this display is $p$-torsion-free. 
To do this, we set $M^{(p)}\coloneqq \varprojlim_{E \in \cX} (\Z_p \otimes_\Z \epsilon_E M_E)$ and claim that $\beta^{(p)}_\ast ( M^{(p)}) = \varprojlim_{E \in \cX} (\Z_p \otimes_\Z \beta_{\ast, E}(\epsilon_E M_E))$. To justify this, we note that $\ker \beta^{(p)}_{\ast, E}$ is a finitely-generated $\Z_p$-module and hence a compact Hausdorff space and so the derived limit ${\varprojlim}^1_{E \in \cX} ( \ker \beta^{(p)}_{\ast, E})$ vanishes. Upon passing to the limit (over $E \in \cX$) of the exact sequences
\[ 0\to \ker \beta^{(p)}_{\ast, E} \to  \Z_p \otimes_\Z \epsilon_E M_E 
\xlongrightarrow{\beta_{\ast, E}^{(p)} } \beta_{\ast, E}^{(p)} ( \Z_p \otimes_\Z \epsilon_E M_E ) \to 0\]
we therefore obtain an exact sequence
\begin{equation} \label{limit of p-completed beta ast sequence}
   0 \to {\varprojlim}_{E \in \cX} ( \ker \beta^{(p)}_{\ast, E})\to M^{(p)} \xlongrightarrow{\beta_\ast^{(p)}} {\varprojlim}_{E \in \cX} (\Z_p \otimes_\Z \beta_{\ast, E}(\epsilon_E M_E) ) \to 0.\end{equation}
This implies the claimed identification and so reduces us to showing that the quotient of $\beta^{(p)}_\ast ( M^{(p)})$ by $\beta_\ast (M)$ is $p$-torsion-free. To do this, it is then enough to show that both outer terms in the tautological short exact sequence
\begin{equation} \label{exact sequence to show torsion free}
0 \to  \frac{ \varprojlim_{F \in \cX} \beta_{\ast, F} ( \epsilon_F M_F)}{\beta_\ast (M)}
\to  
\frac{\beta^{(p)}_\ast ( M^{(p)}) }{\beta_\ast (M)} \to  
\frac{\beta^{(p)}_\ast ( M^{(p)}) }{ \varprojlim_{F \in \cX} \beta_{\ast, F} ( \epsilon_F M_F)} \to 0\end{equation}
are $p$-torsion-free. 

To deal with the first module in this sequence, we note $\beta_\ast$ is injective since the first map that occurs in its definition (\ref{beta ast def}) as a composite is injective by Proposition~\ref{comparison-diamond-projective-limit}\,(i) and the second map is injective by the given assumption (on $\beta$). In particular, upon passing to the limit over $F$ in $\cX$ of the tautological short exact sequences
\[ 0 \to \ker \beta_{\ast, F} \to \epsilon_F M_F \xrightarrow{\beta_{\ast, F}}
  \beta_{\ast, F} (\epsilon_F M_F) \to 0\]
we derive an exact sequence
\[ 0\to M \xrightarrow{\beta_\ast}  {\varprojlim}_{F \in \cX} \beta_{\ast, F} ( \epsilon_F M_F) \to {\varprojlim}^1_{F \in \cX} (\ker \beta_{\ast, F})
\]
that reduces us to showing ${\varprojlim}^1_{F \in \cX}(\ker \beta_{\ast, F})$ is $p$-torsion-free. The key point for this is that $\beta_{F} \: e_F M_F \to N_F$ is assumed to be injective for all $F \in \cX$, and hence that the induced map  $\varprojlim_{F \in \cX} (\Z_p \otimes_\Z e_F M_F) \to  \eullim{F \in \cX}{\emptyset} (\Z_p \otimes_\Z N_F)$ is also injective. From Proposition~\ref{comparison-diamond-projective-limit}\,(i), it then follows that the map $\beta_{\ast}^{(p)} = (\beta_{\ast,F}^{(p)})_F \,: \varprojlim_{F \in \cX} (\Z_p \otimes_\Z \epsilon_F M_F) \to \eullim{F \in \cX}{\emptyset} (\Z_p \otimes_\Z N_F)$ is injective and hence that the sequence (\ref{limit of p-completed beta ast sequence}) implies $\varprojlim_{F \in \cX} (\ker \beta^{(p)}_{\ast, F}) = (0)$ . In particular, since  $\ker \beta^{(p)}_{\ast, F} = \Z_p \otimes_\Z \ker \beta_{\ast, F}$ (as $\Z_p$ is a flat $\Z$-module) and  ${\varprojlim}^1_{F \in \cX} (\ker \beta^{(p)}_{\ast, F}) = (0)$ (as already observed above), upon passing to the limit over $F$ in $\cX$ of the exact sequences
\begin{cdiagram}
0 \arrow{r} & \ker \beta_{\ast, F} \arrow{r} &  \Z_p \otimes_\Z \ker \beta_{\ast, F} \arrow{r} & (\Z_p / \Z) \otimes_\Z \ker \beta_{\ast, F} \arrow{r} & 0
\end{cdiagram}%
one derives an isomorphism 
\[
{\varprojlim}_{F \in \cX} \big ( (\faktor{\Z_p}{\Z}) \otimes_\Z \ker \beta_{\ast, F} \big)
\stackrel{\simeq}{\longrightarrow} 
{\varprojlim}^1_{F \in \cX} (\ker \beta_{\ast, F}). 
\]
Then, since each $(\faktor{\Z_p}{\Z}) \otimes_\Z \ker \beta_{\ast, F}$ is $p$-torsion-free, this isomorphism implies  ${\varprojlim}^1_{F \in \cX} (\ker \beta_{\ast, F})$ is $p$-torsion-free, and hence also proves that the first module in (\ref{exact sequence to show torsion free}) is $p$-torsion-free. \\
Turning now to the third module in (\ref{exact sequence to show torsion free}), we take inverse limit of the exact sequences
\[ 0 \to \beta_{\ast, F}(\epsilon_F M_F) \to \Z_p \otimes_\Z \beta_{\ast, F} (\epsilon_F M_F) \to \big ( \Z_p/\Z \big) \otimes_\Z \beta_{\ast, F} (\epsilon_F M_F) \to 0\]
to deduce the existence of an injective homomorphism 
\[ \frac{\beta^{(p)}_\ast ( M^{(p)}) }{ \varprojlim_{F \in \cX} \beta_{\ast, F} ( \epsilon_F M_F)} \hookrightarrow \varprojlim_{F\in \cX} ( ( \faktor{\Z_p}{\Z}) \otimes_\Z \beta_{\ast, F}(\epsilon_F M_F)).\] 
Then, since each group $( \faktor{\Z_p}{\Z}) \otimes_\Z \beta_{\ast, F} (\epsilon_F M_F)$ is $p$-torsion-free, this injective homomorphism implies that the third module in (\ref{exact sequence to show torsion free}) is also $p$-torsion-free, as required to complete the proof of the claimed result. 
\end{proof}

We can now prove Proposition~\ref{profinite-injective-revisited}.

\begin{proof}[Proof of Proposition~\ref{profinite-injective-revisited}] 
At the outset, we note the Chinese Remainder Theorem implies that the natural map $\widehat{X(a)} \to   \prod_{p\in \mathbb{P}\setminus \mathcal{S}}  \widehat{X(a)}^p$ is injective. We are therefore reduced to showing that the natural homomorphism $i \: X(a) \to \widehat{X(a)}$ is injective. \\
To do this, we note that, since Lemma~\ref{X-torsion-free} implies $X(a)$ is $\Z$-torsion-free, Lemma~\ref{completion-lemma-revisited}\,(d) implies the existence of an exact commutative diagram 
\begin{cdiagram}
0 \arrow{r} & M \arrow{r}{\beta_\ast} \arrow{d}{i_2} & M(a) \arrow{d}{i_1} \arrow{r} & X(a) \arrow{d}{i} \arrow{r} & 0 \\
0 \arrow{r} & \widehat{M} \arrow{r}{\widehat{\beta_\ast}} & \widehat{M(a)} \arrow{r} & \widehat{X(a)} \arrow{r} & 0.
\end{cdiagram}%
In addition, one has $M(a)_\mathrm{div} = (0)$ since $M(a)\subseteq \eullim{F \in \cX}{\emptyset}N_F$ and each $\Z_\cS$-module $N_F$ is finitely-generated. Hence, by Lemma~\ref{completion-lemma-revisited}\,(a), one knows that the map $i_1$, and thus also $i_2$, in the diagram is injective. 
We shall therefore now identify $M$ and $M(a)$ with $\im (i_2) \subseteq \widehat{M}$ and $\im (i_1) \subseteq \widehat{M (a)}$ respectively. Then, by applying the Snake Lemma to the diagram, \textcolor{black}{we derive an isomorphism 
\[
\ker (i) \cong \bigl( \widehat{M} / M \xrightarrow{\widehat{\beta_\ast}} \widehat{M (a)} / M (a) \bigr).
\]
Thus, in order to prove that $i$ is injective, it is enough for us to fix an arbitrary element 
$m \in \widehat{M} $ that satisfies  $\widehat{\beta}_\ast (m) 
\in M(a)$, and to show that $m \in M$.}\\
For each $E \in \cX$, we set $M_E (a) \coloneqq \Z_\mathcal{S} [\cG_E]a_E + \beta_{E, \ast} ( \epsilon_E M_E)$. Then, by assumption, the image $\widehat{\beta_\ast} (m)_E 
$ of $\widehat{\beta_\ast} (m)$ under the natural map $\widehat{M(a)} \to \widehat{M_E (a)}$ belongs to $M_E (a) \subseteq \Q \otimes_\Z \beta_{E, \ast} (\epsilon_E M_E)$. It follows that the image of  $\widehat{\beta_\ast} (m)_E$ in the quotient module 
\[ (\beta_{E, \ast} (\epsilon_E M_E))^\wedge/\beta_{E, \ast} (\epsilon_E M_E) \cong \beta_{E, \ast} (\epsilon_E M_E) \otimes_\Z (\widehat{\Z}/\Z)\]
is annihilated by a natural number. In particular, since $\widehat{\Z}/\Z$ is uniquely divisible, the image of $\widehat{\beta_\ast} (m)_E$ in this module must vanish and so  $\widehat{\beta_\ast} (m)_E \in \beta_{E, \ast} (\epsilon_E M_E)$. 
As a consequence of the (assumed) injectivity of $\beta$, this implies that $e_{E} m_E\in e_E M_E$, or equivalently, that $m_E$ belongs to $\epsilon_E M_E + (1 - e_{E})\widehat{\epsilon_E M_E}$. To prove $m \in M$, we are therefore reduced to showing that 
\[
\widehat{M} \cap \big ( {\prod}_{E\in \cX} ( \epsilon_E M_E + (1 - e_{E})\widehat{\epsilon_E M_E}) \big)\subseteq M. 
\]
To do this, we fix an element $m'$ of the intersection and write $(m'_E)_E$ for its image under the natural map 
\[ \widehat{M} \to {\varprojlim}_{E\in \cX} \widehat{\epsilon_E M_E}\subset {\varprojlim}_{E\in \cX} (\Q \otimes_\Z \widehat{\epsilon_EM_E}).\] 
Then, by applying Proposition~\ref{comparison-diamond-projective-limit}\,(ii) to  the ring extensions $\Q \subseteq \Q\otimes_\ZZ\widehat{\Z}\subseteq \overline{\QQ}\otimes_\ZZ\widehat{\Z}$ and projective system $(\Q \otimes_\Z \epsilon_E M_E, \varphi_{F / E})_E$, we deduce $(m'_E)_E$ belongs to $\varprojlim_{E\in \cX} (\Q \otimes_\Z \epsilon_E M_E)$. By an application of Lemma~\ref{limit-lemma-below} below, we can then directly conclude that $m'$ belongs to $M$, as required to complete the proof.
\end{proof} 

\begin{lem} \label{limit-lemma-below}
Fix a projective system   $\{(Q_E, \rho_{F / E})\}_{E \in \cX}$  of $ \Z_\mathcal{S} [\cG_E]$-lattices and set $Q \coloneqq \varprojlim_{E\in \cX} Q_E$. For each $E\in \cX$, write $\iota_E$ for the  canonical map
$ Q_E \to \widehat{Q_E}$. Then the sequence 
\[
0 \to Q \xrightarrow{\iota}\widehat{Q} \to \prod_{E\in \cX} \frac{\Q \otimes_\Z \widehat{Q_E}}{\Q \otimes_\Z \im (\iota_E)}
\]
is exact,  where $\iota$ is the canonical map and the third arrow is induced by the canonical map $\iota_\cX \: \widehat{Q} \to \varprojlim_{E\in \cX} \widehat{Q_E}$.
\end{lem}

\begin{proof} For  each $E\in \cX$, the  $\Z_\mathcal{S}$-module  $Q_E$ is finitely-generated and so $(Q_E)_{\mathrm{div}} = (0)$. This implies $Q_{\mathrm{div}}= (0)$ and hence, by Lemma~\ref{completion-lemma-revisited}\,(a), that  
$\iota_E$ (for each $E$) and $\iota$ are injective. In addition, since each $Q_E$ is a lattice, the exact sequences $0 \to Q_E \xrightarrow{\times n} Q_E \to Q_E/n \to 0$ combine to imply that the composite map 
\[ \iota_\cX \: \widehat{Q} = {\varprojlim}_n\bigl(\bigl({\varprojlim}_E Q_E\bigr)/n \bigr)\hookrightarrow {\varprojlim}_n\bigl({\varprojlim}_E (Q_E/n)\bigr) = {\varprojlim}_E\bigl({\varprojlim}_n (Q_E/n)\bigr) = {\varprojlim}_E\widehat{Q_E}.\]
(that occurs in the statement of the claimed result) is injective. Hence, to prove exactness of the stated sequence, it is enough for us to fix  $m\in \widehat{Q}$ such that, for every $E\in \cX$, the image $m_E$ of $m$ under the natural map $\widehat{Q} \to \widehat{Q_E} \subset \Q \otimes_\Z \widehat{Q_E}$ is contained in $\QQ \otimes_\ZZ \im(\iota_E)$, and show this implies $\iota_\cX(m)\in \im(\iota_\cX\circ\iota)$. 

To prove this, the key point is to note that, for $E \in \cX$, one has $\widehat{Q_E} \cap (\Q \otimes_\Z \im(\iota_E)) = \im(\iota_E)$ since $( \prod_{\ell \not \in \cS }\Z_\ell) \cap \Q = \Z_\cS$. In particular, the given assumption on $m$ implies that, for every $E\in \cX$, there exists an element $x_E$ of $Q_E$ such that $\iota_E(x_E)=m_E \in \widehat{Q_E}$.  Then, for each $F \in \cX$ with $E \subseteq F$, the commutative diagram 
\begin{cdiagram} Q_F \arrow[r,"\iota_F"] \arrow[d, "\rho_{F/E}"] & \widehat{Q_F} \arrow[d,"\widehat{\rho_{F/E}}"] \\
 Q_E \arrow[r,"\iota_E"] & \widehat{Q_E} \end{cdiagram}%
combines with the injectivity of $\iota_E$ and the fact $(m_E)_E = \iota_\cX(m)$ belongs to ${\varprojlim}_E\widehat{Q_E}$ to imply $\rho_{F/E}(x_F) = x_E$.  It follows that 
the family $x \coloneqq (x_E)_E$ belongs to $Q$ and is such that 
\[  \iota_\cX(m) =   (m_E)_E =( \iota_E (x_E))_E=  \iota_\cX(\iota (x)) \in \im(\iota_\cX\circ\iota),\]
as required.
\end{proof}

\subsubsection{The proof of Theorem~\ref{rigidity theorem}} \label{rigidity proof section}

As a last step before the proof of Theorem~\ref{rigidity theorem}, we establish an auxiliary result regarding the compatibility of Euler limits with $p$-adic completions. 

\begin{lem} \label{p completion euler limits compatibility}
Fix $p\in \mathbb{P}$ and $\cX \subseteq \Omega$ and a triangular system $\{ ( C_{E, F}, \rho_{F / E}, j_{F / E} )\}_{F, E \in \cX}$  such that each map $j_{F / E}$ is injective  and each $C_E$ is $p$-torsion-free. 
Then, for any finite subset $\Pi$ of $S_\fin(k)$, the natural map
\begin{align*}
(\eullim{E \in \cX}{\Pi} C_E)^{\wedge, p} \to \eullim{E \in \cX}{\Pi} \widehat{C_E}^{p},
\quad 
\big( (m_{E, n})_{E \in \cX} \big)_{n \in \N} \mapsto 
\big( (m_{E, n})_{n \in \N}\big)_{E \in \cX}.
\end{align*}
is both well-defined and injective.
\end{lem}

\begin{proof} Since the maps $j_{F / E}$ are injective, as shall for brevity omit explicit reference to them. 

We set $C \coloneqq \eullim{E \in \cX}{\Pi} C_E$ and, for every $E \in \cX$, we take the limit (over $n \in \N$) of the natural maps $C / p^n \to C_E / p^n$ to obtain a map
\begin{align}
C^{\wedge, p} \to \widehat{C_E}^{p}.\label{completion-map}
\end{align}
By taking the product of these maps over $E \in \cX$ we then obtain the further homomorphism
\[
C^{\wedge, p} \to {\prod}_{E\in \cX} \widehat{C_E}^{p},
\]
which we claim has image inside $\eullim{E \in \cX}{\Pi} \widehat{C_E}^{p}$.
To show this we suppose to be given an element $m = (\tilde m_n)_n$ of  $C^{\wedge, p} = \varprojlim_n (C / p^n C)$.
In particular, each $\tilde m_n$ is the class in $C / p^n C$ of some $m_n \in C$.
For every $E \in \cX$, we also write $m_E$ for the image of $m$ under (\ref{completion-map}).
Then, for every $n\in \mathbb{N}$ and $F \in \cX$ with $E \subseteq F$, we may calculate in $C_E / p^n$ that 
\[
\rho_{F / E} (m_F) \equiv   \rho_{F / E} (m_{n, F}) \equiv  P_{F / E, \Pi} \cdot m_{n, E} \equiv   P_{F / E, \Pi} \cdot m_E \mod p^n C_E,
\]
where $m_{n, F}$ denotes the image of $m_n \in C$ in $C_F$. 
Since the above equality holds for all $n$, this shows that $\rho_{F / E} (m_F) = P_{F / E, \Pi} \cdot m_E$, as required. \\
To prove injectivity, we assume $(m_E)_E$ is trivial in $\eullim{E \in \cX}{\Pi} \widehat{C_E}^{p}$ and must show  $m=0$. To do this, we note that, for every $E \in \cX$, one has $m_E = 0$, and so, for every $n$, there exists $z_{n, E}\in C_E$ such that  $m_{n, E} = p^n z_{n, E}$. We next claim that, for each $n$, the collection $z_n \coloneqq (z_{n, E})_E$ is an element of $C$. Indeed, for every $F \in \cX$ containing $E$, one has 
\[
p^n \cdot \rho_{F / E} (z_{n, F}) = \rho_{F / E} (p^n z_{n, F}) =
\rho_{F / E} ( m_{n, F}) = P_{F / E, \Pi} \cdot  m_{n, E}
= p^n \cdot (P_{F / E, \Pi} \cdot  z_{n, E})
\]
and thus (since $C_E$ is $p$-torsion-free)  $\rho_{F / E} (z_{n, F}) = z_{n, E}$, as required. It now follows that, for every $n$, one has  $m_n = p^n z_n \in p^nC$, and hence that $m=0$, as required. 
\end{proof}

We are finally in a position to prove Theorem~\ref{rigidity theorem}.

\begin{proof}[Proof of Theorem~\ref{rigidity theorem}] At the outset we note the map $\beta_\ast$ is injective by Proposition~\ref{comparison-diamond-projective-limit}\,(i).
We next fix $p\in \mathbb{P}\setminus \mathcal{S}$ and  clarify the nature of the map $\theta_p$ that appears in the statement of Theorem~\ref{rigidity theorem}. To do this, we fix a natural number $n$ and note that, upon passing to the limit over $E \in \cX$ of the tautological short exact sequences 
\[ 0 \to \epsilon_E M_E\xrightarrow{\cdot p^n} \epsilon_E M_E \to (\epsilon_E M_E)/p^n \to 0\]
one obtains a canonical short exact sequence 
\begin{equation}\label{key limit seq}
0 \to ( {\varprojlim}_{E \in \cX} (\epsilon_E M))/p^n \to {\varprojlim}_{E \in \cX} \big( \epsilon_E M_E/p^n \big)
\xrightarrow{\theta_{p,n}}  \bigl({\varprojlim}^1_{E \in \cX} (\epsilon_E M_E)\bigr)[p^n] \to 0. 
\end{equation}

The map $\theta_{p}$ is then defined  to be the composite homomorphism 
\begin{align*}
{\varprojlim}_{E \in \cX} (\Z_p\otimes_\Z \epsilon_EM_E) 
& \xrightarrow{\simeq} {\varprojlim}_{E \in \cX} \bigl({\varprojlim}_n  ((\epsilon_EM_E)/p^n)\bigr) \\
& \xrightarrow{\simeq}
{\varprojlim}_n \bigl(\varprojlim_{E \in \cX} ((\epsilon_EM_E)/p^n)\bigr)\\ 
& \to {\varprojlim}_n\bigl({\varprojlim}^1_{E \in \cX} (\epsilon_EM_E)\bigr)[p^n]
\end{align*}
in which the two isomorphisms are the canonical identifications and the unlabelled arrow is the limit (over $n$) of the maps $\theta_{p,n}$ in (\ref{key limit seq}). \\
\textcolor{black}{Having defined $\theta_{p}$ in this way, we next claim $\im \beta_\ast \subseteq \ker \Delta$. To show this, we note  the short exact sequences (\ref{key limit seq}) give rise, upon 
 passing to the limit (over $n$), to an exact sequence
\[
\widehat{M} \to {\varprojlim}_{E \in \cX} (\Z_p\otimes_\Z \epsilon_EM_E)  \xrightarrow{\theta_p} {\varprojlim}_n\bigl({\varprojlim}^1_{E \in \cX} (\epsilon_EM_E)\bigr)[p^n].
\]
This sequence implies the image of the natural map  
$\varprojlim_{F\in \cX} \epsilon_F M_F \to \varprojlim_{E \in \cX} (\Z_p\otimes_\Z \epsilon_EM_E)$ is contained in $\ker \theta_p$ and hence that  $\beta_\ast(\varprojlim_{F\in \cX} \epsilon_F M_F)\subseteq \ker \Delta$, as claimed.}\\
In order to prove Theorem~\ref{rigidity theorem}, it therefore only remains to show any element $a$ of $\ker\Delta$ belongs to $\im \beta_\ast$ or, equivalently, that the class of such an element $a$ in the quotient module $ X(a)$ that occurs in Proposition~\ref{profinite-injective-revisited} vanishes. In view of the latter result, it is thus enough for us to show that, for every prime $p \not \in \cS$, the class of $a$ in $\widehat{X(a)}^p$ vanishes. \\
Now, since $a$ belongs to $\ker\Delta$, there exists  $u^{(p)} = (u^{(p)}_{E})_{E \in \cX}\in \ker (\theta_p)$ such that 
 $a=\beta^{(p)}_\ast ( u^{(p)})$ in $\eullim{F \in \cX}{\emptyset} (\Z_p \otimes_\Z N_F)$. In particular,  for every $E \in \cX$ one has $a_E\in \im \beta^{(p)}_{\ast, E}$ , and therefore $\beta^{(p)}_{\ast, E} ( \Z_p \otimes \epsilon_E M_E) + \Z_p [\cG_E] a_E = \beta^{(p)}_{\ast, E} ( \Z_p \otimes_\Z \epsilon_E M_E)$. Given these equalities, Lemma~\ref{p completion euler limits compatibility} implies the existence of an injective map 
 \[
i_2 \: \widehat{M (a)}^{p} \hookrightarrow \eullim{E \in \cX}{\emptyset} ( \beta^{(p)}_{\ast, E} ( \Z_p \otimes_\Z \epsilon_E M_E))
 \]
that lies in an exact commutative diagram
 \begin{equation} \label{key diagram}
     \begin{tikzcd}
         0 \arrow{r} & \widehat{M}^p \arrow[hookrightarrow]{d}{i_1} \arrow{r}{\widehat{\beta}^p_\ast} & 
         \widehat{M (a)}^{p} \arrow[hookrightarrow]{d}{i_2} \arrow{r} & \widehat{X(a)}^p \arrow{r} & 0 \\
         0 \arrow{r} & \varprojlim_{E \in \cX} (\Z_p \otimes_\Z \epsilon_E M_E) 
         \arrow{r}{\beta^{(p)}_{\ast}} & \eullim{E \in \cX}{\emptyset} ( \beta^{(p)}_{\ast, E} ( \Z_p \otimes_\Z \epsilon_E M_E)), 
     \end{tikzcd}
 \end{equation}
in which exactness of the upper row  follows from Lemma~\ref{completion-lemma-revisited}\,(d) (and the fact  $X(a)$ is $\Z$-torsion-free by Lemma~\ref{X-torsion-free}). Via this diagram, we are thereby reduced to showing that $i_2$ 
 sends the image of $a$ in $\widehat{M (a)}^{p}$ to an element of $\im (\widehat{\beta}^p_\ast)$.\\ 
 To do this, \textcolor{black}{we note  the explicit definition of $\theta_{p}$ combines with the exactness of (\ref{key limit seq})  to identify 
 $\ker (\theta_p)$ with $i_1 (\widehat{M}^p)$, where $i_1$ is the map $\widehat{M}^p \to \varprojlim_{E \in \cX_{S_p}} (\Z_p \otimes_\Z \epsilon_E M_E)$ sending each $x^{(p)}$ to $\big( (x^{(p)}_E)_{n \in \N} \big)_{E \in \cX}$.} 
 Hence, since $u^{(p)}\in \ker (\theta_p)$, there exists an element  $x^{(p)} = \big( (x^{(p)}_E)_{E \in \cX}\big)_{n \in \N}$ of $\widehat M$ with $i_1(x^{(p)}) = u^{(p)}$. 
 By the commutativity of (\ref{key diagram}) we can then conclude that
 \[
 (i_2 \circ \widehat{\beta}^p_\ast ) ( x^{(p)}) 
 = (\beta^{(p)}_\ast \circ i_1) ( x^{(p)}) = \beta^{(p)}_\ast ( u^{(p)}) = i_2 (a),
 \]
 and hence, since $i_2$ is injective, that $a = \widehat{\beta}^p_\ast ( x^{(p)})\in \im (\widehat{\beta}^p_\ast)$.    \\
 This concludes the proof of Theorem~\ref{rigidity theorem}.
\end{proof}

\subsubsection{$p$-adic Euler limits}
In this final subsection  we fix  $p\in \mathbb{P}$ and an isomorphism $\C \cong \C_p$, which we use to regard $\Q_p$ as a subfield of $\C$. We also write $S_p$ for the set of $p$-adic places of $k$. We shall then  prove the following useful result on $p$-adic Euler limits.

\begin{proposition} \label{p-adic-Euler-limit}
Fix  $V \subseteq S_\infty$ and a subset $\cX$ of $\Omega$ satisfying Hypothesis \ref{hyp X} with $\cV = \{ V \}$. Let  $\{ (M_F, \varphi_{F / E}) \}_{F \in \cX_{S_p}}$ be a projective system of $\Z_p [\cG_F]$-lattices that satisfies both Hypothesis \ref{inj-hypothesis} with respect to $p$ (when regarded as a triangular system as in Example~\ref{triangular-system-examples}\,(a)) and also Hypothesis \ref{rigidity hyp}. Then the following homomorphism is bijective
\[
 \varprojlim_{F \in \cX_{S_p}}\epsilon_F M_F \to \eullim{F \in \cX_{S_p}}{\emptyset} e_{F} M_F,
\quad (m_F)_F \mapsto (e_{F} m_F)_F.
\]
\end{proposition}

\begin{proof} We fix a $\Z_p$-extension $k_\infty$ of $k$ in which no finite place splits completely (for example, one may take $k_\infty$ to be the cyclotomic $\Z_p$-extension of $k$). Since all infinite places split completely in $k_\infty$, 
Hypothesis \ref{hyp X}\,(i) implies that, for any field $K \in \cX$, every finite extension of $k$ contained in the composite $K_\infty = K k_\infty$ is also contained in $\cX$. 
We write $K_n$ for the $n$-th layer of $K_\infty / K$ and let $\bLambda_K = \Z_p \llbracket \gal{K_\infty}{k} \rrbracket$ be the relevant (equivariant) Iwasawa algebra. \\
Fix  $E \in \cX_{S_p}$ and $(a_F)_F \in \eullim{F \in \cX_{S_p}}{\emptyset} e_F M_F$. Then, since $S_p \subseteq S (E)$ and $S_\ram(E_\infty / E) \subseteq S_p(E)$, we have $S(E) = S(E_n)$ for all $n \geq 0$. Hence, in this case, the defining property of the Euler limit implies $a_{E_n} = \varphi_{E_m / E_n} (a_{E_m})$ for all $m \geq n$ and so $(a_{E_n})_n\in \varprojlim_n e_{E_n} M_{E_n}$. To study the latter limit we note that each term in the exact sequence
\[ 0 \to (\epsilon_{E_n} M_{E_n}) [e_{E_n}] \xrightarrow{\subseteq} \epsilon_{E_n}  M_{E_n} \xrightarrow{x \mapsto e_{E_n} x}  
e_{E_n} M_{E_n} \to 0\]
is a finitely-generated $\Z_p$-module (and hence endowed with the structure of a compact Hausdorff topological group) so that, upon passing to the inverse limit, one obtains an exact sequence
\begin{equation}\label{useful-isomorphism}
0 \to {\varprojlim}_n \big( (\epsilon_{E_n} M_{E_n})[e_{E_n}] \big) \xrightarrow{\subseteq} {\varprojlim}_n (\epsilon_{E_n}M_{E_n}) \xrightarrow{\kappa_E}  
{\varprojlim}_n (e_{E_n} M_{E_n}) \to 0\end{equation}
in which $\kappa_E$ sends $(x_n)_n$ to $(e_{E_n}x_n)_n$. We next claim  ${\varprojlim}_n \big( (\epsilon_{E_n} M_{E_n})[e_{E_n}] \big) = (0)$ and hence that $\kappa_E$ is bijective.\\
To prove this we note  no finite prime splits completely in $E_\infty / E$ and so  there exists $N \in \mathbb{N}$ such that, for each $n \geq N$, every character in $\Upsilon_{E_n}^\diamond \setminus \Upsilon_{E_n}$ factors through $\cG_{E_N}$. Hence, for any such $n$ and each $x \in (\epsilon_{E_n} M_{E_n}) [e_{E_n}]$ and $\sigma \in \gal{E_n}{E_N}$, one  has 
$\sigma \cdot x = \sigma (1 - e_{E_n}) x = (1 - e_{E_n}) x = x$. It follows that ${\varprojlim}_n \big( (\epsilon_{E_n} M_{E_n})[e_{E_n}] \big) \subseteq  ({\varprojlim}_n(\epsilon_{E_n} M_{E_n}))^{\gal{E_\infty}{E_N}}$ and so it is enough to note that the latter module vanishes as a consequence of Lemma~\ref{some-strange-lemma}. Here we use the fact that, since each $\epsilon_{E_n} M_{E_n}$ is a $\Z_p[\cG_{E_n}]$-lattice (and thus $(\epsilon_{E_n} M_{E_n})_{p\cdot\mathrm{div}} = (0)$), the stated assumptions imply the hypotheses of Lemma~\ref{some-strange-lemma} are satisfied in this case. \\
We now write $b_{E_\infty} = (b_{E_n})_n$ for the unique element of ${\varprojlim}_n (\epsilon_{E_n }M_{E_n})$ with 
$\kappa_E(b_{E_\infty}) = (a_{E_n})_n$ and claim that, as $E$ varies over $\cX_{S_p}$, the family 
\[ b (a) \coloneqq (b_E)_{E \in \cX_{S_p}}\] 
defines an element of $\varprojlim_E (\epsilon_E M_E)$. To prove this, we must show that, for each pair of fields $E$ and $F$ in $\cX_{S_p}$ with $E \subseteq F$, the map $\varphi_{F_\infty / E_\infty} \: \varprojlim_n (\epsilon_{F_n} M_{F_n}) \to \varprojlim_n (\epsilon_{E_n} M_{E_n})$ induced by $(\varphi_{F_n / E_n})_n$ sends $b_{F_\infty}$ to $b_{E_\infty}$.  {As a first step, we recall from Example~\ref{es remark}\,(d) that, 
for any pair of fields $A, B \in \cX$ with $A \subseteq B$ one has $
e_A \cdot P_{B / A, \emptyset} = \pi_{B /A} (e_B) \cdot P_{B / A, \emptyset}$. Using this, we then compute, for each $m \in \N$, that 
}
\begin{align*}
{
e_{E_m} \cdot P_{F_m / E_m, \emptyset} \cdot \varphi_{F_m / E_m} (b_{F_m})} & =
\pi_{F_m / E_m} (e_{F_m}) \cdot P_{F_m / E_m, \emptyset} \cdot \varphi_{F_m / E_m} (b_{F_m}) \\
& = P_{F_m / E_m, \emptyset} \cdot \varphi_{F_m / E_m} ( e_{F_m} b_{F_m}) \\
& = P_{F_m / E_m, \emptyset} \cdot \varphi_{F_m / E_m} ( e_{F_m} a_{F_m}) \\
& = \pi_{F_m / E_m}(e_{F_m}) \cdot P_{F_m / E_m, \emptyset} \cdot a_{E_m} \\
& = e_{E_m} \cdot P_{F_m / E_m, \emptyset} \cdot a_{E_m} \\
& = { e_{E_m} \cdot P_{F_m / E_m, \emptyset} \cdot b_{E_m},}
\end{align*}
where the third equality is true since $\kappa_F(b_{F_\infty}) = (a_{F_n})_n$. Upon taking the limit over $m$ of these equalities, and using the bijectivity of the map $\kappa_E$ in (\ref{useful-isomorphism}), we thereby obtain  an equality  $P_{F / E, \emptyset} \cdot \varphi_{F_\infty / E_\infty} (b_{F_\infty}) = P_{F / E, \emptyset} \cdot b_{E_\infty}.$ To derive the claimed equality $\varphi_{F_\infty / E_\infty} (b_{F_\infty}) = b_{E_\infty}$, it is therefore now  enough to show that no non-zero element of $\varprojlim_n \epsilon_{E_n} M_{E_n}$ lies in the kernel of multiplication by $P_{F / E, \emptyset}$. But this is true since no finite place splits completely in $E_\infty / E$, so that the  Frobenius automorphism $\Frob_v$ of each $v \in S(F) \setminus S(E)$ generates an open subgroup of $\cG_{E_\infty}$ and hence Lemma~\ref{some-strange-lemma} implies each element $1 - \Frob_v$, and thus also $P_{F / E, \emptyset}$, acts injectively
on $\varprojlim_n \epsilon_{E_n} M_{E_n}$. \\
Having shown each  $b (a)$ belongs to  $\varprojlim_{F\in \cX_{S_p}} \epsilon_F M_F$, the proof is reduced to showing that 
\[ \beta: \eullim{F \in \cX_{S_p}}{\emptyset} e_{F} M_F \to \varprojlim_{F \in \cX_{S_p}}\epsilon_F M_F, \, \, \,\, a \mapsto b (a)\]%
defines an inverse to the map $\alpha \: \varprojlim_{F \in \cX_{S_p}}\epsilon_F M_F \to \eullim{F \in \cX_{S_p}}{\emptyset} e_{F} M_F$ sending each $(m_F)_F$ to $(e_{F} m_F)_F$. In addition, the identity $\alpha \circ \beta = \id$ follows directly from the explicit  definition of each $b(a)$. It is thus enough to prove $\beta \circ \alpha = \id$ and to do this we fix $m = (m_F) \in \varprojlim_F \epsilon_F M_F$. Then, for each $n$, one has $e_{F_n} m_{F_n} = e_{F_n} (\beta \circ \alpha)(m)_{F_n}$ and hence also $m_{F_n} = (\beta \circ \alpha)(m)_{F_n}$ (by the bijectivity of $\kappa_F$). In particular, one has $m_F = (\beta \circ\alpha)(m)_F$ for every $F\in \cX_{S_p}$, and thus $m = (\beta \circ \alpha)(m)$, as required to complete the proof. 
\end{proof}

\section{Euler systems and Euler limits}\label{euler systems and limits}

We shall now derive some arithmetic consequences of the general results on Euler limits from the last section. 

\subsection{The Uniformisation Theorem}\label{uniform sec}

Let $\cQ$ be a subfield of $\CC$. Then we recall from Definition~\ref{euler-systems-definition-1} that $\ES_k (\cQ)$ denotes the $\cQ \llbracket G_k^{\mathrm{ab}}\rrbracket$-module $\ES^{\bm{r}_\mathrm{max}}_k (\cQ)$ of $\cQ$-rational Euler systems relative to the maximal rank function $\bm{r}_{\mathrm{max}}$ for $k$ (cf.\@ Definition \ref{rank def}) and that, by Lemma~\ref{RS-properties-lemma}, the Rubin--Stark system $\varepsilon_k$ from Definition \ref{Rubin--Stark-definition}\,(b) is an element of $\ES_k (\mathbb{R})$. We further recall, from Example~\ref{es remark}\,\ref{rk reduction rem}, that for any rank function $\bm{t}$ on $k$ there exists a canonical $\cQ \llbracket G_k^{\mathrm{ab}}\rrbracket$-module $\mathrm{PR}^{\bm{t}}_k (\cQ)$ of Perrin-Riou functionals.  The following result concerning these modules will play a key role in our analysis of Conjecture \ref{scarcity-conjecture}. 

\begin{thm} \label{coleman-over-the-reals}
Assume $\cQ$ is such that $\varepsilon_k\in \ES_k (\cQ)$. Then, given any rank function $\bm{r}$ for $k$, the following claims are valid.
\begin{liste}
\item If $\bm{r} > \bm{r}_\mathrm{max}$, then  $\ES^{\bm{r}}_k (\cQ)$ vanishes.
\item If $\bm{r}\le \bm{r}_\mathrm{max}$, then there exists a canonical isomorphism of $\cQ \llbracket G_k^{\mathrm{ab}}\rrbracket$-modules
\[
\textcolor{black}{
\big \{  (\theta_E)_E \in \mathrm{PR}^{\bm{r}_\mathrm{max} - \bm{r}}_k (\cQ) 
\mid \forall E \in \Omega_k^\mathrm{exc} : e_{\bm{1}_E} \theta_E = 0 \big \}
}
 \xrightarrow{\simeq} \ES^{\bm{r}}_k (\cQ), 
\quad (\theta_E)_E \mapsto (\theta_E ( \varepsilon_{E / k}))_E.
\]
\item $\ES_k (\cQ) = \cQ \llbracket G^{\mathrm{ab}}_k \rrbracket\!\cdot\!\varepsilon_k + \ES_k (\cQ)^{G^{\mathrm{ab}}_k}.$
\item $\cQ \llbracket G^{\mathrm{ab}}_k \rrbracket\!\cdot\!\varepsilon_k$ is free of rank one as a $\cQ \llbracket G_k^{\mathrm{ab}} \rrbracket$-module and also equal to $\ES_k (\cQ)^{\mathrm{sym}}$.
\end{liste}
\end{thm}

\begin{proof} Claim (a) follows directly from Lemma~\ref{faithful-lem}\,(b). To prove claim (b), we recall from Example~\ref{es remark}\,(b) that the image of the displayed map is contained in  $\ES^{\bm{r}}_k (\cQ)$, and so we must prove that this map is bijective.  \\
For this, we recall (from \S\,\ref{RS section}) that  $S^\ast(E)$ denotes $S_\ram(E/k)\cup S_\infty$ for each $E \in \Omega_k$. The Dirichlet regulator map therefore induces an isomorphism of $\R [\cG_E]$-modules $\R \bigO_{E, S(E)}^\times \cong \R X_{E, S^\ast (E)}$ which in turn implies (cf.\@ the discussion in \cite[top of p.~26]{Tate}) that the $\cQ [\cG_E]$-modules $\cQ \bigO_{E, S(E)}^\times$ and $\cQ X_{E, S^\ast (E)}$ are isomorphic (though, in general, non-canonically). The $e_E \cQ [\cG_E]$-module  
$e_E \cQ \bigO_{E, S(E)}^\times \cong e_E \cQ X_{E, S^\ast (E)} = e_E \cQ Y_{E, V_E}$ is therefore free of rank $r_E = |V_E|$ and so the $e_E \cQ [\cG_E]$-module
\[
e_E\cQ \exprod^{r_E}_{\Z [\cG_E]} \bigO^\times_{E, S(E)} = \exprod^{r_E}_{e_E \cQ [\cG_E]} e_E \cQ \bigO^\times_{E, S(E)}
\]
is free of rank one. We fix a basis element $a$ of this module and then specify a unique element $q$ of $e_E \cQ [\cG_E]$ via the equality  
$\varepsilon_{E / k} = q \cdot a$. Then, since $\Upsilon_E (\subseteq \cG_E^\circ)$ is such that 
\begin{equation}\label{vanishing equalities} \{\chi\in \cG_E^\circ: e_\chi e_E \not= 0\} = \Upsilon_E = \{\chi\in \cG_E^\circ: e_\chi \varepsilon_{E / k} \not= 0\},\end{equation}
where the first equality follows directly from (\ref{eK def}) and the second from (\ref{tate order formula}), it follows that $q \in (e_E \cQ [\cG_E])^\times$ and so $\varepsilon_{E / k}$ is also a basis of 
$e_E\cQ \exprod^{r_E}_{\Z [\cG_E]} \bigO^\times_{E, S(E)}$ over $e_E \cQ [\cG_E]$. As this is true for every $E$, one therefore obtains a (well-defined) isomorphism of $\cQ[\![G_k^{\mathrm{ab}}]\!]$-modules 
\begin{align*} f \: \prod_{E \in \Omega_k}\Hom_{\cQ[\cG_E]}\big(e_{E}\cQ \exprod^{r_E}_{\Z [\cG_E]}  \cO_{E, S(E)}^\times, 
e_{E} \cQ \exprod^{\bm{r}(E)}_{\Z [\cG_E]} \cO_{E, S(E)}^\times \big) &\xrightarrow{\simeq}  \prod_{E \in \Omega_k} e_E\cQ\exprod^{\bm{r}(E)}_{\Z [\cG_E]} \cO_{E, S(E)}^\times 
\\
\theta = (\theta_E)_E &\mapsto f(\theta) \coloneqq (\theta_E(\varepsilon_{E / k}))_E.
\end{align*}
\textcolor{black}{
We next recall that $\mathrm{PR}_k^{\bm{r}_\mathrm{max} - \bm{r}} (\cQ)$ is defined to be the Euler limit $\eullim{E \in \Omega_k}{\emptyset} \cQ \exprod^{r_E - \bm{r} (E)}_{\Z [\cG_E]} (\cO_{E, S(E)}^\times)^\ast$ of the triangular system $\cT = \{(M_{F, E}, \rho_{F / E}, j_{F / E})\}_{E, F\in \cX}$ that is described in Example~\ref{triangular-system-examples}\,(c) for $\cX = \Omega_k$, $\bm{t} = \bm{r}$ and $\cR = \cQ$. We also now define $\mathrm{PR}_k^{\bm{r}_\mathrm{max} - \bm{r}} (\cQ)'$ to be the Euler limit $\eullim{E \in \Omega_k}{\emptyset}e_E \cQ \exprod^{r_E - \bm{r} (E)}_{\Z [\cG_E]} (\cO_{E, S(E)}^\times)^\ast$ of  the triangular system that is obtained from the above $\cT$ by replacing each term $M_{F,E}$ by its submodule $e_EM_{F,E}$.  \\
Then the argument of Lemma \ref{faithful-lem}\,(a) implies both that 
\[ \mathrm{PR}_k^{\bm{r}_\mathrm{max} - \bm{r}} (\cQ)' = \big\{ (\theta_E)_E \in \mathrm{PR}_k^{\bm{r}_\mathrm{max} - \bm{r}} (\cQ) \mid \forall E \in \Omega_k^\mathrm{exc} : e_{\bm{1}_E} \theta_E = 0 \big \} \]
and also that $\ES^{\bm{r}}_k (\cQ)$ identifies with a submodule of the codomain of the function $f$ defined above. In particular, if we use the general result of Lemma~\ref{technical-algebraic-lemma} below to regard $\mathrm{PR}_k^{\bm{r}_\mathrm{max} - \bm{r}} (\cQ)'$ as a submodule of the domain of $f$, then 
the proof of claim (b) is reduced to showing that 
\[
f \big (  \mathrm{PR}_k^{\bm{r}_\mathrm{max} - \bm{r}} (\cQ)' \big) = \ES^{\bm{r}}_k (\cQ),
\]
}
or equivalently, that for each element $\theta \in \prod_{E \in \Omega_k}e_{E}\cQ \exprod^{r_E - \bm{r} (E)}_{\Z [\cG_E]} (\cO_{E, S(E)}^\times)^\ast$ one has 
\begin{equation}\label{reinterpretation} \theta \in \eullim{E \in \Omega_k}{\emptyset}  e_{E} \cQ \exprod^{r_E - \bm{r} (E)}_{\Z [\cG_E]} (\cO_{E, S(E)}^\times)^\ast \quad \Longleftrightarrow \quad f(\theta) \in \ES^{\bm{r}}_k (\cQ).\end{equation}
  To verify this, we note that, for every $F\in \Omega_k$ with $E \subseteq F$, one has both 
\begin{align*} ((P_{F / E, \emptyset} \cdot  \Phi_{F / E}^{r_F  - \bm{r} (F)})(\theta_F)) (\varepsilon_{E / k})
    =&\, ( \Phi_{F / E}^{r_F  - \bm{r} (F)} (\theta_F)) \big ( \NN^{r_F}_{F / E} (\varepsilon_{F / k}) \big)\\
    =&\, \NN^{\bm{r}(F)}_{F / E} (\theta_F( \varepsilon_{F / k} ))\\
    =&\, \NN^{\bm{r}(F)}_{F / E} (f(\theta)_F)\end{align*}
and also 
\[ P_{F / E, \emptyset} \cdot f(\theta)_E  =  (P_{F / E, \emptyset} \cdot \theta_E)( \varepsilon_{E / k} ).\]
In particular, since $\varepsilon_{E / k}$ is an $e_{E} \cQ [\cG_E]$-basis of $e_E\cQ \exprod^{r_E}_{\Z [\cG_E]} \bigO^\times_{E, S(E)}$, it follows that    
\[ (P_{F/E, \emptyset} \cdot  \Phi_{F / E}^{r_F  - \bm{r} (F)})(\theta_F) = 
P_{F / E, \emptyset} \cdot \theta_E 
\quad \Longleftrightarrow \quad \NN^{\bm{r}(F)}_{F / E} (f(\theta)_F) = P_{F / E, \emptyset} \cdot f(\theta)_E.\]
This equivalence implies the validity of (\ref{reinterpretation}), and thereby completes the proof of claim (b). \\ 
To prove claim (c) it is enough to show every $c\in \ES_{k} (\cQ)$ belongs to $\cQ \llbracket G_k^{\mathrm{ab}} \rrbracket \cdot \varepsilon_k + \ES_k (\cQ)^{G^{\mathrm{ab}}_k}$. For this, we note claim (b) (with $\bm{r} = \bm{r}_\mathrm{max}$) implies that, for each $c$, there exists a unique 
$q = (q_E)_E\in \eullim{E \in \Omega_k}{\emptyset} e_E \cQ [\cG_E]$ with $c = q \cdot\varepsilon_k$, and we set $q' \coloneqq (e_{\bm{1}_E} q_E)_E\in (\eullim{E \in \Omega_k}{\emptyset} e_{\bm{1}_E} \cQ [\cG_E] )^{G_k^{\mathrm{ab}}}$. Then, 
since $q - q'$ belongs to the kernel of the map $\theta_2$ in Proposition~\ref{comparison-diamond-projective-limit}\,(i) when applied to the projective systems $(\cQ[\cG_F],\pi_{F/E})_{F \in \Omega_k^V}$ for all $V \subseteq S_\infty(k)$, we can deduce the existence of an element $l = (l_E)_E$ of $\cQ \llbracket G_k^{\mathrm{ab}} \rrbracket$ for which one has $e_{E} l_E = q_E - q'_E$ for every $E$. It follows that  
\[
c = (q_E \varepsilon_{E / k})_E = l \cdot \varepsilon_k + q' \cdot \varepsilon_k \in \cQ \llbracket G_k^{\mathrm{ab}} \rrbracket \cdot \varepsilon_k + \ES_k (\cQ)^{G_k^{\mathrm{ab}}},
\]
as required to prove claim (c).\\
To prove the first assertion of claim (d), we assume to be given $q \in \cQ \llbracket G_k^{\mathrm{ab}} \rrbracket$ with $q\cdot \varepsilon_k = 0$, or equivalently 
\begin{equation} \label{annihilator-equation}
q_E \cdot \varepsilon_{E / k} = 0 \quad \text{ for every } E \in \Omega_k, 
\end{equation}
where $q_E$ denotes the image of $q$ in $\cQ [\cG_E]$. Then we must show that $q_E=0$ for each fixed $E$ and to do this we use the following auxiliary fields:  if  $\chi \in \cG_E^\circ$ is ramified at a finite place, we set $E_\chi\coloneqq E^{\ker(\chi)}$; if $\chi\in \cG_E^\circ$ is unramified at all finite places, we fix an auxiliary prime ideal $\p$ of $k$ that has full decomposition group in $\cG_{E^{\ker (\chi)}}$ and is such that $k(\p)/k$ is ramified at $\p$ and set $E_\chi \coloneqq E^{\ker(\chi)} \cdot k(\p)$. We then write $E'$ for the compositum of $E$ with  $E_\chi$ for each $\chi \in \cG_E^\circ$. We note that $e_\chi q_E = e_\chi q_{E'} = e_\chi q_{E_\chi}$ for each  $\chi \in \cG_E^\circ$ and hence that $q_E = 0$ if $e_\chi q_{E_\chi} = 0$ for every $\chi \in \cG_E^\circ$. In addition, for each such $\chi$, one can check that $r_{S(E_\chi)} (\chi) = |V_{E_\chi}| = r_{E_\chi}$ by using the explicit formula (\ref{order-of-vanishing-formula}) and hence that (\ref{tate order formula}) implies $e_\chi \varepsilon_{E_\chi / k} \neq 0$. From the equality (\ref{annihilator-equation}) (with $E$ taken to be $E_\chi$) one can therefore deduce that $e_\chi(q_{E_\chi}) = 0$ for each $\chi \in \cG_E^\circ$, as required.\\
Turning now to the second assertion of claim (d), we recall $ \varepsilon_k\in \ES_k (\cQ)^{\mathrm{sym}}$ by Lemma~\ref{RS-properties-lemma}\,(i). It is thus  enough to show that for every $c\in \ES_k (\cQ)^{\mathrm{sym}}$ there exists $q' \in \cQ \llbracket G_k^{\mathrm{ab}} \rrbracket$ with $c = q' \cdot \varepsilon_k$. In addition, claim (b) implies $c = q \cdot \varepsilon_k$ for some $q = (q_E)_E \in \eullim{E \in \Omega_k}{\emptyset} e_E \cQ [\cG_E]$, and so it is enough to prove the existence of an element $q'$ of $\cQ \llbracket G_k^{\mathrm{ab}} \rrbracket$ such that, for every $E$, $q_E$ is equal to $e_E$ times the image of $q'$ in $\cQ[\cG_E]$. By  applying Proposition~\ref{comparison-diamond-projective-limit}\,(i) (to the projective systems $(\cQ[\cG_E],\pi_{E/E'})_{E \in \Omega_k^V}$ for all $V \subseteq S_\infty(k)$) it is therefore enough for us to show that $(e_{\bm{1}_E} e_E q_E)_E$ belongs to the
image of the map 
\begin{equation}\label{key map} {\varprojlim}_{E\in \Omega_k} \cQ [\cG_E] \to {\prod}_{E\in \Omega_k} \cQ [\cG_E], \, \, \, \, (a_E)_E \mapsto (e_{\bm{1}_E} e_E a_E)_E.\end{equation}
To verify this, we note $e_{\bm{1}_E} e_E = 0$ if either $V_E \not= S_\infty(k)$ or $|S(E)| \not= r_E + 1$. On the other hand, if $E \in \Omega_k$ is such that both $V_E = S_\infty(k)$ and $|S(E)| = r_E + 1$ (and so $|S(E)\cap S_\fin(k)|=1$), then one has  
\begin{align*}
\pi_{E / k} (q_E)\!\cdot\!\varepsilon_{k , k}\! =\! 
\pi_{E / k} (q_E)\!\cdot\!(\Ord_\p\! \circ\! \NN^{r_E}_{E / k}) ( \varepsilon_{E / k}) &\!=\!  (\Ord_\p\! \circ\! \NN^{r_E}_{E / k}) ( e_{\bm{1}_E} q_E \varepsilon_{E / k}) \\
&\! = \! (\Ord_\p\! \circ\! \NN^{r_E}_{E / k}) ( e_{\bm{1}_E} c_E)\! =\!  (\Ord_\p\! \circ\! \NN^{r_E}_{E / k}) (  c_E) = c_k
\end{align*}
where we write $\p$ for the unique place in $S(E)\cap S_\fin(k)$ and so the first and last equalities respectively follow from the fact $\varepsilon_k$ and $c$ are symmetric. In particular, for any other $F \in \Omega_k$ with $e_{\bm{1}_F} e_F \not= 0$, one has $\pi_{E / k} (q_E) \cdot \varepsilon_{k,k} = c_k = \pi_{F / k} (q_F) \cdot \varepsilon_{k,k},$ and thus $\pi_{E / k} (q_E) = \pi_{F / k} (q_F)$ since Lemma~\ref{RS-properties-lemma}\,(i) implies $\varepsilon_{k,k}\not= 0$. In particular, if we now set $q_k \coloneqq \pi_{E / k} (q_E)$ for any $E\in \Omega_k$  with $e_{\bm{1}_E} e_E \not= 0$ (which, by virtue of the above discussion, is independent of $E$), then one can check that $(e_{\bm{1}_E} q_k)_E$ of $ \varprojlim_E \cQ [\cG_E]$ is a preimage of $(e_{\bm{1}_E} e_E q_E)_E$ under the map (\ref{key map}). This therefore concludes the proof of claim (d).
\end{proof}

The following general result was used in the argument above. 

\begin{lem} \label{technical-algebraic-lemma}
Let $R$ be a commutative ring and $F$ a finitely-generated free $R$-module of rank $d$. Then, for any $s \in \N_0$ with $s \leq d$, there exists a canonical 
isomorphism of $R$-modules 
\begin{align*}
 \exprod^s_R F^\ast & \to \Hom_R \Big ( \exprod^d_R F, \exprod^{d - s}_R F \Big), \\
  f_1 \wedge \dots \wedge f_s & \mapsto 
  \Big \{ m_1 \wedge \dots \wedge m_d \mapsto 
\sum_{\sigma} \mathrm{sgn} (\sigma) \det( f_i (m_{\sigma (j)}))_{1 \leq i, j \leq s} m_{\sigma (s +1)} \wedge \dots \wedge m_{\sigma (d)}
\Big \},
\end{align*}
where the sum runs over all permutations $\sigma$ of the set  $\{1,2,\dots, d\}$ with the property that both $\sigma (1) < \dots < \sigma (s)$ and $\sigma (s + 1) < \dots < \sigma (d)$.
\end{lem}

\begin{proof} This is a straightforward exercise that we leave to the reader. 
\end{proof}

\begin{rk}\label{tempting conjecture} In view of Theorem~\ref{coleman-over-the-reals}\,(b) it is  tempting to strengthen Conjecture~\ref{scarcity-conjecture} by predicting any Euler system of rank $\bm{r}\le \bm{r}_{\mathrm{max}}$ that satisfies an appropriate analogue of the requirement to be a symmetric congruence system should arise from the Rubin--Stark system via appropriate rank reduction (as per Example~\ref{es remark}\,(b)). We do not discuss this possibility further, except to observe that 
conjectures connecting particular examples of rank-one Euler systems 
to a higher-rank Euler system have previously appeared in the literature (see, for example, \cite[Conj.\@ 3.5.1]{buyukboduk-et-al}) and even been verified in special cases (see \cite{buyukboduk-lei-2019}). 
\end{rk}

\subsection{Consequences for integral Euler systems}\label{cons section} In this subsection we shall combine the results of Theorems \ref{p adic restriction sequence}, \ref{rigidity theorem} and \ref{coleman-over-the-reals} in order to reduce the proof of Conjecture \ref{scarcity-conjecture} to consideration of the individual components of Euler systems. This result (which is proved as Theorem~\ref{cons scarcity thm} in \S\,\ref{reduction result}) will thereby provide us with an effective means of providing evidence in support of Conjecture \ref{scarcity-conjecture}. However, before discussing it, we must first prove two auxiliary results concerning the modules of isolated and congruence Euler systems introduced in Definitions \ref{local construction} and \ref{con def}.

\subsubsection{Isolated and Congruence systems}

We first clarify the nature of the modules of isolated Euler systems for those Dedekind domains that are of most interest to us. 

\begin{lemma}\label{invariant systems description} 
Let $\cR$ denote either $\Z_\cS$ for a finite subset $\cS$ of $\mathbb{P}$ or $\Z_p$ for some $p \in \mathbb{P}$.
Then, for each rank function $\bm{r}$ for $k$ and subset $\cX\subseteq \Omega_k$, one has $\ES^{\bm{r},\cX}_k(\cR)^{\mathrm{iso}}= \ES_k^{\bm{r}, \cX} (\cR)^{G^{\mathrm{ab}}_k}$. 
\end{lemma}

\begin{proof}
From the explicit construction of isolated systems given in \S\ref{integrality restrictions section}, it is clear that $\ES^{\bm{r},\cX}_k(\cR)^{\mathrm{iso}}\subseteq \ES_k^{\bm{r},\cX} (\cR)^{G_k^{\mathrm{ab}}}$. In addition, if a family $c = (c_F)_F$ belongs to either $\ES^{\bm{r},\cX}_k(\cR)^{\mathrm{iso}}$ or $\ES_k^{\bm{r},\cX}(\cR)^{G^{\mathrm{ab}}_k}$, then one has $c_F = 0$ whenever $S(F)$ contains an archimedean place or at least two finite places. Indeed, for isolated systems this property follows directly from their construction and for $G_k^{\mathrm{ab}}$-invariant systems it is a consequence of Lemma~\ref{faithful-lem} and the explicit formula (\ref{order-of-vanishing-formula}). In the sequel it is therefore enough for us to consider the case $\cX = \Omega^{S_\infty (k) }_k$ and to show that each $G_k^{\mathrm{ab}}$-invariant system $c \in \ES_k^{\bm{r},\cX} (\cR)$ is isolated. 
 
To do this, we first note that, for $F\in \cX$ and each intermediate field $E$ of $F/k$, the $G_k^{\mathrm{ab}}$-invariance of $c$ implies  
 $c_F = [F : E]^{-1} \cdot (\nu_{F / E} \circ \NN^{\bm{r} (F)}_{F / E}) (c_F),$ where $\nu_{F / E}$ is the isomorphism defined just before Lemma~\ref{lattice lemma2}. In particular, if we now fix $\p\in S_\fin(k)$ such that $k(\p^\infty)/k$ is ramified and take $F$ to be any finite ramified extension of $k$ in $k(\p^\infty)$, then $a'_\p \coloneqq \NN_{F / k}^{\bm{r} (F)} (c_F)$ is independent of $F$ and such that $c_F = [F : k]^{-1} \nu_{F / k}(a'_\p)$. Hence, for every such $F$ one has 
\[ a'_\p = [F : k] \cdot \nu_{F / k}^{-1} (c_F) \in [F : k]\cdot |\mu_k|^{-1} \exprod^{\bm{r}(k(\p))}_{\cR} (\cR\cO^{\times}_{k,\{\p\}})^{\ast \ast},\]
where the containment follows from Lemma~\ref{lattice lemma2}\,(d) and the fact $c_F \in (\cR\fL_F^{\bm{r}(F)})^{\cG_F}$. Since, for large enough $F$, the degree $[F:k]$ is divisible by the (finite) order $t_\p$ of $(\cG_{k (\p^\infty)})_{\mathrm{tor}}$, we can therefore conclude that $a_\p \coloneqq t_\p^{-1}\cdot a_\p'$ belongs to $|\mu_k|^{-1} \exprod^{\bm{r}(k(\p))}_{\cR} (\cR\cO^{\times}_{k, \{ \p\}})^{\ast \ast}$ and is such that $c_F = [F : k]^{-1} t_\p \nu_{F / k} (a_\p)$ for all finite ramified extensions $F$ of $k$ in $k(\p^\infty)$.\\
To conclude $c$ is isolated, it is now enough to prove $c_F = 0$ if $S(F) = \{\p\}$ for any $\p\in S_\fin(k)$ for which $k(\p^\infty) / k$ is infinite and, in addition, its residue characteristic is not a unit in $\cR$ (so that $\p$ divides no prime in $\mathcal{S}$ if $\cR = \Z_\cS$ and $\p\in S_p(k)$ if $\cR = \Z_p$). To do this,  we fix such a $\p$, write $\ell$ for its residue characteristic and let $F_\infty$ be a $\Z_\ell$-extension of $F$ contained in $k(\p^\infty)$. Then, with $F_n$ the intermediate field of $F_\infty /F$ of degree $\ell^n$ over $F$, one has 
\[
\nu_{F_n / F} (c_F) = (\nu_{F_n / F} \circ \NN^{\bm{r} (F)}_{F_n / F}) (c_{F_n}) = \NN_{\gal{F_n}{F}} \cdot c_{F_n} = [F_n : F] \cdot c_{F_n} \in \ell^n\cdot (\cR \fL^{\bm{r}(F)}_{F_n})^{\cG_{F_n}},
\]
where the first equality is valid since $S(F_n) = S(F) = \{\p\}$ and the third since $c_{F_n}$ is $\cG_{F_n}$-invariant. In view of Lemma~\ref{lattice lemma2}\,(c), this implies $c_F=0$, as required.
\end{proof}

The following result describes an explicit property of congruence systems that will be key to later arguments.  

\begin{prop} \label{mrs and soogils argument}
Fix $p \in \mathbb{P}$ with $p \mid 2d_k$ and a system $c$ in $\ES_k (\Z_p)^{\mathrm{con}}$ 
that is fixed by an open subgroup of $G_k^{\mathrm{ab}}$.  
Then, for every  $E\in \Omega_k$ with $r_E > 0$, the element  $c_E$ is fixed by $\cG_E$. 
\end{prop}

\begin{proof} We fix an open subgroup $J$ of $G_k^{\mathrm{ab}}$ under which $c$ is invariant and set $N \coloneqq (k^{\mathrm{ab}})^J$. We then fix $K\in \Omega_k$ with $r_K > 0$ and set 
\[ \mathscr{Z}_n \coloneqq \{v \in S_\fin(k): v \,\,\text{splits completely in the compositum}\,\, K(n) \coloneqq N\! \cdot\! K(\mu_{p^n}, (\cO_k^\times)^{1 / p^n})\}.\] 
Finally, we fix $T \in \mathscr{P}_K^{\mathrm{ad}}$ with $T \cap (\mathscr{Z}_1 \cup S_\ram (K / \Q) \cup S_p (k)) = \emptyset$ and set $c_{K, T} \coloneqq \delta_{T, K} \cdot c_K$. Then, since $\delta_{T, K}$ is a non-zero divisor in $\Z_p [\cG_K]$, and $\Z_p \fL_K^{r_K}$ is $\Z$-torsion free, the claimed result will follow if we can prove that, for some $z \in \N$, one has $z(\sigma  - 1) c_{K, T} = 0$ for every $\sigma\in \cG_K$.\\
To do this, we set $U_{K, S(K), T} \coloneqq \Z_p \otimes_\Z \cO_{K, S(K), T}^\times$
\textcolor{black}{and recall each $f \in \exprod^{r_K - 1}_{\Z_p [\cG_K]} U_{K, S(K), T}^{\ast}$ naturally induces a `rank lowering' map $\bidual^{r_K}_{\Z_p [\cG_K]} U_{K, S(K), T} \to U_{K, S(K), T}^{\ast \ast} \cong U_{K, S(K), T}$ (cf.\@ \cite[(A.1)]{sbA}) that, by a slight abuse of notation, we will also denote by $f$.
We then claim that, for any fixed $z \in \N$ and $\sigma \in \cG_K$, 
one has the implication
\begin{equation} \label{the implication we claim}
    z (\sigma - 1) f(c_{K, T}) = 0 \, \, \text{ for every } \,\, f \in \exprod^{r_K - 1}_{\Z_p [\cG_K]} U_{K, S(K), T}^{\ast}
    \quad \Rightarrow \quad z(\sigma  - 1) c_{K, T} = 0.
\end{equation}}
 To verify this claim, we recall (from \cite[Rem.\@ 5.11]{BKS}) that there exists an injective map $\iota \: U_{K, S(K), T} \to P$ of $\Z_p [\cG_K]$-modules in which $P$ is free of finite rank and $\cok (\iota)$ is $\Z_p$-free. \textcolor{black}{We write  
\begin{align*}
\iota_\ast & \: \bidual^{r_K}_{\Z_p [\cG_K]} U_{K, S(K), T} \to \bidual^{r_K}_{\Z_p [\cG_K]} P \cong \exprod_{\Z_p [\cG_K]}^{r_K} P\\
\iota^\ast & \: \exprod^{r_K - 1}_{\Z_p [\cG_K]} P^{\ast} \to \exprod^{r_K - 1}_{\Z_p [\cG_K]} U_{K, S(K), T}^\ast
\end{align*}
for the maps that are functorially induced by $\iota$ (and its $\Z_p$-dual). To prove (\ref{the implication we claim}) is then enough to show that 
\begin{equation} \label{second implication we claim}
z (\sigma - 1) \iota^\ast(\tilde f)(c_{K, T}) = 0\,\,\text{ for every } \,\, \tilde f \in \exprod^{r_K - 1}_{\Z_p [\cG_K]} P^{\ast}\quad  \Rightarrow  \quad z (\sigma - 1)c_{K, T} = 0.
\end{equation} 
However, for every $\tilde f \in \exprod^{r_K - 1}_{\Z_p [\cG_K]} P^{\ast}$, the first equality here implies that 
\[ \tilde f (\iota_\ast(z (\sigma - 1)c_{K, T})) = \iota \big ( \iota^\ast(\tilde f) ( z (\sigma - 1)c_{K, T}) \big) =
\iota (z(\sigma-1) \iota^\ast(\tilde f)(c_{K,T})) = 0.\]
Since $P$ is a free $\Z_p [\cG_K]$-module, this implies $\iota_\ast(z (\sigma - 1)c_{K, T}) = 0$. In particular, since $\iota_\ast$ is injective (by \cite[Lem.~C.1]{Sakamoto20}), one has $z (\sigma - 1)c_{K, T} = 0$. This proves the implication in (\ref{second implication we claim}), and hence also that in (\ref{the implication we claim}).} \\
In the remainder of the argument, we shall therefore fix an arbitrary element $f$ of 
\[ \exprod^{r_K - 1}_{\Z_p [\cG_K]} U_{K, S(K), T}^\ast = {\varprojlim}_{n \in \N} \big( \big( \exprod^{r_K - 1}_{\Z [\cG_K]} (\cO^\times_{K, S(K), T})^\ast \big) / p^n\big) \subset {\prod}_{n \in \N} \big( \big( \exprod^{r_K - 1}_{\Z [\cG_K]} (\cO^\times_{K, S(K), T})^\ast \big) / p^n\big)
\]
and then also fix a pre-image $(f_n)_{n \in \N}$ of it under the natural projection 
\[ \bigl(\exprod^{r_K - 1}_{\Z [\cG_K]} (\cO^\times_{K, S(K), T})^\ast\bigr)^\N \to {\prod}_{n \in \N} \big( \big( \exprod^{r_K - 1}_{\Z [\cG_K]} (\cO^\times_{K, S(K), T})^\ast \big) / p^n\big).\]
In addition, since $\Z_p$ is $\Z$-flat, \cite[Lem.\@ B.12]{Sakamoto20} implies bijectivity of the natural map 
\[
\bidual^{r_K}_{\Z_p [\cG_K]} U_{K, S(K), T} \cong {\varprojlim}_{n} \big (\big ( \bidual^{r_K}_{\Z [\cG_K]} \cO_{K, S(K), T}^\times \big) / p^n\big )
\]
and so, just as above, we shall fix a representative $ (c_{K, T}^{(n)})_{n \in \N}$ of $c_{K,T}$ in $\bigl(\bidual^{r_K}_{\Z [\cG_K]} \cO_{K, S(K), T}^\times\bigr)^\N$. 
\\
The element $((\sigma  - 1) f_n (c_{K, T}^{(n)}))_{n \in \N}$ of $\bigl(\cO^\times_{K,S(K),T}\bigr)^\N$ is then a representative of $(\sigma  - 1) f (c_{K, T})$ and so in order to prove that $z (\sigma - 1) f (c_{K, T}) = 0$ it is enough for us to prove the existence of $z \in \N$ such that  
\[ z (\sigma - 1) f_n (c_{K, T}^{(n)}) \in (\cO^\times_{K,S(K),T})^{p^n} \,\,\,\text{for every }\, n\in \N.
\]
The key idea now is to use the explicit criterion provided by Lemma~\ref{hasse principle lemma} below to verify that this condition is satisfied with $z = 2[N:k]\prod_{t \in T_K}|\mathbb{F}_t^\times|$, where $\mathbb{F}_t$ denotes the residue field of each place $t$. To do this, we write $\iota_\mathfrak{Q} \: K \hookrightarrow K_\mathfrak{Q}$ for the embedding associated to a finite place $\mathfrak{Q}$ of $K$. We also write $\cO_{\mathfrak{Q}}$ for the valuation ring of $K_\mathfrak{Q}$ and  $\mathbb{F}_\mathfrak{Q}$ for its residue field $\cO_{\mathfrak{Q}}/(\mathfrak{Q}\cO_{\mathfrak{Q}})$. \\
Then, by Lemma~\ref{hasse principle lemma}, we are reduced to showing that, for each $\q \in \mathscr{Z}_n \textcolor{black}{\setminus S_p (k)}{}$ and each place $\mathfrak{Q}$ of $K$ above $\q$ the image of $\iota_\mathfrak{Q}(f_n (c_{K, T}^{(n)}))$ in $\mathbb{F}_\mathfrak{Q}$ is a $p^n$-th power. Thus, since the element  $f_m (c_{K, T}^{(m)})\cdot f_n (c_{K, T}^{(n)})^{-1}$ belongs to $(\cO_{K,S(K),T}^\times)^{p^n}$ for each $m > n$, it is enough to show that, for any given $m > n$,  the image of $\iota_\mathfrak{Q}(f_m (c_{K, T}^{(m)}))$ in $\mathbb{F}_\mathfrak{Q}$ is a $p^n$-th power. 
\\
To do this, we fix $\q \in \mathscr{Z}_n\setminus S_p(k)$, write $k^{(p)} (\q)$ for the maximal $p$-power extension of $k$ in $k (\q)$ and set
\[ L = L (\q) \coloneqq K k^{(p)} (\q)\,\, \text{ and }\,\, H\coloneqq \gal{L}{K}.\]
We then claim $c_{L,T}=0$. To verify this, 
it is enough (by Lemma~\ref{faithful-lem}) to show $e_\chi c_L =0$ for all $\chi\in \cG^\circ_{L}$ with $e_\chi \cdot e_{L} \neq 0$. In addition, since \textcolor{black}{$\q$ ramifies in $L$ (by Lemma \ref{rubin lemma}\,(ii)) but} splits completely in $N$, the definition of $e_{L}$ implies the kernel of any such $\chi$ cannot  contain $\gal{ L}{ L \cap N}$. Hence, since the explicit assumptions on $c$ imply that $c_{L}$ is fixed by the action of $\gal{ L}{ L \cap N}$, one must have $e_\chi \cdot c_{L}=0$ for all such $\chi$, and hence that $c_{L,T} = 0$. \\
We note next that our choice of $T$ ensures $T \cap S_\ram (L/\Q) = \emptyset$ and thus $T \in \mathscr{P}^{\mathrm{ad}}_L$, and also that $\q$ splits completely in $K$ and is such that every prime of $K$ above it is tamely ramified in $L/K$ (as $\q\notin S_p(k)$). In particular, since $c$ belongs to $\ES_k (\Z_p)^{\mathrm{con}}$, it satisfies the conditions in Definition \ref{con def} with $v=\q$. In 
particular, one has $\Ord_\q^{-1} (c_{K, T}) \in \bidual^{r_K + 1}_{\Z_p [\cG_K]} \cO_{K, S(L), T}^\times$ and so we can fix a representative $(d_{K, T}^{(n)})_{n \in \N}$ of $\Ord_\q^{-1} (c_{K, T})$ in $(\bidual^{r_K + 1}_{\Z_p [\cG_K]} \cO_{K, S(L), T}^\times)^\N$. It follows that $(\Ord_\q ( d_{K, T}^{(n)}))_{n \in \N}$ is a representative for $c_{K, T}$, and hence that $f_m ( c_{K, T}^{(m)}) \equiv f_m ( \Ord_\q ( d_{K, T}^{(m)})) \mod (\cO_{K, S(K), T})^{p^m}$. It is therefore enough to prove that 
the image of $\iota_\mathfrak{Q}(f_m (\Ord_\q ( d_{K, T}^{(m)})))$ in $\mathbb{F}_\mathfrak{Q}$ is a $p^n$-th power.\\
To do this, we observe that the maps
\begin{align*}
\big ( \bidual^{r_K}_{\Z [\cG_K]} \cO^\times_{K, S(K), T} \big)
\otimes_{\Z} ( I_H / I_H^{2}) 
& \to \mathcal{E}(L) \coloneqq  (\bidual^{r_K}_{\Z [\cG_L]} \cO_{L, S(L), T}^\times ) \otimes_{\Z} ( I_H / I_H^{2}) \\
& \to (\bidual^{r_K}_{\Z [\cG_L]} \cO_{L, S(L), T}^\times ) \otimes_{\Z} ( \Z [H] / I_H^{2}),
\end{align*}
the first of which is induced by $\nu_{L / K}$ and second by the inclusion $I_H \subseteq \Z [H]$,
are both injective (by \cite[Lem.\@ 2.11]{Sano}) and one has $P_{L/K,\{v\}} = 1$ (since $S(L) = S(K)\cup \{\q\}$). In particular, since $\mathcal{E}(L)$ is a finite $p$-group (as $I_H/I_H^2 \cong H$ is a finite $p$-group), we can choose $m$ with $m \geq \ord_p (|\mathcal{E}(L)|)$ and then deduce from this case of the congruences in Definition \ref{con def} (and the vanishing of $c_{L,T}$) that the element   
\[ 
\Rec_\q (d_{K, T}^{(m)}) \in \big ( \bidual^{r_K}_{\Z [\cG_K]} \cO^\times_{K, S(K), T} \big)
\otimes_{\Z} ( I_H / I_H^{2})\]
 vanishes. In the group $(I_H /I_H^2) \otimes_\Z \Z [\cG_K]$ one therefore has 
\begin{align*}\label{contain}
 (\Rec'_\q \circ f_m \circ \Ord_\q)  (d_{K, T}^{(m)})
= &\, \pm (\Ord_\q \wedge f_m) \big( \Rec_\q (d_{K, T}^{(m)}) \big) = 0,
\end{align*}
for the map $\Rec'_\q$ in (\ref{rec-prime def}). To interpret this equality explicitly, we write $\rho_H \: \Z [H] \to \Z [H] / I_H^2$ for the natural projection map and 
use the isomorphism
\[
( I_H / I_H^2) \otimes_\Z \Z [G / H] 
\xrightarrow{\simeq} H \otimes_\Z \Z [G / H], \quad \rho_H ( h - 1)\otimes x \mapsto h \otimes x.
\]
In particular, given this isomorphism, and the explicit definition of the map $\Rec'_\q$, the vanishing of $\Rec'_\q ( f_m ( \Ord_\q (d_{K, T}^{(m)})))$ implies that, for some fixed place $\mathfrak{Q}_1$ of $K$ above $\q$ and every $\sigma \in \cG_K$, the element $\iota_{\mathfrak{Q}_1}(\sigma f_m ( \Ord_\q (d_{K, T}^{(m)})))$ belongs to the kernel of the local reciprocity map $\rec_{\mathfrak{Q}_1}$. It follows that $\iota_\mathfrak{Q}(f_m ( \Ord_\q (d_{K, T}^{(m)})))\in \ker(\rec_{\mathfrak{Q}})$ for every place $\mathfrak{Q}$ of $K$ above $\q$ and so, to complete the proof, it is enough for us to show that these containments imply that the image of each $\iota_\mathfrak{Q}(f_m ( \Ord_\q (d_{K, T}^{(m)})))$ in $\mathbb{F}_\mathfrak{Q}^\times$ is a $p^n$-th power. \\
To do this, we fix a place $\mathfrak{Q}$ of $K$ above $\q$ and a place $\mathfrak{Q}'$ of $L$ above $\mathfrak{Q}$. We write $\mathcal{I}$ for the inertia subgroup of $\gal{L_{\mathfrak{Q}'}}{K_\mathfrak{Q}}$ and recall (from local class field theory) that $\rec_{\mathfrak{Q}}$ induces an isomorphism $\cO_{K_\mathfrak{Q}}^\times/\NN_{L_{\mathfrak{Q}'} / K_\mathfrak{Q}} ( \cO_{L_{\mathfrak{Q}'}}^\times ) \cong \mathcal{I}$. Then, since $\mathcal{I}$ is a $p$-group and the residue characteristic of $\mathfrak{Q}$ is prime-to-$p$, this isomorphism factors through 
the cyclic quotient $\mathbb{F}_{\mathfrak{Q}}^\times \otimes_\Z \Z_p$ of $\cO_{K_\mathfrak{Q}}^\times/\NN_{L_{\mathfrak{Q}'} / K_\mathfrak{Q}} ( \cO_{L_{\mathfrak{Q}'}}^\times )$, and so induces an isomorphism $\mathbb{F}_\mathfrak{Q}^\times \otimes_\Z \big ( \faktor{\Z}{|\mathcal{I}| \Z} \big)  \cong \mathcal{I}$. In particular, since $\iota_\mathfrak{Q}(f_m ( \Ord_\q (d_{K, T}^{(m)}))) \in \cO_{K_\mathfrak{Q}}^\times$, the vanishing of $\rec_{\mathfrak{Q}}(\iota_\mathfrak{Q}(f_m ( \Ord_\q (d_{K, T}^{(m)}))))$  implies the image of $\iota_\mathfrak{Q}(f_m ( \Ord_\q (d_{K, T}^{(m)})))$ in $\mathbb{F}_\mathfrak{Q}^\times$ is an $|\mathcal{I}|$-th power. It is therefore now enough to note that, since $\q\in \mathscr{Z}_n$, Lemma~\ref{rubin lemma}\,(ii) implies that $p^n$ divides $[k ( \q) : k (1)]$ and hence also $|\mathcal{I}|$.
\end{proof} 

The next result was required in the above proof (and uses notation introduced in the proof). 

\begin{lem} \label{hasse principle lemma}
Fix $a\in \cO_{K, S(K)}^\times$ and assume that, for every $\q \in \mathscr{Z}_n \textcolor{black}{\setminus S_p (k)}$ and every $\q$-adic place $\mathfrak{Q}$ of $K$, the image of $\iota_\mathfrak{Q}(a)$ in $\mathbb{F}_\mathfrak{Q}$ is a $p^n$-th power. Then, for every $\sigma \in \cG_K$ the element $(2 [N : k]\prod_{t \in T_K}|\mathbb{F}_t^\times|)(\sigma  - 1) a$ belongs to $ (\cO^\times_{K,S(K),T})^{p^n}$.
\end{lem}

\begin{proof} One has $\cO^\times_{K,S(K)}\cap (K^\times)^{p^n} = (\cO^\times_{K,S(K)})^{p^n}$ and $(\prod_{t \in T_K}|\mathbb{F}_t^\times|)\cdot (\cO^\times_{K,S(K)})^{p^n}\subseteq (\cO^\times_{K,S(K),T})^{p^n}$. It is therefore enough to show, under the stated conditions, that   
\[ 2 [N : k](\sigma  - 1) a \in (K^\times)^{p^n} \,\,\,\text{ for every } \, \sigma \in \cG_K.\]  
To do this, we fix a $\q$-adic place $\mathfrak{Q}$ of $K$ and write $\overline{a}$ for the class of $\iota_\mathfrak{Q} (a)$ in $\mathbb{F}_\mathfrak{Q}$. Then, by assumption, the polynomial $X^{p^n} - \overline{a} \in \mathbb{F}_\mathfrak{Q} [X]$ has a root in $\mathbb{F}_\mathfrak{Q}$. 
Thus, since the characteristic of $\mathbb{F}_\mathfrak{Q}$ is not equal to $p$, Hensel's Lemma implies  $X^{p^n} - a$ has a root in $\cO_{\mathfrak{Q}}$ and so  $\iota_\mathfrak{Q}(a)\in (K^\times_\mathfrak{Q})^{p^n}$. Since $\q$ (and $\mathfrak{Q}$) are arbitrary, this implies $a$ belongs to the kernel of the diagonal map 
\[
 \Delta \: K^\times/(K^\times)^{p^n}\to {\prod}_{\mathfrak{Q} \in (\mathscr{Z}_n)_K \textcolor{black}{\setminus S_p (K)}} K_\mathfrak{Q}^\times/(K_\mathfrak{Q}^\times)^{p^n}.
\]
We now claim that if $b\in K^\times$ represents an element of $\ker(\Delta)$, then $b$ is a $p^n$-th power in $K(n)^\times$. To prove this, we note that every place of $K(n)$ that lies above a place in $\mathscr{Z}_n \textcolor{black}{\setminus S_p (k)}$, and thus \textcolor{black}{all but finitely many of the places} of $k$ that splits completely in $K(n)$, also split completely in $K(n) (\sqrt[p^n]{b})$. Hence, since each Galois extension of $K$ is uniquely determined by \textcolor{black}{any set of all but finitely many of the places that split completely in it (cf.\@ \cite[Th.~13.9]{Neukirch92})}, one has $K(n)(\sqrt[p^n]{b}) = K(n)$, and thus $b \in (K(n)^\times)^{p^n}$, as claimed. \\
At this stage, we know the class of $a$ in $K^\times/(K^\times)^{p^n}$ belongs to the kernel of the natural map 
\[\Delta' \: K^\times / (K^\times)^{p^n} \to K(n)^\times / ( K(n)^\times)^{p^n}\]
and so  are reduced to showing that every element of $\ker(\Delta')$ is annihilated by $2 [N : k] (\sigma - 1)$ for each $\sigma \in \cG_K$. To do this, we note that the inflation-restriction sequence identifies $\ker(\Delta')$ with $H^1 (\gal{K(n)}{K}, \mu_{p^n})$. We then set $K_n \coloneqq K (\mu_{p^n})$ and recall that a result \cite[Satz (4.8)]{Neukirch73} of Neukirch implies 
\[
H^1 ( \gal{N K_n}{N K}, \mu_{p^n}) =
\begin{cases} 
H^1 ( \gal{N K_2}{N K}, \mu_{4}) & \text{ if } p = 2 ,\\
0 \quad & \text{ if } p \not= 2,\\
\end{cases}
\]
and hence that the exponent of this group divides $2$. From the inflation-restriction sequence 
\begin{cdiagram}[column sep=small]
0 \arrow{r} & H^1 ( \gal{N K}{K}, \mu_{p^n} \cap NK ) 
\arrow{r} & H^1 ( \gal{N K_n}{K}, \mu_{p^n} ) \arrow{r} & 
H^1 (\gal{N K_n}{NK}, \mu_{p^n} ),
\end{cdiagram}%
we can therefore deduce that the exponent of $H^1 ( \gal{N K_n}{K}, \mu_{p^n} )$ divides $2 [NK : K]$, and hence also $2 [N : k]$. 
 Taking account of the inflation-restriction sequence
\begin{cdiagram}[column sep=small]
    0 \arrow{r} & H^1  ( \gal{N K_n}{K}, \mu_{p^n} ) 
    \arrow{r} & H^1 (\gal{K(n)}{K}, \mu_{p^n} ) \arrow{r} & H^1 ( \gal{K(n)}{N K_n}, \mu_{p^n})
\end{cdiagram}%
our proof of the claimed result is therefore reduced to showing that the natural $\cG_{NK_n}$-action on $H^1 ( \gal{K(n)}{N K_n}, \mu_{p^n})$ is trivial. To this end, we note that the perfect Kummer pairing combines with the definition of the field $K(n)$ to give a $G_k$-equivariant isomorphism 
\begin{align*}
 H^1 ( \gal{K(n)}{N K_n}, \mu_{p^n})
= \Hom ( \gal{K(n)}{N K_n}, \mu_{p^n})
& \cong \big ( \faktor{\cO_k^\times \cdot  (N K_n^\times)^{p^n}}{(N K_n^\times)^{p^n}} \big)\\
& \cong  \faktor{\cO_k^\times}{(\cO_k^\times \cap (N K_n^\times)^{p^n})}.
\end{align*}
Hence, since elements in the latter quotient module are fixed by the natural action of $\cG_{NK_n}$, the claimed result follows. 
\end{proof}

\subsubsection{A reduction to the Scarcity Conjecture}\label{reduction result}

In this subsection we clarify the extent to which global properties of  Euler systems can be determined by analysis of their individual components. 

In particular, in claim (a) of the following result we show that, up to multiplication by isolated systems, all Euler systems are determined by their components at certain sparse families of fields, and in claim (b) that Conjecture \ref{scarcity-conjecture} is equivalent to explicit predictions regarding individual components. We shall then provide concrete evidence in support of the predictions in (b)\,(ii) of this result, before providing further evidence in support of the predictions in (b)\,(iii) in the next subsection.  

\begin{thm}\label{cons scarcity thm}
Fix a finite subset $\cS$ of $\mathbb{P}$ 
and  a subset $\cX\subseteq \Omega_k$ that  satisfies Hypothesis \ref{hyp X}.
Then the following claims are valid. 
\begin{liste}
\item Fix a rank function $\bm{r}$ for $k$ and a finite subset $\mathcal{M}$ of $S_\fin(k)$. Then each $c\in  \ES_k^{\bm{r}, \cX} (\Z_\mathcal{S})$ is determined, up to multiplication by isolated systems, by the set $\{c_E: E\in \cX, \, \mathcal{M}\subseteq S(E)\}$.
\item Fix a non-empty subset $V$ of $S_\infty(k)$  with $V = S_\infty (k)$ if $2 \notin \cS$.  Assume   $\cX\subseteq \Omega_k^V$ and also that $\varepsilon^\cX_{k}\in \ES_k^\cX ( \Z_\mathcal{S})^{\mathrm{con}}$.
Then, for every $c\in \ES_k^\cX (\Z_\cS)^\mathrm{sym} \cap \ES_k^\cX (\Z_\cS)^\mathrm{con}$, the following assertions are equivalent:
\begin{enumerate}[label=(\roman*)]
    \item $c \in  \Z_\mathcal{S} \llbracket G_k^{\mathrm{ab}} \rrbracket\cdot \varepsilon_k^\cX$ (as predicted by Conjecture \ref{scarcity-conjecture}). 
    \item $c_K \in \Z_\mathcal{S} [\cG_K]\cdot \varepsilon_{K / k}$ for  every $K\in \cX$.
    \item $c_K\in \Z_p [\cG_K]\cdot \varepsilon_{K / k}$ for every $p \in \mathbb{P}\setminus \mathcal{S}$ and every $K\in \cX_{S_p}$. 
\end{enumerate}
\end{liste}
\end{thm}

\begin{proof} To prove claim (a), we use the triangular system $\{ (N_{F, E}, \rho_{F / E}) \}_{E, F \in \cX}$ described  in Example~\ref{triangular-system-examples}\,(b) with  $\cR$ taken to be $\Z_\mathcal{S}$. For this system, Example~\ref{es remark}\,(a) implies that $\ES^{{\bm r},\cX}_k(\Z_\cS)$
\textcolor{black}{is equal to the subset of 
$\eullim{F \in \cX}{\emptyset} N_F$ comprising all elements $(n_F)_F$ for which one has $e_{\bm{1}_F} n_F = 0$ for every $F \in \cX \cap \Omega_k^\mathrm{exc}$}. From Lemma~\ref{invariant systems description}, we can therefore deduce that  
\[ 
\ES^{{\bm r},\cX}_k(\Z_\cS)^{\mathrm{iso}}
= \ES^{{\bm r},\cX}_k(\Z_\cS)^{G_k^{\mathrm{ab}}} 
= 
(\eullim{F \in \cX}{\emptyset} N_F)^{G_k^{\mathrm{ab}}}
\textcolor{black}{ \cap \ES^{{\bm r},\cX}_k(\Z_\cS)}
.\] 
In addition, each diagonal term  $N_F = \Z_\mathcal{S} \fL^{\bm{r}(F)}_F$ is a $\Z_\mathcal{S} [\cG_F]$-lattice and $\{ (N_{F, E}, \rho_{F / E}) \}_{E, F \in \cX}$ validates Hypothesis \ref{inj-hypothesis} with respect to every $p \in \mathbb{P}\setminus \cS$ (cf.\@ Example~\ref{inj examples}\,(c)). Given these facts, claim (a) therefore follows directly from the conclusion of Corollary~\ref{global restriction sequence} for the system  $\{ (N_{F, E}, \rho_{F / E}) \}_{E, F \in \cX}$ and with $\Pi= \emptyset$ and $\Sigma= \cM$. \\
In claim (b), the implications from (i) to (ii) and from (ii) to (iii) are obvious and so it suffices to prove that (iii) implies (i). To do this we fix a system $c\in \ES_k^\cX (\Z_\cS)^\mathrm{sym} \cap \ES_k^\cX (\Z_\cS)^\mathrm{con}$ satisfying the condition in (iii) and prove $c \in  \Z_\mathcal{S} \llbracket G_k^{\mathrm{ab}} \rrbracket\cdot \varepsilon_k^\cX$. 

As a first step, we prove the existence for each $p \in \mathbb{P}\setminus \cS$ of an element $q_p$ of $\Z_p \llbracket G_k^{\mathrm{ab}} \rrbracket$ such that %
\begin{equation}\label{hilflem} \varrho^{\cX_{S_p}} (c) = \varrho^{\cX_{S_p}} ( q_p \cdot \varepsilon_{k}).\end{equation}
For this, we use condition (iii) to fix, for $K \in \cX_{S_p}$ an element $q_K^{(p)}$ of $\Z_p [\cG_K]$ with $c_K = q_K^{(p)} \cdot \varepsilon_{K / k}$. We then claim that $\tilde q_p \coloneqq (q_K^{(p)}e_K)_{K \in \cX_{S_p}}$ belongs to the Euler limit $\eullim{K \in \cX_{S_p}}{\emptyset}  \Z_p [\cG_K]e_K$ of the triangular system associated to the projective system $(\ZZ_p[\cG_L],\pi_{L/K})_{L\in \cX_{S_p}}$ via the construction of Example~\ref{es remark}\,(d). To verify this, we need to show that $P_{L / K, \emptyset}( \pi_{L / K} ( q_L^{(p)}e_L)  -q_K^{(p)}e_K) = 0$ for all $K$ and $L$ in $\cX_{S_p}$ with $K \subseteq L$.  But, since $\varepsilon_{K / k}$ generates a free $ \Z_p [\cG_K] e_K$-module of rank $1$ (by  (\ref{vanishing equalities})), this follows directly from the calculation 
\begin{align*}
P_{L / K, \emptyset}\! \cdot\! \pi_{L / K} ( q_L^{(p)}e_L)\!\cdot\!\varepsilon_{K / k}= 
 \pi_{L / K} ( q_L^{(p)}e_L)\!\cdot\! \NN^{r_L}_{L / K} (\varepsilon_{L / k}) &= \NN^{r_L}_{L / K} ( c_L)\\
&= P_{L / K, \emptyset}\! \cdot\! c_K = P_{L / K, \emptyset} \!\cdot\! q_K^{(p)}e_K\! \cdot\! \varepsilon_{K / k},
\end{align*}
in which the second and last equalities are valid as $\varepsilon_{E/ k} = e_E\cdot\varepsilon_{E / k}$ for $E \in \{L,K\}$ (by (\ref{vanishing equalities})). \\
After recalling Example~\ref{rigidity-hyp-example-group-rings}, we can therefore now apply Proposition~\ref{p-adic-Euler-limit} to the projective system $(\ZZ_p[\cG_L],\pi_{L/K})_{L\in \cX_{S_p}}$ to deduce the existence of $q'_p = (q'_{p,K})_K\in \varprojlim_{K \in \cX_{S_p}} \epsilon_{K,V}\Z_p [\cG_K]$ with $q'_{p,K}e_K = q_K^{(p)}e_K$ for all $K \in \cX_{S_p}$. Now, for every $K$ one has $\epsilon_{K, V} \in \Z_p [\cG_K]$ (this uses the assumption that if $ p=2\notin \cS$ , then  $V = S_\infty (k)$ and hence $\epsilon_{K, V} = 1$) and so $ \epsilon_{K,V}\Z_p [\cG_K]$ is a $\Z_p [\cG_K]$-module direct summand of $\Z_p [\cG_K]$. It follows that there exists an element   $q_p$ of  $\Z_p \llbracket G_k^{\mathrm{ab}} \rrbracket$ that is sent by the natural projection $\Z_p \llbracket G_k^{\mathrm{ab}} \rrbracket \to \varprojlim_{K \in \cX_{S_p}}\Z_p [\cG_K]$ to  $q_p'$ and it is clear that any such  $q_p$ satisfies the  claimed equality (\ref{hilflem}).

We now consider the triangular system $\{ (N_{F, E}, \rho_{F / E}) \}_{E, F \in \cX}$ in Example~\ref{triangular-system-examples}\,(b) with  $\cR=\Z_p$. Then, by the same argument as used in the proof of claim (a), the equality (\ref{hilflem}) can be reinterpreted as asserting that $c - q_p\cdot\varepsilon_{k}^\cX$ belongs to the kernel of the  restriction map 
\[
\res_{S_p} \: \eullim{E \in \cX}{\emptyset} N_E \to \eullim{E \in \cX_{S_p}}{\emptyset} N_E.
\]
Hence, since this triangular system satisfies Hypothesis \ref{inj-hypothesis} for $p$ (by Example~\ref{inj examples}\,(c)), we can apply Theorem~\ref{p adic restriction sequence} (with $\Pi = \emptyset$ and $\Sigma = S_p (k)$) to it to conclude, for each $p\not=2$, that 
\begin{equation} c - q_p\cdot\varepsilon_{k}^\cX \in \bigl(\eullim{E \in \cX}{\emptyset} N_E\bigr)^{\gal{k^{\mathrm{ab}}}{k_\emptyset \langle p \rangle}}\end{equation}
where $k_\emptyset \langle p \rangle$ is the field defined in (\ref{k_{p} def}). In particular, if $p \nmid 2 d_k$, then $k \langle p \rangle = k (\mu_p)$ is totally ramified at every place in $S_p(k)$  so that $k_\emptyset \langle p \rangle = k$ and hence $c - q_p\cdot \varepsilon_{k}^\cX$ is fixed by $G_k^{\mathrm{ab}}$. 

The same result is also true for any $p\in \mathbb{P}\setminus \cS$ that divides $2d_k$. Indeed, if $p\not=2$, respectively if $p=2$, then Theorem~\ref{p adic restriction sequence}, respectively  Proposition~\ref{p adic restriction intermediate step} and the fact any subextension of $k (2, \cT)$ that is unramified at a place in $S_2(k)$ must be of finite degree over  $k$, implies that $c - q_p\cdot \varepsilon_{k}^\cX$ is fixed by an open subgroup of $G_k^{\mathrm{ab}}$. In addition, for every $E \in \cX$ one has $V_E = V\not=\emptyset$ and hence $r_E > 0$ (cf.\@ (\ref{r_K def})). Hence, since both $c$ and $\varepsilon_{k}$, and thus also $c - q_p\cdot \varepsilon_{k}^\cX$, are assumed to belong to $\ES_k^\cX (\Z_\cS)^\mathrm{con}$, in both cases we can deduce from Proposition~\ref{mrs and soogils argument} that $c - q_p\cdot\varepsilon_{k}^\cX$ is fixed by $G_k^{\mathrm{ab}}$, as claimed. \\
In particular, since we are assuming $c\in  \ES_k^\cX (\Z_\cS)^\mathrm{sym}$ and $\varepsilon_{k}^\cX\in  \ES_k^\cX ( \Z_\mathcal{S})$, and hence also $\varepsilon_{k}^\cX\in \ES_k^\cX (\Z_\cS)^\mathrm{sym}$ (by Lemma~\ref{RS-properties-lemma}\,(i)) we now know that, for every $p\in \mathbb{P}\setminus \cS$, one has  
\[ c - q_p\cdot \varepsilon_{k}^\cX \in \ES_k^\cX (\Z_\cS)^\mathrm{sym} \cap \ES_k^\cX (\Z_\cS)^{G_k^{\mathrm{ab}}}.\]
From Lemmas \ref{ivc versus local} and \ref{invariant systems description}, we can therefore conclude that 
$c = q_p\cdot \varepsilon_{k}^\cX$ and hence also  
\begin{equation}\label{hilflem 2} c\in \Z_p \llbracket G_k^{\mathrm{ab}} \rrbracket \cdot \varepsilon_{k}^\cX \quad \text{for every}\,\, p \in \mathbb{P}\setminus \cS.\end{equation}
To complete the proof, we now use the projective system $\{ (\Z_\mathcal{S} [\cG_F], \pi_{F / E})\}_{F \in \cX}$  and triangular system $\{ (N_{F, E}, \rho_{F / E}) \}_{E, F \in \cX}$ from Example~\ref{triangular-system-examples}\,(b) with  $\cR=\Z_\cS$ (so that $N_E = \Z_\mathcal{S} \fL^{r_E}_E$ and $\rho_{F/E}$ is induced by the norm map).  We note that the argument proving Theorem~\ref{coleman-over-the-reals}\,(b) (with $\bm{r} = \bm{r}_{\mathrm{max}}$) shows that, for each $c\in \ES_k^\cX ( \Z_\mathcal{S})$ one has $c_E \in \Q [\cG_E]\cdot\varepsilon_{E / k}$ for every $E$ in $\cX$, and also that the resulting assignment 
\begin{cdiagram}
\beta \: \eullim{E \in \cX}{\emptyset} e_{E} \Z_\mathcal{S} [\cG_E] \to \ES_k^\cX ( \Z_\mathcal{S})\subseteq \eullim{E \in \cX}{\emptyset}\bigl(\Z_\mathcal{S} \fL^{r_E}_E \cap \QQ[\cG_E]\cdot \varepsilon_{E/k}\bigr), \quad
(q_E)_E \mapsto (q_E\cdot\varepsilon_{E / k})_E.
\end{cdiagram}%
is an injective map of $\Z_\cS\llbracket G_k^{\mathrm{ab}}\rrbracket$-modules of the form that occurs in (\ref{linking hom}). In particular, since $\{ (\Z_\mathcal{S} [\cG_F], \pi_{F / E})\}_{F \in \cX}$ satisfies Hypothesis \ref{rigidity hyp} (by Example~\ref{rigidity-hyp-example-group-rings}),  we may apply Theorem~\ref{rigidity theorem} to this choice of data. 

Now, in terms of the notation of the latter result, in this setting one has $\im(\beta_\ast) = \Z_\mathcal{S} \llbracket G_k^{\mathrm{ab}} \rrbracket\cdot \varepsilon_k^\cX$ and also $\Q\otimes_\Z\beta_{\ast,F}(\Z_\mathcal{S} [\cG_F]) = \Q[\cG_F]\cdot \varepsilon_{F/k}$ for every $F \in \cX$. \\
In addition, since $\pi_{F/E}(\epsilon_{F,V}\Z_\cS[\cG_F]) = \epsilon_{E,V}\Z_\cS[\cG_E]$ for all $E$ and $F$ in $\cX$ with $E \subseteq F$, one has ${\varprojlim}^1_{F\in \cX} (\epsilon_{F,V}\Z_\cS[\cG_F]) = (0)$ by the Mittag--Leffler condition. This implies  $\theta_p$ is the zero map for every $p \in \mathbb{P}\setminus \cS$, and hence that $\beta_\ast^{(p)} (\ker (\theta_p)) = \Z_p \llbracket G_k^{\mathrm{ab}} \rrbracket \cdot \varepsilon_{k}^\cX$. \\
In view of these observations, the containments (\ref{hilflem 2}) can be combined with this case of the exact sequence of Theorem~\ref{rigidity theorem} to deduce that $c = q'\cdot \varepsilon_k^\cX$ for some $q'\in \varprojlim_{E \in \cX}\epsilon_{E,V}\Z_\mathcal{S} [\cG_E]$. In addition, since the explicit assumption on $V$ implies $2 \in \Z_\cS^\times$ if $V\not= S_\infty(k)$, and hence that, in all cases, each $ \epsilon_{K,V}\Z_\cS [\cG_K]$ is a $\Z_\cS [\cG_K]$-module direct summand of $\Z_\cS  [\cG_K]$, there exists an element $q$ of  $\Z_\cS \llbracket G_k^{\mathrm{ab}} \rrbracket$ that is sent by the natural projection $\Z_\cS\llbracket G_k^{\mathrm{ab}} \rrbracket \to \varprojlim_{K \in \cX}\Z_\cS [\cG_K]$ to  $q'$. It is then clear that 
\[ c = q'\cdot \varepsilon_k^\cX= q\cdot \varepsilon_k^\cX\in \Z_\cS \llbracket G_k^{\mathrm{ab}} \rrbracket\cdot \varepsilon_k^\cX\]
and hence that $c$ satisfies condition (i) of claim \,(b).\\
This therefore completes the proof of Theorem~\ref{cons scarcity thm}.  
\end{proof}

\begin{rk} \label{hilfslem p adic euler limit}
Fix $p \in \mathbb{P}$ and  $\cX\subseteq \Omega_k$ satisfying Hypothesis \ref{hyp X}. Then a closer reading of the argument used to prove that condition (i) of Theorem~\ref{cons scarcity thm}\,(b) is a consequence of condition (iii) shows it also has the following consequence. For any $c\in \ES^\cX_k (\Z_p)$, one has  
\[ \varrho^{\cX_{S_p}} (c)\in \Z_p \llbracket G_k^{\mathrm{ab}} \rrbracket\!\cdot\!\varepsilon_{k}^{\cX_{S_p}} \quad \Longleftrightarrow \quad c_K\in \Z_p [\cG_K]\!\cdot\! \varepsilon_{K / k} \text{ for every } K\in \cX_{S_p}\]
and, if these equivalent conditions are satisfied, then $c \in \Z_p \llbracket G_k^\mathrm{ab} \rrbracket \cdot \varepsilon_k^\cX + \ES^\cX_k ( \Z_p)^{\gal{k^\mathrm{ab}}{k_\emptyset \langle p \rangle}}$, where $k_\emptyset \langle p \rangle$ is the finite extension of $k$ defined in (\ref{k_{p} def}). 
\end{rk}

To end this subsection, we now combine observations from \S\,\ref{acnf section} with the theory of higher-rank Kolyvagin systems to give evidence in support of the containment in Theorem~\ref{cons scarcity thm}\,(b)\,(ii). 
\smallskip

\begin{thm} \label{bdss-result}
 Fix $p\in \mathbb{P}\setminus\{2\}$ and $E\in \Omega_k^{S_\infty(k)}$ and set $r \coloneqq |S_\infty(k)|$. Write $H$ for the maximal subgroup of $\cG_E$ of order prime-to-$p$, fix a homomorphism $\chi \: H \to \overline{\Q_p}^\times$ and set $\mathcal{O}_\chi \coloneqq \Z_p[\im(\chi)]$. Assume the Rubin--Stark Conjecture to be valid for all abelian extensions of $k$ and, in addition, that all of the following conditions are satisfied:
\begin{itemize}
\item[$\bullet$] $E$ contains the Hilbert $p$-class field of $k$;
\item[$\bullet$] neither $\chi$ nor, if $p=3$, $\chi^2$ is equal to the $p$-adic Teichm\"uller character of $k$;
\item[$\bullet$] $\ker(\chi)$ contains neither the decomposition subgroup in $H$ of any place of $E^H$ that lies above a place in $S(E^H)$ nor the inertia subgroup in $H$ of at least one place of $E^H$ that lies above a place in $S(E)$. 
\end{itemize}
Then for every $c\in \ES_k (\Z)$ one has $e_\chi(1\otimes c_E) \in \mathcal{O}_\chi[\cG_E]\cdot (1\otimes \varepsilon_{E / k}),$ where $e_\chi\in \mathcal{O}_\chi[H]\subseteq \mathcal{O}_\chi[\cG_E]$ acts on $1\otimes c_E\in \mathcal{O}_\chi\otimes_{\Z}\fL_E^{r}$ in the natural way. 
\end{thm}

\begin{proof} Set $L_\chi\coloneqq L^{\ker(\chi)}$. Then, after taking account of the Euler system distribution relations (for both $c$ and $\varepsilon_k$), it is enough to prove the claimed containment after replacing $E$ by $L_\chi$ (and hence $E$ by the compositum $L_\chi E^H$). In this way, one can assume $\chi$ is a faithful character of $H$. Then, in this case, the stated hypotheses imply that all hypotheses of \cite[Th.\@ 4.1]{bdss} are satisfied and so the latter result 
(which relies heavily on the theory of higher-rank Kolyvagin systems) implies that, for each $c\in \ES_k (\Z)$, one has 
\[ e_\chi(1\otimes c_E) \in e_\chi\bigl(\cO_\chi\otimes_\Z\im (\Theta^r_{E / k, S^\ast (E),\emptyset})\bigr),\]
where $\Theta^r_{E / k, S^\ast (E),\emptyset}$ is the map defined in (\ref{projection-map}). \\
In addition, by combining this containment in the case $c = \varepsilon_k$ together with the argument of 
Proposition~\ref{analztic-class-number-formula-argument} one finds that 
\[ e_\chi\bigl(\mathcal{O}_\chi\otimes_\Z\im (\Theta^r_{E / k, S^\ast (E),\emptyset})\bigr) = e_\chi\bigl(\mathcal{O}_\chi[\cG_E]\cdot (1\otimes \varepsilon_{E / k})\bigr) \subseteq \mathcal{O}_\chi[\cG_E]\cdot (1\otimes \varepsilon_{E / k}).\]
The claimed result now follows directly upon comparing this inclusion with the previous containment.
\end{proof}

\section{Iwasawa theory}\label{iwasawa theory section}

In this section, we regard $k$ as fixed and abbreviate $S_\infty(k)$ and $S_p(k)$ for $p \in \mathbb{P}$ to $S_\infty$ and $S_p$. \\
The essential feature  of condition (iii) in Theorem~\ref{cons scarcity thm}\,(b) is then that it only concerns fields $K \in \Omega_k$ with $S_p \subseteq S(K)$. This will allow us to use Iwasawa-theoretic techniques to explicitly reinterpret the condition and then also to verify it in several  important cases. In particular, in this way, we shall obtain proofs for all of the results stated in the Introduction.

\subsection{An Iwasawa-theoretic reduction}\label{iwasawa theory reduction section}

\subsubsection{Statement of the main results}\label{somr section}
For $p\in \mathbb{P}$ we fix an extension $ k_\infty^p$ of $k$ in $k^{\mathrm{ab}}$ for which $\cG_{k^p_\infty}$ is topologically isomorphic to $\Z_p^d$ for some $d \in \N$  and the set of places of $k$ that split completely in $k^p_\infty$ is equal to $S_\infty(k)$. (For example, one can simply take $k^p_\infty$ to be the cyclotomic $\Z_p$-extension of $k$.) 

For any finite abelian extension $E$ of $k$, we then set 
\[  E^p_\infty \coloneqq Ek^p_\infty \quad
\text{ and } \quad \bLambda_{p,E} \coloneqq \Z_p \llbracket \gal{E^p_\infty}{k} \rrbracket\]
and, for any finite set $\Sigma$ of places of $k$ and finite set $T\subset S_\fin(k)\setminus \Sigma$, write $\Cl_{E,\Sigma, T}$ for the $(\Sigma_E, T_E)$-ray class group of $E$ (cf.\@ Lemma~\ref{properties-weil-etale}\,(a)). For any abelian extension $E$ of $k$, we  set
\begin{align*}
U_{E, \Sigma, T} & \coloneqq  {\varprojlim}_F (\Z_p \otimes_\Z \cO^\times_{F, \Sigma, T}),
&  \Cl^p_{E,\Sigma, T}   &\coloneqq {\varprojlim}_F (\Z_p \otimes_\Z \Cl_{F,\Sigma, T}), \\
X^p_{E, \Sigma}  & \coloneqq {\varprojlim}_F (\Z_p \otimes_\Z X_{F, \Sigma}), &
Y^p_{E, \Sigma}  &\coloneqq {\varprojlim}_F (\Z_p \otimes_\Z Y_{F, \Sigma}).
\end{align*}
Here each limit is taken over all finite extensions $F$ of $k$ in $E$  and the respective transition morphisms are induced in the upper cases by the norm maps $\NN_{F'/F}$ and the lower cases by the restriction-of-places maps.

We note that if $p \not \in \cS$, $T\cap S_p = \emptyset$ and $T \in \mathscr{P}_E^{\mathrm{ad}}$ (cf.\@ \S\,\ref{rubin lattice section}), then every  $c\in \ES_k (\Z_\mathcal{S})$ gives rise to a norm-coherent sequence
\[
c_{E_\infty^p, T} \coloneqq (\delta_T \cdot c_F)_F \in 
{\varprojlim}_F \bidual^{r_E}_{\Z_p [\cG_F]} U_{F, S (E^p_\infty), T} 
\cong \bidual^{r_E}_{\bLambda_{p, E}} U_{E^p_\infty, S (E^p_\infty), T}.
\]
Here the limit is over $\{ F \in \Omega_k: F \subset E^p_\infty\}$ and with respect to the norm maps $\NN^{r_E}_{F' / F}$, and the isomorphism is a consequence of the general observation \cite[Lem.\@ B.15]{Sakamoto20} of Sakamoto. In addition, each $\eta\in \bidual^{r_E}_{\bLambda_{p, E}} U_{E^p_\infty, S (E^p_\infty), T}$ is (by definition) a map $\exprod^{r_E}_{\bLambda_{p, E}} U_{E^p_\infty, S(E^p_\infty), T}^\ast \to \bLambda_{p, E}$ and so gives rise to an ideal $\im (\eta)$ of $\bLambda_{p, E}$.\\
The following result is then the main observation that we shall make in \S\,\ref{iwasawa theory section}. It describes  a variety of explicit conditions that are sufficient to ensure either the validity of the Tamagawa Number Conjecture or that a given Euler system validates the Scarcity Conjecture. In several important cases, these conditions will then be verified in \S\,\ref{proof of B section} and \S\,\ref{rank one section}.\\
To state the result, we identify the `reflexive hull' $I^{\ast \ast}$ of an ideal $I$ of $\bLambda_{p, E}$ with its image under the natural evaluation map $I^{\ast \ast} \subseteq (\bLambda_{p, E})^{\ast \ast} \cong \bLambda_{p, E}$.\\
We shall also use the subset $\cX_{S_p}$ of $\cX$ defined in (\ref{cX_Sigma def}) and the field $k_\emptyset \langle p \rangle$ defined in (\ref{k_{p} def}), \textcolor{black}{and we remark that for each  $K \in \cX_{S_p}$ one has $S (K) = S(K_\infty^p)$.}

\begin{thm} \label{reduction to Iwasawa theory} We assume to be given data of the following sort. 
\begin{enumerate}[label=$\bullet$]
    \item A subset $V\subseteq S_\infty$ and a subset $\cX\subseteq \Omega_k$ satisfying Hypothesis \ref{hyp X} with $\cV = \{ V \}$.
    \item A finite subset $\mathcal{S}$ of $\mathbb{P}$ with $2 \in \cS$ if $V \neq S_\infty $.
    \item For every $p \in \mathbb{P}\setminus \mathcal{S}$, an extension $k^p_\infty$ of $k$ as specified at the beginning of this section. 
\end{enumerate}
Then, for every $K\in \mathcal{X}$ and  $p\in \mathbb{P}\setminus \mathcal{S}$, the idempotent $\epsilon_K$ identifies with an element of  $\bLambda_{p,K}$, and we assume the above data satisfies the following three hypotheses:
    \begin{itemize}
        \item[($\mathrm{H}_1$)] The Rubin--Stark system $\varepsilon_k^\cX$ is $\Z_\cS$-integral in the sense of Definition \ref{definition-integral-Euler-systems}. 
        \item[($\mathrm{H}_2$)] For every $p\in \mathbb{P}\setminus \mathcal{S}$, every $K\in \cX_{S_p}$ and every height-one prime $\p$ of $\bLambda_{p,K}$ that lies in the support of $\epsilon_K \bLambda_{p,K}$, the $ (\bLambda_{p,K})_\p$-module $(\Cl^p_{K^p_\infty})_\p$ has finite projective dimension.
        \item[($\mathrm{H}_3$)]  For every $p\in \mathbb{P}\setminus \cS$ and every $K\in \cX_{S_p}$, there exists $T\in \mathscr{P}^\mathrm{ad}_K$ such that %
            \[
            \im (\varepsilon_{K^p_\infty, T})^{\ast \ast} 
        \subseteq \Fitt^0_{\bLambda_{p,K}} ( \Cl^p_{K^p_\infty, S (K), T})^{\ast \ast} 
        \cdot \Fitt^0_{\bLambda_{p,K}} (X^p_{K^p_\infty, S (K)})^{\ast \ast}.
            \]
            \end{itemize}
            Then the following claims are valid.
            \begin{liste}
            \item  The conjecture $\mathrm{TNC} (E, \Z_\cS [\cG_E]\epsilon_{E, V})$ is valid for every $E\in \cX$ if at least one of the following conditions holds:
            \begin{romanliste}
            \item $\cS$ contains all prime divisors of $2 d_k$.
                    \item $V \neq \emptyset$ and $\varepsilon_{k}$ is a congruence system in the sense of Definition \ref{con def}.
                \item $2 \in \cS$ and  for every $p \in \mathbb{P} \setminus \cS$ the conjecture $\mathrm{TNC} (F, \Z_p [\cG_{F}]\epsilon_{ F, V})$ is valid for every  $F\in \cX$ with $F \subseteq k_\emptyset \langle p \rangle L$ for some tamely ramified cyclic $p$-extension $L$ of $k$ that belongs to $\Omega_k^{S_\infty}$.
                \item Assume $2 \in \cS$ and  for every $p \in \mathbb{P} \setminus \cS$ let ${\widetilde k}^p_\infty/k$ be a $\Z_p$-extension in which no finite place splits completely (and which, if $d=1$, may, or may not, be equal to $k^p_\infty$). Then the conjecture $\mathrm{TNC} (F, \Z_p [\cG_{F}]\epsilon_{ F, V})$ is valid for every  $F\in \cX$ with $F \subset k_\emptyset \langle p \rangle {\widetilde k}_\infty^p$. 
               \end{romanliste}
            \item Assume $V \neq \emptyset$ and $\varepsilon_{k}\in \ES_k^\cX (\Z_\cS)^\mathrm{con}$, and fix $c\in \ES_k^\cX (\Z_\cS)^\mathrm{con} \cap \ES_k^\cX (\Z_\cS)^\mathrm{sym}$. Then one has $c\in \Z_\cS \llbracket G_k^{\mathrm{ab}} \rrbracket \cdot \varepsilon^{\cX}_k$ if, for every $p$ and $K$ as in (H$_2$), there exists $T\in \mathscr{P}^\mathrm{ad}_K$ such that
        \begin{equation} \label{doch-in-Hamburg-auf-der-Elbchausee}
        \im (c_{K^p_\infty, T})^{\ast \ast} 
        \subseteq \Fitt^0_{\bLambda_{p,K}} ( \Cl^p_{K^p_\infty, S (K), T})^{\ast \ast} 
        \cdot \Fitt^0_{\bLambda_{p,K}} (X^p_{K^p_\infty, S (K)})^{\ast \ast}.
            \end{equation}  
            \end{liste}
\end{thm}

This result will be proved in \S\,\ref{reduction to Iwasawa theory proof section}, where we shall also (in Lemma~\ref{divisor group lemma})   record some useful facts about the ideals that occur on the right-hand side in the inclusion (\ref{doch-in-Hamburg-auf-der-Elbchausee}). 

\subsubsection{Preliminary observations} Before proving Theorem~\ref{reduction to Iwasawa theory}, it is convenient to take a slightly more general point of view and, for this,  we let $E$ be any finite abelian extension of $k$. 

Then, for a given $p\in \mathbb{P}$, the ring $\bLambda_{p,E}$ need not be regular. To take account of the difficulties that this causes, we refer to a height-one prime $\p$ of $\bLambda_{p,E}$ as `regular' if the localisation of $\bLambda_{p,E}$ at $\p$ is a regular local ring and we label any height-one prime that is not regular as `singular' (note that, if $p \in \p$,  this terminology differs slightly from that in  \cite[\S\,3C1]{BKS2}). 

For convenience, we will first record several useful (and, for the most part, well-known) properties of localisation in such rings 
that are either of a general nature or concern arithmetic modules occurring in later arguments. \textcolor{black}{In this result, we refer to finitely-generated $\bLambda_{p, E}$-submodules $I$ of the total ring of fractions $Q_{p, E}$ of $\bLambda_{p, E}$ as `fractional ideals' of $\bLambda_{p, E}$, and we note that, as in the case of ideals, $I^{\ast \ast}$ naturally identifies with the submodule $\{ a \in Q_{p, E} \mid \forall f \in I^\ast: f (a) \in \bLambda_{p, E} \}$ of $Q_{p, E}$ (cf.\@ \cite[Prop.~A.8]{sbA}).}

\begin{lemma}\label{localisation lemma} For each $p\in \mathbb{P}$, the following claims are valid. 
\begin{liste}
\item If $I$ and $J$ are \textcolor{black}{fractional} ideals of $\bLambda_{p,E}$, then one has $I^{\ast \ast} \subseteq J^{\ast \ast}$ if and only if $I_\p \subseteq J_\p$ for all prime ideals $\p$ of $\bLambda_{p,E}$ of height at most one.
\item Write $\Lambda_{p, E}$ for the subring $\Z_p\llbracket\gal{E^p_\infty}{E}\rrbracket$ of $\bLambda_{p,E}$. Then each singular prime of $\bLambda_{p,E}$  contains $p$ and is outside the support of any finitely-generated $\bLambda_{p,E}$-module that is both torsion, and has vanishing Iwasawa $\mu$-invariant, as a $\Lambda_{p, E}$-module.   
\item For $E \in \cX$ and every finite $M\subset S_\fin(k)$, the $\bLambda_{p, E}$-modules $Y^p_{E^p_\infty, M}$ and $X^p_{E_\infty^p, M}$ are torsion, and their supports contain no singular height-one primes.
\item For $E \in \cX$ and disjoint finite subsets $M$ and $T$ of $S_\fin (k)$, the $\bLambda_{p, E}$-module $\Cl^p_{E^p_\infty, M, T}$ is torsion. In addition, if $M' \subseteq M$ and $T' \subseteq T$, the natural map $\Cl^p_{E^p_\infty, M', T} \to \Cl^p_{E^p_\infty, M, T'}$ is an isomorphism after localising at any singular height-one prime of $\bLambda_{p, E}$.
\end{liste}
\end{lemma}

\begin{proof} Claim (a) follows, for example, from \textcolor{black}{the argument} of Sakamoto in \cite[Lem.\@ C.13]{Sakamoto20}. 

The first assertion of claim (b) follows from the argument of \cite[Lem.\@ 6.2\,(ii)]{BurnsGreither} and the second from an application of Nakayama's Lemma (as in \cite[Lem.\@ 6.3]{BurnsGreither}).

As for (c), since $X^p_{E_\infty^p, M}$ is by definition a submodule of $Y^p_{E^p_\infty, M}$, it suffices to prove the claim for the latter.
To do this, we note that the $\bLambda_{p, E}$-module $Y^p_{E^p_\infty, M}$ is isomorphic to the direct sum $\bigoplus_{v \in M} \Z_p \llbracket \faktor{\cG_{E^p_\infty}}{\cG_{E^p_\infty, v}} \rrbracket$. In particular, since by assumption no finite place splits completely in $k_\infty / k$, it follows that $Y^p_{E^p_\infty, M}$ is isomorphic to a direct sum of modules that are finitely-generated over a power series ring (over $\Z_p$) in at most $d - 1$ variables and therefore a $\bLambda_{p, E}$-torsion module with vanishing $\mu$-invariant as a $\Lambda_{p, E}$-module.
By claim (b), this shows that the localisation of $Y^p_{E^p_\infty, M}$ at any height-one singular prime $\p$ of $\bLambda$-vanishes, as required to prove (c).\\
Turning now to claim (d), 
we observe that $Y_{E^p_\infty, M \setminus M'} \oplus \varprojlim_{F} (\Z_p \otimes_\Z \mathbb{F}_{T \setminus T'}^\times)$ surjects onto the kernel of the natural map $\Cl^p_{E^p_\infty, M', T} \twoheadrightarrow \Cl^p_{E^p_\infty, M, T'}$. By claim (c), $Y_{E^p_\infty, M \setminus M'}$ is a $\bLambda_{p, E}$-torsion module that vanishes when localised at any singular height-one prime, and the same is true for $\varprojlim_{F} (\Z_p \otimes_\Z \mathbb{F}_{T \setminus T'}^\times)$ because no finite place splits completely in $k^p_\infty$ (see \cite[Lem.~4.12]{BullachDaoud} for details). The map $\Cl^p_{E^p_\infty, M', T} \to \Cl^p_{E^p_\infty, M, T'}$ is therefore bijective after localisation at any singular height-one prime. 
In particular, for any singular height-one prime $\p$ we have isomorphisms $(\Cl^p_{E^p_\infty})_\p \xleftarrow{\simeq} (\Cl^p_{E^p_\infty, \emptyset, T})_\p \xrightarrow{\simeq} (\Cl^p_{E^p_\infty, M, T})_\p$.
As it is well-known that $\Cl^p_{E^p_\infty}$ is torsion
as a $\Lambda_{p, E}$-module (cf.~\cite[Prop.~(11.1.4)]{NSW}), and hence also as a $\bLambda_{p, E}$-module, these isomorphisms imply that also 
 $\Cl^p_{E^p_\infty, M, T}$ is a $\bLambda_{p, E}$-torsion module, as claimed.
\end{proof}

\subsubsection{Main conjectures in equivariant Iwasawa theory}

We now fix $p\in \mathbb{P}$ and, for simplicity, suppress many explicit references to it in the notation of this section, thereby writing $k_\infty$, $E_\infty$, \textcolor{black}{and $\bLambda_E$} in place of $k_\infty^p$, $E_\infty^p$, \textcolor{black}{and $\bLambda_{p, E}$} etc.

Note that, by assumption, $S(k_\infty)\subset S_\fin(k)$. 
Given a finite set $\Sigma$ of places of $k$ containing $S^\ast (E)$ and $T\in \mathscr{P}^\mathrm{ad}_E$ with $T \cap \Sigma = \emptyset$, we write 
\[ D^\bullet_{E, \Sigma, T} \coloneqq \Z_p \otimes_\Z^\mathbb{L} C^\bullet_{E, \Sigma, T}\]
for the $p$-completion  of the complex $C^\bullet_{E, \Sigma, T}$ introduced in \S\,\ref{acnf section}.
If $S^\ast (E_\infty) \subseteq \Sigma$, then we may also define the
complex %
\[ D^\bullet_{E_\infty, \Sigma, T} \coloneqq \mathrm{R}{\varprojlim}_F D^\bullet_{F, \Sigma, T},\]
where in the (homotopy) limit $F$ ranges over all finite extensions of $k$ in $E_\infty$ and the transition maps are induced by the map $\nabla_{F / F', \Sigma, T}$ from (\ref{def nabla map}). 
As before, we drop the subscript $T$ if $T = \emptyset$.\\
Then, since each complex $D^\bullet_{F, \Sigma, T}$ can be represented by a bounded complex of finitely-generated $\Z_p$-modules, and hence of compact Hausdorff spaces, the result of Lemma~\ref{properties-weil-etale}\,(a) induces (upon passing to the limit over all intermediate fields of $E_\infty/E$) both an identification $H^0 (D^\bullet_{E_\infty, \Sigma, T}) = U_{E_\infty, \Sigma, T}$ and a natural exact sequence
\begin{equation} \label{exact-sequence-H1-limit}
    \begin{tikzcd}
    0 \arrow{r} & \Cl^p_{E_\infty, \Sigma, T} \arrow{r} & H^1 (D^\bullet_{E_\infty, \Sigma, T}) \arrow{r} & 
    X^p_{E_\infty, \Sigma} \arrow{r} & 0.
    \end{tikzcd}
\end{equation}
In addition, since $T\in \mathscr{P}^\mathrm{ad}_E$, for every  finite subextension $E'$ of $E_\infty / E$, the group $U_{E', \Sigma, T}$ is $\Z_p$-torsion free  (cf.\@ the general result of \cite[Prop.\@ (1.6.12)]{NSW}). By a well-known argument in homological algebra (see, for example, \cite[Prop.\@ 3.2]{BullachHofer}), it then follows that  $D^\bullet_{E_\infty, \Sigma, T}$ admits a `quadratic standard representative' (in the sense of \cite[\S\,A.2]{sbA}) 
\begin{equation}\label{complex rep} 
P_{E_\infty} \longrightarrow P_{E_\infty},\end{equation}
where $P_{E_\infty}$ is a free $\bLambda_E$-module of finite rank and the first term occurs in degree zero. 
\\
In the next result we shall also use the homomorphism of $\bLambda_E$-modules $\Theta_{E_\infty, \Sigma, T}$ that, in analogy to the map $\Theta_{E / k, \Sigma, T}^{r_E}$ from (\ref{projection-map}), is defined by means of the following composite (where, for simplicity, we abbreviate $\bLambda_E$ to $\bLambda$ and write $Q$ for the total field of fractions of $\bLambda$)
\begin{align*}
 \epsilon_E \Det_{\bLambda} ( D^\bullet_{E_\infty, \Sigma, T}) 
& \hookrightarrow \epsilon_E Q \otimes_{\bLambda} \Det_{\bLambda} ( D^\bullet_{E_\infty, \Sigma, T}) \\
& \xrightarrow{\simeq} 
\Det_{\epsilon_E Q} \big ( \epsilon_E Q  \otimes_{\bLambda}  H^0 (D^\bullet_{E_\infty, \Sigma, T}) \big)
\otimes_{Q} 
\Det_{\epsilon_E Q} \big( \epsilon_E Q  \otimes_{\bLambda}  H^1 (D^\bullet_{E_\infty, \Sigma, T}) \big)^\ast \\
& \xrightarrow{\simeq}  \Big ( \epsilon_E Q \otimes_{\bLambda} \exprod^{r_E}_{\bLambda} U_{E_\infty, \Sigma, T} \Big) 
\otimes_Q \Big ( \epsilon_E Q \otimes_{\bLambda} \exprod^{r_E}_{\bLambda} Y^p_{E_\infty, V_E} \Big)^\ast \\
& \xrightarrow{\simeq}   \epsilon_E Q  \otimes_{\bLambda} \exprod^{r_E}_{\bLambda} U_{E_\infty, \Sigma, T}.
\end{align*}
Here the first isomorphism is the natural `passage-to-cohomology' map,
the second isomorphism is induced by (\ref{exact-sequence-H1-limit}) and the fact that $\Cl^p_{E_\infty, \Sigma, T}$ and $X^p_{E_\infty, \Sigma \setminus S_\infty }$ are both $\bLambda$-torsion modules (by parts (c) and (d) of Lemma~\ref{localisation lemma}) and
$Y_{E_\infty, S_\infty \setminus V_E}^p$ is annihilated by $\epsilon_E$,
and the last isomorphism is due to our fixed choice of extensions of places in $V_E$ to $E_\infty$ (and hence choice of isomorphism $Y^p_{E_\infty, V_E} \cong \bLambda_E^{r_E}$). We recall that this map can be explicitly described in terms of certain `rank reduction maps' (see \cite[Lem.\@ A.7\,(i)]{sbA}) and that, by using this explicit description, it can be shown that $\im( \Theta_{E_\infty, \Sigma, T})\subseteq \bidual^{r_E}_{\bLambda_E} U_{E_\infty, \Sigma , T}$, and that $\Theta_{E_\infty, \Sigma, T}$ agrees with the limit (over all finite layers $F / k$ of $E_\infty / k$) of the maps $\epsilon_E (\Z_p \otimes_\Z \Theta^{r_F}_{F, \Sigma, T})$  defined in (\ref{projection-map}) (see \cite[Lem.~3.12 and Lem.~3.19]{BullachDaoud} for details).
\medskip \\
The following result will play a key role in the proof of Theorem~\ref{reduction to Iwasawa theory}. 

\begin{proposition} \label{elbchausse-implies-eimc}
Fix a prime number $p$, a finite abelian extension $E$ of $k$ with $V_E = S_\infty$ if $p = 2$, a finite set $\Sigma$ of places of $k$ with $S^\ast (E_\infty^p) \subseteq \Sigma$ and a set $T\in \mathscr{P}^\mathrm{ad}_E$ with $T \cap \Sigma = \emptyset$. Then the following claims are valid. 
\begin{liste}
\item Let $\fz_{E^p_\infty}$ be an $\epsilon_E \bLambda_{p, E}$-basis of $\epsilon_E \Det_{\bLambda_{p, E}} ( D^\bullet_{E^p_\infty, \Sigma, T})$ and set $z^\mathrm{b}_{E^p_\infty} \coloneqq \Theta_{E^p_\infty, \Sigma, T} ( \fz_{E^p_\infty})$.
Then 
\begin{equation} \label{Elbchaussee-for-basic-Eulers}
\im ( z^\mathrm{b}_{E^p_\infty})^{\ast \ast} = \Fitt^0_{\bLambda_{p, E}} ( \Cl^p_{E^p_\infty, \Sigma, T})^{\ast \ast} 
        \cdot \Fitt^0_{\bLambda_{p, E}} (X^p_{E^p_\infty, \Sigma \setminus V_E})^{\ast \ast}. 
\end{equation}
\item Assume that, for every height-one prime $\p$ of $\bLambda_{p,E}$ that is contained in the support of $\epsilon_E \bLambda_{p,E}$, the $ (\bLambda_{p,E})_\p$-module $(\Cl^p_{E^p_\infty})_\p$ has finite projective dimension.
\begin{enumerate}[label=(\roman*)]
    \item For each $\eta\in (\bidual^{r_E}_{\bLambda_{p,E}} U_{E_\infty^p, \Sigma, T})[1 - \epsilon_E]$ one has 
        \begin{equation} \label{doch-in-Hamburg-auf-der-Elbchausee-variante}\hskip-0.2truein
        \eta\in \im(\Theta_{E^p_\infty, \Sigma, T}) \; \Longleftrightarrow \; \im (\eta)^{\ast \ast} 
        \subseteq \Fitt^0_{\bLambda_{p,E}} ( \Cl^p_{E^p_\infty, \Sigma, T})^{\ast \ast} 
     \Fitt^0_{\bLambda_{p,E}} (X^p_{E^p_\infty, \Sigma \setminus V_E})^{\ast \ast}.
            \end{equation}
    \item If the equivalent conditions in (\ref{doch-in-Hamburg-auf-der-Elbchausee-variante}) are satisfied with $\Sigma = S^\ast (E_\infty^p)$ and $\eta$   
    the family of Rubin--Stark elements $\varepsilon_{E^p_\infty, T} \coloneqq (\varepsilon^{V_F}_{F / k, S^\ast (E^p_\infty), T})_F$, then there exists  an $\epsilon_E \bLambda_{p,E}$-basis $\mathcal{L}_{E^p_\infty, S^\ast (E_\infty^p), T}$ of $\epsilon_E \Det_{\bLambda_{p, E}} (D^\bullet_{E^p_\infty, S^\ast (E^p_\infty), T})$ such that 
    \[ \Theta_{E^p_\infty, S^\ast (E_\infty^p), T} (\mathcal{L}_{E^p_\infty, S^\ast (E_\infty^p), T}) = \varepsilon_{E^p_\infty, T}.\]
\end{enumerate}
\end{liste}
\end{proposition}

\begin{proof} 
 As $p$ is fixed during this proof, we continue to abbreviate $E_\infty^p$ to $E_\infty$ and $\bLambda_{p,E}$ to $\bLambda$. \\
We then first note that $Y^p_{E_\infty, V_E}$ and, due to our assumption that $V_E = S_\infty $ if $p = 2$, also $Y^p_{E_\infty, S_\infty \setminus V_E}$ are both projective $\bLambda$-modules. 
We can therefore find a $\bLambda$-module $Z$ such that the exact sequence (\ref{exact-sequence-H1-limit}) induces an isomorphism
\begin{equation}\label{Z def}
H^1 (D^\bullet_{E_\infty, \Sigma, T}) \cong Z \oplus Y^p_{E_\infty, S_\infty \setminus V_E} \oplus Y^p_{E_\infty, V_E}
\end{equation}
and, setting $\Sigma_\fin \coloneqq \Sigma \cap S_\fin(k)$, an exact sequence
\begin{equation} \label{exact-sequence-Z}
    \begin{tikzcd}
   0 \arrow{r} & \Cl^p_{E_\infty, \Sigma, T} \arrow{r} & 
Z \arrow{r} & X^p_{E_\infty, \Sigma_\fin} \arrow{r} & 0.  
    \end{tikzcd}
\end{equation}
Moreover, by choosing a section to the (surjective) composite map
\[ P_{E_\infty} \twoheadrightarrow H^1 (D^\bullet_{E_\infty, T}) \twoheadrightarrow Y^p_{E_\infty, V_E}\]
in which the first map is induced by the representative (\ref{complex rep}) of the complex $D^\bullet_{E_\infty, \Sigma, T}$ and the second by the isomorphism (\ref{Z def}), we may identify $Y^p_{E_\infty, V_E}$ with a free direct summand of $P_{E_\infty}$. In this way, we deduce the existence of a projective, and hence free (since $\bLambda$ is semilocal and $Y_{E_\infty, V_E}^p$ is free), $\bLambda$-submodule $P'_{E_\infty}$ of $P_{E_\infty}$ for which there is an isomorphism of $\bLambda$-modules 
\begin{equation}\label{P' iso} P_{E_\infty} \cong P'_{E_\infty} \oplus Y^p_{E_\infty, V_E},\end{equation}
and also an exact sequence of $\bLambda$-modules 
\begin{cdiagram}
0 \arrow{r} & U_{E_\infty, \Sigma, T} \arrow{r} & P_{E_\infty} \arrow{r} & P'_{E_\infty} \arrow{r} & Z \oplus Y^p_{E_\infty, S_\infty \setminus V_E} \arrow{r} & 0,
\end{cdiagram}%
in which the third arrow is the composite of the differential of  (\ref{complex rep}) and the projection $P_\infty \to P_\infty'$ induced by (\ref{P' iso}), and the fourth is the restriction to $P_{E_\infty}'$ of the map $P_{E_\infty} \to Z \oplus Y^p_{E_\infty, S_\infty  \setminus V_E}$ induced by (\ref{complex rep}) and the decomposition (\ref{Z def}). 

By now applying \cite[Lem.\@ 2.7\,(c)]{BullachDaoud} to this sequence  
and recalling that $\Theta_{E_\infty, \Sigma, T}$ can be explicitly described as a rank reduction map (see \cite[Lem.\@ A.7\,(i)]{sbA}),
we derive an equality   
\begin{align} \nonumber
\im ( z^\mathrm{b}_{E_\infty})^{\ast \ast}
& = \Fitt^{0}_\bLambda ( Z \oplus Y^p_{E_\infty, S_\infty  \setminus V_E})^{\ast \ast} \\ \label{Fitt-description}
  &      = \Fitt^0_{\bLambda} ( Z)^{\ast \ast} \cdot \Fitt^0_\bLambda (Y^p_{E_\infty, S_\infty \setminus V_E})^{\ast \ast}.
\end{align}
Next we note that if $\p$ is a regular height-one prime of $\bLambda$, then $\bLambda_\p$ is a discrete valuation domain. In particular, since over such a ring initial Fitting ideals are multiplicative on short exact sequences, we may deduce from the above equality and the exact sequence (\ref{exact-sequence-Z}) that
\begin{align} \nonumber 
\im ( z^\mathrm{b}_{E_\infty})_\p &= \bigl(\im ( z^\mathrm{b}_{E_\infty})^{\ast \ast}\bigr)_\p\\ \nonumber 
& = \Fitt^0_{\bLambda} ( \Cl^p_{E_\infty, \Sigma, T})_\p 
        \cdot \Fitt^0_{\bLambda} (X^p_{E_\infty, \Sigma_\fin})_\p \cdot \Fitt^0_\bLambda (Y^p_{E_\infty, S_\infty \setminus V_E})_\p \\
        & = \Fitt^0_{\bLambda} ( \Cl^p_{E_\infty, \Sigma, T})_\p 
        \cdot \Fitt^0_{\bLambda} (X^p_{E_\infty, \Sigma \setminus V_E})_\p,
        \label{im z basic}
\end{align}
where the last equality is true because there exists a natural exact sequence \begin{equation} \label{exact sequence X term}
\begin{tikzcd}
   0 \arrow{r} &  X^p_{E_\infty, \Sigma_\fin} \arrow{r} & X^p_{E_\infty, \Sigma \setminus V_E} \arrow{r} & Y^p_{E_\infty, S_\infty  \setminus V_E} \arrow{r} & 0.
   \end{tikzcd}
\end{equation}
This description of $\im ( z^\mathrm{b}_{E_\infty})_\p$ is also valid for any prime $\p$ of $\bLambda$ of height zero since this simply amounts to a further localisation of (\ref{im z basic}). \\ 
If $\p$ is a singular height-one prime of $\bLambda$, on the other hand, then Lemma  \ref{localisation lemma}\,(c) shows that 
 the localisation $(X_{E_\infty, \Sigma_\fin}^p)_\p$ vanishes and hence that the exact sequence (\ref{exact-sequence-Z}) induces a natural isomorphism $ (\Cl^p_{E_\infty, \Sigma, T})_\p \cong Z_\p$. 
This isomorphism then combines with the equality (\ref{Fitt-description}) to imply that \begin{align*}
\im ( z^\mathrm{b}_{E_\infty})_\p  & = \Fitt^0_{\bLambda} (\Cl^p_{E_\infty, \Sigma, T})_\p \cdot \Fitt^0_\bLambda (Y^p_{E_\infty, S_\infty\setminus V_E})_\p \\
& = 
\Fitt^0_{\bLambda} ( \Cl^p_{E_\infty, \Sigma, T})_\p 
        \cdot \Fitt^0_{\bLambda} (X^p_{E_\infty, \Sigma \setminus V_E})_\p,
\end{align*}
where the second equality is true since $(X_{E_\infty, \Sigma_\fin}^p)_\p$ vanishes and (\ref{exact sequence X term}) is exact. \\
At this stage we have established the last displayed equality for all primes of $\bLambda$ of height at most one and so it follows from Lemma~\ref{localisation lemma}\,(a) that the equality (\ref{Elbchaussee-for-basic-Eulers}) required to establish claim (a) holds. \\ 
Turning now to the proof of claim (b)\,(i), we first note that the `only if' part of the claim follows from (a). To prove the `if' part, we observe that the definition of $\Theta_{E_\infty, \Sigma, T}$ ensures that $z^\mathrm{b}_{E_\infty}$ is a generator of the $\epsilon_E Q (\bLambda)$-module spanned by $(\bidual^{r_E}_{\bLambda} U_{E_\infty, \Sigma, T})[1 - \epsilon_E]$. Given an element $\eta$ of $(\bidual^{r_E}_{\bLambda} U_{E_\infty, \Sigma, T})[1 - \epsilon_E]$, we can therefore fix an element $q$ of $\epsilon_E Q (\bLambda)$ with the property that 
\begin{equation}\label{q def} 
\eta = q \cdot z^\mathrm{b}_{E_\infty}.
\end{equation}
This equality then combines with (\ref{Elbchaussee-for-basic-Eulers}) and the assumed inclusion  (\ref{doch-in-Hamburg-auf-der-Elbchausee-variante}) to imply an inclusion 
\begin{equation}\label{last-displayed-2}
q \cdot \im ( z^\mathrm{b}_{E_\infty})^{\ast \ast} 
 = \im (\eta)^{\ast \ast}
 \subseteq \im ( z^\mathrm{b}_{E_\infty})^{\ast \ast} . 
\end{equation}
We next claim that, for every height-one prime ideal $\p$ of $\bLambda$ that is contained in the support of $\epsilon_E \bLambda$, the ideal $\Fitt^0_{\bLambda} (Z)_\p$ is principal and generated by a non-zero divisor. Indeed, if $\p$ is regular, this is automatically satisfied because $\bLambda_\p$ is a discrete valuation ring and $\Fitt^0_{\bLambda} (Z)_\p$ is a nonzero ideal of $\bLambda_\p$. If $\p$ is singular, on the other hand, then the exact sequence (\ref{exact-sequence-Z}) implies that $Z_\p$ is isomorphic to $(\Cl^p_{E_\infty, \Sigma, T})_\p$.
 In addition, Lemma~\ref{localisation lemma}\,(d) implies that the natural maps  $(\Cl^p_{E_\infty, \emptyset, T})_\p\to  (\Cl^p_{E_\infty, \Sigma, T})_\p$ and  $(\Cl^p_{E_\infty, \emptyset, T})_\p \to (\Cl^p_{E_\infty})_\p$ are both bijective.
Thus, our assumption that the projective dimension of the $\bLambda_\p$-module  $(\Cl^p_{E_\infty})_\p$ is finite, and hence at most one (as a consequence of the Auslander--Buchsbaum formula since $\bLambda$ is Gorenstein and $\p$ has height one),
implies that also $Z_\p$ has projective dimension one. It follows that there is an exact sequence of the form $0 \to \bLambda_\p^{n} \to \bLambda_\p^{ n} \to Z_\p \to 0$ for some natural number $n$, and computing $\Fitt^0_{\bLambda} (Z)_\p$ with respect to this presentation of $Z_\p$ shows that $\Fitt^0_{\bLambda} (Z)_\p$ is a principal ideal generated by a nonzero divisor. 
This shows that in both cases $\Fitt^0_{\bLambda} (Z)_\p$ is indeed principal and generated by a non-zero divisor.
\\ 
In addition, writing $\tau_v \in \cG_{E_\infty}$ for the generator of the decomposition group $\cG_{E_\infty, v}$ of a place $v \in S_\infty \setminus V_E$, the isomorphism $Y^p_{E_\infty, S_\infty \setminus V_E} \cong \bigoplus_{v \in S_\infty  \setminus V_E} \Z_p \llbracket \cG_{E_\infty} / \cG_{E_\infty, v} \rrbracket$
gives rise to an exact sequence
\[ 0 \to \bigoplus_{v \in S_\infty \setminus V_E} \bLambda \xrightarrow{ \cdot (1 - \tau_v)_v} \bigoplus_{v \in S_\infty  \setminus V_E} \bLambda \to Y^p_{E_\infty, S_\infty \setminus V_E} \to 0,
\]
which we may use to compute that
 $\Fitt^0_\bLambda (Y^p_{E_\infty, S_\infty\setminus V_E}) = \bLambda\epsilon_E$
 because we are assuming $2 \in \bLambda^\times$ if $V_E \neq S_\infty $.
 It therefore follows from (\ref{Fitt-description}) that, if $\p$ is in the support of $\epsilon_E \bLambda$, then the ideal $\im ( z^\mathrm{b}_{E_\infty})_\p$ is 
 equal to $\epsilon_E \Fitt^0_{\bLambda} (Z)_\p$ and hence
 generated by a non-zero divisor in $\bLambda_\p = (\epsilon_E \bLambda)_\p$. We therefore deduce from (\ref{last-displayed-2}), by cancellation, that $q$ belongs to $\bLambda_\p$. Since, by construction,  $q$ belongs to $\epsilon_E Q (\bLambda)$, the last assertion is also  clear both for height-one primes $\p$ that are not in the support of $\epsilon_E \bLambda$ and for primes of height zero, and so Lemma~\ref{localisation lemma}\,(a) implies that $q$ belongs to $(q \bLambda)^{\ast \ast} \subseteq (\epsilon_E \bLambda)^{\ast\ast} = \epsilon_E \bLambda$, as required to prove claim (b)\,(i). \\
To prove claim (b)\,(ii) we note that, for each finite extension $E'$ of $E$ in $E_\infty$ one has $V_{E'} = V_E$ and hence $r_{E'} = r_E$. Using this fact, we write $z^\mathrm{b}_{E'}$ for the image of $z^\mathrm{b}_{E_\infty}$ under the natural projection map
\[
\bidual^{r_E}_\bLambda U_{E_\infty, S (E_\infty), T} \to \bidual^{r_{E'}}_{\Z_p [\cG_{E'}]} U_{E', S (E_\infty), T}.
\]
Then, by claim (b)\,(i), we know that $\varepsilon^{V_{E'}}_{E' / k, S^\ast (E_\infty), T}$ belongs to $\Z_p [\cG_{E'}] \cdot z^\mathrm{b}_{E'} = \Z_p\cdot \im (\Theta^{r_{E'}}_{E', S^\ast (E_\infty), T})$ and hence, by Proposition~\ref{analztic-class-number-formula-argument}, that $\Z_p\cdot \im (\Theta^{r_{E'}}_{E', S^\ast (E_\infty), T}) = \Z_p[\cG_{E'}]\cdot\varepsilon^{V_{E'}}_{E' / k, S^\ast (E_\infty), T}$. 

By passing to the limit over all such fields $E'$
and recalling that $\Theta_{E_\infty, S^\ast (E_\infty), T}$ agrees with the limit of the maps $\Z_p \otimes_\Z \Theta^{r_E}_{E', S^\ast (E_\infty), T}$, these equalities combine to imply that the image of $\Theta_{E_\infty, S^\ast (E_\infty), T}$ is generated as a $\bLambda$-module by $\varepsilon_{E_\infty, T}$.
As a consequence, the element $q$ that verifies (\ref{q def}) with $\eta$ taken to be $\varepsilon_{E_\infty, T}$ must be a unit of $\epsilon_E\bLambda$. \\
Given this observation, it is then easily checked that the element 
$\mathcal{L}_{E_\infty, S^\ast (E_\infty), T} \coloneqq q^{-1}\cdot \fz_{E_\infty}$ has the properties that are required to validate claim (b)\,(ii).
\end{proof}

\subsubsection{Criteria for the validity of the eTNC}

In the following result, we describe a concrete link between the results of Theorems \ref{cons scarcity thm} and \ref{reduction to Iwasawa theory}\,(a). In order to state it, we fix $V \subseteq S_\infty$ and a subset $\cX\subseteq \Omega_k^V$ satisfying Hypothesis \ref{hyp X} with respect to $\cV = \{ V \}$. We also fix $p \in \mathbb{P}$ and, for each finite abelian extension $F$ of $k$, use the convenient abbreviation
\[ \mathrm{TNC}_{F} \coloneqq \mathrm{TNC} (F,  \Z_p [\cG_F]\epsilon_{F, V}),\]
where $\epsilon_{F, V}$ is the idempotent of $\Z[1/2][\cG_F]$ defined in (\ref{epsKV def}).

As in Proposition~\ref{basic-Euler-systems}, we define a homomorphism of $\Z_p\llbracket G_k^{\mathrm{ab}}\rrbracket$-modules 
\[
\Theta_{k}^p \coloneqq (\Theta^{r_E}_{E, S^\ast (E)})_E \:  {\varprojlim}_{E \in \Omega_k} \Det_{\Z_p [\cG_E]} (D^\bullet_{E, S^\ast (E)}) \to \ES_k (\Z_p),
\]
where the limit is taken with respect to the maps $i_{F / E}$ in (\ref{definition-transition-map}). 
Finally, we recall the extension $k_\emptyset \langle p \rangle$ of $k$ defined in (\ref{k_{p} def}).

\begin{prop} \label{etnc new criterion}
Assume $p \not= 2$.
Then  $\mathrm{TNC}_K$ is valid for every field $K \in \cX$ provided that the following two conditions are satisfied:
\begin{enumerate}[label=(\roman*)]
    \item $\im (\Theta^p_k) \subseteq \Z_p \llbracket G_k^{\mathrm{ab}} \rrbracket \cdot \varepsilon_{k}^\cX + \ES_k^\cX (\Z_p)^{\gal{k^{\mathrm{ab}}}{k_\emptyset \langle p \rangle}}$.

    \item At least one of the following is valid: 
    \begin{itemize}
        \item[(a)] $\mathrm{TNC}_F$ is valid for all fields $F$ of the form $k_\emptyset \langle p \rangle \cdot Z$ with $Z$ a tamely ramified cyclic $p$-extension of $k$ that belongs to $\Omega_k^{S_\infty }$.
        \item[(b)] $\mathrm{TNC}_F$ is valid for all fields $F$ of the form $k_\emptyset \langle p \rangle \cdot {\widetilde k}_n^p$ with ${\widetilde k}_n^p$ the $n$-th layer of a $\Z_p$-extension ${\widetilde k}^p_\infty$ of $k$ in which no place in $S_\fin(k)$ splits completely. 
        \end{itemize}
\end{enumerate}
\end{prop}

To prove this result, we shall use the  `zeta element' associated to each finite abelian extension $F$ of $k$ and finite set of places $\Sigma$ of $k$ with $S^\ast (F)\subseteq \Sigma$. We recall that this is the unique element $\cL_{F, \Sigma}$ of $\R \otimes_\Z \Det_{\Z [\cG_F]} ( C^\bullet_{F, \Sigma})$ with 
\[
 \vartheta_{C^\bullet_{F, \Sigma}, \lambda_{F, \Sigma}} ( \cL_{F, \Sigma}) = \theta^\ast_{F, \Sigma} (0).
\]
Here $\vartheta_{C^\bullet_{F, \Sigma}, \lambda_{F, \Sigma}}$ denotes for the isomorphism of $\R[\cG_F]$-modules $\R \otimes_{\Z} \Det_{\Z [\cG_F]} ( C^\bullet_{F, \Sigma}) \xrightarrow{\simeq} \R [\cG_F]$ induced by the Dirichlet regulator $\lambda_{F, \Sigma}$ (and  Lemma~\ref{norm-map-Lemma}\,(c)). By fixing an isomorphism $\C \cong \C_p$ we will always regard $\cL_{F, \Sigma}$ as an element of $\C_p \otimes_{\Z} \Det_{\Z [\cG_F]} ( C^\bullet_{F, \Sigma}) \cong \C_p \otimes_{\Z_p} \Det_{\Z_p [\cG_F]} ( D^\bullet_{F, \Sigma})$. Then the following result records several useful (and essentially well-known) functorial properties of these elements. 

\begin{lem} \label{functoriality of zeta elements}
    The following claims are valid for any finite abelian extension $F$ of $k$ and any finite set $\Sigma$ of places of $k$ that contains $S^\ast (F)$.
    \begin{liste}
        \item $\mathrm{TNC}_F$ is valid if and only if $\epsilon_{F, V} \cL_{F, \Sigma}$ is a $\Z_p[\cG_F]\epsilon_{F, V}$-basis of $\epsilon_{F, V} \Det_{\Z_p [\cG_F]} ( D^\bullet_{F, \Sigma})$. 
        \item If $W$ is a proper subset of $\Sigma$ comprising places that split completely in $F$, then one has $\Theta^{|W|}_{F, \Sigma} ( \cL_{F, \Sigma}) = \varepsilon^{W}_{F / k, \Sigma}$.
        \item If $\Sigma'$ is a finite set of places of $k$ that contains $\Sigma$, then  the map $\Delta_{F, \Sigma', \Sigma}$ from (\ref{def delta map}) sends 
        $\cL_{F, \Sigma'}$ to $\cL_{F, \Sigma}$.  
        \item If $F'$ is an intermediate field of $F/k$, then  the map $\nabla_{F / F', \Sigma}$ from (\ref{def nabla map}) sends 
        $\cL_{F, \Sigma}$ to $\cL_{F', \Sigma}$.
    \end{liste}
\end{lem}

\begin{proof}
    Claim (a) is proved in \cite[Prop.~3.4]{BKS}, and the proof of the cited result also shows claim (c). Claim (b) is \cite[Th.~5.14]{BKS}. Claim (d) follows from  the commutative diagram 
    \begin{equation*}
    \begin{tikzcd}[column sep=huge]
        \R \otimes_\Z \Det_{\Z [\cG_F]} ( C^\bullet_{F, \Sigma})  \arrow{d}{\nabla_{F / F', \Sigma}} \arrow{r}{\vartheta_{C^\bullet_{F, \Sigma}, \lambda_{F, \Sigma}}} & \R [\cG_F] \arrow{d}{\pi_{F / F'}} \\
         \R \otimes_\Z \Det_{\Z [\cG_{F'}]} ( C^\bullet_{F', \Sigma}) \arrow{r}{\vartheta_{C^\bullet_{F', \Sigma}, \lambda_{F', \Sigma}}} & \R [\cG_{F'}]. 
    \end{tikzcd}%
    \end{equation*}
    and the fact that $\pi_{F / F'} ( \theta^\ast_{F / k, \Sigma} (0)) = \theta^\ast_{F' / k, \Sigma} (0)$.
\end{proof}

We can now give the proof of Proposition~\ref{etnc new criterion}.

\begin{proof}[Proof of Proposition~\ref{etnc new criterion}:] 
Since $\im(\Theta_k) \subseteq \im (\Theta_k^p)$,  Proposition~\ref{every-basic-RS-implies-etnc} implies that $\mathrm{TNC}_K$ is valid for every $K \in \cX$ provided  $\varrho^\cX (\im (\Theta_k^p))\subseteq \Z_p \llbracket G_k^{\mathrm{ab}} \rrbracket \cdot \varepsilon_{k}^\cX$. 
To verify the latter inclusion we note that, for every $E \in \Omega_k$, the $\Z_p [\cG_E]$-module $\Det_{\Z_p [\cG_E]} (D^\bullet_{E, S^\ast (E)})$ is free of rank one. Since all of the transition maps $i_{F / E}$ (from (\ref{definition-transition-map})) are surjective, it follows that the $\Z_p \llbracket G_k^{\mathrm{ab}} \rrbracket$-module $\varprojlim_{E \in \Omega_k} \Det_{\Z_p [\cG_E]} (D^\bullet_{E, S^\ast (E)})$ is free of rank one (see the argument of \cite[Prop.\@ 3.7]{bdss} for more details). We may therefore choose a $\Z_p \llbracket G_k^{\mathrm{ab}} \rrbracket$-basis $\fz = (\fz_E)_{E \in \Omega_k}$ of the latter module and set $z^\mathrm{b}\coloneqq \Theta^p_k ( \fz)$. Then, by condition (i), there exists an element 
\[ r_p = (r_{p, E})_{E \in \Omega_k}\in {\varprojlim}_{E \in \Omega_k} ( \Z_p [\cG_E]\epsilon_{E, V}) \]
such that $ r_{p, E}\cdot\varepsilon_{E / k} - z^\mathrm{b}_E$ is fixed by $\gal{E}{E \cap k_\emptyset \langle p \rangle}$ for all $E \in \cX$. It therefore suffices to show that either of the conditions stated in (ii) implies that  
\begin{equation}\label{break down} r_{p, E}\cdot\varepsilon_{E / k} =  z^\mathrm{b}_E \,\text{ for every }\,\, E \in \cX.\end{equation}
This will be proved by comparing $r_p = (r_{p, E})_E$ to an element $q_p = (q_{p, E})_E$ that measures the discrepancy between $\fz$ and the family of zeta elements $(\cL_{E, S^\ast (E)})_E$ and so is related to the conjectures $\mathrm{TNC}_E$ via Lemma  \ref{functoriality of zeta elements}\,(a).
\textcolor{black}{The essential difficulty is then that there need not a priori be an explicit relation between $r_{p, E}$ and $q_{p, E}$ even if $\mathrm{TNC}_E$ is valid. However, the point of condition (ii) is that, by assuming the validity of $\mathrm{TNC}_{\tilde E}$ for certain collections of fields $\tilde E$ containing $E$, we are able to show that $r_{p, E} q_{p, E} = 1$. The last step of the proof is then to deduce from the latter  equalities that (\ref{break down}) is valid for each field $E$ even if $\mathrm{TNC}_E$ is not a priori known to be valid.}\\
To make this idea precise, we note that 
$(r_{p, E})_{E \in \Omega_k}$ and $(\fz_E)_{E \in \Omega_k}$
can be regarded as elements $(r_{p, F})_F$ and $(\fz_F)_F$  of the inverse limits $\varprojlim_{F} (\Z_p [\cG_F] \epsilon_{F, V})$ 
and $\varprojlim_F \Det_{\Z_p [\G_F]} ( D^\bullet_{F, S^\ast (F)})$
with $F$ ranging over all finite abelian extensions of $k$, and we will continue to denote these elements as $r_p$ and $\fz$. In particular, $r_{p, F}$ and $\fz_F$ are also well-defined if $F$ is an unramified extension of $k$ even though such $F$ are not contained in $\Omega_k$.
For every finite abelian extension $F$ of $k$, we then define $q_{p, F} \in \C_{p} [\cG_F] \epsilon_{F, V}$ to be the unique element such that
\[
q_{p, F} \cdot \fz_{F} = \epsilon_{F, V}\cL_{F, S^\ast (F)}.
\]
Claims (c) and (d) of Lemma~\ref{functoriality of zeta elements} combine to imply that, if $F' \subseteq F$, then
\[
i_{F / F'} ( \cL_{F, S^\ast (F)}) = ( \Delta_{F', S^\ast (F), S^\ast (F')} \circ \nabla_{F / F', S^\ast (F)}) ( \cL_{F, S^\ast (F)}) = \cL_{F', S^\ast (F')}.
\]
In particular, since $\fz_{F'}$ is a $\C_p [\cG_{F'}]$-basis of $\C_p \otimes_{\Z_p} \Det_{\Z_p [\cG_{F'}]} ( D^\bullet_{F', S^\ast (F')})$, the equality  
    \[
    q_{p, F'} \cdot \fz_{F'} = \epsilon_{F', V} \cL_{F', S^\ast (F')} = i_{F / F'} ( \epsilon_{F, V} \cL_{F, S^\ast (F)}) = i_{F / F'} ( q_{p, F} \cdot \textcolor{black}{\fz_{F}}) = \pi_{F  / F'} ( q_{p, F}) \cdot \fz_{F'}
    \]
implies that
\begin{equation} \label{more about the units}
    \pi_{F / F'} (q_{p, F}) = q_{p, F'}.
\end{equation}
In addition, if $F \in \Omega_k$, then Lemma~\ref{functoriality of zeta elements}\,(b) implies that the element $q_{p, F}$ satisfies
\begin{equation}\label{units} q_{p, F} \cdot z^\mathrm{b}_F =  
\Theta^{|V|}_{F, S^\ast (F)} ( q_{p, F} \cdot 
\fz_F ) = \Theta^{|V|}_{F, S^\ast (F)} ( \epsilon_{F, V} \cdot 
\cL_{F, S^\ast (F)}) =
\epsilon_{F, V}\cdot\varepsilon_{F / k}.
\end{equation}
To proceed, we shall consider the collections of fields specified by 
\begin{align*}
\Omega^{\mathrm{tnc}}_k & \coloneqq \{F \mid F/k \,\,\text{is finite, abelian and such that}\,\, \mathrm{TNC}_F\,\,\text{is valid}\},\\ 
 \cX^\mathrm{tnc} & \coloneqq \{ F\in \Omega^{\mathrm{tnc}}_k \cap \cX \mid   \epsilon_{F, V} \cdot (r_{p, F} \varepsilon_{F / k} - z^\mathrm{b}_F)= 0\} \subseteq \Omega_k^\mathrm{tnc}.
\end{align*}
We observe that $\Omega^{\mathrm{tnc}}_k$ is not a subset of $\Omega_k$ since it contains $k$ (and possibly other non-trivial unramified extensions of $k$). However, the functoriality properties of zeta elements recalled in Lemma~\ref{functoriality of zeta elements} do guarantee that, for every $F \in \Omega^{\mathrm{tnc}}_k$, one has both 
 \begin{itemize}
\item[\textcolor{black}{(P1)}] $F'\in \Omega^{\mathrm{tnc}}_k$ for all fields $F'$ with $k \subseteq F'\subset F$, and
\item[\textcolor{black}{(P2)}] the zeta element $\epsilon_{F, V} \cL_{F, S^\ast (F)}$ is a $\Z_p [\cG_F]\epsilon_{F, V} $-basis of $\epsilon_{F, V} \Det_{\Z_p [\cG_F]} (D^\bullet_{F, S^\ast (F)})$ and so  
 $q_{p, F}\in (\Z_p [\cG_F]\epsilon_{F, V} )^\times$.
\end{itemize}
We now consider an arbitrary $F$ in $\cX^\mathrm{tnc}$. Then the (assumed) equality $\epsilon_{F, V} \cdot (r_{p, F} \varepsilon_{F / k} - z^\mathrm{b}_F)=0$  combines with  (\ref{units}) to imply that 
\[
\epsilon_{F, V} \cdot z^\mathrm{b}_F = \epsilon_{F, V} \cdot r_{p, F} \varepsilon_{F / k} = r_{p, F} q_{p, F} \cdot z^\mathrm{b}_F.
\]
In particular, since $z^\mathrm{b}_F$ generates a free $e_F \Z_p [\cG_F]$-module of rank one, it follows that the term $q_{p, F}^{-1} - r_{p, F}$ is annihilated by $\epsilon_{F, V} e_F$. In addition, as an element of $\Z_p [\cG_F] \epsilon_{F, V}$, this term is also annihilated by $1 - \epsilon_{F, V}$ and so the idempotent  
\[ \epsilon_{F, V} ( 1 - e_F) = 1 - ( \epsilon_{F, V} e_F + (1 - \epsilon_{F, V}))\] 
acts as the identity on $q_{p, F}^{-1} - r_{p, F}$. It follows that every element $\sigma$ of $\cG_F$ acts on $q_{p, F}^{-1} - r_{p, F}$ as multiplication by $\sum_{\chi\in \Upsilon_F} \chi (\sigma) e_\chi$ with  $\Upsilon_F \coloneqq \{\chi \in \cG_F^\circ \mid \chi (\epsilon_{F, V} (1 - e_F)) \neq 0\}$, and therefore 
\begin{equation}\label{trivial action}\sigma\in {\bigcap}_{\chi\in \Upsilon_F}\ker(\chi) 
\quad \Rightarrow \quad \sigma (q_{p, F}^{-1} - r_{p, F}) = q_{p, F}^{-1} - r_{p, F}.\end{equation}

To analyse $\bigcap_{\chi\in \Upsilon_F}\ker(\chi)$ we fix $\chi \in \Upsilon_F$. Then $\chi ( \epsilon_{F, V}) \neq 0$ and so, since $F \in \cX^{\mathrm{tnc}} \subseteq \cX \subseteq \Omega^V_k$, one has $\chi ( \cG_{F, v}) \neq 1$ for every $v \in S_\infty  \setminus V$. On the other hand,  since $\chi ( 1 - e_F) \neq 0$ (or, equivalently, $\chi (e_F) = 0$), one must also have $\chi ( \cG_{F, v}) = 1$ for some $v \in S (F)$. Taken together, these facts imply that $\chi ( \cG_{F, v}) = 1$ for some $v \in S (F)\setminus S_\infty$ and hence that $\bigcap_{\chi\in \Upsilon_F}\ker(\chi)$ contains the subgroup $H_F \coloneqq \bigcap_{v \in S(F) \setminus S_\infty} \cG_{F, v}$ of $\cG_F$. 

\textcolor{black}{
We now additionally choose $F' \in \cX$ with $F'\subseteq F$. Then, since $F \in \cX^\mathrm{tnc}\subseteq \Omega_k^\mathrm{tnc}$, property (P1) implies that $F'\in \Omega_k^\mathrm{tnc}$ and so property (P2) (for both $F$ and $F'$) implies $q_{p, F}\in (\Z_p [\cG_{F}] \epsilon_{F, V})^{\times}$ and $q_{p, F'}\in (\Z_p [\cG_{F'}] \epsilon_{F', V})^{\times}$. From (\ref{more about the units}), we can therefore derive an equality  
\[  \pi_{F / F'} ( q_{p, F}^{-1} - r_{p, F}) = q_{p, F'}^{-1} - r_{p, F'}\]
of elements of $\Z_p [\cG_{F'}] \epsilon_{F', V}$. In particular, if we set $H_{F/F'} \coloneqq H_F \cap \Gal(F/F') \subseteq \bigcap_{\chi\in \Upsilon_F}\ker(\chi)$, then this equality  combines with (\ref{trivial action}) to imply that  
\begin{equation}
    \label{last displayed containment}
 q_{p, F'}^{-1} - r_{p, F'} \in \pi_{F / F'} \big( (\Z_p [\cG_{F}] \epsilon_{F, V})^{H_{F/ F'}} \big) = |H_{F/F'}|\cdot \Z_p [\cG_{F'}] \epsilon_{F', V}.
 \end{equation}
The key step now is the following: for each $F'\in \cX$ with $F'\subseteq k_\emptyset \langle p \rangle$, we will show that, under either of the conditions stated in (ii),  the set 
\begin{equation}\label{unbouned set} \{ \ord_p( |H_{F/F'}|) \mid F\in \cX^{\mathrm{tnc}} \text{ and } F'\subseteq F\}\end{equation} 
is unbounded. This fact combines with (\ref{last displayed containment}) to imply $q_{p, F'}^{-1} = r_{p, F'}$ and hence, by (\ref{units}), that $r_{p, F'}\cdot\varepsilon_{F' / k} =  z^\mathrm{b}_{F'}$. In this way, we shall verify (under condition (ii)) the required equality (\ref{break down}) for all $E\in \cX$ with $E\subseteq k_\emptyset \langle p \rangle$.  }
 \\
For the moment, then, we fix  $F'\in \cX$ with $F'\subseteq k_\emptyset \langle p \rangle$ and prove that the set (\ref{unbouned set}) is unbounded. To do this, our argument naturally splits according to the two  conditions listed in (ii). \\
In order to apply condition (ii)\,(a), we first fix $N \in \N$, \textcolor{black}{and set $n \coloneqq N + \ord_p ( [F' : k])$. 
We also choose a finite subset $\cT$ of $S_\fin (k)$ that contains $S_p \cup (S ( k_\emptyset \langle p \rangle)\setminus S_\infty)$ and is such that the $\cT$-class group of $k$ is trivial. }
 We then apply Proposition~\ref{ultimate neukirch result} (with $K = k_\emptyset \langle p \rangle$, $\sigma = 1$ \textcolor{black}{and $n$ replaced by $n' \coloneqq n + \ord_p ( |\mu_k|)$}) to deduce the existence of a cyclic Galois extension $L_n$ of $k$ in which all places that belong to $\cT$  are  unramified and have decomposition group of order at least \textcolor{black}{$p^{n'}$ (and hence at least $p^n$)}, and all places in $S(L_n)$ are  non-archimedean, totally split in $k_\emptyset \langle p \rangle$ and have inertia subgroup of order at least  \textcolor{black}{$p^{n'}$ (and hence at least $p^n$)}. In particular, each such field $L_n$ is a tamely ramified cyclic extension of $k$. \\
We now apply condition (ii)\,(a) with  $Z$ taken to be $L_n$. In this case, then, we know $k_\emptyset \langle p \rangle L_n \in \Omega_k^{\mathrm{tnc}}$, and hence also, \textcolor{black}{by property (P1)}, that the subfield $F'_n \coloneqq F'\cdot L_n$ of $k_\emptyset \langle p \rangle L_n$ belongs to $\Omega_k^{\mathrm{tnc}}$. In addition, one has $F'_n\in \cX$ by Hypothesis~\ref{hyp X}\,(i) and so $F'_n\in \cX^{\mathrm{tnc}}$ 
provided one has $\epsilon_{F'_n, V} \cdot (r_{p, F'_n} \varepsilon_{F'_n / k} - z^\mathrm{b}_{F'_n})= 0$. 
To verify this last equality, it suffices to prove that $e_\chi \cdot (r_{p, F'_n} \varepsilon_{F'_n / k} - z^\mathrm{b}_{F'_n}) = 0$ for every $\chi\in \cG^\circ_{F'_n}$ with $e_\chi \epsilon_{F'_n, V} e_{F'_n} \neq 0$.  Since $r_{p, F'_n} \varepsilon_{F'_n / k} - z^\mathrm{b}_{F'_n}$ is fixed by $\gal{F'_n}{F'_n \cap k_\emptyset \langle p \rangle}$, the required vanishing is clearly valid for every $\chi \in \cG_{F'_n}^\circ$ that does not factor through $\cG_{k_\emptyset \langle p \rangle}$. If, on the other hand, $\chi\in \cG_{F'_n}^\circ$ factors through $\cG_{k_\emptyset \langle p \rangle}$, then it is trivial on the decomposition group of a place in $S(L_n) \subseteq S(F'_n)$ (because all places in $S (L_n)\setminus S_\infty$ are split in $k_\emptyset \langle p \rangle$) and so $e_\chi e_{F'_n} = 0$, contradicting the assumption $e_\chi \epsilon_{F'_n, V} e_{F'_n} \neq 0$. It follows that  $F'_n$ does indeed belong to $\cX^{\mathrm{tnc}}$.\\
\textcolor{black}{
Next we note that the natural isomorphism $\gal{F_n'}{F'} \cong \gal{L_n}{L_n \cap F'}$ implies $\gal{F_n'}{F'}$ is a cyclic $p$-group.
Moreover, for each $v \in S(L_n)\setminus S_\infty$ one has $\cG_{F'_n, v} \subseteq \gal{F_n'}{F'}$  because $v$ splits completely in $F'$ and so
$|\cG_{L_n, v}|\ge p^n$ implies   $|\gal{F_n'}{F'}|\ge p^n$. 
 In particular, $\gal{F_n'}{F'}$ contains a unique cyclic subgroup $U_N$ of order $p^N$ and one has $U_N \subseteq \cG_{F'_n, v}$ for all $v \in S (L_n)\setminus S_\infty$ because $N \leq n$ and $|\cG_{F'_n, v}|\ge p^n$.\\
 If $v \in S(F')$, then condition (ii) in Proposition \ref{ultimate neukirch result} implies  $v$ is unramified in $L_n$ and $\cG_{L_n, v}$ contains an element $\sigma$ of order at least $p^n$. The restriction of $\sigma^{[F' : k]}$ to $F'$ is a $[F' : k]$-th power in $\cG_{F'}$ and therefore trivial, hence $\sigma^{[F' : k]} \in \gal{F'_n}{F'}$. Since $\sigma^{[F' : k]}$ has order at least $p^{n - \ord_p ( [F' : k])} = p^N$, this 
 shows $\gal{F'_n}{F'} \cap \cG_{L_n, v}$ has order at least $p^N$, and hence that $U_N \subseteq \gal{F'_n}{F'} \cap \cG_{L_n, v} \subseteq \cG_{L_n, v}$. 
 We have thereby proved that for every finite place $v$ in $S(F'_n) = S(F') \cup S(L_n)$ one has $U_N \subseteq \cG_{L_n, v}$, and hence that 
 $U_N \subseteq H_{F'_n} \cap \gal{F'_n}{F'} = H_{F'_n / F'}$. It follows that $\ord_p(|H_{F'_n/F'}|)$ is unbounded as $N$ varies, and hence also that (\ref{unbouned set}) is unbounded under condition (ii)\,(a).} 

\textcolor{black}{
To prove the same is true under condition (ii)\,(b), we fix $m \in \N$ so that, for the $\ZZ_p$-extension ${\widetilde k}^p_\infty$ in the stated condition, one has $S ( {\widetilde k}^p_\infty) = S ( {\widetilde k}^p_n)$ for every $n \geq m$. In particular, if we set $F'_n \coloneqq {\widetilde k}^p_n F'$ for $n\ge m$, then  $(r_{p, F'_n} \varepsilon_{F'_n / k} - z^\mathrm{b}_{F'_n})_{n \geq m}$ is a norm-coherent sequence that, by condition (i), is fixed by a finite-index subgroup of $\cG_{F' {\widetilde k}^p_\infty}$. We now apply Lemma~\ref{some-strange-lemma} with $\{ (M_{K_n}, \varphi_{K_n / K_{n'}}) \}_{n \in \N_0}$ equal to $\{ ( \fL_{F'_n}, \NN^{r_{F'}}_{F'_n / F'_{n'}}) \}_{n \ge m}$ (this is permissible by Example~\ref{inj examples}\,(c) and the fact no place in $S_\fin(k)$ splits completely in ${\widetilde k}^p_\infty$) to deduce  $r_{p, F'_n} \varepsilon_{F'_n / k} - z^\mathrm{b}_{F'_n} = 0$ for all $n \geq m$, and hence that (\ref{break down}) is valid for all $E = F'_n$ with $n \geq m$. Since condition (ii)\,(b) implies $\mathrm{TNC}_{F'_n}$ is also valid, it follows that  $F'_n \in \cX^\mathrm{tnc}$ for all $n\ge m$. 
In addition, since no place in $S_\fin(k)$ splits completely in ${\widetilde k}^p_\infty$, we can fix $t \in \N$ such that, for every $n \geq t$, every place in $S ( F'_n)\setminus S_\infty$ has full decomposition group in $\gal{F'_n}{F'_t}$.
Given these facts, an analysis very similar to that in the preceding paragraph shows that the set $\{\ord_p(|H_{F'_n/F'}|) \mid n > \max \{ m, t \}\}$ is unbounded, and hence also that the set (\ref{unbouned set}) is unbounded under condition (ii)\,(b). 
}

Having verified (under condition (ii)) that  (\ref{unbouned set}) is unbounded for all  $F'\in \cX$ with $F'\subseteq k_\emptyset \langle p \rangle$, we can now finally prove (\ref{break down}), or equivalently that, for each fixed $E\in \cX$, one has   
\begin{equation*}e_\chi\cdot( r_{p, E}\cdot\varepsilon_{E / k} -  z^\mathrm{b}_E) = 0 \,\,\,\text{for every } \chi\in\cG^\circ_E.\end{equation*}
Firstly, since $ r_{p, E}\cdot\varepsilon_{E / k} - z^\mathrm{b}_E$ is fixed by $\gal{E}{E \cap k_\emptyset \langle p \rangle}$, this equality is clear for those $\chi$ with either $\gal{E}{E \cap k_\emptyset \langle p \rangle}\not\subseteq \ker(\chi)$ or $\chi ( \epsilon_{E, V}) =0$. We shall therefore assume both that $\chi$ factors through $\cG_{k_\emptyset \langle p \rangle}$ and  also that $\chi ( \epsilon_{E, V}) \neq 0$. Setting $E_\chi\coloneqq E^{\ker(\chi)}\subseteq k_\emptyset \langle p \rangle$,
it is then sufficient to prove the vanishing of the element 
\[
 \NN^{\bm{r} (E)}_{E / E_\chi} ( 
r_{p, E} \cdot \varepsilon_{E / k} - z^\mathrm{b}_{E}
) = 
\NN^{\bm{r} (E)}_{E / E_\chi} ( r_{p, E} \cdot \varepsilon_{E / k}) - \NN^{\bm{r} (E)}_{E / E_\chi} ( z^\mathrm{b}_{E}).
\]
To do this, we use the unboundedness of (\ref{unbouned set}) with $F'= E_\chi$ to deduce that $r_{p, E_\chi} = q_{p, E_\chi}^{-1}$, 
and hence that 
\begin{equation} \label{the crucial comparison of elements}
\pi_{E / E_\chi} ( q_{p, E}) = 
    q_{p, E_\chi}^{-1} = r_{p, E_\chi} = \pi_{E / E_\chi} (r_{p, E}).
\end{equation}
Given this equality, we may then compute that
\begin{align*}
    \NN^{\bm{r} (E)}_{E / E_\chi} ( z^\mathrm{b}_{E}) & = \NN^{\bm{r} (E)}_{E / E_\chi} ( \Theta^{\bm{r} (E)}_{E, S^\ast (E)} ( \fz_E)) \\
    & = (\Theta^{\bm{r} (E)}_{E_\chi, S^\ast (E)} \circ \nabla_{E / E_\chi, S^\ast (E)}) ( \fz_{E}) \\
    & = \Theta^{\bm{r} (E)}_{E_\chi, S^\ast (E)} ( \pi_{E / E_\chi} (q_{p, E})^{-1} \cdot \nabla_{E / E_\chi, S^\ast (E)}( \epsilon_{E, V} \cL_{E, S^\ast (E)})) \\
    & = 
    \Theta^{\bm{r} (E)}_{E_\chi, S^\ast (E)} ( \pi_{E / E_\chi} (r_{p, E}) \cdot \nabla_{E / E_\chi, S^\ast (E)}( \epsilon_{E, V} \cL_{E, S^\ast (E)})) \\ 
    & = 
    \NN^{\bm{r} (E)}_{E / E_\chi} ( \Theta^{\bm{r} (E)}_{E, S^\ast (E)} ( r_{p, E} \cdot \epsilon_{E, V} \cL_{E, S^\ast (E)})) \\
    & = \NN^{\bm{r} (E)}_{E / E_\chi} (  r_{p, E} \cdot \varepsilon_{E / k}).
\end{align*}
Here the first equality is by definition of $z^\mathrm{b}_E$, the second equality is by  Lemma~\ref{commutative-diagrams-lemma}\,(a), the third equality is by definition of $q_{p, E}$, the fourth equality is by (\ref{the crucial comparison of elements}), the fifth equality is again by Lemma~\ref{commutative-diagrams-lemma}\,(a), and the final equality is by Lemma~\ref{functoriality of zeta elements}\,(b). \\
This concludes the proof of (\ref{break down}), and hence of Proposition~\ref{etnc new criterion}.
\end{proof}

\subsubsection{The proof of Theorem~\ref{reduction to Iwasawa theory}}
\label{reduction to Iwasawa theory proof section} 

In this section, we prove Theorem~\ref{reduction to Iwasawa theory} and also establish some useful facts about the ideals that occur in Theorem~\ref{reduction to Iwasawa theory}\,(a) and (b). \\
We start by proving Theorem~\ref{reduction to Iwasawa theory} and so fix data as in the statement of that result. 

It is convenient to prove claim (b) first and to do this we fix $c\in \ES^\cX_k (\Z_\mathcal{S})$. We then also fix $p \in \mathbb{P}\setminus  \mathcal{S}$, $K \in \cX_{S_p}$ and $T\in \mathscr{P}^\mathrm{ad}_K$ with respect to which $c$ satisfies the condition (\ref{doch-in-Hamburg-auf-der-Elbchausee}). 
\\
We next note that the assumption in Proposition~\ref{elbchausse-implies-eimc}\,(b) is valid by hypothesis ($\mathrm{H}_2$), and that $c_{K^p_\infty, T}$ belongs to $( \bidual^{r_K}_{\bLambda_{p, K}} U_{K_\infty^p, S(K), T}) [1 - \epsilon_K]$ by Lemma~\ref{faithful-lem}. The assumed validity of 
 (\ref{doch-in-Hamburg-auf-der-Elbchausee}) then combines with Proposition~\ref{elbchausse-implies-eimc}\,(b)\,(i) to imply that  $c_{K^p_\infty, T}$ belongs to $\im (\Theta_{K^p_\infty, S^\ast (K), T})$. In addition, Proposition~\ref{elbchausse-implies-eimc}\,(b)\,(ii) combines with assumption ($\mathrm{H}_3$) in Theorem~\ref{reduction to Iwasawa theory}, namely that the Rubin--Stark system satisfies (\ref{doch-in-Hamburg-auf-der-Elbchausee}), to imply that $\im( \Theta_{K^p_\infty, S^\ast (K), T})$ is generated over $\bLambda_{p, K}$ by $ \varepsilon_{K_\infty, T}$, and hence that there exists an element $q_{p, K}$ of $\epsilon_K \bLambda_{p, K}$ such that $c_{K^p_\infty, T} = q_{p, K} \cdot \varepsilon_{K^p_\infty, T}$.\\
Taking images under the projection map
 $\bidual^{r_K}_{\bLambda_{p, K}} U_{K^p_\infty, S(K), T} \to \bidual^{r_K}_{\Z_p [\cG_K]} U_{K, S(K), T}$, it follows that  $\delta_{T, K} \cdot c_K = \delta_{T,K} \cdot q_{p, K}\cdot \varepsilon_{K / k}$ and hence, since $\delta_{T, K}$ is a non-zero divisor in $\ZZ_p[\cG_{K}]$, that  
 \[ c_K = q_{p, K} \cdot  \varepsilon_{K / k} \in \Z_p[\cG_K]\cdot \varepsilon_{K / k}.\]
In particular, since this containment is valid for all primes $p$ outside $\mathcal{S}$, we find that condition (b)\,(iii) in Theorem~\ref{cons scarcity thm} is satisfied in this case.
The implication `(iii) $\Rightarrow$ (i)' in Theorem~\ref{cons scarcity thm} now gives, as claimed in Theorem~\ref{reduction to Iwasawa theory}\,(b), that $c$ belongs to $\Z_\cS \llbracket G_k^{\mathrm{ab}} \rrbracket \varepsilon_{k}^\cX$
if $\varepsilon_k^\cX \in \ES_k^\cX (\Z_\cS)^\mathrm{con}$ and $c \in \ES_k^\cX (\Z_\cS)^\mathrm{sym} \cap \ES_k^\cX (\Z_\cS)^\mathrm{con}$. 
 \\
Turning now to the proof of Theorem~\ref{reduction to Iwasawa theory}\,(a), we 
note that $\mathrm{TNC} ( E, \Z_\cS [\cG_E] \epsilon_{E, V})$ is valid if and only if $\mathrm{TNC} ( E, \Z_p [\cG_E] \epsilon_{E, V})$ is valid for every $p \in \mathbb{P} \setminus \cS$. We therefore now fix $p \in \mathbb{P} \setminus \cS$ and show that the assumptions of Theorem~\ref{reduction to Iwasawa theory}\,(a) imply the validity of  $\mathrm{TNC} ( E, \Z_p [\cG_E] \epsilon_{E, V})$ for every $E \in \cX$. To do this, it is enough, by Proposition~\ref{every-basic-RS-implies-etnc}, to show that if $c$ is any Euler system in the image of the map $\Theta_k$ defined in Proposition~\ref{basic-Euler-systems}, then  $\varrho^\cX (c)\in \ZZ_p \llbracket G_k^{\mathrm{ab}}\rrbracket \cdot\varepsilon_k^\cX$.
In particular, it is sufficient to verify the latter inclusion for every system $c$ in the image of $\Theta^p_k$. 
\\
As a first step in this direction we note the exact triangle in Lemma~\ref{properties-weil-etale}\,(b)
combines with Remark~\ref{T-modification Fitt remark} to imply that, for any finite subextension $F$ of $K^p_\infty / K$, one has 
$\im (\Theta^{r_K}_{F, S^\ast (K), T}) = \delta_{T, F} \cdot \im (\Theta^{r_K}_{F, S^\ast (K), \emptyset})$. In particular, the $p$-adic completion of the latter  module contains $c_{F, T}$ for every such $F$ and so, by taking the limit over such $F$, we deduce $c_{K^p_\infty, T}\in \im (\Theta_{K^p_\infty, S^\ast (K), T})$. \\
Next we recall that, as observed earlier, 
the hypotheses ($\mathrm{H}_2$) and ($\mathrm{H}_3$) combine with Proposition~\ref{elbchausse-implies-eimc}\,(b)\,(ii) 
to imply that $\im (\Theta_{K^p_\infty, S^\ast (K), T})= \bLambda_{p, K} \cdot \varepsilon_{K_\infty, T}$, and hence, by an argument similar to above, that $c_K\in \Z_p [\cG_K]\cdot \varepsilon_{K / k}$.
By Hypothesis ($\mathrm{H}_1$) and Remark~\ref{hilfslem p adic euler limit} one therefore has
\begin{equation}\label{partial step} c\in \Z_p \llbracket G_k^{\mathrm{ab}} \rrbracket \varepsilon_k^\cX + \ES_k^\cX (\Z_p)^{\gal{k^{\mathrm{ab}}}{k_\emptyset \langle p \rangle}}.\end{equation}
In the remainder of the argument we shall explain how this containment combines with the respective conditions stated in Theorem~\ref{reduction to Iwasawa theory}\,(a) to imply the containment
$c \in \Z_p \llbracket G_k^{\mathrm{ab}} \rrbracket \varepsilon_{k}^\cX$ that is required to complete the proof of the claim.\\
To start, we observe that if $k$ contains no primitive $p$-th root of unity, then Lemma~\ref{max ab subext lemma} implies $k \langle p \rangle \cap k^{\mathrm{ab}} = k (\mu_p)$.
In particular, if $p \nmid 2 d_k$, as follows from condition (a)\,(i) of Theorem~\ref{reduction to Iwasawa theory}, then every place in $S_p$ is totally ramified in  $k \langle p \rangle \cap k^{\mathrm{ab}}=k (\mu_p)$ and so $k_\emptyset \langle p \rangle = k$. Hence, in this case, (\ref{partial step}) asserts  $c\in \Z_p \llbracket G_k^{\mathrm{ab}} \rrbracket \varepsilon_{k}^\cX + \ES_k^\cX (\Z_p)^{G_k^{\mathrm{ab}}}$.
 Then, since $\varepsilon^\cX_k$ and $c$ are both symmetric (by Lemma~\ref{RS-properties-lemma}\,(i) and Proposition~\ref{basic-Euler-systems} respectively) 
and $\ES^\cX_k (\Z_p)^\mathrm{sym}\cap \ES_k^\cX (\Z_p)^{G_k^{\mathrm{ab}}} = (0)$ (as a consequence of Lemmas \ref{ivc versus local} and 
\ref{invariant systems description}), one has $c\in \Z_p \llbracket G_k^{\mathrm{ab}} \rrbracket \varepsilon_{k}^\cX$, as required. \\
Let us next assume $V \neq \emptyset$ and $\varepsilon_{k}$ is a congruence system (as in condition (a)\,(ii)). We also note that $c$ is a congruence system (by  Proposition~\ref{basic-Euler-systems}). In particular, if $p$ divides $2d_k$ (so that condition (a)\,(i) is not satisfied), then in this case we may combine the containment (\ref{partial step}) with the result of Proposition~\ref{mrs and soogils argument}   to deduce  $c\in \Z_p \llbracket G_k^{\mathrm{ab}} \rrbracket \varepsilon_{k}^\cX + \ES_k^\cX (\Z_p)^{G_k^{\mathrm{ab}}}$. Then, just as above, this containment  implies that $c\in \Z_p \llbracket G_k^{\mathrm{ab}} \rrbracket \varepsilon_{k}^\cX$, as required. \\
It thus only remains to prove the claim in the cases of conditions (a)\,(iii) and (a)\,(iv) and for this we use Proposition~\ref{etnc new criterion}. In this regard we note that the inclusion 
\[ \im (\Theta^p_k) \subseteq \Z_p \llbracket G_k^{\mathrm{ab}} \rrbracket \varepsilon_{k}^\cX + \ES_k^\cX (\Z_p)^{\gal{k^{\mathrm{ab}}}{k_\emptyset \langle p \rangle}}\]
implied by (\ref{partial step}) verifies condition (i) of Proposition~\ref{etnc new criterion} and also that the 
conditions (a)\,(iii) and (a)\,(iv) in Theorem~\ref{reduction to Iwasawa theory} coincide with the respective conditions (a) and (b) in Proposition~\ref{etnc new criterion}\,(ii). Having verified all hypotheses of the latter result, we may then directly deduce from it the result of Theorem~\ref{reduction to Iwasawa theory}\,(a) under either  of the conditions  (a)\,(iii) or (a)\,(iv).  

This concludes the proof of Theorem~\ref{reduction to Iwasawa theory}. 
\qed
\medskip \\
The next result establishes some useful facts about the ideals in Theorem~\ref{reduction to Iwasawa theory}\,(a) and (b). 

\begin{lem}\label{divisor group lemma} 
Fix a finite abelian extension $E$ of $k$, prime number $p \in \mathbb{P}\setminus \mathcal{S}$, set $T \in \mathscr{P}^\mathrm{ad}_E$ with $T \cap S(k^p_\infty)= \emptyset$, and Euler system $c\in \ES_k (\Z_\mathcal{S})$. Then the following claims are valid.  
\begin{liste}
\item $\Fitt^0_{\bLambda_{p, E}} (X^p_{E^p_\infty,S (k^p_\infty) })^{\ast \ast} \cdot \im (c_{E^p_\infty, T})^{\ast \ast} \subseteq \Fitt^0_{\bLambda_{p, E}} (X^p_{E^p_\infty, S(E^p_\infty)})^{\ast \ast}$.
\item If $|S(k^p_\infty) | = 1$, then $\im (c_{E^p_\infty, T})^{\ast \ast} \subseteq \Fitt^0_{\bLambda_{p, E}} (X^p_{E^p_\infty, S(E^p_\infty)})^{\ast \ast}$.
\item If the $\Z_p$-rank of $\gal{E^p_\infty}{E}$ is at least $2$, and every place in $S(k^p_\infty)\cap S_\fin(k)$ has open decomposition subgroup in $\cG_{k^p_\infty}$, then $\Fitt^0_{\bLambda_{p, E}} (X_{E^p_\infty, S (k^p_\infty)})^{\ast \ast} = \bLambda_{p, E}$. 
\item If \textcolor{black}{$S(E_\infty^p)\cap S_\infty=\emptyset$,} the $\Z_p$-rank of $\gal{E^p_\infty}{E}$ is $1$, and  $(\Cl^p_{E^p_\infty, S(k^p_\infty)})^{J}$ is finite for every open subgroup $J$ of $\gal{E^p_\infty}{E}$, then 
\begin{multline*} \Fitt^0_{\bLambda_{p, E}} ( \Cl^p_{E^p_\infty, S (E^p_\infty), T})^{\ast \ast} 
        \cdot \Fitt^0_{\bLambda_{p, E}} (X^p_{E^p_\infty, S (E^p_\infty)})^{\ast \ast}\\
        = \Fitt^0_{\bLambda_{p, E}} ( \Cl^p_{E^p_\infty, S (E^p_\infty), T})^{\ast \ast} 
        \cap \Fitt^0_{\bLambda_{p, E}} (X^p_{E^p_\infty, S (E^p_\infty)})^{\ast \ast}. \end{multline*}
\end{liste}
\end{lem}

\begin{proof}
As $p$ is fixed in this argument, we often suppress explicit reference to it in the notation and also write $\bLambda$ in place of $\bLambda_{p, E}$. \\
The exact sequence
\begin{cdiagram}
0 \arrow{r} & X^p_{E_\infty,S (k_\infty)} \arrow{r} & 
X^p_{E_\infty, S (E_\infty)} \arrow{r} & 
Y^p_{E_\infty, S (E_\infty) \setminus S(k_\infty)} \arrow{r} & 0
\end{cdiagram}%
implies that $\Fitt^0_\bLambda (X^p_{E_\infty,S (k_\infty)}) \cdot \Fitt^0_\bLambda (Y^p_{E_\infty, S (E_\infty) \setminus S(k_\infty)})$ is contained in $\Fitt^0_\bLambda ( X^p_{E_\infty, S (E_\infty)})$, and hence combines with Lemma~\ref{localisation lemma}\,(a) to
reduce part (a) to the claim that, for each height-one prime $\p$ of $\bLambda$, one has 
\[ \im (c_{E_\infty, T})_\p
        \subseteq   \Fitt^0_{\bLambda} (Y^p_{E_\infty, S (E_\infty) \setminus S(k_\infty)})_\p.\]
Firstly, if $\p$ is a height-one prime $\p$ of $\bLambda$ that contains $p$, then  $Y^p_{E_\infty, S (E_\infty) \setminus (S (k_\infty) \cup S_\infty)}$ vanishes when localised at $\p$ as a consequence of Lemma~\ref{localisation lemma}\,(c). From the exact sequence
\begin{cdiagram}
0 \arrow{r} & 
Y^p_{E_\infty, S_\infty \setminus V_E} \arrow{r} & 
    Y^p_{E_\infty, S (E_\infty) \setminus S (k_\infty)}
    \arrow{r} & Y^p_{E_\infty, S (E_\infty) \setminus (S (k_\infty) \cup S_\infty)} \arrow{r} & 0
\end{cdiagram}%
we therefore deduce that
\[
\Fitt^0_{\bLambda} (Y^p_{E_\infty, S (E_\infty) \setminus S(k_\infty)})_\p = \Fitt^0_{\bLambda} ( Y^p_{E_\infty, S_\infty \setminus V_E} )_\p = \bLambda_\p \epsilon_{E}. 
\]
Now, Lemma~\ref{faithful-lem}\,(a) shows that $\epsilon_{E} \cdot c_{E_\infty, T} = c_{E_\infty, T}$, and so $\im (c_{E_\infty, T})_\p$ is contained in $\bLambda_\p \epsilon_{E}$, as required to prove the claimed containment in the case of a height-one prime $\p$ with $p \in \p$.\\
To proceed, we write $\Delta$ for the (finite) torsion subgroup of $\cG_{E_\infty}$ and fix a splitting of groups
\[ \cG_{E_\infty} \cong \Delta \times \Gamma.\] 
We note that if $\p$ is a height-one prime of $\bLambda$ that does not contain $p$, then there exists a character $\chi$ in $\Delta^\circ$ and a height-one prime $\wp_\chi$ of the ring $\Lambda_\chi \coloneqq \Z_p [\im \chi] \llbracket{ \Gamma \rrbracket}$ such that $\bLambda_\p = \Lambda_{\chi,\wp_\chi}$ (cf.\@ \cite[\S\,3C1]{BKS2}). \\
Setting $E_\chi \coloneqq E^{\ker(\chi)}$ and $L \coloneqq E_\infty^{\Gamma}$, the elements  $\NN_{\gal{L}{E_\chi}}$ and $e_\chi$ are units in $\bLambda_\p$ and so 
\[
c_{E_\infty, T} \in \bLambda_\p \cdot \NN_{\gal{L}{E_\chi}} c_{E_\infty, T}  =  \Lambda_{\chi,\wp_\chi} \cdot P_{L / E_\chi, \emptyset} \cdot c_{E_{\chi, \infty}, T},
\]
where the Euler factor $P_{L / E_\chi, \emptyset}$ is as defined in Definition \ref{S euler def}\,(a).
The required inclusion is therefore true since %
\[\Lambda_{\chi,\wp_\chi}\cdot P_{L / E_\chi, \emptyset} =   e_\chi \Fitt^0_{\Lambda_\chi} (Y^p_{E_{\chi, \infty}, S(E_\infty) \setminus S(k_\infty)})_{\wp_\chi} = \Fitt^0_{\bLambda} (Y^p_{E_\infty, S (E_\infty) \setminus S(k_\infty)})_\p,\]
and this concludes the proof of claim (a).\\
By the above discussion, it is sufficient to prove, for each $\chi \in \Delta^\circ$, that the ideal $\im (c_{E_{\chi, \infty}, T})_{\wp_\chi}$ is contained in $e_\chi \Fitt^0_{\Lambda_\chi} ( X^p_{E_{\chi, \infty}, S (k_\infty)})_{\wp_\chi}$ in order to establish claim (b). 
Note that the module $e_\chi X^p_{E_{\chi, \infty}, S (k_\infty)}$ vanishes if the unique $p$-adic place $\wp \in S (k_\infty) \setminus S_\infty$ satisfies $\chi (\wp) \neq 1$ and so we may therefore assume that $\wp$ is completely split in $E_\chi / k$.

We now write $n$ for the unique integer with the property that $E_{\chi, n}$ is the decomposition field of $\wp$ in $E_{\chi, \infty} / k$. The ideal $\Fitt^0_{\Lambda_\chi} ( X^p_{E_{\chi, \infty}, S (k_\infty)})_{\wp_\chi}$ is then generated by $\gamma^{p^n} - 1$, where $\gamma$ is any choice of topological generator of $\Gamma_{E_\chi} \coloneqq \gal{E_{\chi, \infty}}{E_\chi}$. It therefore suffices to prove that $c_{E_{\chi, \infty}, T}$ is divisible by $\gamma^{p^n} - 1$ in $\bidual^{r_{E_\chi}}_{\Lambda_\chi} U_{E_{\chi, \infty}, S (E_{\chi, \infty})}$. To do this, we observe that 
        \[
        \NN_{E_{\chi, m} / E_{\chi, n}} (c_{E_{\chi, m}, T}) = (1 - \Frob^{-1}_\wp) \cdot c_{E_{\chi, n}} = 0
        \]
        for all $m > n$ since $\wp$ is assumed to split completely in $E_{\chi, n} / k$. This shows that $c_{E_{\chi, \infty}, T}$ is in the kernel of the natural codescent map
        \[
        \bidual^{r_{E_\chi}}_{\Lambda_\chi} U_{E_{\chi, \infty}, S (E_{\chi, \infty}), T} \to \bidual^{r_{E_\chi}}_{\Z_p [\cG_{E_\chi, n}]} U_{E_{\chi,n}, S(E_{\chi, \infty}), T}
        \]
        which, by the argument of \cite[Th.\@ 3.8\,(b)]{BullachDaoud}, is equal to $(\gamma^{p^n} - 1) \bidual^{r_{E_\chi}}_{\Lambda_\chi} U_{E_{\chi, \infty}, S (K_\chi, p), T}$. This completes the proof of (b). \\
To prove (c), it is enough to note that the given assumption ensures that $X^p_{E_\infty, S (k_\infty)}$ is a finitely-generated $\Z_p$-module and hence pseudo-null as a $\bLambda$-module if $\mathrm{rk}_{\Z_p} (\cG_{E_\infty}) > 1$. \\
Regarding  (d), we first observe that, by Lemma~\ref{localisation lemma}\,(a), it suffices to verify the stated equality after localisation at each height-one prime $\fp$ of $\bLambda$. Since the claim is trivial if $\fp$ does not belong to the support of $X^p_{E_\infty, S (E_\infty)}$, we now fix a height-one prime $\fp$ in said support. 

We then note that, since \textcolor{black}{(d) assumes $S(E_\infty)\subset S_\fin(k)$}, and also no place in $S_\fin(k)$ splits completely in $k_\infty / k$, the $\bLambda$-module $X^p_{E_\infty, S(E_\infty)}$ is annihilated by an element of the form $\gamma^{p^n} - 1$, where $\gamma$ denotes a choice of topological generator of $\Gamma$. From this, it follows that $\p$
must contain an element of the form $\gamma^{p^n} - 1$.

Further, since the stated assumption implies that the kernel and cokernel of multiplication by $\gamma^{p^n} - 1$ on $\Cl^p_{E_\infty, S (k_\infty)}$, and hence also on $\Cl^p_{E_\infty, S (E_\infty), T}$ (cf.\@ \cite[Lem.\@ 4.5]{BullachHofer}), are finite, it follows that $\gamma^{p^n} - 1$ acts bijectively on $(\Cl^p_{E_\infty, S (E_\infty), T})_\p$. Hence, since $\gamma^{p^n} - 1$ belongs to $\p$, the tensor product $(\Cl^p_{E_\infty, S (E_\infty), T})_\p \otimes_{\bLambda_\p} (\bLambda_\p/(\p \bLambda_\p))$ vanishes. By appealing to Nakayama's Lemma, this in turn implies that $(\Cl^p_{E_\infty, S (E_\infty), T})_\p$ vanishes. 

This argument shows that any height-one prime $\p$ of $\bLambda$ that belongs to the support of $X^p_{E_\infty, S(E_\infty)}$ cannot also belong to the support of $\Cl^p_{E_\infty, S (E_\infty), T}$. The displayed equality in claim (d) 
follows easily from this fact. 
\end{proof}

\begin{rk}\label{gross-kuzmin-remark} The modules $(\Cl^p_{E^p_\infty, S(k^p_\infty)})^J$ in Lemma~\ref{divisor group lemma} (d) are known to be finite if either $E$ is an abelian extension of $\Q$ or if a single place of $E$ ramifies in $E^p_\infty$, and also in several  other situations including cases in which $k$ is imaginary quadratic (for an overview of results, see \cite[Rem.\@ 4.4 and Rem. 4.13]{BullachHofer} and, more recently, \cite{maksoud}). 
Further, if $k^p_\infty$ is the cyclotomic $\Z_p$-extension of $k$, then the Gross--Kuz'min Conjecture predicts that this module is always finite.  
\end{rk}

\subsection{Results in rank zero and the minus part of Kato's Conjecture}\label{proof of B section}

In this section we focus on the case of CM extensions of totally real fields and, in particular, prove claim \,(a) of Theorem~\ref{intro etnc} in the Introduction. 

To do this we fix a (finite, abelian) CM extension $K$ of a totally real field $k$. We write $K^+$ for the maximal totally real subfield of $K$ and $\tau$ for the (unique) non-trivial element of $\gal{K}{K^+}$. We then write $e^-$ for the idempotent $(1-\tau)/2$ and define the \textit{minus part} of a $\Z [\cG_K]$-module $M$ by setting  
 \[ M^- \coloneqq (\Z [1/2] [\cG_K]e^-)\otimes_{\Z [\cG_K]}M .\]
We note that the assignment $M \mapsto M^-$ gives an exact functor from the category of $\Z[\cG_K]$-modules to the category of modules over the ring $\Z[\cG_K]^- = \Z [1/2] [\cG_K]e^-$. 

In this section, for each prime $p$ we always take the field $k_\infty^p$  (as fixed in \S\,\ref{somr section}) to be the cyclotomic $\Z_p$-extension of $k$. For each abelian extension $E$ of $k$ and each $n \in \N_0$ we write $E_n$ for the unique intermediate field of $E_\infty^p/E$ of degree $p^n$.  

We shall prove claim (a) of Theorem~\ref{intro etnc} by combining Theorem~\ref{reduction to Iwasawa theory}\,(a) with the following seminal result of Dasgupta and Kakde. 

\begin{thm}[Dasgupta--Kakde] \label{dasgupta-kakde}
    Fix an extension $K/k$ as above, $p \in \mathbb{P}\setminus \{2\}$ and 
     $T\in \mathscr{P}^\mathrm{ad}_K$ with $T \cap S_p= \emptyset$. Then one has 
    \begin{equation} 
    \theta_{K^p_\infty / k, S^\ast (K^p_\infty), T} (0) \in \iota_\#\bigl(\Fitt^0_{\bLambda_{K, p}^-}( {\varprojlim}_n (\Cl^p_{K^p_n, \emptyset, T})^{\vee, - })\bigr),
    \end{equation}
    where $\iota_\#$ is the $\Z_p$-linear involution of $\bLambda_{p, K}$ that sends each $\sigma \in \cG_{K^p_\infty}$ to $\sigma^{-1}$.
\end{thm}

\begin{proof} The verification of the Strong Brumer--Stark Conjecture given in \cite[Cor.\@ 3.8]{DasguptaKakde} implies, for every natural number $n$, a containment
\begin{equation} \label{strong-brumer-stark}
\theta_{K_n^p / k, S^\ast (K^p_\infty), T} (0) \in \iota_\#\bigl(\Fitt^0_{\Z_p [\cG_{K_n^p}]^-} ( (\Cl^{p}_{K_n^p, \emptyset, T})^{\vee, -})\bigr).
\end{equation}
In addition, the natural maps $\Cl^{p, -}_{K^p_n, \emptyset} \to \Cl^{p, -}_{K^p_{n + 1}, \emptyset}$ are injective (by \cite[Prop.\@ 13.26]{Washington}) and so the Pontryagin-dual maps $(\Cl^{p}_{K^p_{n + 1}, \emptyset, T})^{\vee, -} \to (\Cl^{p}_{K^p_n, \emptyset, T})^{\vee, -}$ are surjective. Given this fact, the claimed containment is obtained  by simply passing to the limit over $n$ of (\ref{strong-brumer-stark}) and then taking account of the general result concerning inverse limits of Fitting ideals proved by Greither and Kurihara in \cite[Th.\@ 2.1]{GreitherKurihara2008}. \end{proof}

A key step in the deduction of Theorem~\ref{intro etnc}\,(a) will be  provided by the following technical result. This result reinterprets the containment in Theorem~\ref{dasgupta-kakde} in a way that is better suited to our purposes and, at the same time, verifies condition (ii) of Theorem~\ref{reduction to Iwasawa theory} in the relevant case.  

\begin{proposition} \label{everything-is-true-for-minus-parts-and-CM-fields}
Fix a finite abelian CM extension $K$ of a totally real field $k$ and an odd prime $p$. Then the following claims are valid.
\begin{liste}
\item For every height-one prime $\p$ of $\bLambda_{p, K}^-$, the $(\bLambda_{p, K}^-)_\p$-module $(\Cl^{p,-} _{K^p_\infty})_\p$ has projective dimension at most one.
\item For each $T\in \mathscr{P}^\mathrm{ad}_K$ with $T \cap S_p = \emptyset$ one has 
\begin{equation} \label{claim-c-equation}
\theta_{K^p_\infty / k, S^\ast (K^p_\infty), T} (0) \in \Fitt^0_{\bLambda_{p, K}^-} ( \Cl^{p,-}_{K^p_\infty,S (K^p_\infty), T})^{\ast \ast} \cdot \Fitt^0_{\bLambda_{p, K}^-} ( X_{K^p_\infty, S (K^p_\infty)}^{p, -})^{\ast \ast}. 
\end{equation}
\end{liste}
\end{proposition}

\begin{proof} 
Regarding $p$ as clear from context, we abbreviate  $\bLambda_{p,K}, K^p_\infty$ and $k_\infty^p$ to $\bLambda_K, K_\infty$ and $k_\infty$ respectively.\\
To prove claim (a), it is enough to consider a singular prime $\p$ of $\bLambda_K$ since for every regular height-one prime $\p$ the ring $(\bLambda_{p, K}^-)_\p$ is a discrete valuation ring and the claim therefore clear. To deal with the case of singular height-one primes $\p$ we fix  $T\in \mathscr{P}^{\mathrm{ad}}_K$ with $T \cap S_p = \emptyset$.\\ 
We then note first that Lemma~\ref{localisation lemma}\,(d) implies  the natural projection maps %
\[ (\Cl^p_{K_\infty})_\p \leftarrow (\Cl^p_{K_\infty,\emptyset, T})_\p   \quad \text{and}\quad (\Cl^p_{K_\infty,\emptyset, T})_\p \to (\Cl^p_{K_\infty,S (K_\infty), T})_\p\]
are both bijective. For the same reason, one finds that the $\p$-localisation of the $\bLambda^-_{K}$-module $X_{K_\infty, S (K_\infty) \setminus S_\infty}^{p, -}$, which agrees with $X^{p,-}_{K_\infty, S^\ast (K_\infty)}$, vanishes (cf.\@ Lemma  \ref{localisation lemma}\,(c)). Upon taking minus parts of the exact sequence (\ref{exact-sequence-H1-limit}), it follows that the $\bLambda^-_{K,\p}$-module $(\Cl^{p,-}_{K_\infty,S (K_\infty), T})_\p$, and hence also $(\Cl^{p,-}_{K_\infty})_\p$, is isomorphic to $H^1 (D^\bullet_{K_\infty, S (K_\infty), T})^-_\p$. \\
Next we note that $H^1 (D^\bullet_{K_\infty, S (K_\infty), T})^-$ is a $\bLambda_K^-$-torsion module and so, by analysing the exact sequence 
\begin{cdiagram}
0 \arrow{r} & U^-_{K_\infty, S (K_\infty), T} \arrow{r} & P_{K_\infty}^- \arrow{r} & P_{K_\infty}^- \arrow{r} & H^1 (D^\bullet_{K_\infty, S (K_\infty), T})^- \arrow{r} & 0
\end{cdiagram}%
obtained by taking minus parts of the representative (\ref{complex rep}) of the complex $D^\bullet_{K_\infty, T}$, we deduce that $U^-_{K_\infty, S (K_\infty), T}$ is $\bLambda_K^-$-torsion as well. On the other hand, $U^-_{K_\infty, S (K_\infty), T}$ embeds into the $\bLambda_K^-$-torsion free module $P_{K_\infty}^-$ and so must in fact vanish. 

One therefore obtains a short exact sequence of $\bLambda_K^-$-modules
 \begin{equation} \label{free-resolution}
\begin{tikzcd}
    0 \arrow{r} & P_{K_\infty}^- \arrow{r} & P_{K_\infty}^- \arrow{r} & H^1 (D^\bullet_{K_\infty, S (K_\infty), T})^- \arrow{r} & 0.
\end{tikzcd}%
\end{equation}
This sequence directly implies that the torsion $\bLambda_{K,\p}^-$-module $(\Cl^{p,-}_{K_\infty})_\p\cong H^1 (D^\bullet_{K_\infty, S (K_\infty), T})^-_\p$ is of projective dimension at most one, as required to prove claim (a).  \\
To prepare for the proof of claim (b), 
it is convenient to first make some observations that are valid not only for the fixed field $K$ but for any finite abelian CM extension of $k$. To avoid confusion, we therefore now let $E$ denote a finite abelian CM extension of $k$ and we fix a set $T\in \mathscr{P}^\mathrm{ad}_E$ with $T \cap S_p = \emptyset$, and, as in the case of $E = K$, we abbreviate $\bLambda_{p, E}$ and $E_\infty^p$ to just $\bLambda_E$ and $E_\infty$.
We shall then
prove, for every height-one prime $\p$ of $\bLambda_E^-$, that 
\begin{equation} \label{inclusion-for-E}
\theta_{E_\infty / k, S^\ast (E_\infty), T} (0) \in \Fitt^0_{\bLambda_E^-} ( \Cl^{p, -}_{E_\infty, S (E_\infty), T})_\p
\cdot \Fitt^0_{\bLambda_E^-} (Y^{p, -}_{E_\infty, S_p})_\p.
\end{equation}
For this, we observe that there exists a natural number $n$, an $(n \times n)$-matrix $A$ with coefficients in $\bLambda_{E,\p}^-$, and an exact sequence of $\bLambda_{E,\p}$-modules of the form 
\begin{equation} \label{resolution of class group 1}
\begin{tikzcd}
0 \arrow{r} & (\bLambda_{E,\p}^-)^n \arrow{rr}{v\mapsto  A \cdot v} & &
(\bLambda_{E,\p}^-)^n \arrow{r} & (\Cl^{p, -}_{E_\infty, \emptyset, T})_\p \arrow{r} & 0.
\end{tikzcd}%
\end{equation}%
Indeed, if $\p$ is regular, then the existence of such a sequence is clear since $ (\Cl^{p, -}_{E_\infty, \emptyset, T})_\p$ is a finitely-generated torsion module over the discrete valuation domain
$\bLambda_{E,\p}$ and, if $\p$ is singular, then it follows by localising the exact sequence (\ref{free-resolution}) (with $K$ replaced by $E$) at $\p$. 

Upon taking $\bLambda_{E,\p}$-linear duals of (\ref{resolution of class group 1})
and identifying $\bLambda_{E, \p}$ with $\bLambda_{E,\p}^\ast$ via its dual basis, we obtain the exact sequence
\begin{equation}\label{resolution of class group 2}
\begin{tikzcd}
0 \arrow{r} & (\bLambda_{E,\p}^-)^n \arrow{rr}{v\mapsto  A^{\#, \mathrm{t}} \cdot v} & &
(\bLambda_{E,\p}^-)^n \arrow{r} & \alpha (\Cl^{p, -}_{E_\infty, \emptyset, T})_\p \arrow{r} & 0,
\end{tikzcd}%
\end{equation}%
in which $\alpha (- ) \coloneqq \Ext^1_{\bLambda_E} ( - , \bLambda_E)$ denotes the Iwasawa adjoint of a $\bLambda_E$-module and $A^{\#, \mathrm{t}}$ is the matrix obtained from the transpose $A^\mathrm{t}$ of $A$ by applying $\iota_\#$ to each of its entries. \\ 
By computing Fitting ideals with respect to the resolutions (\ref{resolution of class group 1}) and (\ref{resolution of class group 2}), we then find that
\begin{equation} \label{Fitting ideal comparison with adjoint}
\Fitt^0_{\bLambda_E^-} ( \Cl^{p, -}_{E_\infty, \emptyset, T})_\p = (\bLambda_{p, K} \cdot \det (A))_\p = \iota_\#\bigl( \bLambda_{p, K} \cdot \det (A^{\#, \mathrm{t}})\bigr)_\p = \iota_\#\bigl(\Fitt^0_{\bLambda_E^-} ( \alpha ( \Cl^{p, -}_{E_\infty, \emptyset, T} ))\bigr)_\p.
\end{equation}
 To proceed, we  use the composite isomorphism of $\bLambda_{E,\p}$-modules
\[ 
 \bigl(\bLambda_E\otimes_{\bLambda_E,\iota_\#}\alpha ( \Cl^{p, -}_{E_\infty, \emptyset, T} )\bigr)_\p \cong ({\varinjlim}_n \Cl^{p, -}_{E_n, \emptyset, T})^{\vee}_\p \cong ({\varprojlim}_n (\Cl^{p, -}_{E_n, \emptyset, T})^{\vee })_\p.\]
Here in the tensor product $\bLambda_E$ is regarded as a $\bLambda_E$-module via the homomorphism $\iota_\#$ (instead of the identity) so that the first isomorphism follows from \cite[Prop.\@ 15.34]{Washington} and the second isomorphism is an easy consequence of  the fact that Pontryagin duals are exact. In particular, this isomorphism combines with (\ref{Fitting ideal comparison with adjoint}) and Theorem~\ref{dasgupta-kakde} to imply that, for every height-one prime $\p$ of $\bLambda_E^-$, one has 
\begin{equation}\label{almost E inclusion}
\theta_{E_\infty / k, S^\ast (E_\infty), T} (0) \in \iota_\#\bigl(\Fitt^0_{\bLambda_E^-} ( ({\varprojlim}_n (\Cl^{p, -}_{E_n, \emptyset, T})^{\vee }))\bigr)_\p = \Fitt^0_{\bLambda_E^-} ( \Cl^{p, -}_{E_\infty, \emptyset, T})_\p. 
\end{equation}
To derive (\ref{inclusion-for-E}) from here, we note that, since $S_p\subseteq S (E_\infty)$, there exists a canonical exact sequence of $\bLambda_E$-modules 
\begin{cdiagram}
    U_{E_\infty, S (E_\infty), T} \arrow{r} & Y^p_{E_\infty, S_p} \arrow{r} & \Cl^p_{E_\infty, \emptyset, T} \arrow{r} & \Cl^p_{E_\infty, S (E_\infty), T}  \arrow{r} & 0
\end{cdiagram}
that induces, upon taking minus parts, a short exact sequence of $\bLambda^-_E$-modules 
\begin{cdiagram}
    0 \arrow{r} & Y^{p, -}_{E_\infty, S_p} \arrow{r} & \Cl^{p, -}_{E_\infty, \emptyset, T} \arrow{r} & \Cl^{p, -}_{E_\infty, S (E_\infty), T} \arrow{r} & 0.
\end{cdiagram}%
Now, since the torsion $\bLambda_{E,\p}^-$-module $(\Cl^{p, -}_{E_\infty, S (E_\infty), T})_\p$ has projective dimension one (by the above  argument), its zeroth Fitting ideal contains a non-zero divisor.  Given these observations, a general property of Fitting ideals (cf.~\cite[Prop.\@ 2.2.3]{Greither04}) combines with the above short exact sequence to imply that 
\[
\Fitt^0_{\bLambda_E^-} ( \Cl^{p, -}_{E_\infty, \emptyset, T})_\p = 
\Fitt^0_{\bLambda_E^-} ( \Cl^{p, -}_{E_\infty, S (E_\infty), T})_\p
\cdot \Fitt^0_{\bLambda_E^-} (Y^{p, -}_{E_\infty, S_p})_\p.
\]
Given this, the claimed containment (\ref{inclusion-for-E}) follows directly from (\ref{almost E inclusion}). 

Turning now to the proof of claim (b), we observe that, by Lemma~\ref{localisation lemma}\,(a), the containment (\ref{claim-c-equation}) can be verified after localisation at height-one primes of $\bLambda_K^-$. In addition, if $\p$ is a singular height-one prime, then Lemma~\ref{localisation lemma}\,(c) implies $(X^{p, -}_{K_\infty, S(K_\infty)})_\p$ vanishes, and hence that its zeroth Fitting ideal over $(\bLambda^-_K)_\p$ is equal to $(\bLambda^-_K)_\p$. In  this case, therefore, the localisation at $\p$ of the containment (\ref{claim-c-equation}) is a direct consequence of (\ref{inclusion-for-E}) with $E=K$. \\
In the remainder of this argument we may therefore assume that $\p$ is a regular height-one prime of $\bLambda_K^-$. In this case, we fix a group isomorphism $\cG_{K_\infty} \cong \Delta \times \Gamma$ where  $\Delta$ is finite and $\Gamma$ isomorphic to $\Z_p$, and  set $L \coloneqq K_\infty^\Gamma$. For each totally odd $\chi\in \Delta^\circ$, we set $L_\chi\coloneqq L^{\ker(\chi)}$ and $\bLambda_\chi\coloneqq \bLambda_{L_\chi}$. We also recall there exists 
$\chi\in \Delta^\circ$ for which $\bLambda_{K,\p}^-$ identifies with the localisation of $\bLambda_{\chi}^-$ at a suitable height-one prime ideal $\wp_\chi$  (cf.\@ \cite[\S\,3C1]{BKS2}).\\
Writing $P_{K_\infty / L_{\chi, \infty}}$ for the Euler factor $ P_{K_\infty / L_{\chi, \infty}, \emptyset}$ in $\bLambda_{\chi}$ that is defined as in Definition~\ref{S euler def}\,(a), we observe first that 
\begin{align*}
     (\bLambda_K^-)_\p \cdot \theta_{K_\infty / k, S^\ast (K_\infty), T}  & =  (\bLambda_{\chi}^-)_{\wp_\chi} \cdot 
     P_{K_\infty / L_{\chi, \infty}} \cdot 
     \theta_{L_{\chi, \infty} / k, S^\ast (L_{\chi, \infty}), T}  \\
     &\, \in  P_{K_\infty / L_{\chi, \infty}} \cdot \Fitt^0_{\bLambda_{\chi}^-} 
     ( \Cl^{p, -}_{L_{\chi, \infty}, S (L_{\chi, \infty}), T})_{\wp_\chi}
\cdot \Fitt^0_{\bLambda_{\chi}^-} (Y^{p, -}_{L_{\chi, \infty}, S_p})_{\wp_\chi} \\
 &\, =  \Fitt^0_{\bLambda_{\chi}^-} ( \Cl^{p, -}_{L_{\chi, \infty}, S (K_{\infty}), T})_{\wp_\chi} \cdot \Fitt^0_{\bLambda_{\chi}^-} ( X^{p, -}_{L_{\chi, \infty}, S (K_\infty)})_{\wp_\chi}.
\end{align*}
Here the first equality follows directly from the functorial properties of Dirichlet $L$-series and the containment from (\ref{inclusion-for-E}) with $E = L_\chi$. The second equality is valid because explicit computation (as in \cite[Lem.\@ 5.5]{Flach04}) shows that
\[ 
P_{K_\infty / L_{\chi, \infty}} \cdot \Fitt^0_{\bLambda_{\chi}^-} (Y^{p, -}_{L_{\chi, \infty}, S_p})_{\wp_\chi} = \Fitt^0_{\bLambda_{\chi}^-} (Y^{p, -}_{L_{\chi, \infty}, S (K_{\infty})})_{\wp_\chi} = \Fitt^0_{\bLambda_{\chi}^-} (X^{p, -}_{L_{\chi, \infty}, S (K_{\infty})})_{\wp_\chi}
\]
and, in addition, one has $\Fitt^0_{\bLambda_{\chi}^-} ( \Cl^{p, -}_{L_{\chi, \infty}, S (L_{\chi, \infty}), T})  \subseteq \Fitt^0_{\bLambda_{\chi}^-} ( \Cl^{p, -}_{L_{\chi, \infty}, S (K_{\infty}), T})$ since the natural map  $\Cl^p_{L_{\chi, \infty}, S (L_{\chi, \infty}), T} \to  \Cl^p_{L_{\chi, \infty}, S (K_{\infty}), T}$ is surjective.\\
 To derive the $\p$-localisation of (\ref{claim-c-equation}) from the above containment, it is then enough to note that there are equalities of ideals 
\begin{align*}
    & \phantom{\,=\,} \Fitt^0_{\bLambda_{\chi}^-} ( \Cl^{p, -}_{L_{\chi, \infty}, S (K_{\infty}), T})_{\wp_\chi}
\cdot \Fitt^0_{\bLambda_{\chi}^-} (X^{p, -}_{L_{\chi, \infty}, S (K_\infty)})_{\wp_\chi} \\
    & = \Fitt^0_{\bLambda_{\chi}^-} (H^1 (D^\bullet_{L_{\chi,\infty}, S (K_\infty), T})^-)_{\wp_\chi} \\
    & = \Fitt^0_{\bLambda_K^-} (H^1 (D^\bullet_{K_\infty, S (K_\infty), T})^-)_\p \\
    & = \Fitt^0_{\bLambda_K^-} ( \Cl^{p, -}_{K_\infty, S (K_\infty), T})_\p \cdot \Fitt^0_{\bLambda_K^-} ( X^{p, -}_{K_\infty, S (K_\infty)})_\p.
\end{align*}
Here the first and third equalities follow from the relevant cases of the exact sequence (\ref{exact-sequence-H1-limit}) (and the fact that the primes $\p$ and $\wp_\chi$ are regular). To derive the second equality we note that the result of Lemma~\ref{properties-weil-etale}\,(d) induces (upon applying $\Z_p\otimes_\Z-$ and then passing to the limit over intermediate fields $E$ of $K_\infty/k$) an isomorphism $D^\bullet_{K_\infty, S (K_\infty), T} \otimes^\mathbb{L}_{\bLambda_K} \bLambda_{\chi} \cong D^\bullet_{L_{\chi,\infty}, S (K_\infty), T}$ in $D^{\mathrm{perf}}(\bLambda_{\chi})$. This isomorphism in turn induces an isomorphism of $\bLambda_{\chi}$-modules
\[
H^1 (D^\bullet_{K_\infty, S (K_\infty), T}) \otimes_{\bLambda_K} \bLambda_{\chi}
\cong H^1 (D^\bullet_{L_{\chi,\infty}, S (K_\infty), T}),
\]
which directly implies the second equality in the above display. \\
This concludes the proof of claim (b). 
\end{proof}

By combining  Proposition~\ref{everything-is-true-for-minus-parts-and-CM-fields} with a particular case of the criterion of Theorem~\ref{reduction to Iwasawa theory}\,(a), we can now finally derive the main result of this section (which, we observe, verifies Theorem~\ref{intro etnc}\,(a) in the Introduction). 

\begin{thm}\label{th B(a) is valid} Assume $k$ is a totally real field. Then, for any finite abelian CM extension $E$ of $k$, the conjecture 
$\mathrm{TNC}(E,\mathbb{Z}[\cG_E]^-)$ is valid.
\end{thm}

\begin{proof} 
To prove this result, we aim to apply Theorem~\ref{reduction to Iwasawa theory}\,(a)\,(iii) with $V=\emptyset$, $\cX$ the set of all finite abelian CM extensions of $k$ (cf.\@ Example~\ref{cX examples}\,(b)), $\mathcal{S} = \{2\}$, and $k^p_\infty$ the cyclotomic $\Z_p$-extension of $k$, and so must verify the hypotheses of said result for this data.\\
Firstly, the Rubin--Stark system $\varepsilon_k^\emptyset$ is $\Z$-integral in the sense of Definition \ref{definition-integral-Euler-systems}. Indeed, given the explicit description of $\varepsilon_k^\emptyset$ that follows, in this setting, from Example~\ref{Rubin--Stark-examples}\,(b), this fact is a well-known consequence of work of Deligne and Ribet in \cite{DeligneRibet} (cf.\@ \cite[Prop.\@ 3.7]{Gross88}). In particular, hypothesis $(\mathrm{H}_1)$ in Theorem~\ref{reduction to Iwasawa theory} is valid.\\
Next we note that the hypotheses $(\mathrm{H}_2)$ and $(\mathrm{H}_3)$ in Theorem~\ref{reduction to Iwasawa theory} are valid by parts (a) and (b), respectively, of Proposition~\ref{everything-is-true-for-minus-parts-and-CM-fields}.\\
Next we fix $p \in \mathbb{P}\setminus \{2\}$ and write $\mathcal{Y}_p$ for the set of all fields of the form $k_\emptyset \langle p \rangle L$, where $k_\emptyset \langle p \rangle$ is as defined in (\ref{k_{p} def}) and $L$ is a totally real, tamely ramified cyclic $p$-extension of $k$. \\
Then, since $p\not= 2$ and $k$ is totally real, Lemma~\ref{max ab subext lemma}\,(b) implies $k_\emptyset \langle p \rangle\subseteq k (\mu_p)$. In particular, every field $F$ in $\mathcal{Y}_p$ is a tamely ramified abelian extension of $k$ for which one has $\epsilon_{F, \emptyset} = e^-$ and so the validity of  
$\mathrm{TNC} (F,  \Z_p [\cG_F]\epsilon_{F,\emptyset})$ is proved by Nickel in \cite[Th.~2]{nickel2021strong}.
We have thereby also verified that the criterion of Theorem~\ref{reduction to Iwasawa theory}\,(a)\,(iii) is valid. 
\\
Given this fact, the validity of $\mathrm{TNC}(E,\mathbb{Z}_p[\cG_E]\epsilon_{E,\emptyset})$ for any finite abelian CM extension $E$ of $k$ follows from an application of Theorem~\ref{reduction to Iwasawa theory}\,(a)\,(iii) with the data specified at the beginning of this proof. 
\\
Then, since $\mathrm{TNC}(E,\mathbb{Z}_p[\cG_E]\epsilon_{E,\emptyset})$ is valid for every odd prime $p$, to deduce the validity of $\mathrm{TNC}(E,\mathbb{Z}[\cG_E]^-)$ it is enough to note that $\epsilon_{E, \emptyset} = e^-$. \end{proof}

\begin{rk}
In \cite{JoNi20} Johnston and Nickel use arguments similar to those in the proof of Proposition~\ref{everything-is-true-for-minus-parts-and-CM-fields} to deduce the `equivariant Iwasawa Main Conjecture' of Ritter and Weiss from the known  validity of the Strong Brumer--Stark Conjecture.
\end{rk}

\subsection{Results in rank one} \label{rank one section}

In this section we use the classical theory of (rank-one) Euler systems to investigate the inclusion  (\ref{doch-in-Hamburg-auf-der-Elbchausee}) in cases where both  $S_\infty = S_\infty(k)$ has cardinality $1$ (so that $k$ is either $\Q$ or an imaginary quadratic field) and also $\mathcal{V} = \{ S_\infty \}$. 
In this way we shall, in particular, complete the proofs of Theorems \ref{rank one intro} and \ref{intro etnc}, as stated in the Introduction.

\subsubsection{Statement of the main results}\label{somr rank one}

For $p\in \mathbb{P}$ and $\mathbf{p} \in S_p$, we write $k_\infty^{\mathbf{p}}$ for the maximal $\Z_p$-power extension of $k$ that is unramified outside $\mathbf{p}$. Following (\ref{cX_Sigma def}), we also write $\Omega_{k,S_p}$ for the set of finite abelian extensions of $k$ in which every $p$-adic place of $k$ is ramified. 

We will then often use the following set of prime numbers. 

\begin{definition} \label{imag quad mu hyp}
We write $\cS_k^\mu$ for the set of  $p\in \mathbb{P}$ for which there exists a place $\mathbf{p} \in S_p$ and field $K \in \Omega_{k,S_p}$  such that $\gal{Kk^{\mathbf{p}}_\infty}{k}$ has an element of order $p$ and the $\Z_p \llbracket \gal{Kk^{\mathbf{p}}_\infty}{K} \rrbracket$-module $\Cl^p_{Kk^{\mathbf{p}}_\infty}$ has non-vanishing Iwasawa $\mu$-invariant.
\end{definition}

\begin{rk} \label{mu vanishing rk}
The theorem of Ferrero--Washington \cite{ferrero-washington} implies  $S^\mu_\Q = \emptyset$. If $k$ is an imaginary quadratic field, then 
an affirmative answer to Iwasawa's conjecture on the vanishing of $\mu$-invariants in cyclotomic $\Z_p$-extensions would similarly imply $S^\mu_k=\emptyset$ (cf.\@ \cite[Prop.~6.7\,(c) and Rem.~6.8]{BullachHofer}), and it is known 
$S^\mu_k$ does not contain any  $p$ that splits in $k$. The latter fact was proved by Gillard \cite{Gil85} if $p > 3$ and by Oukhaba--Vigui{\'e} \cite{OuVi16} if $p \in \{2, 3 \}$. 
\end{rk}

Then Theorem~\ref{reduction to Iwasawa theory} enables us to prove the following result towards the scarcity conjecture.
We also note that this result combines with Remark~\ref{mu vanishing rk} to establish a precise version of Theorem~\ref{rank one intro} stated in the Introduction.

\begin{thm} \label{thm-scarcity-rank-one}
Let $k$ be $\Q$ or an imaginary quadratic field. Assume $\cS_k^\mu$ is finite, and choose a finite subset $\cS$ of $\mathbb{P}$ containing both $\cS^\mu_k$ and all primes dividing $2d_k$. Then 
Conjecture~\ref{scarcity-conjecture} is valid for the pair $(\Omega^{S_\infty }_k,\cS)$.
\end{thm}

In addition, we will prove the following precise version of Theorem~\ref{intro etnc}\,(b).

\begin{thm}\label{precise intro B(ii)}
Let $k$ be an imaginary quadratic field for which $\cS_k^\mu$ is finite, and choose a finite subset $\cS$ of $\mathbb{P}$ containing $\{2 \} \cup S^\mu_k$. Then $\mathrm{TNC}(E,\mathbb{Z}_{\cS} [\cG_E])$ is valid for every finite abelian extension $E$ of $k$.
\end{thm}

The key step in proving both of these results is the verification of the relevant hypotheses in Theorem~\ref{reduction to Iwasawa theory}, and this is accomplished by  the following technical result.

\begin{prop} \label{verification of hypotheses for Iwasawa theory thm}
    Let $k$ be  $\Q$ or an imaginary quadratic field, and set $\cX \coloneqq \Omega^{S_\infty }_k$ and $\cS \coloneqq \{ 2 \} \cup \cS_k^\mu$.
    Then the following claims are valid.
    \begin{liste}
        \item The Rubin--Stark system $\varepsilon_k^\cX$ is $\Z$-integral.
        \item For every $p\in \mathbb{P}\setminus \mathcal{S}$, $\mathbf{p} \in S_p$, $K\in \cX_{S_p}$ and height-one prime $\p$ of $\bLambda_{p,K} \coloneqq \Z_p \llbracket \cG_{K k_\infty^{\mathbf{p}}} \rrbracket$, the $ (\bLambda_{p,K})_\p$-module $(\Cl^p_{K k^{\mathbf{p}}_\infty})_\p$ has finite projective dimension.
        \item For every  $p \in \mathbb{P} \setminus \cS$, $\mathbf{p} \in S_p $,  $K \in \cX_{S_p}$,  $T \in \mathscr{P}^\mathrm{ad}_K$ and $c \in \ES^\cX_k (\Z_\cS)$ one has
        \[ \im (c_{K k_\infty^{\mathbf{p}}, T})^{\ast \ast} 
        \subseteq \Fitt^0_{\bLambda_{p,K}} ( \Cl^p_{K k_\infty^{\mathbf{p}}, S (K), T})^{\ast \ast} 
        \cdot \Fitt^0_{\bLambda_{p,K}} (X^p_{K k_\infty^{\mathbf{p}}, S (K)})^{\ast \ast}.
        \]
    \end{liste}
\end{prop}

Before proving Proposition~\ref{verification of hypotheses for Iwasawa theory thm}, in \S\,\ref{verification of hyps sec}, we first use it to prove Theorems \ref{thm-scarcity-rank-one} and \ref{precise intro B(ii)}.

\begin{proof}[Proofs of Theorems \ref{thm-scarcity-rank-one} and \ref{precise intro B(ii)}]
Writing $\infty$ for the unique infinite place of $k$, we verify the hypotheses ($\mathrm{H}_1$), ($\mathrm{H}_2$) and ($\mathrm{H}_3$) of Theorem~\ref{reduction to Iwasawa theory} with respect to the data $V=\{ \infty \}$, $\cX = \Omega^{\{ \infty \}}_k$, $\mathcal{S} \subseteq \mathbb{P}$ a finite set containing $\{2 \} \cup \cS_k^\mu$, and  $k^p_\infty \coloneqq k_\infty^{\mathbf{p}}$ with $\mathbf{p} \in S_p$.\\
As for $(\mathrm{H}_1)$ and $(\mathrm{H}_2)$, this follows directly from parts (a) and (b), respectively, of Proposition~\ref{verification of hypotheses for Iwasawa theory thm}. In particular, $\varepsilon^\cX_k \in \ES_k^\cX (\Z)$ and so we can apply part (c) of Proposition~\ref{verification of hypotheses for Iwasawa theory thm}\,(c) with $c = \varepsilon_k^\cX$ to deduce that  $(\mathrm{H}_3)$ is valid. \\
Now, if $\cS$ additionally contains all prime divisors of $d_k$, as assumed in Theorem~\ref{thm-scarcity-rank-one}, then $\ES_k^\cX (\Z_\cS) = \ES_k^\cX (\Z_\cS)^\mathrm{con}$ (cf.\@ Remark~\ref{scarcity remarks}\,(c)) and so $(\mathrm{H}_1)$ implies $\varepsilon_k \in \ES^\cX_k (\Z_\cS)^\mathrm{con}$. Since  Proposition~\ref{verification of hypotheses for Iwasawa theory thm}\,(c) implies (\ref{doch-in-Hamburg-auf-der-Elbchausee}) is valid for every $c \in \ES_k^\cX (\Z_\cS)$, $p \in \mathbb{P} \setminus \cS$, $K \in \cX_{S_p}$ and $T \in \mathscr{P}^\mathrm{ad}_K$, we may therefore now apply Theorem~\ref{reduction to Iwasawa theory}\,(b) with the specified data to deduce that Theorem~\ref{thm-scarcity-rank-one} is valid. \\
Turning next to the proof of Theorem~\ref{precise intro B(ii)}, 
we fix an odd $p$ and claim that the criterion of Theorem~\ref{reduction to Iwasawa theory}\,(a)\,(iv) is valid with ${\widetilde k}_\infty^p$ taken to be the cyclotomic $\Z_p$-extension of $k$. To do this, we first observe that the field $k \langle p \rangle$ defined in (\ref{angle field}) is abelian over $\Q$ because $\cO_k^\times$ is finite. Since also ${\widetilde k}_\infty^p$ is abelian over $\Q$, the set $\cY_p$ of finite extensions of $k$ in $k_{\emptyset} \langle p \rangle \cdot {\widetilde k}_\infty^p$ 
only comprises extensions of $k$ that are abelian over $\Q$. In particular, $\mathrm{TNC}(E,\mathbb{Z}_p [\cG_E])$ is valid for any field $E$ in $\cY_p$ by the main result of Greither and the second author in \cite{BurnsGreither}. \\
Having thereby verified the validity of all hypotheses necessary for an application of Theorem~\ref{reduction to Iwasawa theory}\,(a) 
with the data specified at the beginning of this proof, said result implies the validity of $\mathrm{TNC}(E,\mathbb{Z}_p [\cG_E])$ for every $E \in \Omega_k$ and $p\in \mathbb{P}\setminus ( \{ 2 \} \cup \cS_k^\mu)$.\\
This completes the proof of Theorem~\ref{precise intro B(ii)}. 
\end{proof}

\begin{rk} As noted in the Introduction, one can also combine Theorem~\ref{thm-scarcity-rank-one}  with Theorem~\ref{scarcity-implies-etnc} to obtain a simpler proof of the known validity of $\mathrm{TNC}(K, \Z [1 / 2] [\cG_K])$ for any finite real abelian extension $K$ of $k=\Q$.
\end{rk}

\subsubsection{The proof of Proposition~\ref{verification of hypotheses for Iwasawa theory thm}}  \label{verification of hyps sec}
We first establish a useful preliminary result.

\begin{lem} \label{T-modification-Iwasawa-theory}
Fix a finite subset $\cS$ of $\mathbb{P}$, a number field $k$, $c\in \ES_k (\Z_\mathcal{S})$, $p \in \mathbb{P}\setminus \mathcal{S}$, $\mathbf{p} \in S_p$, and $K\in \Omega_k$ with $S_p\subseteq S(K)$. Set 
$K_\infty \coloneqq K k^{\mathbf{p}}_\infty$, and write $\Omega'$ for the set of finite extensions of $K$ in $K_\infty$.  Then the following claims are valid. 
\begin{liste}
    \item For  $E\in \Omega'$, write $\mathcal{C}_E$ for the $\Z_p[\cG_E]$-module generated by $\{  \delta_{T, E}c_{E}: T \in \mathscr{P}^{\mathrm{ad}}_E\} \subset U_{E, S(E)}$. Then, for any  $E'\in \Omega'$ with $E \subseteq E'$, one has $\NN^{r_E}_{E'/E}(\mathcal{C}_{E'})\subseteq \mathcal{C}_E$.
    \item Set $\mathcal{C}_{K_\infty} \coloneqq \varprojlim_{E\in \Omega'} \mathcal{C}_E$ and $\mu_{K_\infty} \coloneqq \varprojlim_{E\in \Omega'}(\Z_p\otimes_\Z\mu_{E})$, where the transition maps are induced by field-theoretic norms (and, in the first case, the result of claim (a)). Then, for any regular height-one prime $\p$ of 
 $\bLambda$ that belongs to the support of 
 $\bLambda_{p, K} c_{K_\infty, T}$ one has 
 \[
 \Fitt^0_{\bLambda_{p, K}} \big ( \mathcal{C}_{K_\infty}/(\bLambda_{p, K} c_{K_\infty, T})\big)_\p = \delta_{T, K_\infty} \cdot \Ann_{\bLambda_{p, K}} (\mu_{K_\infty})_\p^{-1},
 \]
 where $\delta_{T, K_\infty}$ is the element of $\bLambda_{p, K}$ given by the family $(\delta_{T, E})_{E \in \Omega'}$. 
\end{liste}
\end{lem}

\begin{proof} 
We abbreviate $\bLambda_{p,K}$ to $\bLambda$. For each pair of fields $E$ and $E'$ in $\Omega'$ with $E \subseteq E'$, one has $S(E') = S(E)$ and $r_{E'} = r_E$, and hence $\NN^{r_E}_{E'/E}(c_{E'}) = c_E$. Since $\mathscr{P}^{\mathrm{ad}}_{E'}\subseteq \mathscr{P}^{\mathrm{ad}}_E$, it is therefore clear that $\NN^{r_E}_{E'/E}(\mathcal{C}_{E'})\subseteq \mathcal{C}_E$, as required to prove claim (a).\\ 
To prove claim (b) we recall that, for  $E\in \Omega'$, the module $\Ann_{\Z_p [\cG_E]} (\Z_p\otimes_\Z\mu_E)$ is generated over $\Z_p$ by  $\{ \delta_{T, E} \mid T \in \mathscr{P}^\mathrm{ad}_E \}$ (cf.\@ \S\,\ref{rubin lattice section}). This fact gives rise to an exact commutative diagram of $\Z_p[\cG_E]$-modules in which all vertical maps are surjective:
\begin{cdiagram}[column sep=small]
0 \arrow{r} & \Z_p [\cG_E]  \delta_{T, E} \arrow[twoheadrightarrow]{d}{x \mapsto x c_E} \arrow{r} &
\Ann_{\Z_p [\cG_E]} (\Z_p\otimes_\Z\mu_E) \arrow{r} \arrow[twoheadrightarrow]{d}{x \mapsto x c_E} & \Ann_{\Z_p [\cG_E]} (\Z_p\otimes_\Z\mu_E)/(\Z_p [\cG_E]\delta_{T, E}) \arrow{r} \arrow[twoheadrightarrow]{d} & 0\\
    0 \arrow{r} &  \Z_p [\cG_E]  \delta_{T, E} c_E \arrow{r} & \mathcal{C}_E \arrow{r} & 
    \mathcal{C}_E/(\Z_p [\cG_E]  \delta_{T, E} c_E) \arrow{r} & 0.
\end{cdiagram}%
By passing to the limit over $E$ of these diagrams (which preserves exactness since all occurring modules are compact) and then localising at $\p$, we deduce the existence of an exact commutative diagram of  $\bLambda_\p$-modules in which all vertical maps are surjective:
\begin{equation} 
\begin{tikzcd}
    0 \arrow{r} &  \bLambda_\p\delta_{T, K_\infty} \arrow[twoheadrightarrow]{d}{\lambda\delta_T \mapsto \lambda c_{K_\infty,T}} \arrow{r} & 
    \Ann_{\bLambda} (\mu_{K_\infty})_\p \arrow[twoheadrightarrow]{d}{x \mapsto x c_{K_\infty,T}} \arrow{r} &  \Ann_{\bLambda} (\mu_{K_\infty})_\p/(\bLambda \delta_{T, K_\infty})_\p \arrow[twoheadrightarrow]{d} \arrow{r} & 0 \\
    0 \arrow{r} & \bLambda_\p c_{K_\infty, T} \arrow{r} & \mathcal{C}_{K_\infty,\p} \arrow{r} & \mathcal{C}_{K_\infty,\p}/(\bLambda_\p c_{K_\infty, T}) \arrow{r} & 0.
\end{tikzcd}%
\end{equation}%
In particular, the surjectivity of the left hand vertical map implies that, if $(\bLambda c_{K_\infty, T})_\p$ is non-zero, then $(\bLambda_\p\delta_{T, K_\infty})_\p$ and $\Ann_{\bLambda} (\mu_{K_\infty})_\p$ are also non-zero. As non-zero ideals of the discrete valuation domain $\bLambda_\p$, the latter are therefore free rank-one $\bLambda_\p$-modules. This implies that the first and second vertical arrows in the above diagram are bijective, and hence that also the third vertical arrow is bijective by an application of the Snake Lemma.
We also note that this shows that $\Ann_{\bLambda} (\mu_{K_\infty})_\p/(\bLambda \delta_{T, K_\infty})_\p$ is a cyclic $\bLambda_\p$-module so that its Fitting ideal is equal to its annihilator.
For any prime $\p$ as in the statement one therefore has  
\begin{align*}
\Fitt^0_\bLambda \big ( \mathcal{C}_{K_\infty}/(\bLambda c_{K_\infty, T}) \big)_\p & =  \Fitt^0_{\bLambda_\p} \big ( \mathcal{C}_{K_\infty,\p}/(\bLambda c_{K_\infty, T})_\p \big)\\
& = \Fitt^0_{\bLambda_\p} \big ( \Ann_{\bLambda} (\mu_{K_\infty})_\p/(\bLambda \delta_{T, K_\infty})_\p \big)\\
& =  \Ann_{\bLambda_\p} \big ( \Ann_{\bLambda} (\mu_{K_\infty})_\p/(\bLambda \delta_{T, K_\infty})_\p \big)\\
& = \delta_{T, K_\infty} \cdot \Ann_{\bLambda} (\mu_{K_\infty})_\p^{-1},
\end{align*}
as claimed.
\end{proof}

We shall now prove Proposition~\ref{verification of hypotheses for Iwasawa theory thm}.

\begin{proof}[Proof of Proposition~\ref{verification of hypotheses for Iwasawa theory thm}]
 Let $k$ be either $\Q$ or an imaginary quadratic field. Then the integrality of $\varepsilon_k^\cX$ claimed in (a) is a consequence of the explicit descriptions of $\varepsilon_k^\cX$ in terms of cyclotomic, respectively elliptic, units recalled 
 in Example~\ref{Rubin--Stark-examples}.
\\
To proceed, we set $\cS \coloneqq \{2 \} \cup \cS_k^\mu$ and fix $p \in \mathbb{P} \setminus \cS$ and, if $k$ is imaginary quadratic, also $\mathbf{p} \in S_p$ (this involves a choice if $p$ splits in $k$). As $p$ is fixed, we also occasionally suppress dependency on $p$ in the notation. We take $k_\infty$ to be $k_\infty^{\mathbf{p}}$, noting this is the cyclotomic $\Z_p$-extension if $k = \Q$. We write $\infty$ for the unique archimedean place of $k$ and 
consider the subset 
\[ \Omega^p_k \coloneqq \Omega^{\{ \infty \}}_{k,S_p }\]
of $\Omega_k$ comprising all finite abelian extensions of $k$ that are both ramified at all places in $S_p$ but unramified at $\infty$. For $K \in \Omega^p_k$, we also  set $K_\infty \coloneqq K k_\infty$ and $\bLambda_K \coloneqq \Z_p \llbracket \cG_{K_\infty} \rrbracket$.\\
Now, if $\p$ is a height-one prime of $\bLambda_K$ such that $(\Cl^p_{K_\infty})_\p$ has infinite projective dimension, then $\p$ must be singular (since $(\bLambda_K)_\p$ is a regular local ring if $\p$ is regular). In addition, Lemma~\ref{localisation lemma}\,(b) implies  $\bLambda_K$ has no singular primes if $\cG_{K_\infty}$ is $p$-torsion free. Moreover, if $\p$ is a singular prime and the $\mu$-invariant of the $\Z_p \llbracket \gal{K_\infty}{K}\rrbracket$-module $\Cl^p_{K_\infty}$  vanishes, then  
Lemma~\ref{localisation lemma}\,(b)  implies  $(\Cl^p_{K_\infty})_\p$ vanishes and so has finite projective dimension. Taken together, these observations  imply that $p\in \cS_k^\mu$ whenever there exists a height-one prime $\p$ of $\bLambda_K$ such that $(\Cl^p_{K_\infty})_\p$ has infinite projective dimension, In particular, since our choice of $\cS$ ensures that no such $p$ is contained in $\mathbb{P} \setminus \cS$, and this proves claim (b).\\
We now fix $K\in \Omega^p_k$ and set $\bLambda \coloneqq \bLambda_K$. We also note that if $p$ has trivial decomposition subgroup in $\gal{k}{\Q}$, then Remark~\ref{gross-kuzmin-remark} implies  $(\Cl^p_{K_\infty, S (k_\infty)})^{J}$ is finite for every open subgroup $J$ of $\gal{K_\infty}{K}$. In this case, therefore, claims (b) and (d) of Lemma~\ref{divisor group lemma} reduce the proof of claim (c) to showing that, for any $c \in \ES_k (\Z_\mathcal{S})$ and $T\in \mathscr{P}^\mathrm{ad}_K$, one has 
\begin{equation} \label{inclusion-to-verify-1}
\im (c_{K_\infty, T})^{\ast \ast} \subseteq \Fitt^0_{\bLambda}  (\Cl^p_{K_\infty, S(K), T})^{\ast \ast}.
\end{equation}
On the other hand, if  the decomposition subgroup of $p$ in $\gal{k}{\Q}$ is non-trivial (so that $k$ is imaginary quadratic), then the $\Z_p$-rank of $\cG_{K_\infty}$ is $2$ and so claims  (a) and (c) of Lemma~\ref{divisor group lemma} again reduce the proof of claim (c) to the verification of (\ref{inclusion-to-verify-1}). \\
Taking account of Lemma~\ref{localisation lemma}\,(a), it is thus enough to show  (\ref{inclusion-to-verify-1}) is valid after localisation at each height-one prime $\p$ of $\bLambda$. In addition, as observed in the proof of claim (b), if $\p$ is singular, then the assumption $p \not \in \cS_k^\mu$ implies $(\Cl^p_{K_\infty})_\p = (0)$ and hence also, by Lemma~\ref{localisation lemma}\,(d), that $(\Cl^p_{K_\infty, S (K), T})_\p=(0)$ and so the $\p$-localisation of (\ref{inclusion-to-verify-1}) is valid trivially. In the rest of the argument, it is therefore enough for us to fix a regular height-one prime $\p$ of $\bLambda$ that belongs to the support of $\bLambda c_{K_\infty, T}$ and verify the $\p$-localisation of (\ref{inclusion-to-verify-1}).\\
To do this, we set $\mathbb{F}_{K_\infty, T}^{\times, p} \coloneqq \varprojlim_E (\Z_p \otimes_\Z \mathbb{F}_{E, T}^\times)$ where the limit ranges  over all finite extensions $E$ of $k$ in $K_\infty$ and the transition maps are induced by the relevant norm maps. Then, by passing to the limit over $E$ of the $\ZZ_p$-linear extensions of the exact sequences in  Remark~\ref{T-modification Fitt remark}, one computes that 
\begin{equation}\label{fitt equality} \Fitt^0_{\bLambda} (\mathbb{F}_{K_\infty, T}^{\times, p} ) = \bLambda\cdot\delta_{T, K_\infty},\end{equation}
with $\delta_{T, K_\infty} \coloneqq (\delta_{T, E})_E\in \bLambda$. 
In a similar way, by passing to the limit over $E$ of the $\Z_p$-linear extension of the long exact cohomology sequence of the exact triangle in Lemma~\ref{properties-weil-etale}\,(b), one obtains an exact sequence of $\bLambda$-modules %
\begin{equation} \label{an 5-term exact sequence}
\begin{tikzcd}[column sep=small]
    0 \arrow{r} & U_{K_\infty, S(K), T} \arrow{r} & U_{K_\infty, S(K)} \arrow{r} & \mathbb{F}_{K_\infty, T}^{\times, p} \arrow{r} & 
    \Cl^p_{K_\infty, S(K), T} \arrow{r} & \Cl^p_{K_\infty, S(K)} \arrow{r} & 0.
\end{tikzcd}%
\end{equation}
Now, since $\bLambda_\p$ is a discrete valuation ring, the $\bLambda_\p$-module $(U_{K_\infty, S(K), T})_\p$ is free  and a standard Euler characteristic calculation  implies its $\bLambda_\p$-rank is $1$ (cf.~\cite[Th.~3.8\,(a)]{BullachDaoud}). The $\bLambda_\p$-module $(U_{K_\infty, S(K), T}/(\bLambda c_{K_\infty, T}))_\p$ is therefore torsion and also such that 
\begin{equation} \label{image Fitt lem replacement}
\Fitt^0_{\bLambda}( U_{K_\infty, S(K), T}/(\bLambda c_{K_\infty, T}))_\p = \im (c_{K_\infty, T})_\p
\end{equation}
since any $\bLambda_\p$-basis $f_0$ of $(U_{K_\infty, S(K), T}^\ast)_\p$ induces an isomorphism
\begin{equation*} 
\big( U_{K_\infty, S(K), T}/(\bLambda c_{K_\infty, T}) \big)_\p \xrightarrow{\simeq} \bLambda_\p / f_0 (  c_{K_\infty, T}) \bLambda_\p = \big( \bLambda / \im (c_{K_\infty, T}) \big)_\p. 
\end{equation*}

Given this, the equality (\ref{fitt equality}) can be combined with the sequence (\ref{an 5-term exact sequence}) to reduce the proof of the $\p$-localisation of   (\ref{inclusion-to-verify-1}) to verifying the inclusion  
\[
\Fitt^0_{\bLambda} \bigl( U_{K_\infty, S(K)}/(\bLambda c_{K_\infty, T}) \bigr)_\p \subseteq 
\delta_{T, K_\infty} \cdot \Fitt^0_{\bLambda} ( \Cl^p_{K_\infty, S(K)})_\p.
\]
To do this, we note first that $(\mu_{K_\infty})_\p = (0)$. Indeed, $\mu_{K_\infty}$ is finite if either $k$ is $\Q$ (since $K_\infty$ is then totally real) or if $k$ is imaginary quadratic and $p$ is split in $k$ (since $k_\infty / k$ is only ramified at $\mathbf{p}$), whilst if $k$ is imaginary quadratic and $p$ does not split in $k$, then the $\Z_p$-rank of $\mu_{K_\infty}$ is $1$ whilst the $\Z_p$-rank of $\gal{K_\infty}{k}$ is $2$ (hence $\mu_{K_\infty}$ is pseudo-null in all cases). In particular, since $\mu_{K_\infty}$ is isomorphic to $\bLambda/\Ann_{\bLambda} (\mu_{K_\infty})$, one has $\Ann_{\bLambda} (\mu_{K_\infty})_\p = \bLambda_\p$. Given this, Lemma~\ref{T-modification-Iwasawa-theory}\,(b) implies that the required inclusion is valid provided that the $\bLambda$-module $\mathcal{C}_{K_\infty}$ that occurs in the latter result satisfies   
\[
\Fitt^0_{\bLambda} \bigl(U_{K_\infty, S(K)}/\mathcal{C}_{K_\infty} \bigr)_\p \subseteq  \Fitt^0_{\bLambda} ( \Cl^p_{K_\infty, S(K)})_\p.
\]
To verify this, we note that in the same way that (\ref{image Fitt lem replacement}) was derived, one can also show that 
\begin{equation}\label{next step}\Fitt_{\bLambda}^0\bigl(U_{K_\infty, S(K)}/\mathcal{C}_{K_\infty}\bigr)_\p = \bLambda_{K,\p}\cdot \{ f (x) \mid x \in \mathcal{C}_{K_\infty},\, f \in U_{K_\infty, S(K)}^\ast \}.\end{equation}
We are thereby reduced to proving that for any $x \in \mathcal{C}_{K_\infty}$ one has $f (x)\in  \Fitt^0_{\bLambda} ( \Cl^p_{K_\infty, S(K)})_\p$ for every  $f \in U_{K_\infty, S(K)}^\ast$. To show this, we recall that the argument of Lemma~\ref{T-modification-Iwasawa-theory}\,(b) implies every such $x$ can be written as
$\lambda\cdot c_{K_\infty}$ for some $\lambda= \lambda_x\in \Ann_{\bLambda} (\mu_{K_\infty})$. 
After taking account of Remark~\ref{surj of annihilator}, we can then also fix a pre-image $\lambda'$ of $\lambda$ under the projection map $\varprojlim_{E \in \Omega_k} \Ann_{\Z_p [\cG_E]}(\Z_p\otimes_\Z\mu_E) \to \Ann_{\bLambda} (\mu_{K_\infty})$.
For each $E\in \Omega^{\{ \infty \}}_k$, we set 
\[
c'_E \coloneqq \lambda'_{E} \cdot \NN_{EK / E} (c_{EK}).
\]
Then the family 
$c' \coloneqq (c'_E)_E$ can be seen to 
satisfy the definition \cite[Def.~2.1.1]{Rubin-euler} of an Euler system 
for the representation $\Z_p (1)$ (in the notation of loc.\@ cit.\@ we are taking the extension $\cK$ to be the maximal abelian extension of $k$ in which $\infty$ splits completely and the ideal $\cN$ to be $\prod_{v \in S(K)} v$), and is also such that  
$c'_{K_\infty} = x$. In view of (\ref{next step}), the required inclusion will thus follow if,  for every Euler system $c$ for $\Z_p (1)$ (with $\cN = \prod_{v \in S(K)} v$) and every $f\in U_{K_\infty, S(K)}^\ast $, one has 
\begin{equation}\label{last one}
 f (c_{K_\infty}) \in \Fitt^0_{\bLambda} ( \Cl^p_{K_\infty,S(K)})_\p .
\end{equation}
We shall now derive this containment from results of Rubin. To do this, we write $\Delta$ for the torsion subgroup $\gal{K_\infty}{k_\infty}$ of $\gal{K_\infty}{k}$ and note there exists $\chi\in \Delta^\circ$ and a height-one prime $\wp_\chi$ of the ring $\Lambda_\chi \coloneqq \bLambda \otimes_{\Z_p [\Delta]} \Z_p [\im \chi]$ such that $\bLambda_\p = (\Lambda_\chi)_{\wp_\chi}$  (cf.\@ \cite[\S\,3C1]{BKS2}). We also recall that, by the argument of Rubin in \cite[Prop.\@ 6.2.1 and Cor.\@ 6.2.2]{Rubin-euler}, for the dual canonical Selmer structure $\mathcal{F}^\ast_\mathrm{can}$ one has a natural isomorphism of dual Selmer modules
\[ \bigl(H^1_{\mathcal{F}^\ast_\mathrm{can}} ( \cO_{k, S (K)}, \bLambda^\vee (1))^\vee\bigr) \otimes_{\bLambda}\Lambda_\chi \cong H^1_{\mathcal{F}^\ast_\mathrm{can}} ( \cO_{k, S (K)}, \Lambda_\chi^\vee (1)(\chi))^\vee \cong  H^1_{\mathcal{F}^\ast_\mathrm{can}} ( \cO_{k, S (K)},(\cT_{\chi, k_\infty})^\vee)^\vee,\]
where $\cT_{\chi, k_\infty}$ is the induction from $G_{k_\infty}$ to $G_k$ of the $G_k$-representation $\cT_\chi \coloneqq \Z_p [\im \chi] (1) (\chi^{-1})$. Taking account of the canonical  isomorphism $\Cl^p_{K_\infty,S(K)} \cong H^1_{\mathcal{F}^\ast_\mathrm{can}} ( \cO_{k, S (K)}, \bLambda^\vee (1))^\vee$ induced by \cite[Prop.\@ 1.6.1]{Rubin-euler}, one therefore obtains an  isomorphism of $\bLambda_{\p}$-modules  $(\Cl^p_{K_\infty,S(K)})_\p \cong (H^1_{\mathcal{F}^\ast_\mathrm{can}} ( \cO_{k, S (K)},(\cT_{\chi, k_\infty})^\vee)^\vee)_{\wp_\chi}$ and hence an equality of ideals 
\begin{align*}
\Fitt^0_{\bLambda} ( \Cl^p_{K_\infty,S(K)})_\p & = \Fitt^0_{\Lambda_\chi} ( H^1_{\mathcal{F}^\ast_\mathrm{can}} ( \cO_{k, S (K)}, (\cT_{\chi, k_\infty})^\vee)^\vee)_{\wp_\chi}\\
& = \cchar_{\Lambda_\chi} ( H^1_{\mathcal{F}^\ast_\mathrm{can}} ( \cO_{k, S (K)}, (\cT_{\chi, k_\infty})^\vee)^\vee)_{\wp_\chi}.
\end{align*}
Here the second equality follows from the general fact that the structure theorem for finitely-generated modules over a discrete valuation ring $R$ with maximal ideal $\mathfrak{m}$ implies  $\Fitt^0_{R} (M) = \mathfrak{m}^{^{\mathrm{length}_{R} (M)}}$ for any finitely-generated torsion $R$-module $M$.\\
Via the last displayed equality, the proof of (\ref{last one}) is reduced to showing that    
\[
\im (c^\chi_{k_\infty}) \subseteq \cchar_{\Lambda_\chi} ( H^1_{\mathcal{F}^\ast_\mathrm{can}} ( \cO_{k, S (K)}, (\cT_{\chi, k_\infty})^\vee)^\vee),
\]
where $c^\chi$ is the Euler system for $\cT_\chi$ obtained from $c$ by twisting by $\chi^{-1}$ (for details see \cite[Prop.\@ 2.4.2]{Rubin-euler}). It is therefore  now enough to note that, after taking account of the observation in \cite[Rem.\@ 2.1.3]{Rubin-euler},  this inclusion is a direct consequence of \cite[Th.~2.3.3]{Rubin-euler}.\\
This completes the proof of Proposition~\ref{verification of hypotheses for Iwasawa theory thm}.
\end{proof}

\renewcommand{\emph}[1]{\textit{#1}}

\addcontentsline{toc}{section}{References}
\tiny
\printbibliography

\small

(Bullach)
\textsc{Instituto de Ciencias Matem\'aticas, c/ Nicol\'as Cabrera 13-15, Campus de Cantoblanco UAM, 28049 Madrid, Spain.}
\textit{Email address:} \href{mailto:dominik.bullach@icmat.es}{dominik.bullach@icmat.es} 
\medskip \\ 
(Burns, Daoud) \textsc{King's College London,
Department of Mathematics,
London WC2R 2LS,
UK.} 
\textit{Email addresses:}
\href{mailto:david.burns@kcl.ac.uk}{david.burns@kcl.ac.uk},
\href{mailto:alexandre.daoud@kcl.ac.uk}{alexandre.daoud@kcl.ac.uk} \medskip \\
(Seo) \textsc{Yonsei University,
Department of Mathematics,
Seoul,
Korea.}\\
\textit{Email address:} \href{mailto:sgseo@yonsei.ac.kr}{sgseo@yonsei.ac.kr}

\end{document}